\documentclass[11pt,reqno]{amsart}

\usepackage[latin1]{inputenc}
\usepackage{amsmath,amsthm,amssymb,amscd,mathrsfs,amsbsy}
\usepackage{mathtools,color,esint,bm,float}
\usepackage{booktabs}
\usepackage[perpage]{footmisc}
\usepackage{subfig,caption}
\usepackage[shortlabels]{enumitem}

\definecolor{darkgreen}{rgb}{0.0, 0.6, 0.13}

\usepackage{hyperref}

\newtheorem{thm}{Theorem}[section]
 \newtheorem{cor}[thm]{Corollary}
 \newtheorem{lem}[thm]{Lemma}
 \newtheorem{prop}[thm]{Proposition}

 \theoremstyle{definition}
 \newtheorem{df}[thm]{Definition}
 \theoremstyle{remark}
 \newtheorem{rem}[thm]{Remark}
 
 \numberwithin{equation}{section}

 \newcommand{\les}{\lesssim}

\newcommand{\Ab}{\mathbb A}
\newcommand{\Bb}{\mathbb B}
\newcommand{\Cb}{\mathbb C}
\newcommand{\Db}{\mathbb D}
\newcommand{\Eb}{\mathbb E}

\newcommand{\Hb}{\mathbb H}

\newcommand{\Mb}{\mathbb M}
\newcommand{\Nb}{\mathbb N}

\newcommand{\Rb}{\mathbb R}

\newcommand{\Tb}{\mathbb T}
\newcommand{\Ub}{\mathbb U}
\newcommand{\Vb}{\mathbb V}

\newcommand{\Zb}{\mathbb Z}
\newcommand{\Ac}{\mathcal A}
\newcommand{\Bc}{\mathcal B}
\newcommand{\Cc}{\mathcal C}
\newcommand{\Dc}{\mathcal D}
\newcommand{\Ec}{\mathcal E}
\newcommand{\Fc}{\mathcal F}
\newcommand{\Gc}{\mathcal{G}}

\newcommand{\Ic}{\mathcal I}
\newcommand {\Jc}{\mathcal{J}}
\newcommand{\Kc}{\mathcal K}
\newcommand{\Lc}{\mathcal L}
\renewcommand{\Mc}{\mathcal M}
\newcommand{\Nc}{\mathcal N}
\newcommand{\Oc}{\mathcal O}

\newcommand{\Qc}{\mathcal Q}
\newcommand{\Rc}{\mathcal R}

\newcommand{\Tc}{\mathcal T}
\newcommand{\Uc}{\mathcal U}
\newcommand{\Vc}{\mathcal V}
\newcommand{\Wc}{\mathcal W}
\newcommand{\Xc}{\mathcal X}
\newcommand{\Yc}{\mathcal Y}
\newcommand{\Zc}{\mathcal Z}
\newcommand{\As}{\mathscr A}
\newcommand{\Bs}{\mathscr B}
\newcommand{\Cs}{\mathscr C}
\newcommand{\Ds}{\mathscr D}
\newcommand{\Es}{\mathscr E}

\newcommand{\Ks}{\mathscr K}
\newcommand{\Ls}{\mathscr L}
\newcommand{\Ms}{\mathscr M}

\newcommand{\Ps}{\mathscr P}
\newcommand{\Qs}{\mathscr Q}
\newcommand{\Rs}{\mathscr R}

\newcommand{\Ts}{\mathscr T}
\newcommand{\Us}{\mathscr U}
\newcommand{\Vs}{\mathscr V}
\newcommand{\Ws}{\mathscr W}
\newcommand{\Xs}{\mathscr X}
\newcommand{\Ys}{\mathscr Y}

\newcommand{\Af}{\mathfrak A}

\newcommand{\Cf}{\mathfrak C}
\newcommand{\Df}{\mathfrak D}

\newcommand{\Ff}{\mathfrak F}

\newcommand{\Pf}{\mathfrak P}

\newcommand{\Sf}{\mathfrak S}

\newcommand{\Xf}{\mathfrak X}

\newcommand{\ff}{\mathfrak f}

\newcommand{\lf}{\mathfrak l}
\newcommand{\mf}{\mathfrak m}
\newcommand{\nf}{\mathfrak n}

\newcommand{\rf}{\mathfrak r}

\newcommand{\uf}{\mathfrak u}

\newcommand{\wf}{\mathfrak w}

\newcommand{\dirac}{\boldsymbol{\delta}}

\begin{document}
\title{Derivation of the wave kinetic equation: Full range of scaling laws}
\author{Yu Deng and Zaher Hani}
\maketitle

\begin{abstract}
This paper completes the program started in \cite{DH21, DH21-2} aiming at providing a full rigorous justification of the wave kinetic theory for the nonlinear Schr\"odinger (NLS) equation. Here, we cover the full range of scaling laws for the NLS on an arbitrary periodic rectangular box, and derive the wave kinetic equation up to small multiples of the kinetic time.

The proof is based on a diagrammatic expansion and a deep analysis of the resulting Feynman diagrams. The main novelties of this work are three-fold: (1) we present a robust way to identify arbitrarily large ``bad" diagrams which obstruct the convergence of the Feynman diagram expansion, (2) we systematically uncover intricate cancellations among these large ``bad" diagrams, and (3) we present a new robust algorithm to bound all remaining diagrams and prove convergence of the expansion. These ingredients are highly robust, and constitute a powerful new approach in the general mathematical study of Feynman diagrams.

\end{abstract}
\tableofcontents
\section{Introduction}
\subsection{Setup and the main result}\label{intro-nls} In this paper we derive the wave kinetic equation, from the continuum cubic nonlinear Schr\"{o}dinger equation and at the kinetic time scale, for the \emph{full range of scaling laws} between the large box and weak nonlinearity limits. This completes the program initiated in \cite{DH19,DH21}, aiming at providing rigorous mathematical foundation for the wave turbulence theory.
\smallskip

In dimension $d\geq 3$, consider the cubic nonlinear Schr\"{o}dinger equation
\begin{equation}\label{nls}\tag{NLS}
\left\{
\begin{split}&(i\partial_t-\Delta)u+\alpha|u|^2u=0,\quad x\in \Tb_L^d=[0,L]^d,\\
&u(0,x)=u_{\mathrm{in}}(x)
\end{split}
\right.
\end{equation} 
on the square torus $\Tb_L^d=[0,L]^d$ of size $L$ (all results and proofs extend without change to arbitrary rectangular tori). Here $\alpha$ is a parameter indicating the strength of the nonlinearity, and $\Delta:=\frac{1}{2\pi}(\partial_{x_1}^2+\cdots +\partial_{x_d}^2)$ is the normalized Laplacian. We also set the space Fourier transform as
\begin{equation}\label{fourier}
\widehat u(t, k) =\frac{1}{L^{d/2}}\int_{\Tb^d_L} u(t, x) e^{-2\pi i k\cdot x} \, dx, \qquad u(t,x) =\frac{1}{L^{d/2}}\sum_{k\in\Zb_L^d}\widehat{u}(k)e^{2\pi ik\cdot x},
\end{equation} where $\Zb_L^d:=(L^{-1}\Zb)^d$. Note that this convention is different from (but equivalent to) the one in \cite{DH21,DH21-2}; the parameters $\lambda$ in \cite{DH21,DH21-2} and $\alpha$ in the current paper are related by $\alpha=\lambda^2L^{-d}$.

Assume the initial data of (\ref{nls}) is given by
\begin{equation}
\label{data}\tag{DAT}u_{\mathrm{in}}(x)=\frac{1}{L^{d/2}}\sum_{k\in\Zb_L^d}\widehat{u_{\mathrm{in}}}(k)e^{2\pi ik\cdot x},\quad \widehat{u_{\mathrm{in}}}(k)=\sqrt{n_{\mathrm{in}}(k)}g_k(\omega),
\end{equation}
where $n_{\mathrm{in}}:\Rb^d\to[0,\infty)$ is a given Schwartz function, and $\{g_k(\omega)\}$ is a collection of i.i.d. random variables. For concreteness, we will assume each $g_k$ is a standard normalized Gaussian.

Define the \emph{kinetic (or Van Hove) time}
\[T_{\mathrm{kin}}:=\frac{1}{2\alpha^2}.\] For a fixed value $\gamma\in(0,1)$, we will assume that the \emph{scaling law} between $L$ and $\alpha$ is $\alpha=L^{-\gamma}$, so we have $T_{\mathrm{kin}}=\frac{1}{2} L^{2\gamma}$.
\subsubsection{The wave kinetic equation} The wave kinetic equation is given by:

\begin{equation}\label{wke}\tag{WKE}
\left\{
\begin{split}&\partial_t n(t,k)=\Kc(n(t),n(t),n(t))(k),\\
&n(0,k)=n_{\mathrm{in}}(k),
\end{split}
\right.
\end{equation} where $n_{\mathrm{in}}$ is as in Section \ref{intro-nls}, and the nonlinearity $\Kc$ is given by
\begin{multline}\label{wke2}\tag{COL}
\Kc(\phi_1,\phi_2,\phi_3)(k)=\int_{(\Rb^d)^3}\big\{\phi_1(k_1)\phi_2(k_2)\phi_3(k_3)-\phi_1(k)\phi_2(k_2)\phi_3(k_3)+\phi_1(k_1)\phi_2(k)\phi_3(k_3)\\-\phi_1(k_1)\phi_2(k_2)\phi_3(k)\big\}\times\dirac(k_1-k_2+k_3-k)\cdot\dirac(|k_1|^2-|k_2|^2+|k_3|^2-|k|^2)\,\mathrm{d}k_1\mathrm{d}k_2\mathrm{d}k_3.
\end{multline} Here and below $\dirac$ denotes the Dirac delta, and we define
\[|k|^2:=\langle k,k\rangle,\quad \langle k,\ell\rangle:=k^1 \ell^1+\cdots +k^d\ell^d,\] where $k=(k^1,\cdots,k^d)$ and $\ell=(\ell^1,\cdots,\ell^d)$ are $\Zb_L^d$ or $\Rb^d$ vectors. 

Given any Schwartz initial data $n_{\mathrm{in}}(k)$, the equation \eqref{wke} has a unique local solution $n=n(t,k)$ on some short time interval depending on $n_{\mathrm{in}}$.
\subsubsection{The main result} The main result is stated as follows.
\begin{thm}\label{main} Fix $d\geq 3$ and $\gamma\in(0,1)$. Fix a Schwartz function $n_{\mathrm{in}}\geq 0$, and fix $\delta\ll 1$ depending only on $(d,\gamma,n_{\mathrm{in}})$. Consider the equation (\ref{nls}) with random initial data (\ref{data}), and assume $\alpha=L^{-\gamma}$ so that $T_{\mathrm{kin}}=\frac{1}{2} L^{2\gamma}$.

Then, for sufficiently large $L$ (depending on $\delta$), the equation has a smooth solution up to time \[T=\delta\cdot\frac{L^{2\gamma}}{2}=\delta\cdot T_{\mathrm{kin}},\] with probability $\geq 1-e^{-(\log L)^2}$. Moreover we have
\begin{equation}\label{limit}\lim_{L\to\infty}\sup_{t\in[0,T]}\sup_{k\in\Zb_L^d}\left|\Eb\,|\widehat{u}(t,k)|^2-n\bigg(\frac{t}{T_{\mathrm{kin}}},k\bigg)\right|=0,
\end{equation} where $\widehat{u}$ is as in \eqref{fourier}, and $n(\tau,k)$ is the solution to (\ref{wke}).  In (\ref{limit}) and below we understand that the expectation $\Eb$ is taken under the assumption that (\ref{nls}) has a smooth solution on $[0,T]$, which is an event with overwhelming probability.
\end{thm}
A few comments about this result are in order.
\begin{itemize}
\item The nonlinear Schr\"{o}dinger equation (\ref{nls}) is studied here and in \cite{DH19,DH21,DH21-2}, as a representative model in nonlinear wave theory. In fact, it is the \emph{universal} Hamiltonian nonlinear dispersive equation, in the sense that any such equation gives (NLS) in a suitable limiting regime \cite{Sul99}. The methods we develop here also apply to other dispersive models modulo technical differences.
\item Theorem \ref{main} holds for the rectangular torus $\Tb_{L,\boldsymbol{\lambda}}^d=[0,\lambda_1L]\times\cdots [0,\lambda_dL]$ for \emph{any} $\lambda_j>0$ (rational or irrational) without genericity assumption. Here we only present the proof for the square torus, but the general case can be treated by the same arguments.
\item In the same way as \cite{DH21}, {the assumption that $n_{\mathrm{in}}(k)$ is Schwartz is unnecessary. In fact it suffices to assume that its first $40d$ derivatives decay like $\langle k\rangle^{-40d}$. Moreover} the error term defined in (\ref{limit}) enjoys the explicit decay rate $L^{-c}$, which is uniform in $t$ and $k$, for some absolute constant $c>0$. The value of $c$ we get, though, is likely non-optimal.
\item The exceptional probability $e^{-(\log L)^2}$ in Theorem \ref{main} is better than \cite{DH21}, but this is just due to the choice of the order $N$ of the expansion (see Section \ref{norms}). In fact the same bound also holds in the setting of \cite{DH21}, as is already demonstrated in \cite{DH21-2}.
\item The main results in \cite{DH21-2} (evolution of higher order moments, propagation of chaos, law evolution for non-Gaussian data, derivation of wave kinetic hierarchy) also extend to the current setting. In particular, we can replace the i.i.d. Gaussians $g_k(\omega)$ by any centered-normalized i.i.d. random variables $\eta_k(\omega)$ whose law is rotationally symmetric and has exponential tails. This is easily shown by combining the arguments in this paper and in \cite{DH21-2} with obvious modifications.
\end{itemize}
\subsection{Background and literature}\label{intro-back} The theory of wave turbulence describes the non-equilibrium statistical behavior of systems of interacting waves, in the thermodynamic limit where the number of degrees of freedom goes to infinity. It is the wave analog of the classical kinetic theory of Boltzmann for particles, and its rigorous justification corresponds to the Hilbert's sixth problem for nonlinear waves.

The basic setup of the theory is as follows. Start with a nonlinear dispersive equation as the microscopic system of nonlinear waves. This is (\ref{nls}) in our case, but can also be replaced by other equations. Such system is studied in a \emph{large box} $\Tb_L^d$ with a \emph{weak nonlinearity} $\alpha|u|^2u$, where $L\to\infty$ (so the number of degrees of freedom diverges as $\sim L^d$) and $\alpha\to 0$ in the limit. Assume the initial data is random and \emph{well-prepared} as in (\ref{data}), i.e. the different Fourier modes $\widehat{u}(k)$ are independent and satisfy a random phase (RP) condition. Then, among other things, the following kinetic description is expected \emph{at the kinetic time} $T_{\mathrm{kin}}$:
\begin{itemize}
\item Propagation of chaos: different Fourier modes should remain independent in the limit;
\item The wave kinetic equation: the evolution of energy density $|\widehat{u}(k)|^2$ should be governed by (\ref{wke}) in the limit.
\end{itemize}

In the physics literature, the very first kinetic description for waves appeared in Peierls \cite{Pei29} in the study of anharmonic crystals, leading to the so-called phonon Boltzmann equation. Since then, the kinetic theory has been developed for various models, and has become a systematic paradigm starting in the 1960s, with immense applications in various fields of physics and science \cite{BS66,BN69,Dav72,Has62,Has63,Jan08,Naz11,Spo06,Spo08,Ved67,WMO98,ZS67}. The name \emph{wave turbulence theory} comes from the spectral energy dynamics and cascades that the wave kinetic equation predicts for nonlinear wave systems, which yields similar conclusions to Kolmogorov spectra in hydrodynamic turbulence; this connection was is a major contribution of Zakharov \cite{Zak65,ZLF92}.

On the other hand, the rigorous mathematical treatment of wave turbulence had to wait until much later for the appropriate conceptual and technical ingredients to be invented. While it was clear in the theoretical physics community that a Feynman diagram expansion is the right approach to the problem, the main mathematical issue here was to prove the convergence of such an expansion. Naturally, the progress started in a linear setting (e.g. electron moving through random impurities), namely with the work of Spohn \cite{Spohn1977} for short kinetic times. This was later extended to much longer times in the celebrated works of Erd\"{o}s-Yau \cite{EY00} and Erd\"{o}s-Salmhofer-Yau \cite{ESY08}. Obviously, the next level of progress is to advance this understanding to the nonlinear setting, where the randomness is only coming from the initial distribution of the data as explained in \cite{Spohn1994}. The first breakthrough proving the convergence of the diagrammatic expansion in a nonlinear setting was that of Lukkarinen-Spohn \cite{LS11}, which considered the lattice (NLS) and studied the time correlations of the invariant Gibbs measure in the thermodynamic limit. Even though the above works only dealt with linear or equilibrium settings, they managed to draw substantial interest to this field from the mathematical community, and inspired subsequent research. In the last decade, partial results have been proved regarding the derivation of (\ref{wke}) in the nonlinear out-of-equilibrium setting, starting with works that addressed certain aspects of the problem (second-order expansions, near-equilibrium dynamics, shorter time scales etc.), see \cite{BGHS19,CG19,CG20,DH19,DK19,DK19-2,Fa18} and references therein. In particular, the authors' earlier work \cite{DH19}, as well as Collot-Germain \cite{CG19,CG20}, provides the justification of (\ref{wke}) up to the almost sharp time scale $T_{\mathrm{kin}}^{1-\varepsilon}$ for any $\varepsilon>0$.

In April 2021, the authors \cite{DH21} completed the first rigorous derivation of (\ref{wke}) up to time $T_{\mathrm{kin}}$ for scaling laws $\gamma=1$ or close to $1$. Subsequently, propagation of chaos and other predictions of wave turbulence theory were proved in \cite{DH21-2}. This includes the asymptotics of higher order correlations, derivation of the wave kinetic hierarchy, limit equations for the law of $\widehat u(t,k)$ when the initial distribution is not necessarily Gaussian, and propagation of Gaussianity in the case the initial distribution is Gaussian. 

We should mention that, after the work \cite{DH21}, some other results in a similar vein were also obtained, but for equations with special time-dependent random forcing. In \cite{ST21}, Staffilani and Tran derived the wave kinetic equation for the Zakharov-Kuznetsov equation in the presence of a time-dependent noise that {provides an additional randomization effect for angles in Fourier space.} Recently they extended their result to the spatial inhomogeneous setting with a different noise, in joint work with Hannani and Rosenzweig \cite{HRST22}. At this time, \cite{DH21,DH21-2,HRST22,ST21} are the only results that reach the kinetic time $T_{\mathrm{kin}}$ in the non-equilibrium setting. Some more recent results that cover shorter time scales, but do not include forcing, can be found in \cite{ACG21,Ma22}.

In addition to the derivation of (\ref{wke}), there are also many works devoted to the study of the behavior of solutions to wave kinetic equations like (\ref{wke}), see for example \cite{CDG22,EV15,EV15-2,GIT20, RST21,SoT18}. This is another very important question, but is less related to the focus of this paper, so we will not elaborate on its state of art here.
\subsection{The scaling laws}\label{intro-scale} Note that the kinetic description of wave turbulence theory involves the two limits $L\to\infty$ and $\alpha\to 0$. In fact, it is very important to specify the exact manner in which these two limits are taken. The most general form of such limits would be 
\[\alpha= L^{-\gamma}\] for some $\gamma\in[0,\infty]$, which is called a \emph{scaling law}. Note that, the endpoint case $\gamma=0$ is understood as the iterated limit where first $L\to\infty$ with $\alpha$ fixed and then $\alpha\to 0$; the case $\gamma=\infty$ is the opposite. The purpose of this section is to explain the necessary conditions on the scaling laws for a kinetic theory to hold. This will justify why $\gamma\in(0, 1)$ is the full range of scaling laws for (\ref{nls}) on the {square torus}.

To the best of our knowledge, the role of the scaling law in wave turbulence theory has not been adequately clarified in the physics literature, prior to the recent rigorous mathematical studies. In fact, this was one of the contribution of the authors' recent works, and is explained clearly in the expository paper \cite{DH22}. For completeness of the discussion, we elaborate on this here as well.

 First of all, not all scaling laws\footnote{A common knowledge in physical literature is that the limit $\alpha\to \infty$ should not be taken before the limit $L\to 0$ {(see Remark \ref{scaling0})}, which excludes the scaling law $\gamma=\infty$. This may lead to some mistaken belief that the only other option is $\gamma=0$, i.e. to take $L\to \infty $ first followed by $\alpha\to 0$. In fact, as we shall see in Section \ref{intro-endpoint}, the latter is also not compatible with equations on continuum domains, {so in continuum setting one has to restrict to scaling laws $0<\gamma<\infty$}. In the discrete setting, the scaling law $\gamma=0$ is compatible.} $\alpha=L^{-\gamma}$ allow for the kinetic description in Section \ref{intro-back}. To see this, consider the equation (\ref{nls}) with initial data (\ref{data}), but with a general dispersion relation $\omega(\nabla/i)$ instead of $-\Delta$. Then $\Eb|\widehat{u}(t,k)|^2$ admits an expansion with the first term being $n_{\mathrm{in}}(k)$, and (part of) the second term being
\begin{equation}\label{intro-duhamel}\alpha^2t\cdot L^{-2d}\sum_{k_1-k_2+k_3=k}n_{\mathrm{in}}(k_1)n_{\mathrm{in}}(k_2)n_{\mathrm{in}}(k_3)\cdot t\bigg|\frac{\sin(\pi\Omega t)}{\pi\Omega t}\bigg|^2;\quad \Omega:=\omega(k_1)-\omega(k_2)+\omega(k_3)-\omega(k),\end{equation} where $k_j\in (L^{-1}\mathbb{Z})^d$, due to a calculation of Duhamel iterations. At time $|t|\sim T_{\mathrm{kin}}\sim \alpha^{-2}$, and when $L\to\infty$ and $\alpha\to 0$, this expression formally matches one of the terms in the second iteration of (\ref{wke}) (cf. the first term in (\ref{wke2})), using the fact that $t|\sin(\pi \Omega t)/(\pi\Omega t)|^2\to\dirac(\Omega)$ as $t\to\infty$.

In order for this formal approximation to be legitamite, the one and only restriction is that {\bf the values of $\Omega$, as $k_j$ range over the lattice $\Zb^d_L=(L^{-1}\Zb)^d$, must be equidistributed at scale $T_{\mathrm{kin}}^{-1}\sim\alpha^2$.} In fact, suppose $|\alpha^2t|\sim 1$, then the convergence of (\ref{intro-duhamel}) is intimately tied to the bound
\begin{equation}\label{intro-res}
\begin{aligned}&\qquad\qquad\qquad\qquad\frac{\#(A\cap (L^{-1}\mathbb{Z})^{2d})}{L^{2d}}\sim |t|^{-1}\sim \mathrm{Vol}(A),\mathrm{\ where}\\&A=\{(k_1,k_2)\in\Rb^{2d}:|k_1|,|k_2|\lesssim 1,\,\,|\omega(k_1)-\omega(k_2)+\omega(k+k_2-k_1)-\omega(k)|\lesssim |t|^{-1}\}.
\end{aligned}\end{equation}
The bound \eqref{intro-res} follows from the convergence of \eqref{intro-duhamel}, if we replace the $L^1$ function $|\sin x/x|^2$ by a cutoff function, and similarly for $n_{\mathrm{in}}$. This means that the probability of a lattice point in $\Zb^{2d}_L$ falling into the set $A$---the level set of the function $\Omega$---is proportional to the volume of $A$, which is exactly equidistribution of $\Omega$.

Another implication of the equidistribution property (\ref{intro-res}) is that 
\begin{equation}\label{intro-res2}\#(A_0\cap(L^{-1}\Zb)^d)\lesssim L^{2d}|t|^{-1}\end{equation}where $A_0$ is defined as $A$ above but with $\Omega=0$. In fact, the sets $A$ and $A_0$ are referred to as sets of \emph{quasi resonance} and \emph{exact resonance} by physicists, and the latter inequality just states that {\bf the contribution of exact resonances should be dominated by volume-counting estimates of quasi resonances in \eqref{intro-res}}. This is certainly necessary for the kinetic formalism to hold, and is consistent with the discussions in the physical literature.
\subsubsection{Admissible scaling laws}\label{intro-admis} We say a scaling law $\gamma\in[0,\infty]$ is \emph{admissible}, if the above equidistribution property holds for $\alpha=L^{-\gamma}$ (equivalently $t=\alpha^{-2}\sim L^{2\gamma}$ in \eqref{intro-res}). Clearly, the range of admissibility depends on the precise properties of the dispersion relation $\omega$. Note that a \emph{sufficient} condition is given by
\[\frac{1}{L}\bigg|\frac{\partial \omega}{\partial k}\bigg|\lesssim \alpha^2\] which corresponds to $\gamma\leq 1/2$ \cite{Naz11, GBE22}: in this range, the equidistribution property holds for \emph{any} reasonably behaved dispersion relation $\omega$ without the need for any number theoretic arguments.

However, for a given (or a class of) dispersion relation $\omega$, the above sufficient condition is usually \emph{not necessary}. Specifying to the Schr\"{o}dinger case (\ref{nls}), one can see in multiple ways that the admissible range is in fact \fbox{$\gamma<1$} for arbitrary (including square) tori, and \fbox{$\gamma<d/2$} for tori satisfying a genericity assumption. For example, for the square torus one has $L^2\Omega\in\Zb$, so $\Omega$ cannot be equidistributed at scales $\ll L^{-2}$. Alternatively, the cardinality of the exact resonance set $A_0$ can be shown to be $\sim L^{2d-2}$ for square tori and $\sim L^{d}$ under genericity assumption, which leads (using \eqref{intro-res2}) to the same range of $\gamma$. In fact, we shall see that different values of $\gamma\in(0,1)$ or $(0,d/2)$ represent a range of different physical and mathematical phenomena; see Section \ref{twoscaling} for two special cases.

For $d\geq 3$, the results of the authors' earlier work \cite{DH21} covers the range of scaling laws $\gamma\in(1-c,1)$ for arbitrary tori, where $c$ is a small dimensional constant, as well as $\gamma=1$ under a genericity assumption. The goal of the current work, as stated in Theorem \ref{main}, is to extend the results to the full range $\gamma\in(0,1)$ (we discuss the endpoint $\gamma=0$ in Section \ref{intro-endpoint}).
\subsubsection{Two important scaling laws}\label{twoscaling} For the Schr\"{o}dinger equation (\ref{nls}), there are two scaling laws of particular mathematical and physical interest. The first one is $\gamma=1$, so that $\alpha =L^{-1}$ and $T_{\mathrm{kin}}\sim L^2$. This is consistent with the natural parabolic scaling for (\ref{nls}), by which solutions to (\ref{nls}) on torus of size $L$ and at time scale $\sim L^2$ can be rescaled to solutions on the unit torus and at time $O(1)$; namely, if $u$ solves (\ref{nls}) and $v(t,x)=L^{1/2}u(L^2t,Lx)$ with $x\in\Tb^d$, then $v$ solves the equation $(i\partial_t-\Delta)v+|v|^2v=0$. 
This means that, the predictions of wave turbulence theory under this scaling law, can be translated into conclusions on the unit torus. In three dimensions, this is closely related to the famous Gibbs measure invariance problem for cubic NLS (i.e. invariance of the $\Phi_3^4$ measure under the Schr\"{o}dinger dynamics), which is the only Gibbs measure invariance problem that still remains open after the works \cite{Bou94,Bou96,BDNY22,DNY19,OT20,Zhi94}. In addition, energy cascade behavior for NLS can also be observed at the level of (\ref{wke}) \cite{EV15-2,Naz11}, and proving such cascade dynamics for the NLS equation on the unit torus is a problem of great interest \cite{Bou00}.

\medskip

Another important scaling law is $\gamma=\frac12$ for which $T_{\mathrm{kin}}=L$. We may call this the \emph{ballistic scaling law} because it equates the kinetic timescale with the ballistic timescale needed for a wave packet at frequency $O(1)$ to traverse the domain $\Tb^d_L$. In some sense, this is analogous to the Boltzmann-Grad scaling law adopted in Lanford's theorem justifying the Boltzmann equation, in which the so-called mean-free path is also equated to the transport length scale.

It should be pointed out that such wave packet considerations are more relevant in the \emph{inhomogeneous} setting of the problem, where the initial field is not homogeneous in space as in \eqref{data}. An example of such data is when one sets \eqref{nls} on $\Rb^d$ with random data $u_{\mathrm{in}}(x)$ whose Wigner transform \begin{equation}\label{wigner}\mathbb{E}\bigg(\int_{\mathbb{R}^d}e^{iLy\cdot\eta}\,\overline{\widehat{u_{\mathrm{in}}}\big(\xi-\frac{\eta}{2}\big)}\widehat{u_{\mathrm{in}}}\big(\xi+\frac{\eta}{2}\big)\,\mathrm{d}\eta\bigg)\to W_0(y,\xi)\quad (\mathrm{as\ }L\to\infty),
\end{equation} possibly in a weak sense, where $W_0=W_0(x,\xi):\mathbb{R}^d\times\mathbb{R}^d\to\mathbb{R}_{\geq 0}$ decays rapidly in $\xi$ and $x$.
This is achieved, for example, by setting the random data as
\begin{equation}\label{data0}
u_{\mathrm{in}}(x)=L^{-\frac{d}{2}}\sum_{k\in(L^{-1}\mathbb{Z})^d}\psi\bigg(\frac{x}{L},k\bigg)\cdot g_k\cdot e^{ik\cdot x};\quad \psi(y,k)=\sqrt{W_0(y,k)},
\end{equation} which can be viewed as an inhomogeneous generalization of that in \eqref{data}. Then, the solution to (\ref{nls}) has the form
\begin{equation}\label{solution}
u(t,x)=L^{-\frac{d}{2}}\sum_{k\in(L^{-1}\mathbb{Z})^d}A\bigg(t,\frac{x}{L},k\bigg)\cdot e^{ik\cdot x}.
\end{equation} Denoting $N(t,y,k):=\mathbb{E}|A(t,y,k)|^2$, which corresponds to the Wigner transform of $u(t)$, and performing a formal expansion, we find that $N$ satisfies
\begin{equation}\label{wke1}\partial_tN+\frac{1}{L}(k\cdot\nabla_y)N\approx\alpha^2 \mathcal{C}(N,N,N), \qquad N(0,y,k)=W_0(y,k).
\end{equation}
This gives the inhomogeneous wave kinetic equation provided one equates the transport timescale $L$ with the kinetic timescale $\alpha^{-2}$, which is the scaling law $\gamma=\frac12$ with $T_{\mathrm{kin}}=L$. 

Note that, if one wants to view the homogeneous WKE as a limit of the inhomogeneous one, then one has to introduce an additional parameter to the data in \eqref{data0}, namely one measuring the scale of the inhomogeneity. This can be done by rescaling $W_0$, or equivalently by replacing $\psi(\frac{x}{L},k)$ with $\psi(\frac{x}{M},k)$ in \eqref{data0}, where $M$ is the new inhomogeneity scale. This leads to the flexibility of scaling laws in the homogeneous setting; in fact all the admissible scaling laws $\gamma \in (0, 1)$ described above arise as suitable limits with $L\to\infty$ and $M/L\to \infty$.
\subsubsection{The scaling law $\gamma=0$}\label{intro-endpoint} Note that Theorem \ref{main} covers the full range of scaling laws $\gamma\in(0,1)$, except the endpoint $\gamma=0$. This endpoint does not seem to be compatible with the continuum setting; indeed, formally taking the $L\to\infty$ first will lead to (\ref{nls}) on $\Rb^d$ with initial data
\[u(0,x)=u_{\mathrm{in}}^\infty(x);\quad \Eb(u_{\mathrm{in}}^\infty(x)\overline{u_{\mathrm{in}}^\infty(y)})=(\Fc^{-1}n_{\mathrm{in}})(x-y),\] which is a Gaussian random field with covariance operator $n_{\mathrm{in}}(\nabla/i)$ \emph{that has uniform strength at every point of $\mathbb{R}^d$}. In particular, this initial data, and any possible remainder term that may occur, belongs only to $L^\infty(\Rb^d)$ (with logarithmic growth at infinity). However, for $L^\infty$ data,  there is no known solution theory to (\ref{nls}) (or even the linear Schr\"odinger equation) in any function space, due to infinite speed of propagation and the unboundedness of the linear propagator $e^{it\Delta}$.

Nevertheless, in the \emph{discrete} setting where $\Delta$ is replaced by a discrete difference operator, it is completely plausible to solve (\ref{nls}) in (weighted) $L^\infty$, so in this case $\gamma=0$ is a compatible scaling law, and the corresponding justification of (\ref{wke}) for $\gamma=0$ may be possible \cite{LV22}.
\begin{rem}\label{scaling0} In some early physical literature, the limiting procedure was described as ``the $L\to\infty$ limit should be taken before the $\alpha\to 0$ limit, and not after". This should not be understood as these two limits being taken independently; rather, it simply means that the rate $L\to\infty$ should not be slower than that of $\alpha\to 0$. In other words, we must have $\alpha L^{\gamma_0}\to\infty\Leftrightarrow \alpha\gg L^{-\gamma_0}$, or $\gamma<\gamma_0$ in the context of scaling laws, where $\gamma_0$ is a constant depending on the setting of the problem. This is clearly consistent with all the above discussions.
\end{rem}
\subsection{Ingredients of the proof} We briefly describe here the main difficulties and new ingredients in the proof of Theorem \ref{main}; see Section \ref{overview} for a more substantial description, as that requires the notations set up in Section \ref{setup}.

While the general methodology here follows that in \cite{DH21} which dealt with the scaling law $\gamma=1$, fundamentally new structures and ideas appear for scaling laws $\gamma<2/3$ as we shall explain below. The analysis of these new structures requires introducing new ideas to isolate, analyze, and uncover novel cancellations between some of them. Moreover, it requires upgrading our previous combinatorial algorithm to a much more robust and streamlined apparatus. 

The first steps of the proof of Theorem \ref{main} are essentially the same as in \cite{DH21}: one expands the solution $u$ to (\ref{nls}) into terms indexed by ternary trees, which allows to express the correlations of these terms using \emph{couples}. These couples (which are the Feynman diagrams in this game) are pairs of trees whose leafs are completely paired to each other. The analysis of such couples goes through parallel analytical and combinatorial approaches. The leading couples, which we call \emph{regular couples}, are studied and computed analytically to isolate from them the iterates of (\ref{wke}). It then suffices to show that the contribution of non-regular couples is of lower order. Here, the novel idea of \emph{molecules} was introduced in \cite{DH21} to study the combinatorial problems associated with non-regular couples. This molecular picture will prove to be even more indispensable in this paper. 

The same algorithm used in \cite{DH21} to analyze these molecules breaks down, as soon as $\gamma<2/3$. On a superficial technical level, this is due to the failure of a particular two-vector counting estimate (namely the $q=2$ case of (\ref{atomiccount})). However, this break down is much more fundamental and cannot be saved by simply modifying the algorithm. Indeed, when $\gamma<2/3$, the molecule may contain new bad structures (in fact \emph{multiple families} of them) other than those already observed in \cite{DH21}. {\bf Such bad structures are harmless at scaling laws $\gamma>2/3$, but can overtake the leading terms for $\gamma<2/3$.} Note that this difficulty is of very different nature from that of \cite{DH21}, which mainly revolves around overcoming the factorial divergence caused by \emph{generic} molecules. While this is still a problem here, the extra difficulty imposed by these \emph{special} bad structures requires substantially new ideas beyond the proof in \cite{DH21}.

The strategy here is to first (i) identify all the possible bad structures---there are \emph{eight} families of them that we call \emph{vines} (Figure \ref{fig:vines}), then (ii) recover a good estimate for any molecule absent of these bad structures (in the form of a \emph{rigidity theorem} similar to Proposition 9.10 of \cite{DH21}), and finally (iii) control the contribution of these bad structures.

Parts (i)--(ii) can in fact be done together at the level of the molecule picture, by introducing a powerful new operation that is absent in the algorithm of \cite{DH21}, called the \emph{cutting} operation (Figure \ref{fig:cutintro}). This seemingly simple operation allows us to \emph{isolate} all the possible bad scenarios into ``local" post-surgery connected components, and locate only finitely many families of connected components that are problematic. It is precisely this small addition that leads to the complete classification of bad structures in this paper, namely the vines. We believe that this is the one missing piece in the algorithm of \cite{DH21} that makes it much more robust. We also notice that the combinatorial difficulty caused by vines is not specific to the (NLS) case, but is actually universal (at least in the $4$-wave setting) independent of dispersion relation and multilinear multipliers. As such, we believe that the algorithm in \cite{DH21}, equipped with the cutting operation, should be directly applicable in many other settings.

Part (iii) of this plan, which is another main novelty of this work, relies on extremely delicate, and somewhat miraculous, cancellations observed between the bad structures identified in (i). Indeed, starting from these bad structures identified at the level of the molecule, one can reconstruct the various possibilities of couples that have this same molecular structure. These couples, which may have arbitrarily large size, can be grouped into pairs defined as \emph{twists} of each other (Figures \ref{fig:vinescancel} and \ref{fig:twist_dec}). The cancellation structure is then found by studying expressions associated with couples that are twists of each other. It is worth mentioning that this cancellation is so involved and intricate that there is little-to-no chance of uncovering it if one only looks at the couple picture, and does not turn to the molecule picture (cf. Figure \ref{fig:flowchart}). This strongly suggests that the molecules introduced in \cite{DH21} are fundamental objects, and not mere technical tools.
 
To the best of our knowledge, the cancellation identified in this paper has not appeared in earlier mathematical or physical literature (such structures only become significant in higher oder terms, so it's not surprising that they do not play a role in the formal derivation of physicists that only involve second order expansions). Therefore, we believe that the ingredients of this paper and \cite{DH21}---including cancellations of vine structures and the algorithm in \cite{DH21} with cutting---constitute {\bf the next major step beyond \cite{EY00,ESY08,LS11} in the study of Feynman diagrams. This development allows us to effectively estimate diagrams of much higher order than those in \cite{EY00,ESY08,LS11},} which results in the proof of Theorem \ref{main} in the non-equilibrium setting and without noise.

Finally, we remark that, this new Feynman diagram analysis is robust enough to be applicable in a wide range of semilinear dispersive equations. The only major difference for other dispersion relations $\omega$ would be the equidistribution property (\ref{intro-res}), which may restrict the range of scaling laws $\gamma$ depending on the fine number theoretic properties of $\omega$. However, these number theoretic ingredients are only needed when $\gamma>\frac12$; for $\gamma\leq 1/2$, we expect that results like Theorem \ref{main} should hold for arbitrary $\omega$, as demonstrated in Section \ref{intro-admis}.
\subsection{Future Horizons} We conclude this introduction by listing, what we believe to be some of the next major frontiers in this line of research, after the resolution (here and in \cite{DH21, DH21-2}) of the first fundamental question that is the rigorous justification of the wave kinetic theory.

\medskip 

(1) \emph{Longer times}: the obvious question after Theorem \ref{main} would be whether the same result can be extended to time $|t|\leq C\cdot T_{\mathrm{kin}}$ for constants $C\gg1$. This is a tremendous open problem, and its resolution is unknown not only in the wave turbulence setting, but also in the classical particle setting of Lanford's theorem justifying Boltzmann's equation. Note that (\ref{wke}) may have finite time blowup (which is even expected to be generic, see \cite{EV15,EV15-2}), so the best one can hope for, in terms of  the approximation (\ref{limit}), would be the following conjecture:
\begin{itemize}
\item Suppose the solution to (\ref{wke}) stays smooth up to time $\tau$, then the approximation (\ref{limit}) holds for all time $|t|\leq\tau\cdot T_{\mathrm{kin}}$.
\end{itemize}

Answering this conjecture is highly challenging, and would require ideas and techniques completely different from the current and earlier works. Moreover, a positive answer would have profound implications on the study of long-time dynamics of (\ref{nls}), especially on energy cascades.

(2) \emph{Post-singularity dynamics}: Suppose that the conjecture in (1) has been proved or is assumed to be true. Moreover, suppose a specific solution to (\ref{wke}) exhibits a $\dirac$ singularity at a particular time $\tau_0$. The analysis in \cite{EV15, EV15-2} suggests that such singularity formation is somewhat generic (formation of condensate). Then we may ask the following question: what is the asymptotic behavior of
\[\Eb|\widehat{u}(\tau_0\cdot T_{\mathrm{kin}},0)|^2?\] 
In other words, can one prove rigorously the dynamical formation of condensate for (\ref{nls})? More interestingly, for $\tau>\tau_0$, can one still track the macroscopic behavior of $\Eb|\widehat{u}(\tau\cdot T_{\mathrm{kin}},k)|^2$? Does it converge to a finite limit? If so, can it be defined as a weak solution to (\ref{wke}) in some sense? If not, then should we somehow modify (\ref{nls}) (and/or the wave kinetic equation) beyond the time $\tau_0\cdot T_{\mathrm{kin}}$, in view of the condensate formed for (\ref{wke}) at $\tau_0$? These questions may be even more challenging than the conjecture in (1), but their resolution would bring new insights, both physical and mathematical, to the study of (\ref{nls}) and its condensates.

(3) \emph{Properties of solutions to (\ref{wke})}: turning now to the solution theory to (\ref{wke}), an important question is to describe more precisely the formation of condensate \cite{EV15,EV15-2}, and perhaps justify its genericity, for sufficiently strong classes of solutions. Another venue of immense physical interest, is to rigorously study solutions that may asymptote to (or resemble in some meaningful sense) the Zakharov spectra, see for example \cite{CDG22} for a step in this direction. These specific solutions, when combined with possible results in (1) and (2), may lead to the discovery of very interesting behavior of solutions to (\ref{nls}).

One may also consider the inhomogeneous version of (\ref{wke}), whose derivation is expected to be similar to (\ref{wke}) with only technical differences. However, solutions to the inhomogeneous (\ref{wke}) may behave quite differently, for the transport term may prevent blowup. If these solution exhibit diffusive behavior for long times, this may lead to a nonlinear version of the quantum diffusion behavior described in Erd\"{o}s-Salmhofer-Yau \cite{ESY08} in the linear setting.

\subsection{Acknowledgements} The first author is supported in part by NSF grant DMS-1900251 and a Sloan Fellowship. The second author is supported in part by NSF grant DMS-1654692 and a Simons Collaboration Grant on Wave Turbulence. The authors would like to thank Herbert Spohn for several discussions, that explained the importance of other scaling laws (particularly $\gamma=\frac12$). This was a major drive to study the full range of scaling laws in this paper. The authors also thank Jani Lukkarinen for explaining the work \cite{LS11} and several other illuminating discussions.

\section{Preparations}\label{setup}
\subsection{Preliminary reductions}Start from the equation (\ref{nls}), let $u$ be a solution, and recall $\alpha=L^{-\gamma}$. Let $M=\fint|u|^2$ be the conserved mass of $u$ (where $\fint$ takes the average on $\Tb_L^d$), and define $v:=e^{-2iL^{-\gamma} Mt}\cdot u$, then $v$ satisfies the Wick ordered equation
\begin{equation}(i\partial_t-\Delta)v+L^{-\gamma}\bigg(|v|^2v-2\fint|v|^2\cdot v\bigg)=0.
\end{equation} By switching to Fourier space, rescaling in time and reverting the linear Schr\"{o}dinger flow, we define
\begin{equation}a_k(t)=e^{-\pi i\cdot\delta L^{2\gamma}|k|^2t}\cdot\widehat{v}( \delta T_{\mathrm{kin}}\cdot t,k)
\end{equation} with $\widehat{v}$ as in (\ref{fourier}), then $\textit{\textbf{a}}:=a_k(t)$ will satisfy the equation
 \begin{equation}\label{akeqn}
 \left\{
\begin{aligned}
\partial_ta_k &= \Cc_+(\textit{\textbf{a}},\overline{\textit{\textbf{a}}},\textit{\textbf{a}})_k(t),\\
a_k(0) &=(a_k)_{\mathrm{in}}=\sqrt{n_{\mathrm{in}}(k)}g_k(\omega),
\end{aligned}
\right.
\end{equation} with the nonlinearity
\begin{equation}\label{akeqn2} \Cc_\zeta(\textit{\textbf{f}},\textit{\textbf{g}},\textit{\textbf{h}})_k(t):=\frac{\delta}{2L^{d-\gamma}}\cdot(i\zeta)\sum_{k_1-k_2+k_3=k}\epsilon_{k_1k_2k_3}
e^{\zeta\pi i\cdot\delta L^{2\gamma}\Omega(k_1,k_2,k_3,k)t}f_{k_1}(t)g_{k_2}(t)h_{k_3}(t).
\end{equation}for $\zeta\in\{\pm\}$. Here in (\ref{akeqn2}) and below, the summation is taken over $(k_1,k_2,k_3)\in(\Zb_L^d)^3$, and \begin{equation}\label{defcoef0}\epsilon_{k_1k_2k_3}=
\left\{
\begin{aligned}+&1,&&\mathrm{if\ }k_2\not\in\{k_1,k_3\};\\
-&1,&&\mathrm{if\ }k_1=k_2=k_3;\\
&0,&&\mathrm{otherwise},
\end{aligned}
\right.\end{equation} and the resonance factor
\begin{equation}\label{res}
\Omega=\Omega(k_1,k_2,k_3,k):=|k_1|^2-|k_2|^2+|k_3|^2-|k|^2=2\langle k_1-k,k-k_3\rangle.\end{equation} Note that $\epsilon_{k_1k_2k_3}$ is always supported in the set \begin{equation}\label{defset}\Sf:=\big\{(k_1,k_2,k_3):\mathrm{\ either\ }k_2\not\in\{k_1,k_3\},\mathrm{\ or\ }k_1=k_2=k_3\big\}.\end{equation}

The rest of this paper is focused on the system (\ref{akeqn})--(\ref{akeqn2}) for $\textit{\textbf{a}}$, with the relevant terms defined in (\ref{defcoef0})--(\ref{res}), in the time interval $t\in[0,1]$.
\subsection{Parameters, notations and norms}\label{norms} In this subsection we list some notations and fix some parameters that will be useful below. Recall that $d\geq 3$ and $0<\gamma<1$, and Schwartz data $n_{\mathrm{in}}$ are fixed. Define \begin{equation}\label{othergamma}\gamma_0:=\min(\gamma,1-\gamma),\quad \gamma_1:=\min(2\gamma,1,2(d-1)(1-\gamma)).\end{equation} Fix $\eta$ as a small absolute constant such that $\eta\ll_{d,\gamma}1$, and let $C$ be any large constant depending on $(d,\gamma)$ and $\eta$. Let also $C^+$ be any large constant depending on $C$ and $n_{\mathrm{in}}$, and fix $\delta$ as a small constant such that $\delta\ll_{C^+}1$. Unless otherwise stated, the implicit constants in $\lesssim$ symbols may depend on $C^+$, but those in $O(\cdot)$ symbols depend only on $C$. Let $L$ be large enough depending on $\delta$, and define $N=\lfloor(\log L)^4\rfloor$.

Let $\chi_0=\chi_0(z)\in C^\infty(\Rb\to\Rb_{\geq 0})$ be such that $\chi_0=1$ for $|z|\leq 1/2$ and $\chi_0=0$ for $|z|\geq 1$; define $\chi_0(z^1,\cdots,z^d)=\chi_0(z^1)\cdots \chi_0(z^d)$ and $\chi_\infty=1-\chi_0$, where $z^j$ are coordinates of vectors $z\in\Rb^d$ (we use this notation throughout). By abusing notation, sometimes we may also use $\chi_0$ to denote other cutoff functions with slightly different supports. These functions, as well as the other cutoff functions, will be in Gevrey class $2$ (i.e. the $k$-th order derivatives are bounded by $(2k)!$). For a multi-index $\rho=(\rho_1,\cdots,\rho_m)$, we adopt the usual notations $|\rho|=\rho_1+\cdots+\rho_m$ and $\rho!=(\rho_1)!\cdots(\rho_m)!$, etc. For an index set $A$, we use the vector notation $\alpha[A]=(\alpha_j)_{j\in A}$ and $\mathrm{d}\alpha[A]=\prod_{j\in A}\mathrm{d}\alpha_j$, etc.

Denote $z^+=z$ for a complex number $z$, and $z^-=\overline{z}$. In the rest of this paper, we will not use the space Fourier transform notation as in (\ref{fourier}). We will use $\widehat{\cdot}$ only for the time Fourier transform, which is defined as \[\widehat{u}(\lambda)=\int_\Rb u(t) e^{-2\pi i\lambda t}\,\mathrm{d}t,\quad u(t)=\int_\Rb \widehat{u}(\lambda)e^{2\pi i\lambda t}\,\mathrm{d}\lambda,\] and similarly for higher dimensional versions. If a function $F=F(t_j,k_j)$ depends on several time variables $t_j$ and several vector variables $k_j$, we shall define its $X^{\theta,\beta}$ norm by \[\|F\|_{X^{\theta,\beta}}=\int\big(\max_j\langle\lambda_j\rangle\big)^{\theta}\cdot\bigg[\sup_{k_j}\big(\max_j\langle k_j\rangle\big)^{\beta}|\widehat{F}(\lambda_j,k_j)|\bigg]\,\prod_j\mathrm{d}\lambda_j,\]
If $F=F(t_j)$ does not depend on any $k_j$, the norms are modified accordingly; they do not depend on $\beta$ so we call it $X^\theta$. Define the localized version $X_{\mathrm{loc}}^{\theta,\beta}$, and associated auxiliary $Y_{\mathrm{loc}}^\theta$ norm, by 
\[\|F\|_{X_{\mathrm{loc}}^{\theta,\beta}}=\inf\big\{\|\widetilde{F}\|_{X^{\theta,\beta}}:\widetilde{F}=F\mathrm{\ for\ }0\leq t_j\leq 1\big\},\quad\|F\|_{Y_{\mathrm{loc}}^\theta}:=\sup_{(k_j^0)}\|F\cdot\mathbf{1}_{|k_j-k_j^0|\leq 1\,(\forall j)}\|_{X_{\mathrm{loc}}^{\theta,0}}.\] If we will only use the value of $F$ in some subset (for example $\{t_1>t_2\}$, see the second part of Proposition \ref{regcpltreeasymp}), then in the above definition we may only require $\widetilde{F}=F$ in this set. Finally, define the $Z$ norm for function $a=a_k(t)$,
\begin{equation}\label{defznorm}\|a\|_Z^2=\sup_{0\leq t\leq 1}L^{-d}\sum_{k\in\Zb_L^d}\langle k\rangle^{10d}|a_k(t)|^2
\end{equation}All these norms are readily extended to Banach space valued functions.
\subsection{Trees, couples and decorations} Recall the notions of trees, couples and decorations, which are defined in \cite{DH21}.
\begin{df}[Trees]\label{deftree} A \emph{ternary tree} $\Tc$ (we will simply say a \emph{tree} below) is a rooted tree where each non-leaf (or \emph{branching}) node has exactly three children nodes, which we shall distinguish as the \emph{left}, \emph{mid} and \emph{right} ones. A node $\mf$ is a \emph{descendant} of a node $\nf$, or $\nf$ is an \emph{ancestor} of $\mf$, if $\mf$ belongs to the subtree rooted at $\nf$ (we allow $\mf=\nf$). We say $\Tc$ is \emph{trivial} (and write $\Tc=\bullet$) if it has only the root, in which case this root is also viewed as a leaf.

We denote generic nodes by $\nf$, generic leaves by $\lf$, the root by $\rf$, the set of leaves by $\Lc$ and the set of branching nodes by $\Nc$. The \emph{order} of a tree $\Tc$ is defined by $n(\Tc)=|\Nc|$ (this is called scale in \cite{DH21}), so if $n(\Tc)=n$ then $|\Lc|=2n+1$ and $|\Tc|=3n+1$.

A tree $\Tc$ may have sign $+$ or $-$. If its sign is fixed then we decide the signs of its nodes as follows: the root $\rf$ has the same sign as $\Tc$, and for any branching node $\nf\in\Nc$, the signs of the three children nodes of $\nf$ from left to right are $(\zeta,-\zeta,\zeta)$ if $\nf$ has sign $\zeta\in\{\pm\}$. Once the sign of $\Tc$ is fixed, we will denote the sign of $\nf\in\Tc$ by $\zeta_\nf$. Define $\zeta(\Tc)=\prod_{\nf\in\Nc}(i\zeta_\nf)$. We also define the conjugate $\overline{\Tc}$ of a tree $\Tc$ to be the same tree but with opposite sign.
\end{df}
\begin{df}[Couples]\label{defcouple} A \emph{couple} $\Qc$ is an unordered pair $\{\Tc^+,\Tc^-\}$ of two trees $\Tc^\pm$ with signs $+$ and $-$ respectively, together with a partition $\Ps$ of the set $\Lc^+\cup\Lc^-$ into $(n+1)$ pairwise disjoint two-element subsets, where $\Lc^\pm$ is the set of leaves for $\Tc^\pm$, and $n=n^++n^-$ where $n^\pm$ is the order of $\Tc^\pm$. This $n$ is also called the \emph{order} of $\Qc$, denoted by $n(\Qc)$. The subsets $\{\lf,\lf'\}\in\Ps$ are referred to as \emph{pairs}, and we require that $\zeta_{\lf'}=-\zeta_\lf$, i.e. the signs of paired leaves must be opposite. If both $\Tc^\pm$ are trivial, we call $\Qc$ the \emph{trivial couple} (and write $\Qc=\times$).

For a couple $\Qc=\{\Tc^+,\Tc^-,\Ps\}$ we denote the set of branching nodes by $\Nc=\Nc^+\cup\Nc^-$, and the set of leaves by $\Lc=\Lc^+\cup\Lc^-$; for simplicity we will abuse notation and write $\Qc=\Tc^+\cup\Tc^-$. Define $\zeta(\Qc)=\prod_{\nf\in\Nc}(i\zeta_\nf)$. We also define a \emph{paired tree} to be a tree where \emph{some} leaves are paired to each other, according to the same pairing rule for couples. We say a paired tree is \emph{saturated} if there is only one unpaired leaf (called the \emph{lone leaf}). In this case the tree forms a couple with the trivial tree $\bullet$. Finally, we define the conjugate of a couple $\Qc=\{\Tc^+,\Tc^-\}$ as $\overline{\Qc}=\{\overline{\Tc^-},\overline{\Tc^+}\}$ with the same pairings; for a paired tree $\Tc$ we also define its conjugate as $\overline{\Tc}$ with the same pairings, where $\overline{\Tc}$ is as in Definition \ref{deftree}.
\end{df}
\begin{df}[Decorations]\label{defdec} A \emph{decoration} $\Ds$ of a tree $\Tc$ is a set of vectors $(k_\nf)_{\nf\in\Tc}$, such that $k_\nf\in\Zb_L^d$ for each node $\nf$, and that \[k_\nf=k_{\nf_1}-k_{\nf_2}+k_{\nf_3},\quad \mathrm{or\ equivalently}\quad \zeta_\nf k_\nf=\zeta_{\nf_1}k_{\nf_1}+\zeta_{\nf_2}k_{\nf_2}+\zeta_{\nf_3}k_{\nf_3},\] for each branching node $\nf\in\Nc$, where $\zeta_\nf$ is the sign of $\nf$ as in Definition \ref{deftree}, and $\nf_1,\nf_2,\nf_3$ are the three children nodes of $\nf$ from left to right. Clearly a decoration $\Ds$ is uniquely determined by the values of $(k_\lf)_{\lf\in\Lc}$. For $k\in\Zb_L^d$, we say $\Ds$ is a $k$-decoration if $k_\rf=k$ for the root $\rf$.

Given a decoration $\Ds$, we define the coefficient
\begin{equation}\label{defcoef}\epsilon_\Ds:=\prod_{\nf\in\Nc}\epsilon_{k_{\nf_1}k_{\nf_2}k_{\nf_3}}\end{equation} where $\epsilon_{k_1k_2k_3}$ is as in (\ref{defcoef0}). Note that in the support of $\epsilon_\Ds$ we have that $(k_{\nf_1},k_{\nf_2},k_{\nf_3})\in\Sf$ for each $\nf\in\Nc$. We also define the resonance factor $\Omega_\nf$ for each $\nf\in\Nc$ by
\begin{equation}\label{defres}\Omega_\nf=\Omega(k_{\nf_1},k_{\nf_2},k_{\nf_3},k_\nf)=|k_{\nf_1}|^2-|k_{\nf_2}|^2+|k_{\nf_3}|^2-|k_\nf|^2.\end{equation}

A decoration $\Es$ of a couple $\Qc=\{\Tc^+,\Tc^-,\Ps\}$, is a set of vectors $(k_\nf)_{\nf\in\Qc}$, such that $\Ds^\pm:=(k_\nf)_{\nf\in\Tc^\pm}$ is a decoration of $\Tc^\pm$, and moreover $k_\lf=k_{\lf'}$ for each pair $\{\lf,\lf'\}\in\Ps$. We define $\epsilon_\Es:=\epsilon_{\Ds^+}\epsilon_{\Ds^-}$, and define the resonance factors $\Omega_\nf$ for $\nf\in\Nc$ as in (\ref{defres}). Note that we must have $k_{\rf^+}=k_{\rf^-}$ where $\rf^\pm$ is the root of $\Tc^\pm$; again we say $\Es$ is a $k$-decoration if $k_{\rf^+}=k_{\rf^-}=k$. We also define decorations $\Ds$ of paired trees, as well as $\epsilon_\Ds$ and $\Omega_\nf$ etc., similar to the above (except that we don't pair all leaves).
\end{df}
\subsection{The ansatz and main estimates} We now state the ansatz for the solution $\textit{\textbf{a}}$ to the system (\ref{akeqn})--(\ref{akeqn2}), as well as the main estimates.
\subsubsection{The expressions $\Jc_\Tc$ and $\Kc_\Qc$} For any tree $\Tc$ of order $n$, define the expression
\begin{equation}\label{defjt}(\Jc_\Tc)_k(t)=\bigg(\frac{\delta}{2L^{d-\gamma}}\bigg)^n\zeta(\Tc)\sum_\Ds\epsilon_\Ds\cdot\int_\Dc\prod_{\nf\in\Nc}e^{\zeta_\nf\pi i\cdot\delta L^{2\gamma}\Omega_\nf t_\nf}\,\mathrm{d}t_\nf\cdot\prod_{\lf\in\Lc}\sqrt{n_{\mathrm{in}}(k_\lf)}\eta_{k_\lf}^{\zeta_\lf}(\omega)
\end{equation} where the sum is taken over all $k$-decorations $\Ds$ of $\Tc$, and the domain
\begin{equation}\label{defdomaind}\Dc=\big\{t[\Nc]:0<t_{\nf'}<t_\nf<t\mathrm{\ whenever\ }\nf'\mathrm{\ is\ a\ child\ node\ of\ }\nf\big\}.
\end{equation} For any couple $\Qc$ of order $n$, define the expression
\begin{equation}\label{defkq}\Kc_\Qc(t,s,k)=\bigg(\frac{\delta}{2L^{d-\gamma}}\bigg)^n\zeta(\Qc)\sum_\Es\epsilon_\Es\cdot\int_\Ec\prod_{\nf\in\Nc}e^{\zeta_\nf\pi i\cdot\delta L^{2\gamma}\Omega_\nf t_\nf}\,\mathrm{d}t_\nf\cdot\prod_{\lf\in\Lc}^{(+)}n_{\mathrm{in}}(k_\lf),
\end{equation} where the sum is taken over all $k$-decorations $\Es$ of $\Qc$, the product $\prod_{\lf\in\Lc}^{(+)}$ is taken over all leaves $\lf\in\Lc$ with $+$ sign, and the domain
\begin{multline}\label{defdomaine}\Ec=\big\{t[\Nc]:0<t_{\nf'}<t_\nf\mathrm{\ whenever\ }\nf'\mathrm{\ is\ a\ child\ node\ of\ }\nf;\\t_\nf<t\mathrm{\ whenever\ }\nf\in\Nc^+\mathrm{\ and\ }t_\nf<s\mathrm{\ whenever\ }\nf\in\Nc^-\big\}.
\end{multline}
\subsubsection{The ansatz for $a_k(t)$} Let
\begin{equation}\label{defjn}(\Jc_n)_k(t)=\sum_{n(\Tc^+)=n}(\Jc_{\Tc^+})_k(t)
\end{equation} where the sum is taken over all trees $\Tc^+$ of order $n$ and sign $+$, and define $\textit{\textbf{b}}=b_k(t)$ by
\begin{equation}\label{defb}a_k(t)=\sum_{0\leq n\leq N}(\Jc_n)_k(t)+b_k(t),
\end{equation} where $N=\lfloor(\log L)^4\rfloor$ as defined in Section \ref{norms}. Then $\textit{\textbf{b}}$ satisfies an equation of form
\begin{equation}\label{eqnbk}\textit{\textbf{b}}=\Rc+\Ls \textit{\textbf{b}}+\Ls_2(\textit{\textbf{b}},\textit{\textbf{b}})+\Ls_3(\textit{\textbf{b}},\textit{\textbf{b}},\textit{\textbf{b}}),
\end{equation} or equivalently 
\begin{equation}\label{eqnbk2}\textit{\textbf{b}}=(1-\Ls)^{-1}(\Rc+\Ls_2(\textit{\textbf{b}},\textit{\textbf{b}})+\Ls_3(\textit{\textbf{b}},\textit{\textbf{b}},\textit{\textbf{b}})),
\end{equation}where the relevant terms are defined as
\begin{equation}\label{eqnbk1.5}\Rc=\sum_{(0)}\Ic\Cc_+(\textit{\textbf{u}},\overline{\textit{\textbf{v}}},\textit{\textbf{w}}),\quad\Ls \textit{\textbf{b}}=\sum_{(1)}\Ic\Cc_+(\textit{\textbf{u}},\overline{\textit{\textbf{v}}},\textit{\textbf{w}}),\quad \Ls_2(\textit{\textbf{b}},\textit{\textbf{b}})=\sum_{(2)}\Ic\Cc_+(\textit{\textbf{u}},\overline{\textit{\textbf{v}}},\textit{\textbf{w}}),\end{equation} and $\Ls_3(\textit{\textbf{b}},\textit{\textbf{b}},\textit{\textbf{b}})=\Ic\Cc_+(\textit{\textbf{b}},\overline{\textit{\textbf{b}}},\textit{\textbf{b}})$. The sums in (\ref{eqnbk1.5}) are taken over $(\textit{\textbf{u}},\textit{\textbf{v}},\textit{\textbf{w}})$, each of which being either $\textit{\textbf{b}}$ or $\Jc_n$ for some $0\leq n\leq N$. In the sum $\sum_{(j)}$ for $0\leq j\leq 2$, exactly $j$ inputs in $(\textit{\textbf{u}},\textit{\textbf{v}},\textit{\textbf{w}})$ equals $\textit{\textbf{b}}$, and in the sum $\sum_{(0)}$ we require that $(\textit{\textbf{u}},\textit{\textbf{v}},\textit{\textbf{w}})=(\Jc_{n_1},\Jc_{n_2},\Jc_{n_3})$ with $n_1+n_2+n_3\geq N$. Note that $\Ls$, $\Ls_2$ and $\Ls_3$ are $\Rb$-linear, $\Rb$-bilinear and $\Rb$-trilinear operators respectively.
\subsubsection{Correlations, and expansion of (\ref{wke})} For any $n_1,n_2\geq 0$, by using Isserlis' theorem as in Section 2.2.3 of \cite{DH21}, we have that
\begin{equation}\label{correlation}\Eb\big((\Jc_{n_1})_{k}(t)\overline{(\Jc_{n_2})_k(t)}\big)=\sum_{\Qc}\Kc_\Qc(t,t,k),\end{equation} where the summation is taken over all couples $\Qc=\{\Tc^+,\Tc^-\}$ such that $n(\Tc^+)=n_1$ and $n(\Tc^-)=n_2$ (the partition $\Ps$ can be arbitrary).

Now consider the equation (\ref{wke}). Due to the smallness of $\delta$, the solution $n=n(t,k)$ to (\ref{wke}) has the Taylor expansion
\begin{equation}\label{wketaylor}n(\delta t,k)=\sum_{n=0}^\infty\Mc_n(t,k),
\end{equation} where $\Mc_n(t,k)$ is defined such that
\begin{equation}\label{wketaylor2}\Mc_0(t,k)=n_{\mathrm{in}}(k),\quad \Mc_{n}(t,k)=\delta\sum_{n_1+n_2+n_3=n-1}\int_0^t\Kc(\Mc_{n_1}(t'),\Mc_{n_2}(t'),\Mc_{n_3}(t'))(k)\,\mathrm{d}t'.
\end{equation} It is easy to see that $|\Mc_n(t,k)|\lesssim\langle k\rangle^{-20d}(C^+\delta)^n$ uniformly in $(t,k)$.
\subsubsection{The main estimates}
The followings are the main estimates of this paper. Their proofs will occupy up to Section \ref{linoper1}; once they are proved, Theorem \ref{main} will then be proved in Section \ref{linoper2}, similar to Section 12 of \cite{DH21}.
\begin{prop}\label{mainprop1}Let $\Kc_\Qc$ be defined in (\ref{defkq}). Then for each $0\leq n\leq N^2$, $k\in\Zb_L^d$ and $t\in[0,1]$ we have
\begin{equation}\label{mainest1}\bigg|\sum_{\Qc}\Kc_\Qc(t,t,k)\bigg|\lesssim\langle k\rangle^{-20d}(C^+\sqrt{\delta})^n,
\end{equation} where the summation is taken over all couples $\Qc=\{\Tc^+,\Tc^-\}$ such that $n(\Tc^+)=n(\Tc^-)=n$.
\end{prop}
\begin{prop}\label{mainprop4} Let $\Mc_n(t,k)$ be defined as in (\ref{wketaylor2}). Then for each $0\leq n\leq N^2$, $k\in\Zb_L^d$ and $t\in[0,1]$, we have that
\begin{equation}\label{mainest1.5}\bigg|\sum_{n(\Qc)=2n}\Kc_\Qc(t,t,k)-\Mc_n(t,k)\bigg|\lesssim\langle k\rangle^{-20d} (C^+\sqrt{\delta})^n L^{-\eta^8},\end{equation} where the summation is taken over all couples $\Qc$ of order $2n$. If $2n$ is replaced by $2n+1$, then the same result holds with $\Mc_n(t,k)$ replaced by $0$.
\end{prop}
\begin{prop}\label{mainprop2} For any $\Rb$-linear operator $\Ks$ define its kernel $\Ks_{kk'}^\zeta(t,s)$ (which we assume is supported in $t,s\in[0,1]$ and $t>s$), where $\zeta\in\{\pm\}$, such that \begin{equation}\label{parametrix}(\Ks\textit{\textbf{b}})_k(t)=\sum_{\zeta\in\{\pm\}}\sum_{k'}\int_0^t \Ks_{kk'}^\zeta(t,s) b_{k'}(s)^\zeta\,\mathrm{d}s;\quad \textit{\textbf{b}}=(b_{k'}(s)).\end{equation} Now let $\Ls$ be defined as in (\ref{eqnbk1.5}). Then there exists an $\Rb$-linear operator $\Xs$, and $\Ys=(1-\Ls)\Xs$ and $\Ws=\Xs(1-\Ls)$, such that
\begin{equation}\label{kernelexp}\Xs=1+\sum_{m=1}^N\Xs_m,\quad \Ys=1+\sum_{m=N+1}^{3N+1}\Ys_m,\quad \Ws=1+\sum_{m=N+1}^{3N+1}\Ws_m\end{equation} The kernels of $\Xs_m$, $\Ys_m$ and $\Ws_m$ are Weiner chaos of order $2m$, and they satisfy that
\begin{equation}\label{mainest2}\Eb|(\Xs_m)_{kk'}^{\zeta}(t,s)|^2\lesssim \langle k-k'\rangle^{-20d}(C^+\sqrt{\delta})^mL^{40d}\end{equation} for any $1\leq m\leq N$ and $k,k'\in\Zb^d_L$ and $t,s\in[0,1]$ with $t>s$. The same holds for $\Ys_m$ and $\Ws_m$ and any $N+1\leq m\leq 3N+1$.
\end{prop}
\begin{rem}\label{nonresrem}
Recall that the support of the coefficient $\epsilon_{k_1k_2k_3}$ in (\ref{defcoef0}), which is the set $\Sf$ in (\ref{defset}), allows for the \emph{degenerate} case $k_2\in\{k_1,k_3\}$. However in this case we must have $k_1=k_2=k_3$, and such case is always easily treated (see for example Section 9.3.1 of \cite{DH21} which deals with degenerate atoms).

For simplicity of presentation, we will neglect the degenerate case for most of this paper and assume that $\epsilon_{k_1k_2k_3}$ is always supported in the set where $k_2\not\in\{k_1,k_3\}$. In Section \ref{extradegen} after the main proof, we briefly discuss how to treat degenerate cases, which only requires minor modifications.
\end{rem}
\section{Overview of the proof}\label{overview}
\subsection{Previous strategy, and the main difficulty} Let us start from the proof of Propositions \ref{mainprop1} and \ref{mainprop4}, which requires us to analyze the quantity $\Kc_\Qc$ . The early steps in the analysis are essentially the same as in \cite{DH21} (which deals with the case $\gamma\approx 1$). First one identifies and analyzes the leading couples, called \emph{regular couples}, which are ones built by concatenating specific building blocks called mini couples (cf. Definition \ref{defmini} and Figure \ref{fig:oper}). The analysis of those couples was done in \cite{DH21} and is basically independent of the chosen scaling law. As such, the heart of the matter is showing that the contribution of non-regular couples is lower order. The analysis of non-regular couples proceeds by applying a structure theorem (Proposition \ref{skeleton}) to reduce to prime couples $\Qc$ (Section \ref{primered}), which are couples that contain no regular sub-couples inside them. Then, one uses the almost $L^1$ bound for the time integral in (\ref{bigformula}) (see for example (\ref{1stsumbound})) to reduce the estimates on such prime couples to a counting problem for decorations $(k_\nf)$ of $\Qc$, which has the form
\begin{equation}\label{cplcount0}
\left\{
\begin{aligned}&k_{\nf_1}-k_{\nf_2}+k_{\nf_3}=k_\nf;\quad |k_{\nf_1}|,\,|k_{\nf_2}|,\,|k_{\nf_3}|,\,|k_{\nf}|\leq 1,\\
&|k_{\nf_1}|^2-|k_{\nf_2}|^2+|k_{\nf_3}|^2-|k_\nf|^2=(\mathrm{Const})+O(L^{-2\gamma}),
\end{aligned}
\right.\quad\forall \mathrm{\ branching\ nodes\ }\nf
\end{equation}
in the notions of Definition \ref{defdec}.

To study the counting problem, we then introduce the notion of \emph{molecules} $\Mb$ (Definition \ref{defmole}), which is of fundamental importance in \cite{DH21} and even more so in the current paper. Basically, a molecule is a directed graph formed by \emph{atoms} which are $4$-element subsets $\{\nf,\nf_1,\nf_2,\nf_3\}$ as in (\ref{cplcount0}), and \emph{bonds} which are common elements of these subsets under the given pairing structure; such a molecule coming from a couple will have all degree $4$ atoms except for exactly two degree $3$ atoms, and moreover contains no triple bond if the couple is prime. The system (\ref{cplcount0}) is then reduced to a system for decorations this molecule (Definition \ref{decmole}), where each bond is decorated by a vector $k_\ell$ and each atom gives an equation of form (\ref{cplcount0}) with variables corresponding to the bonds at this atom. As a central component of the proof, we need to establish a suitable \emph{rigidity theorem}, which (schematically) states that
\begin{itemize}
\item Apart from some explicitly defined special structures, the counting problem provides sufficient control for $\Kc_\Qc$, with an additional power gain\footnote{This extra gain is needed to cancel the factorial divergence coming from the number of generic couples and molecules, which is the major difficulty in \cite{DH21}. Exactly the same gain is needed also in the current work.} $L^{-cn}$ proportional to the size $n$ of the molecule $\Mb$.
\end{itemize} It is at this point that the arguments for $\gamma\approx 1$ and $\gamma\in(0,1)$ start to differ, and the new structures and ideas start to emerge.
\subsubsection{The $\gamma=1$ case} When $\gamma=1$ (and similarly when $\gamma\approx 1$), the counting problem for the molecule is solved by designing a \emph{reduction algorithm} (Section 9.4 of \cite{DH21}), where in each step we remove one or two atoms and their bonds, and reduce to the counting problem for a smaller molecule. Each such operation is either \emph{good} and favorable for counting (i.e. the desired bound for the smaller molecule implies strictly better than desired bound for the original molecule) or \emph{normal} and neutral for counting (i.e. the desired bound for the smaller molecule implies precisely the desired bound for the original molecule). Note that there is no bad operations, and the proof goes by a careful analysis of the algorithm, using suitable invariants and monotonic quantities, that shows at least a fixed small portion of all operations are good, apart from two special structures called type I and II molecular chains.

As for the special structures, type I chains are formed by double bonds only; type II chains are formed by single and double bonds (the lower right corner of Figure \ref{fig:vines}), and in the current paper we will refer to them as \emph{ladders}. These ladders are neutral for counting, and their presence does not help or harm anything, so we will mostly ignore them for the rest of this section. In comparison, type I chains cause a logarithmic loss when $\gamma=1$, but this is compensated by a delicate \emph{cancellation} between different type I chains (which come from \emph{irregular chains} in the corresponding couples), see Section 8.3 of \cite{DH21}. Such cancellation was previously unknown in either mathematical or physical literature, and is another key component of the proof in \cite{DH21}.
\subsubsection{The $\gamma=1/2$ case: Main difficulties} We now turn to the general range $\gamma\in(0,1)$. In fact, the most typical case, and in some sense the hardest case, is the ballistic scaling $\gamma=1/2$. For simplicity, we will assume this scaling for the rest of this section.

Recall that in \cite{DH21}, when $\gamma=1$, the key quality of the operations in the algorithm---namely \emph{good} and \emph{normal}---relies on what we may call the \emph{atomic} counting estimates, which involve one or two atoms (i.e. systems similar to (\ref{cplcount0})), and between $2$ and $5$ bonds (i.e. unknown vectors). For example, favorable $2$- and $3$-vector counting estimates take the form
\begin{equation}\label{atomiccount}\#\big\{(k_1,\cdots,k_q)\in \Zb_L^d\cap[0,1]^d:k_1-\cdots +k_q=k_*,|k_1|^2-\cdots+|k_q|^2=O(L^{-2\gamma})\big\}\lesssim L^{(q-1)(d-\gamma)}
\end{equation} for $q\in\{2,3\}$ and $k_*$ fixed; the $4$- and $5$-vector counting estimates involve two systems, see Lemma \ref{basiccount} (3) and (4).

Now, when $\gamma=1/2$, we can prove the same (good and normal) bounds for the $3$-, $4$- and $5$-vector counting problems just as in the $\gamma=1$ case; however, the \emph{two} vector counting bound, namely (\ref{atomiccount}) with $q=2$, breaks down. Indeed, in the worst case scenario $|k_*|\sim L^{-1}$, the left hand side of (\ref{atomiccount}) is of order $L^{d}$ which loses $L^{1/2}$ compared to (\ref{atomiccount}).

This, being the only difference between the $\gamma=1/2$ and $\gamma=1$ cases, seems to be a minor issue at first sight; however it turns out the have a huge effect on the analysis of molecules. First of all, this leads to the presence of \emph{bad} operations in the algorithm (say when one removes one atom of degree $2$), where the desired bound for the smaller molecule does \emph{not} imply the desired bound for the original molecule, so the approach in \cite{DH21} breaks down; in fact it breaks down in a much more essential way, due to the occurrence of new bad special structures.

Of course, the double bond chain---called type I chain in \cite{DH21}---are now quite bad in terms of counting, but as in \cite{DH21} they are compensated by cancellation, albeit in a slightly more subtle manner. Nevertheless, there are families of much more complicated structures, which are favorable for counting when $\gamma=1$, but becomes bad when $\gamma=1/2$, for example the one shown in Figure \ref{fig:vinesintro}. Moreover, there are also \emph{six more families} of structures which are favorable for counting when $\gamma=1$, but become neutral when $\gamma=1/2$. These are depicted in Figure \ref{fig:vines}, including the one in Figure \ref{fig:vinesintro}, and are collectively referred to as \emph{vines}.
\begin{figure}[h!]
\includegraphics[scale=0.14]{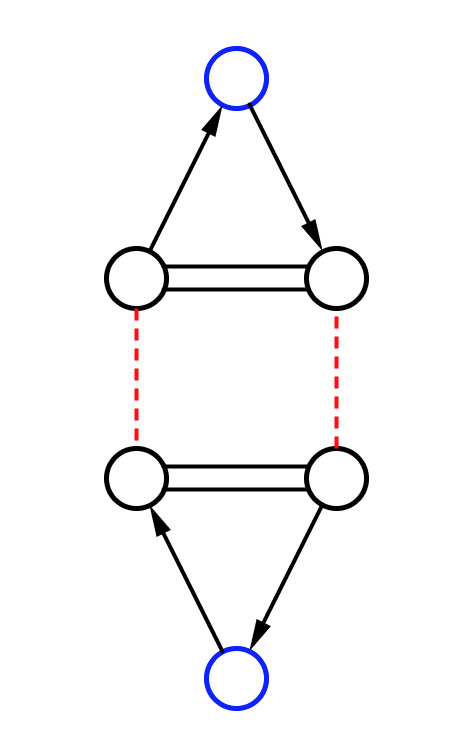}
\caption{An example of a new bad structure for $\gamma=1/2$ (see also Figure \ref{fig:vines}), called \emph{vine (II)}; here the red bonds may be replaced by a ladder.}
\label{fig:vinesintro}
\end{figure}

This means that, even in the statement of the rigidity theorem, one has to exclude, in addition to double bond chains and ladders, all the different vines as well as chains formed by them. Dealing with these new structures is the major challenge (which we discuss further in Section \ref{introcancel}), but even if we assume absence of these vines, it is not at all clear why the rigidity theorem would hold. In particular, the vines occurring in Figure \ref{fig:vines} seem sporadic and unrelated to each other, so how can we identify them from all the other structures and show they are the only bad ones?

All these considerations have led to an important modification to the algorithm, namely the addition of a new operation called \emph{cutting}. This not only makes the algorithm in \cite{DH21} much more robust, but also allows one to naturally see the occurrence of all the vines in Figure \ref{fig:vines}, which would otherwise seem to be coming from nowhere. We shall elaborate on on this further below.

\subsection{Vines, derived from cutting}\label{introcut} Now we discuss the main new ingredient in our modification to the algorithm in \cite{DH21}, namely the new operation called cutting.

Recall that the only \emph{bad} operation that can ever occur in the algorithm is $2$-vector counting, and this typically come from removing a degree $2$ atom $v$, which has two bonds $(\ell_1,\ell_2)$ of opposite directions, such that $k_{\ell_1}-k_{\ell_2}=k_*$ as in (\ref{atomiccount}) with $|k_*|\lesssim L^{-1}$. Suppose $v$ has degree $4$ \emph{before any operation} (i.e. in the original molecule), then due to the first equation in (\ref{cplcount0}), the other two bonds $(\ell_3,\ell_4)$ at $v$ must also satisfy $|k_{\ell_3}-k_{\ell_4}|\lesssim L^{-1}$. We then call such atoms $v$ a \emph{small gap}, or SG atom. In comparison, assume (under a small simplification) that all other atoms $v'$ with bonds $\ell_j'\,(1\leq j\leq 4)$ must have that $|k_{\ell_i'}-k_{\ell_j'}|\sim 1$ for any pair $(\ell_i',\ell_j')$ of opposite directions, and call them \emph{large gap} or LG atoms.

Therefore, it is the SG atoms that will cause bad operations to occur, and such bad operations are hard to control once they enter the main algorithm. To isolate the difficulties, the natural idea is then to get rid of these potentially bad atoms \emph{in the first place}, before entering the algorithmic phase. The precise operation we perform in this pre-processing stage is cutting. As its name suggests, at each step we choose a degree $4$ SG atom $v$, and split it into two atoms of degree $2$, as in Figure \ref{fig:cutintro} below (see also Figure \ref{fig:cut}).
\begin{figure}[h!]
\includegraphics[scale=0.45]{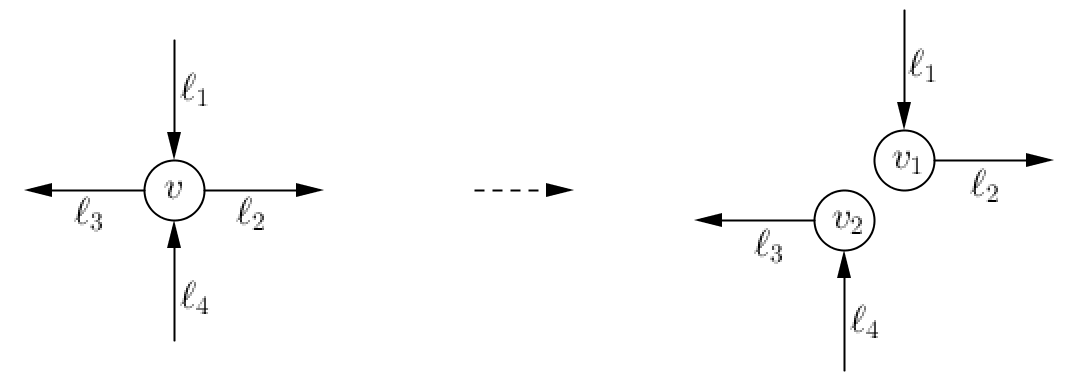}
\caption{Cutting a degree $4$ atom $v$ (Definition \ref{defcut}).}
\label{fig:cutintro}
\end{figure}

There are two cases of cutting: when this operation does not create a new connected component (which we call $\alpha$-cutting, see Definition \ref{defcut}), or when it creates a new connected component (which we call $\beta$-cutting). It turns out that $\alpha$ cuttings are favorable for counting in a certain sense, and leads to power gains instead of losses, so below we will focus on $\beta$-cuttings.
\subsubsection{Local rigidity theorems} Suppose we have done all possible cuttings (say they are all $\beta$-cuttings), then the resulting graph is composed of finitely many connected components. A typical component will contain several degree $2$ atoms that result from cutting, as well as degree $4$ atoms\footnote{There is only one component with two degree $3$ atoms, which enjoys much better estimates compared to other components, and will be neglected in the discussions below.}; the point here is that all the degree $4$ atoms must have \emph{large gap}, and so will not be involved in any bad operations. Moreover, since the number $q$ of degree $2$ atoms and the number $p$ of components satisfy $q=2p-2$, we expect that a typical component will contain exactly two degree $2$ atoms. We will fix such a component $\Mb_0$ below.

Before getting to the counting problem, we shall make one more reduction to this component. For each degree $2$ atom $v$, suppose it has two bonds $(\ell_1,\ell_2)$, then we may remove $v$ and these two bonds, and replace them by a single bond $\ell$ connecting their two other endpoints. If this operation introduces a new triple bond, then we may further remove the two endpoint atoms of this triple bond and add one new bond between the two other atoms connected to these two endpoints, and keep doing so until no more new triple bonds appear. The combined effect of this sequence of operations, which we may call the \emph{(Y) sequence}, is shown in Figure \ref{fig:yoperintro}. Let the graph resulting from this sequence be $\Mb_1$. Note that this sequence essentially corresponds to removing a degree $2$ atom together with a ladder attached to it, and recall that ladders are neutral objects for counting. In fact, this sequence is also neutral for counting, in the sense that each decoration of $\Mb_0$ provides also a decoration of $\Mb_1$, and the desired bound for the counting problem for $\Mb_1$ implies precisely the desired bound for the counting problem for $\Mb_0$.
\begin{figure}[h!]
\includegraphics[scale=0.39]{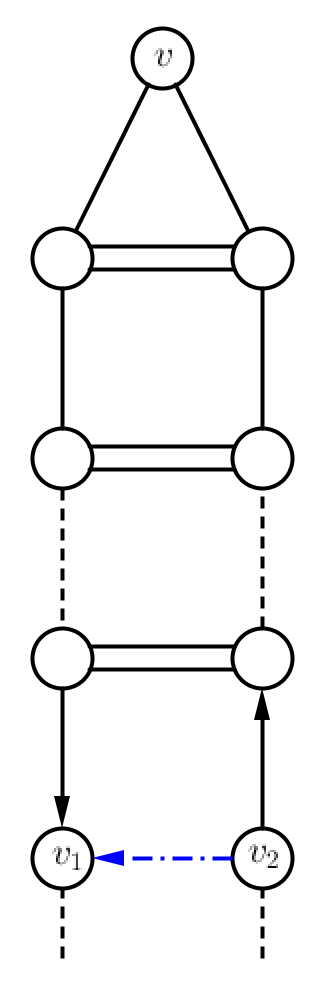}
\caption{The (Y) sequence of operations starting from a degree $2$ atom $v$ (compare also with a similar related sequence (Y2) in Figure \ref{fig:yoper}). The combined effect is removing all atoms and bonds above $v_1$ and $v_2$, and adding a new bond (colored in blue) between $v_1$ and $v_2$.}
\label{fig:yoperintro}
\end{figure}

Now the goal is to study the counting problem associated with $\Mb_1$; let the number of solutions to this counting problem be $\Cf$, and let $\chi:=E-V+F$ be the characteristics of $\Mb_1$ (where $E$ and $V$ are number of bonds and atoms in $\Mb_1$, and $F$ is the number of components which is $1$ for now). We would like to compare $\Cf$ with the quantity $L^{(d-1/2)\chi}$, and thus define $\Af:=\Cf \cdot L^{-(d-1/2)\chi}$. Then, the main result for $\Mb$ can be stated in the form of a ``local" rigidity theorem (Proposition \ref{lgmolect}, see also Proposition \ref{moleprop4}), as follows:
\begin{itemize}
\item If $\Mb$ does not equal to one of the finitely many explicitly defined ``bad" graphs, then, apart from ladders, we have $\Af\lesssim L^{-c\cdot V}$ for some constant $c>0$.
\end{itemize}

The proof of the local rigidity theorem, as well as the arguments leading to the bad graphs, relies on the exact algorithm described in Section 9.4 of \cite{DH21}; of course this is owing to the fact that we no longer have any SG atoms in $\Mb_1$. However there is also price to pay, namely that now a typical component $\Mb_1$ contains only degree $4$ atoms, as opposed having two degree $3$ atoms before doing all the cuttings. Therefore, the first step in the algorithm necessarily have to be removing a degree $4$ atom, which corresponds to a counting problem of form (\ref{atomiccount}) but with $q=4$ (note this is different from the $4$-vector counting in Lemma \ref{basiccount} (2)). This is one---but the \emph{only} one---bad operation in the algorithm, which loses power $L^{1/2}$, in the sense that the $\Af$ value before the operation is only bounded by $L^{1/2}$ times the $\Af$ value after the operation.

After this first bad operation, we no longer need to consider (\ref{atomiccount}) with $q=4$, so the remaining operations are all good or normal, due to absence of SG molecules; in summary, we have one bad operation per component (as opposed to potentially many bad operations due to SG atoms, had we not done the cuttings in the first place). Moreover, it can be shown that each good operation gains power $L^{-1/2}$ (opposite to the above, so that the $\Af$ value before the operation is bounded by $L^{-1/2}$ times the $\Af$ value after the operation). A subsequent discussion, in the same spirit as Section 9.5 of \cite{DH21}, then allows us to bound the number of good operations from below. More precisely, apart from ladders, there is only one case for $\Mb_1$ in which there is no good operation (so $\Af=L^{1/2}$), namely when $\Mb_1$ is a quadruple bond; there is also only one case for $\Mb_1$ in which there is exactly one good operation (so $\Af=1$), namely then $\Mb_1$ is a triangle of three double bonds. These are shown in Figure \ref{fig:badmol}. In \emph{all the other} cases, the number of good operations is at least two, so we have $\Af\lesssim L^{-1/2}$. If we choose $c$ small enough, this already proves the local rigidity theorem when $V\leq 100$; if $V>100$, then we can repeat the proof in Section 9.5 of \cite{DH21}---almost word by word---to get that the power gain is at least proportional to the size $V$ of the graph, which makes the $1/2$ loss in the only bad operation negligible. In the end, this allows us to prove the local rigidity theorem.
\begin{figure}[h!]
\includegraphics[scale=0.43]{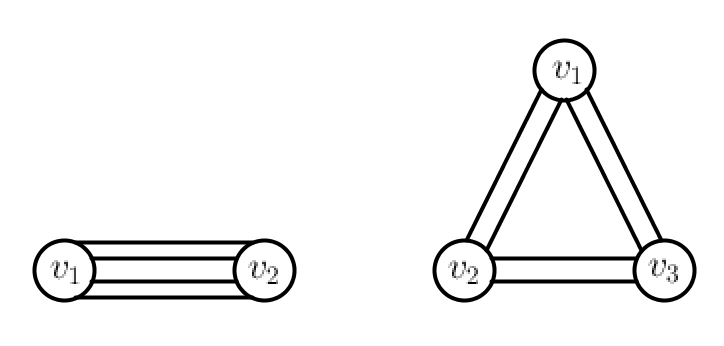}
\caption{The two bad graphs: the quadruple bond (left) and the triangle of three double bonds.}
\label{fig:badmol}
\end{figure}

\subsubsection{Vines} With the local rigidity theorem proved, we only need to check all the possibilities for components with exactly two degree $2$ atoms, which lead to the two bad cases---a quadruple bond and a triangle of three double bonds---after at most two (Y) sequences described above.  These possibilities can be found by enumeration, and exactly correspond to \emph{the families of vines (II)--(VIII)} (except vines (I) which are double bonds) in Figure \ref{fig:vines}; in fact, this is exactly how these vines are discovered. The following Figure \ref{fig:vinesintro2} shows an example of how a specific case of vine (VI) is reduced to a triangle of three double bonds after two (Y) sequences; the other cases can be shown similarly (see the proof of Lemma \ref{ylem1}).

It is now easy to get the main rigidity theorem, namely Proposition \ref{kqmainest2}, by combining the local rigidity estimates for all components after cutting; however, we still need to analyze the vines. In fact, vines (III)--(VIII) in Figure \ref{fig:vines} are neutral for counting, and do not cause any gain or loss in powers; nevertheless, each individual vine (I) or (II) would cause a serious $L^{1/2}$ power loss. These will be controlled, by some surprisingly delicate cancellations, which we discuss next. 
\begin{figure}[h!]
\includegraphics[scale=0.26]{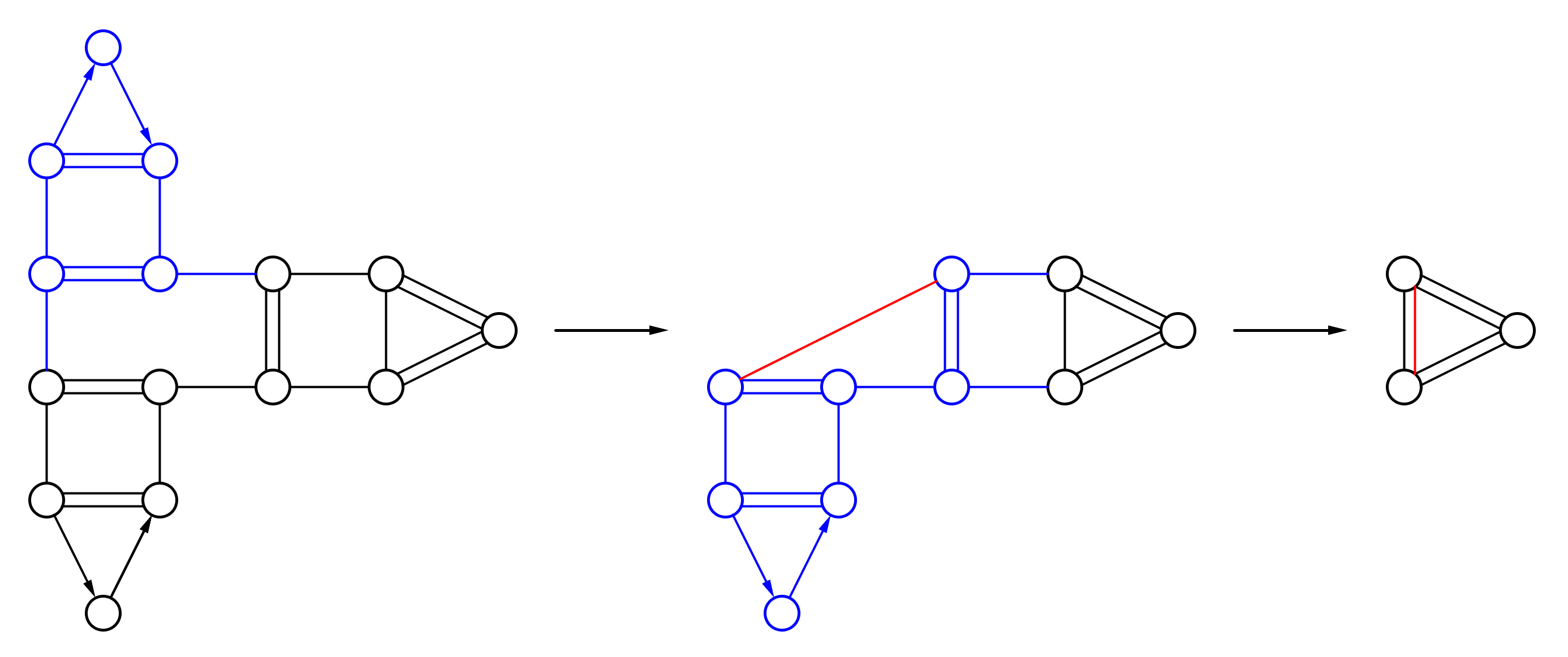}
\caption{The process of reducing a vine (VI) to a triangle of three triple bonds, using two (Y) sequences. Each time, the blue objects are removed in each sequence, and the red objects are added in the previous sequence.}
\label{fig:vinesintro2}
\end{figure}
\subsection{The miraculous cancellation}\label{introcancel} As described above, we are now left with the analysis of vines (I) and (II) (called \emph{bad vines}) in Figure \ref{fig:vines}. Since double bonds are essentially the same as in \cite{DH21}, we will focus on vines (II) in this subsection.

To exhibit the cancellation, the idea is to go back to the couple picture and enumerate all the possible (parts of) couples that correspond to a given vine (II), in the same way that chains of double bonds are shown to come from irregular chains in \cite{DH21}. In fact, it will suffice to consider only a triangle (of one double and two single bonds, see the triangle at the top of Figure \ref{fig:vinesintro}) at one end of the vine instead of the full length vine. By definition of molecule, each bond either corresponds to a branching node (\emph{parent-child} or PC bonds) that belongs to two different subsets of form $\{\nf,\nf_1,\nf_2,\nf_3\}$ as in (\ref{cplcount0}), or corresponds to a pair of leaves (\emph{leaf-pair} or LP bonds, see Definition \ref{defmole}). By considering all possibilities of each pair being either PC or LP, we get five different couple structures corresponding to a given vine (II), namely vines (II-a)--(II-e) as in Proposition \ref{molecpl} (Figure \ref{fig:block_mole}). Among the five structures (II-a)--(II-e), it turns out that vines (II-a) can be uniquely paired with vines (II-b), and vines (II-c) with vines (II-d), to give desired cancellations (vines (II-e) entails a cancellation structure in itself). Such pairs of couple structures are called \emph{twists} of each other, see Definition \ref{twist} (Figure \ref{fig:twist_dec}).
\begin{figure}[h!]
\includegraphics[scale=0.16]{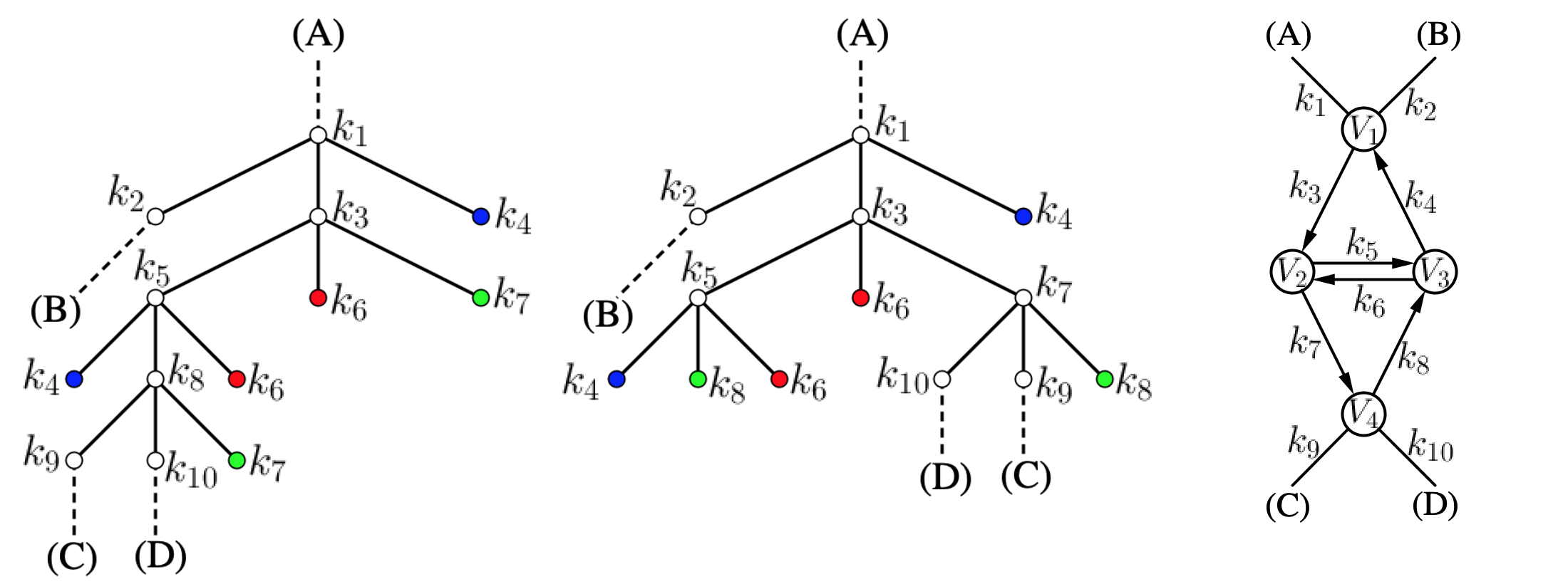}
\caption{Two couple structures, namely vines (II-c) and (II-d), that are twists of each other. Here (A)--(D) represent the remaining parts of the couples, which are the same in both cases. The corresponding molecule (which coincide in both cases) and decorations (which are in one-to-one correspondence) are also illustrated.}
\label{fig:vinescancel}
\end{figure}

In Figure \ref{fig:vinescancel}, we show one example of vine (II-c) and its twist, which is vine (II-d), which exhibit cancellation on the couple level. Note that these two structures are \emph{not} isomorphic as couples (which may be naturally defined by isomorphism of ternary trees), however, their corresponding molecules are the same. Therefore, it is reasonable to say that, compared to the notion of ternary trees and couples, which are immediately associated with the Duhamel evolution (\ref{akeqn}), it is really the notion of molecules that captures the essence of the hidden cancellation structure associated with the problem. This can also be compared to the simple cancellation of irregular chains in \cite{DH21}, shown in Figure \ref{fig:vinescancel2}, in which case the couple structure and its twist are actually isomorphic as couples.
\begin{figure}[h!]
\includegraphics[scale=0.34]{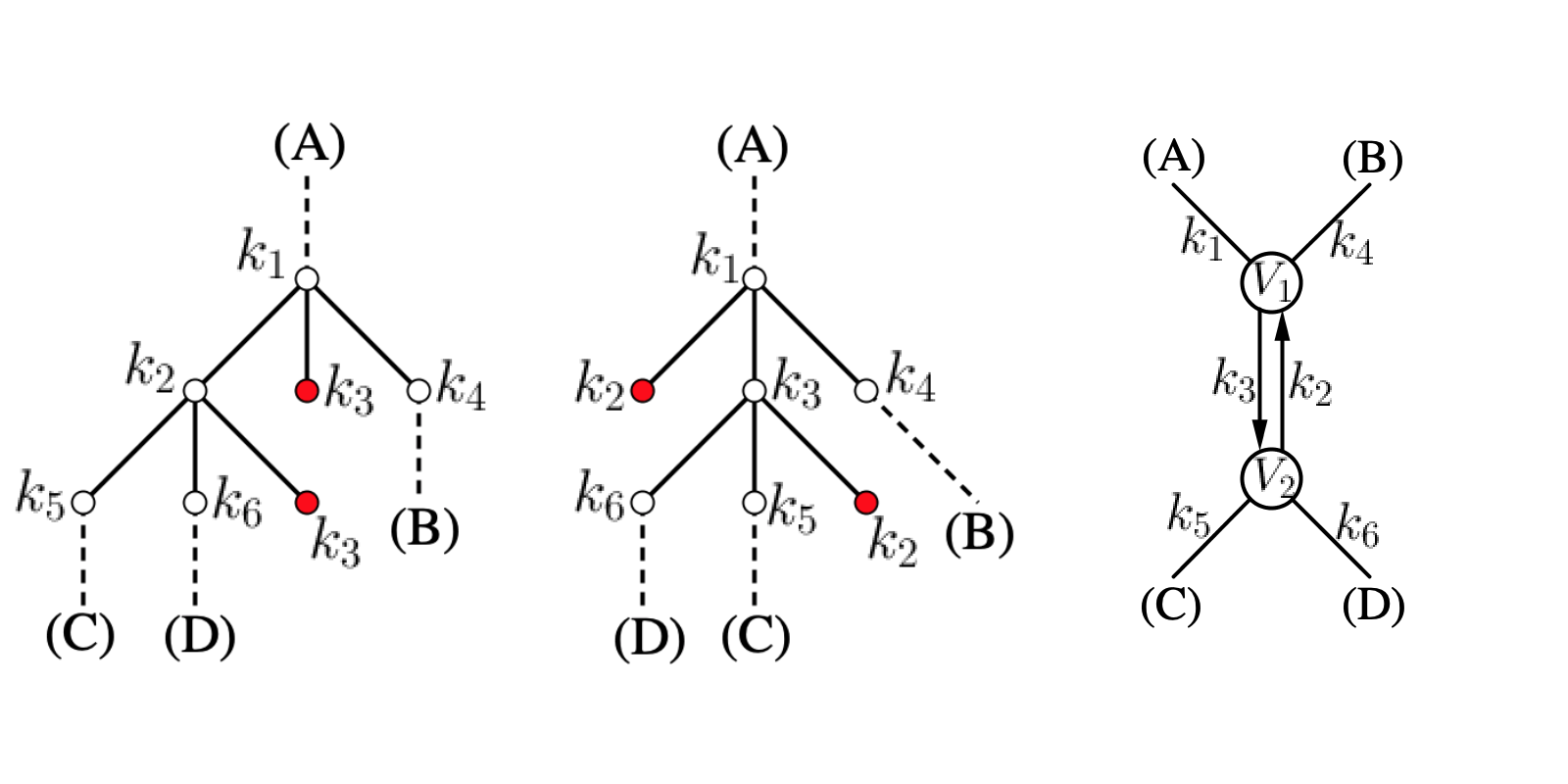}
\caption{Two couple structures corresponding to irregular chains in \cite{DH21} (or vines (I-a) and (I-b) as in Proposition \ref{molecpl}), that are twists of each other; notations are the same as in Figure \ref{fig:vinescancel}.}
\label{fig:vinescancel2}
\end{figure}

We now explain how cancellation takes place between vines (II-c) and (II-d), as shown in Figure \ref{fig:vinescancel}. Recall the expression $\Kc_\Qc$ defined in (\ref{defkq}). By examining the decorations as shown in Figure \ref{fig:vinescancel}, it is easy to see for the two corresponding couples that (i) the signs $\zeta(\Qc)$ are the opposite, (ii) the integrands $\exp(\zeta_\nf \pi i\cdot \delta L^{2\gamma}\Omega_\nf t_\nf)$ are exactly the same. The only differences are that (iii) the initial data $n_{\mathrm{in}}$ factors are
\begin{equation}\label{diffnin}n_{\mathrm{in}}(k_4)n_{\mathrm{in}}(k_6)n_{\mathrm{in}}(k_7)\mathrm{\ \ for\ the\ left\ couple},\qquad n_{\mathrm{in}}(k_4)n_{\mathrm{in}}(k_6)n_{\mathrm{in}}(k_8)\mathrm{\ \ for\ the\ right\ couple},\end{equation} and (iv) the (part of) domain $\Ec$ of time integration are
\begin{equation}\label{diffint}t_1>t_2>t_3>t_4\mathrm{\ \ for\ the\ left\ couple},\qquad t_1>t_2>\max(t_3,t_4)\mathrm{\ \ for\ the\ right\ couple}.\end{equation} Here in (iv) the time variables $t_j$ correspond to the atom $V_j$ in the molecule, which also correspond to the branching nodes decorated by $k_1,k_3,k_5,k_8$ (for the left couple) or $k_1,k_3,k_5,k_7$ (for the right couple).

We may assume the atoms $V_1$ and $V_4$ have SG (otherwise this vine would not correspond to a bad operation), which implies that $|k_7-k_8|\lesssim L^{-1}$, and thus $n_{\mathrm{in}}(k_7)\approx n_{\mathrm{in}}(k_8)$ up to negligible error. Therefore, we only need to treat the difference in the time integration in (\ref{defkq}) caused by (\ref{diffint}), which leads to the domain $t_2>t_4>t_3$. Now consider the $\zeta_\nf\Omega_\nf$ factors for the branching nodes $\nf$ decorated by $k_3$ and $k_5$, and denote them by $\Gamma_2$ and $\Gamma_3$, then from Figure \ref{fig:vinescancel} we see that
\[\Gamma_2=|k_3|^2-|k_5|^2+|k_6|^2-|k_7|^2=2(k_3-k_5)\cdot (k_3-k_7),\quad \Gamma_2+\Gamma_3=|k_3|^2-|k_4|^2+|k_8|^2-|k_7|^2=O(L^{-1}),\] noticing that $|k_3-k_4|=|k_7-k_8|\lesssim L^{-1}$. We may then assume $\Gamma_2+\Gamma_3=0$, so the expression (\ref{defkq}) will involve a part that essentially has the form
\[\int_{t_2>t_4>t_3}\sum_{x,y} e^{2\pi iL^{2\gamma}(t_2-t_3)(x\cdot y)}W(x,y)\,\mathrm{d}t_2\mathrm{d}t_3\mathrm{d}t_4,\] where $(x,y)=(k_3-k_5,k_3-k_7)$ and $W$ is a well-behaved function. The sum in $(x,y)$ can be calculated similar to regular couples in \cite{DH21}, however the leading term \emph{vanishes precisely because $t_2>t_4>t_3$}, due to the time integral having zero average as a function of $x\cdot y$, see Lemma \ref{sumintest1}. This cancellation then provides enough decay and allows us to control the contribution of vines.
\begin{rem}Such cancellation for vines, as described above, seems to be new in both the mathematical and physical literature. It seems quite miraculous, and it might have some physical interpretation, or be part of a more general cancellation mechanism for Feynman diagrams. However, such interpretation is still unclear at this point.
\end{rem}
We summarize the proof of Propositions \ref{mainprop1} and \ref{mainprop4} in the following flowchart:
\begin{figure}[h!]
\includegraphics[scale=0.34]{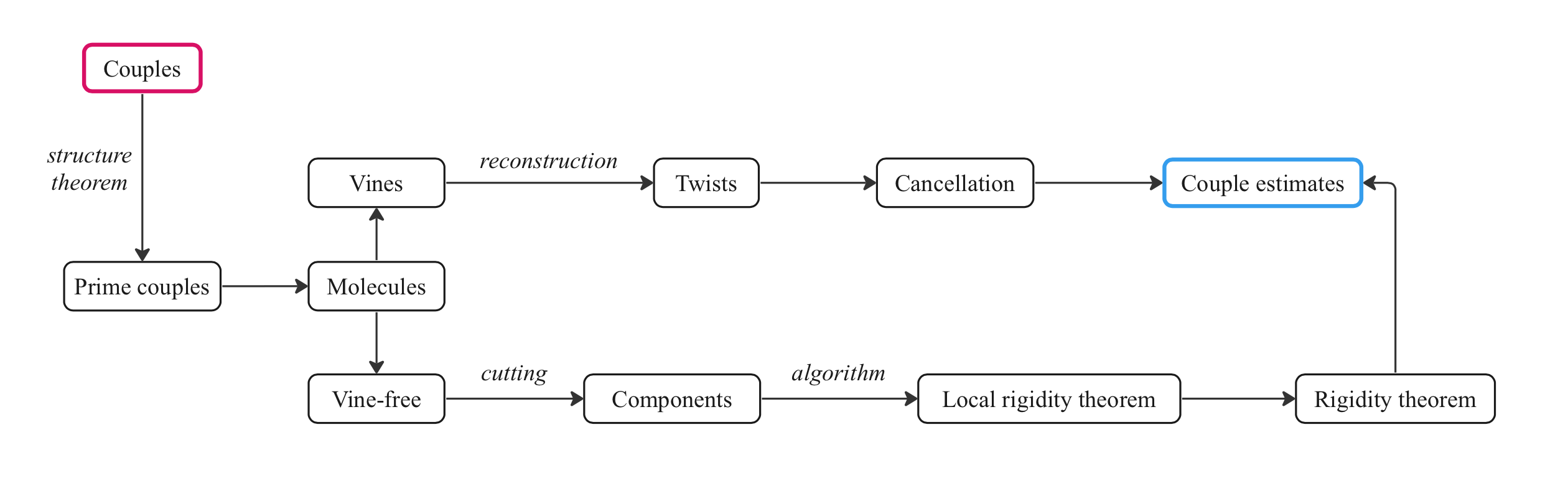}
\caption{A flowchart explaining the process of the proof of Propositions \ref{mainprop1} and \ref{mainprop4}. After the reduction to prime couples, one runs parallel arguments to (a) bound the vines contained in the couple, using cancellation arguments if needed, and (b) perform the combinatorial analysis of the couples absent the problematic vines. The latter proceeds via the cutting operation followed by a local rigidity theorem in the spirit of that in \cite{DH21}, resulting in a global rigidity theorem that gives the needed gain from such couples.}
\label{fig:flowchart}
\end{figure}
\subsection{Construction of a parametrix} Finally we turn to Proposition \ref{mainprop2}. Recall the equation (\ref{eqnbk2}) satisfied by $\textit{\textbf{b}}$. As pointed out in \cite{DH21}, we do not need to bound the norm of $\Ls$ in any function space in order to solve this equation, but only need to invert the operator $1-\Ls$. This then requires to construct a parametrix $\Xs$ to $1-\Ls$, as stated in Proposition \ref{mainprop2}. In \cite{DH21}, this parametrix is simply defined, using Neumann series, as $\Xs=1+\Ls+\cdots +\Ls^N$ for large $N$; but such construction would run into a problem here, because it is not compatible with the vine cancellation structure.

The solution is to use the notion of \emph{flower trees} and \emph{flower couples} introduced in \cite{DH21}; this is natural, as these structures are already used to obtain bounds for powers of $\Ls$ in \cite{DH21}. Here, instead of sticking to powers of $\Ls$, we construct $\Xs$ by using these structures directly, which allows us to group the flower couples that occur, in the precise way that allows for all the needed cancellations. Apart from these, the proof of Proposition \ref{mainprop2} relies on the same arguments as in the proof of Propositions \ref{mainprop1} and \ref{mainprop4}, with only minor modifications. See Section \ref{linoper1} for details.
\subsection{Plan of this paper} In Section \ref{molecules} we define and study the structure of molecules. In Section \ref{secvine} we introduce the key new objects called \emph{vines}. In Section \ref{regular} we study expressions associated with regular couples and regular trees, and prove estimates which are the same as in \cite{DH21} but with more precise error bounds and a new cancellation structure. In Section \ref{funcgroup} we study similar expressions associated wth vines, and prove two key estimates exploiting the cancellation between vines.

With these preparations, we present the proof of Propositions \ref{mainprop1} and \ref{mainprop4} in Sections \ref{reduct1}--\ref{lgmole}: in Section \ref{reduct1} (stage 1) we reduce them to Proposition \ref{kqmainest1} and then \ref{kqmainest2}, in Section \ref{reduct2} (stage 2) we further reduce them to Proposition \ref{lgmolect}, and in Section \ref{lgmole} we prove Proposition \ref{lgmolect}. Finally, in Section \ref{linoper1} we prove Proposition \ref{mainprop2}, and in Section \ref{linoper2} we prove Theorem \ref{main}. The various auxiliary results used in this paper are listed and proved in Appendix \ref{aux}.
\section{Couples and molecules}\label{molecules}
\subsection{Definition of molecules} We start by defining molecules and related notions as in \cite{DH21}.
\begin{df}[Molecules]\label{defmole0} A \emph{molecule} $\Mb$ is a directed graph, formed by vertices (called \emph{atoms}) and edges (called \emph{bonds}), where multiple bonds are allowed\footnote{We do not allow self-connecting bonds here; see Remark \ref{moleremark} and Section \ref{extradegen}.}, and each atom has out-degree at most $2$ and in-degree at most $2$. We write $v\in \Mb$ and $\ell\in\Mb$ for atoms $v$ and bonds $\ell$ of $\Mb$, and write $\ell\sim v$ if $v$ is an endpoint of $\ell$. We further require that $\Mb$ does not have any connected component with only degree $4$ atoms (we call such components \emph{saturated}), where connectivity is always understood in terms of undirected graphs. For distinction, if a directed graph is otherwise like molecules but may contain saturated components, we will call it a \emph{pseudomolecule}.

An \emph{atomic group} in a molecule is a subset of atoms, together with all bonds between these atoms. Some particular atomic groups, or families of atomic groups, will play important roles in our proof (such as the \emph{vines} (I)--(VIII) defined in Section \ref{subsecvine}). Given any molecule $\Mb$, we define $V$ to be the number of atoms, $E$ the number of bonds, and $F$ the number of connected components. Define the characteristics $\chi:=E-V+F$.
\end{df}
\begin{df}[Molecule of couples]\label{defmole} Let $\Qc$ be a nontrivial couple, we will define a directed graph $\Mb=\Mb(\Qc)$ associated with $\Qc$ as follows. The atoms are all the branching nodes $\nf\in\Nc$ of $\Qc$. For any two atoms $\nf_1$ and $\nf_2$, we connect them by a bond if either (i) one of them is the parent of the other, or (ii) a child of $\nf_1$ is paired to a child of $\nf_2$ as leaves. In case (i) we label this bond by PC, and place a label P at the parent atom, and place a label C at the child atom in case (ii) we label this bond by LP. Note that one atom $v$ may have multiple P and C labels coming from different bonds $\ell\sim v$.

We fix the direction of each bond as follows. Any LP bond should go from the atom whose paired child has $-$ sign to the one whose paired child has $+$ sign. Any PC bond should go from the P atom to the C atom if the C atom has $-$ sign as a branching node in $\Qc$, and should go from the C atom to the P atom if the C atom has $+$ sign.

For any atom $v\in\Mb(\Qc)$, let $\nf=\nf(v)$ be the corresponding branching node in $\Qc$. For any bond $\ell\sim v$, define also $\mf=\mf(v,\ell)$ such that (i) if $\ell$ is PC with $v$ labeled C, then $\mf=\nf$; (ii) if $\ell$ is PC with $v$ labeled P, then $\mf$ is the branching node corresponding to the other endpoint of $\ell$ (which is a child of $\nf$); (iii) if $\ell$ is LP then $\mf$ is the leaf in the leaf pair defining $\ell$ that is a child of $\nf$.
\end{df}
\begin{rem}\label{remlable} The molecule $\Mb(\Qc)$ is actually a \emph{labeled molecule} because of the labels LP and PC on bonds (and P and C on atoms). This feature is specific to molecules coming from a couple. Below we will not be too strict in distinguishing a molecule (which is just a direct graph) with a labeled molecule, but this difference does sometimes occur (see e.g. Remark \ref{twistexplain} (a)).
\end{rem}
\begin{prop}\label{moleproperty} For any nontrivial couple $\Qc$ with order $n$, the directed graph $\Mb=\Mb(\Qc)$ defined in Definition \ref{defmole} is a molecule. It has $n$ atoms and $2n-1$ bonds, in particular it is connected and has either two atoms of degree $3$ or one atom of degree $2$, with the remaining atoms all having degree $4$.

For any atom $v$, let $\nf=\nf(v)$, then the values of $\mf(v,\ell)$ where $\ell\sim v$ form a subset of $\{\nf,\nf_1,\nf_2,\nf_3\}$ where $\nf_j$ are children of $\nf$. When $v$ has degree $4$ the equality holds, and when $v$ has degree $2$ or $3$, some of the nodes in $\{\nf,\nf_1,\nf_2,\nf_3\}$ will not correspond to a bond $\ell$.
\end{prop}
\begin{proof} For connectivity see Proposition 9.4 of \cite{DH21}. The rest follows directly from definitions.
\end{proof}
\begin{prop}\label{recover} Given a molecule $\Mb$ with $n$ atoms as in Definition \ref{defmole0}, the number of couples $\Qc$ (if any) such that $\Mb(\Qc)=\Mb$ is at most $C^n$.
\end{prop}
\begin{proof} See Proposition 9.6 of \cite{DH21}.
\end{proof}
\begin{df}[Decoration of molecules]\label{decmole} Given a molecule or pseudomolecule $\Mb$ (Definition \ref{defmole}), suppose we also fix the vectors $c_v\in\Zb_L^d$ for each $v\in\Mb$ such that $c_v=0$ when $v$ has degree $4$, then we can define a $(c_v)$-\emph{decoration} (or just a decoration) of $\Mb$ to be a set of vectors $(k_\ell)$ for all bonds $\ell\in\Mb$, such that $k_\ell\in\Zb_L^d$ and
\begin{equation}\label{decmole1}\sum_{\ell\sim v}\zeta_{v,\ell}k_\ell=c_v
\end{equation} for each atom $v\in\Mb$. Here the sum is taken over all bonds $\ell\sim v$, and $\zeta_{v,\ell}$ equals $1$ if $\ell$ is outgoing from $v$, and equals $-1$ otherwise. For each such decoration and each atom $v$, define also that
\begin{equation}\label{defomegadec}\Gamma_v=\sum_{\ell\sim v}\zeta_{v,\ell}|k_\ell|^2.
\end{equation}

Suppose $\Mb=\Mb(\Qc)$ comes from a nontrivial couple $\Qc$, and for $k\in\Zb_L^d$, we define a $k$-decoration of $\Mb$ to be a $(c_v)$-decoration where $(c_v)$ is given by 
\begin{equation}\label{molecv}
c_v=\left\{
\begin{aligned}&0,&\textrm{if }&v\textrm{ has degree }2\textrm{ or }4,\\
+&k,&\textrm{if }&v\textrm{ has out-degree }2\textrm{ and in-degree }1,\\
-&k,&\textrm{if }&v\textrm{ has out-degree }1\textrm{ and in-degree }2. 
\end{aligned}
\right.
\end{equation} Given any $k$-decoration of $\Qc$ in the sense of Definition \ref{defdec}, define a $k$-decoration of $\Mb(\Qc)$ such that $k_\ell=k_{\mf(v,\ell)}$ for an endpoint $v$ of $\ell$. It is easy to verify that this $k_\ell$ is well-defined  (i.e. does not depend on the choice of $v$), and it gives a one-to-one correspondence between $k$-decorations of $\Qc$ and $k$-decoration of $\Mb(\Qc)$. Moreover for such decorations we have
\begin{equation}\label{molegammav}
\Gamma_v=\left\{
\begin{aligned}&0,&\textrm{if }&v\textrm{ has degree }2,\\
-&\zeta_{\nf(v)}\Omega_{\nf(v)},&\textrm{if }&v\textrm{ has degree }4,\\
-&\zeta_{\nf(v)}\Omega_{\nf(v)}+|k|^2,&\textrm{if }&v\textrm{ has out-degree }2\textrm{ and in-degree }1,\\
-&\zeta_{\nf(v)}\Omega_{\nf(v)}-|k|^2,&\textrm{if }&v\textrm{ has out-degree }1\textrm{ and in-degree }2. 
\end{aligned}
\right.
\end{equation}

Finally, given $\beta_v\in\Rb$ for each $v\in\Mb$ and $k_\ell^0\in\Zb_L^d$ for each $\ell\in\Mb$, we define a decoration $(k_\ell)$ to be \emph{restricted by} $(\beta_v)$ and/or $(k_\ell^0)$, if we have $|\Gamma_v-\beta_v|\leq \delta^{-1}L^{-2\gamma}$ for each $v$ and/or $|k_\ell-k_\ell^0|\leq 1$ for each $\ell$.
\end{df}
\begin{rem}\label{moleremark} As stated in Remark \ref{nonresrem}, for simplification, we will assume there is no degenerate case (i.e. $k_2\in\{k_1,k_3\})$ in any $k$-decoration of any couple $\Qc$ we are considering. In particular, we may assume that no two siblings are paired as leaves in $\Qc$, and more generally there are no two sibling nodes such that the leaves of the subtrees rooted at them are completely paired. In particular, there is no degree $2$ atom or self-connecting bond in the associated molecule $\Mb=\Mb(\Qc)$, as well as the skeleton $\Qc_{\mathrm{sk}}$ defined in Proposition \ref{skeleton} below (so they must have two atoms of degree $3$ with the rest atoms having degree $4$, due to Proposition \ref{moleproperty}). These may be violated if there is degenerate case, but the latter is easily addressed, see Section \ref{extradegen}.
\end{rem}
\subsection{Regular couples and regular trees}
Recall the following definitions of regular couples and regular trees in \cite{DH21}.
\begin{df}\label{defmini} A \emph{$(1,1)$-mini couple} is a couple formed by two trees of order $1$ with no siblings paired. A \emph{mini tree} is a saturated tree of order $2$, again with no siblings paired; see Figure \ref{fig:mini}. Note that if $\Qc$ is a $(1,1)$-mini couple or a couple formed by a mini tree with the trivial tree (called a $(2,0)$-mini couple in \cite{DH21}), then $\Mb(\Qc)$ is exactly  one triple bond.

For any couple $\Qc$ we can define two operations: operation $A$ where a leaf pair is replaced by a $(1,1)$-mini couple, and operation $B$ where a node is replaced by a mini tree, see Figure \ref{fig:oper}. Then, we define a couple $\Qc$ to be \emph{regular} if it can be formed, starting from the trivial couple $\times$, by operations $A$ and $B$. We also define a saturated paired tree $\Tc$ to be a \emph{regular tree}, if $\Tc$ forms a regular couple with the trivial tree $\bullet$. Clearly the order of any regular couple and regular tree must be even.
\end{df}
\begin{figure}[h!]
\includegraphics[scale=0.2]{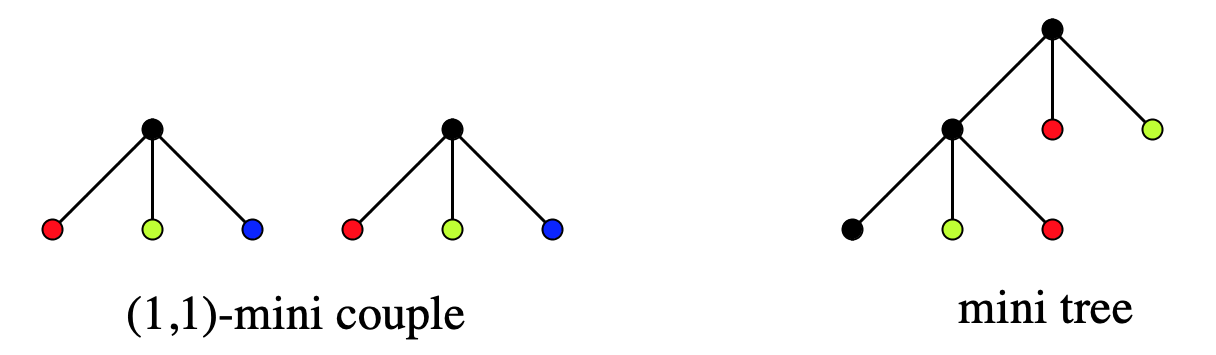}
\caption[$(1,1)$-mini couple and mini tree]{Example of a $(1,1)$-mini couple and a mini tree defined in Definition \ref{defmini}; the exact positions of nodes and pairings may vary (for example the coloring of the three leaves at the right branching node in the $(1,1)$-mini couple might be blue, green, red instead of red, green, blue, etc.).}
\label{fig:mini}
\end{figure}
\begin{figure}[h!]
\includegraphics[scale=0.4]{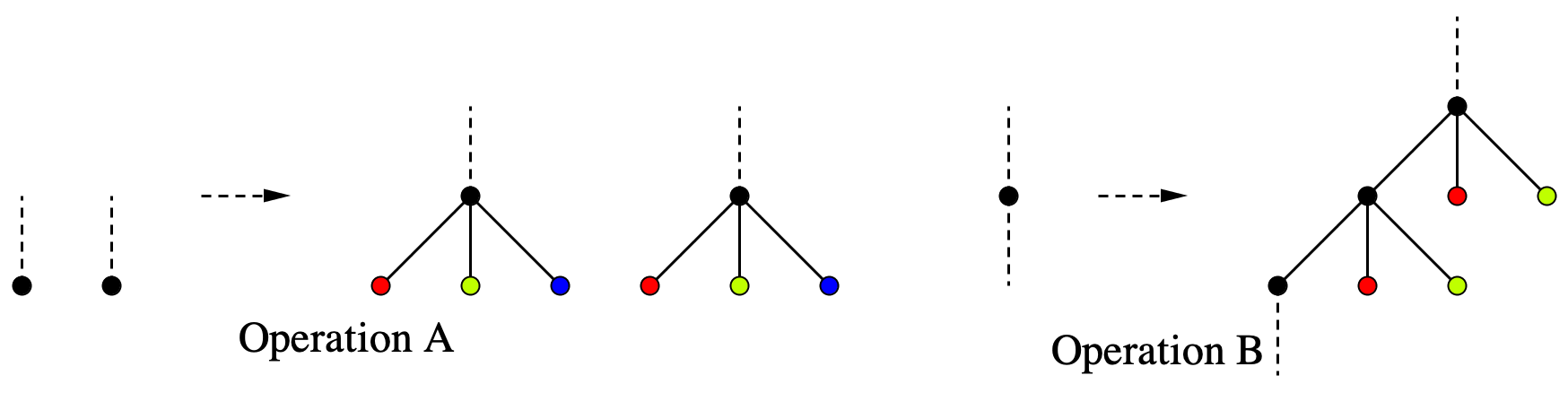}
\caption[Operations $A$ and $B$]{Example of operations $A$ and $B$ defined in Definition \ref{defmini}. The exact positions of nodes may vary.}
\label{fig:oper}
\end{figure}
\begin{prop}\label{skeleton} For any couple $\Qc$ there is a unique couple $\Qc_{\mathrm{sk}}$, which is \emph{prime} in the sense that it does not contain any $(1,1)$-mini couple or mini tree, such that $\Qc$ is constructed from $\Qc_{\mathrm{sk}}$ in a unique way, by replacing each leaf pair with a regular couple, and each branching node with a regular tree. This $\Qc_{\mathrm{sk}}$ is called the \emph{skeleton} of $\Qc$;  see Figure \ref{fig:skeleton} for an illustration. The molecule $\Mb(\Qc_{\mathrm{sk}})$ does not contain a triple bond, and $\Qc$ is regular if and only if $\Qc_{\mathrm{sk}}$ is trivial. Moreover, the number of couples $\Qc$ with order $n$ and fixed skeleton $\Qc_{\mathrm{sk}}$ is at most $C^n$.

More generally, let $\Qc_0$ be any couple (not necessarily prime), one may form a couple $\Qc$ by replacing each leaf pair $(\lf,\lf')$ with a regular couple $\Qc^{(\lf,\lf')}$ and each branching node $\mf$ with a regular tree $\Tc^{(\mf)}$. We shall denote the the collection of all these $\Qc^{(\lf,\lf')}$ and $\Tc^{(\mf)}$ by $\As$, and write $\Qc\sim(\Qc_0,\As)$. Define $n(\As)$ to be the total order of regular couples and regular trees $\As$; we may use $\Bs$ etc. to denote suitable sub-collections if $\As$, and $n(\Bs)$ etc. are defined similarly.
\begin{figure}[h!]
\includegraphics[scale=0.4]{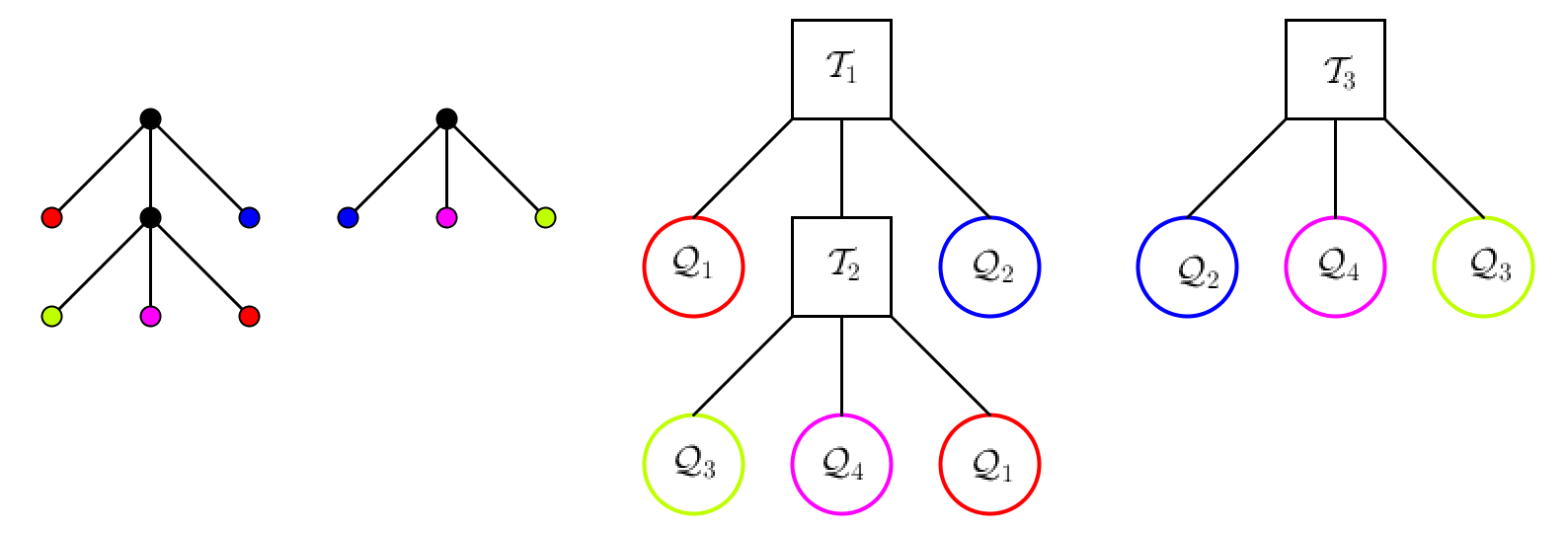}
\caption[A couple with given skeleton]{A couple $\Qc$ (on the right) together with its skeleton $\Qc_{\mathrm{sk}}$ (on the left), which is a prime couple. The structure of $\Qc$ is as in Proposition \ref{skeleton}, where each $\Tc_j$ (drawn as a black box) represents a regular tree, and each $\Qc_j$ (drawn as two circles of same color) represents a regular couple.}
\label{fig:skeleton}
\end{figure}
\end{prop}
\begin{proof} See Proposition 4.14 and Remark 4.15 of \cite{DH21}. The molecule $\Mb(\Qc_{\mathrm{sk}})$ does not have triple bond, because $\Qc_{\mathrm{sk}}$ is a prime couple.
\end{proof}
\subsection{Blocks}\label{deffuncgroup} We next define the notion of blocks (and hyper-blocks), which is an important class of atomic groups that occur in our proof.
\begin{df}[Blocks]\label{defblock}Given a molecule $\Mb$, an atomic group $\Bb\subset\Mb$ is called a \emph{block}, if all atoms in $\Bb$ have degree $4$ within $\Bb$, except for exactly two atoms $v_1$ and $v_2$ (called \emph{joints} of the block), each of which having out-degree $1$ and in-degree $1$ (hence total degree $2$) within $\Bb$, see Figure \ref{fig:block}. Define $\sigma(\Bb)$ as the number of bonds between $v_1$ and $v_2$. Note that $\sigma(\Bb)\in\{0,1,2\}$, and $\sigma(\Bb)=2$ if and only if $\Bb$ is a double bond. Moreover, we define a \emph{hyper-block} $\Hb$ to be the atomic group formed by adding one bond between the two joints $v_1$ and $v_2$ of a block $\Bb$ (we shall call this $\Hb$ the \emph{adjoint} of $\Bb$), and define $\sigma(\Hb)=\sigma(\Bb)+1$.

If two blocks share one common joint and no other common atom, then their union (or concatenation) is either a block or a hyper-block (depending on whether the two other joints of the two blocks are connected by a bond), see Figure \ref{fig:blockconc}. Note that a hyper-block cannot be concatenated with another block or hyper-block in this way. In general any finitely many (at least two) blocks can be concatenated to form a new block $\Bb$, or a new hyper-block $\Hb$, in which case we must have $\sigma(\Bb)=0$ and $\sigma(\Hb)=1$.
\end{df}
  \begin{figure}[h!]
  \includegraphics[scale=.35]{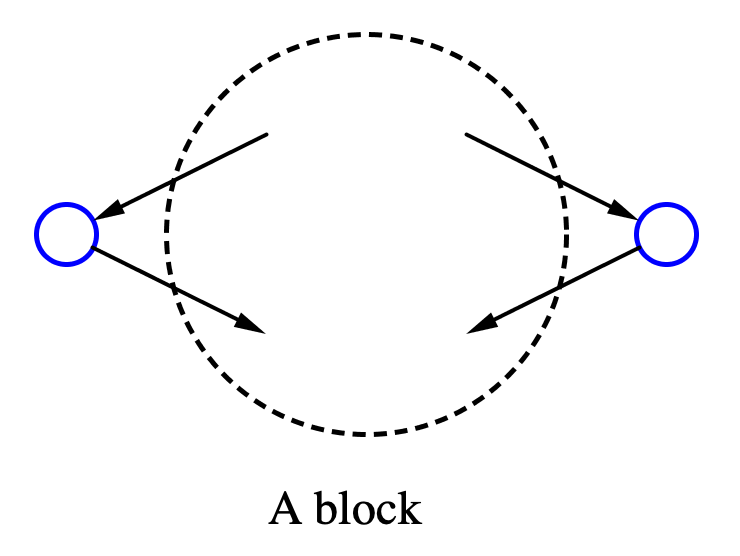}
  \caption[Block]{A block as in Definition \ref{defblock}. Here the two joint atoms are colored blue, and all atoms in the circle have degree $4$.}
  \label{fig:block}
\end{figure} 
  \begin{figure}[h!]
  \includegraphics[scale=.35]{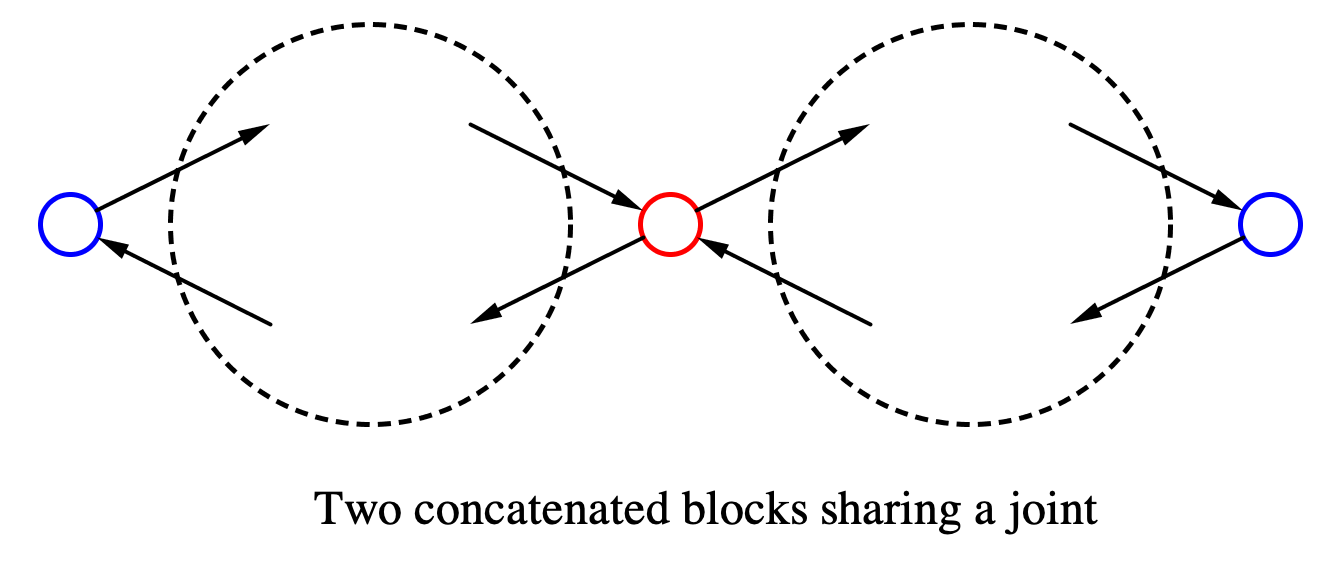}
  \caption[Chain of two blocks]{A chain of two blocks as in Definition \ref{defblock}. Here the two joint atoms at the end are colored blue, the one common joint atom colored red, and all atoms in the circles have degree $4$. Clearly longer chains can be constructed similarly.}
  \label{fig:blockconc}
\end{figure} 
\begin{lem}\label{disjointlem} Let $\Mb$ be a molecule. Suppose $\Ab,\Bb\subset\Mb$, each of them is a block or a hyper-block, and $\Ab\not\subset\Bb$, $\Bb\not\subset\Ab$ and $\Ab\cap\Bb\neq\varnothing$.

Let $a_1$ and $a_2$ be the joints of $\Ab$, and $b_1$ and $b_2$ be the joints of $\Bb$. Suppose further that (i) $\Bb\backslash\{b_1,b_2\}$ is connected, and (ii) for any $v\in\Bb\backslash\{b_1,b_2\}$, the subset $\Bb\backslash\{v\}$ is either connected, or has two connected components containing $b_1$ and $b_2$ respectively, and (iii) the same holds for $\Ab$. 

Then $\Ab$ and $\Bb$ are both blocks, and exactly one of the three following scenarios happens: (a) $\Ab$ and $\Bb$ share two common joints and no other common atom, and $\sigma(\Ab)=\sigma(\Bb)=1$, (b) $\Ab$ and $\Bb$ share one common joint and no other common atom, and can be concatenated like in Definition \ref{defblock}; (c) $\Ab$ is formed by concatenating two blocks $\Cb_0$ and $\Cb_1$, and $\Bb$ is formed by concatenating $\Cb_1$ with another block $\Cb_2$ (where $\Cb_0\cap\Cb_2=\varnothing$).
\end{lem}
\begin{proof} (1) Suppose $\Ab$ and $\Bb$ share two common joints, say $a_1=b_1$ and $a_2=b_2$. If a third atom $u\in\Ab\cap\Bb$, since $\Ab\not\subset\Bb$, there must exist another atom $v\in\Ab\backslash\Bb$. Since $\Ab\backslash\{a_1,a_2\}$ is connected by assumption (i), we can find a path from $u$ to $v$ that remains in $\Ab$ but does not include either $a_1$ or $a_2$. However we have $u\in\Bb$ and $v\not\in\Bb$, so any path from $u$ to $v$ must include either $b_1$ or $b_2$, contradiction. This tells us that $\Ab\cap\Bb=\{a_1,a_2\}$. In this case there must be one (and exactly one) bond between $a_1$ and $a_2$, so $\sigma(\Ab)=\sigma(\Bb)=1$ and we are in scenario (a). In fact, if $\sigma(\Ab)=\sigma(\Bb)=0$, then $a_1$ has two bonds connecting to atoms in $\Ab$ and two other bonds connecting to atoms in $\Bb$, and same for $a_2$.  Therefore \emph{every} atom in $\Ab\cup\Bb$ will have degree $4$ (including $a_1$ and $a_2$), which contradicts the definition of molecule. The other cases are treated similarly.

(2) Suppose $\Ab$ and $\Bb$ share no common joint. Choose $u\in\Ab\cap\Bb$ and $v\in\Ab\backslash\Bb$, the same argument in (1) implies that either $b_1$ or $b_2$ must be an \emph{interior} (i.e. non-joint) atom of $\Ab$. Similarly, either $a_1$ or $a_2$ must be an interior atom of $\Bb$. However, these four atoms cannot be all interior atoms because otherwise every atom in $\Ab\cup\Bb$ will again have degree $4$. By symmetry, we may assume that $a_1$ is an interior atom of $\Bb$, $b_1$ is an interior atom of $\Ab$, and $b_2\not\in\Ab$.

Now consider the atomic group $\Ab\backslash\{b_1\}$, which is the disjoint union of $\Ab\backslash\Bb$ and $(\Ab\cap\Bb)\backslash\{b_1\}$. If two atoms $u$ and $v$ from these two subsets are connected by a path in $\Ab\backslash\{b_1\}$, then we again have a contradiction because this path cannot include either $b_1$ or $b_2$. Using also assumption (ii), we know that $\Ab\backslash\Bb$ and $(\Ab\cap\Bb)\backslash\{b_1\}$ are two connected components of $\Ab\backslash\{b_1\}$, and $a_2\in\Ab\backslash\Bb$. It is now easy to see that $\Cb_0:=(\Ab\backslash\Bb)\cup\{b_1\}$ and $\Cb_1:=\Ab\cap\Bb$ are two blocks that are concatenated at the common joint $b_1$ to form $\Ab$. Now by switching $\Ab$ and $\Bb$ and arguing similarly, we can see that $\Cb_2=(\Bb\backslash\Ab)\cup\{a_1\}$ is also a block, and is concatenated with $\Cb_1$ at the common joint $a_1$ to form $\Bb$. Therefore, we are in scenario (c).

(3) Finally, suppose $\Ab$ and $\Bb$ share only one common joint, say $a_1=b_1$. If $\Ab\cap\Bb=\{a_1\}$, then clearly we are in scenario (b). If not, then there is a second atom $u\in\Ab\cap\Bb$. By repeating the arguments in (1) and (2), we know that $a_2$ is an interior atom of $\Bb$ and $b_2$ is an interior atom of $\Ab$. Then all atoms in $\Ab\cup\Bb$ except $a_1$ will have degree $4$, thus $a_1$ can only have degree $2$ (in-degree $1$ and out-degree $1$, as total in-degree must equal total out-degree), which means that $a_1$ has two bonds connecting to atoms in $\Ab\cap\Bb$. Now we can apply the same argument in (2) and conclude that $\Ab\backslash\Bb$ and $(\Ab\cap\Bb)\backslash\{b_2\}$ are two connected components of $\Ab\backslash\{b_2\}$, and $a_2\in\Ab\backslash\Bb$. But we already know $a_2$ is an interior atom of $\Bb$, which is impossible. This contradiction completes the proof.
\end{proof}
\subsubsection{Blocks in a couple} We now discuss the relative position of a block $\Bb\subset\Mb(\Qc)$ in a couple $\Qc$.
\begin{prop}\label{block_clcn} Let $\Qc$ be a couple and $\Bb\subset\Mb(\Qc)$ be a block with two joints $v_1$ and $v_2$, and let $\uf_j=\nf(v_j)$.
\begin{enumerate}[{(1)}]
\item Then (up to symmetry) exactly one of the following two scenarios happens.
\begin{itemize}
\item  \emph{\bf{(CL) or ``cancellation" blocks:}} There is a child $\uf_{11}$ of $\uf_1$ and two children $\uf_{21},\uf_{22}$ of $\uf_2$, such that (i) $\uf_{11}$ has the same sign as $\uf_1$, $\uf_{21}$ has sign $+$ and $\uf_{22}$ has sign $-$, (ii) $\uf_2$ is a descendant $\uf_1$ but not of $\uf_{11}$, and (iii) all the leaves in the set $\Qc[\Bb]$ are completely paired, where $\Qc[\Bb]$ denotes all nodes that are descendants of $\uf_1$ but not of $\uf_{11},\uf_{21}$ or $\uf_{22}$ (in particular $\uf_1\in\Qc[\Bb]$ and $\uf_{11},\uf_{21},\uf_{22}\not\in\Qc[\Bb]$). See Figure \ref{fig:couples_cl}.
\item \emph{\bf{(CN) or ``connectivity" blocks:}} There is a child $\uf_{11}$ of $\uf_1$ and $\uf_{21}$ of $\uf_2$, such that (i) $\uf_{11}$ has the same sign as $\uf_1$ and $\uf_{21}$ has the same sign as $\uf_2$, (ii) $\uf_2$ is either a descendant of $\uf_{11}$ or not a descendant of $\uf_1$ (similar for $\uf_1$), and (iii) all the leaves in the set $\Qc[\Bb]$ are completely paired, where $\Qc[\Bb]$ denotes all the nodes that are descendants of $\uf_1$ but not of $\uf_{11}$, and all the nodes that are descendants of $\uf_2$ but not of $\uf_{21}$ (in particular $\uf_1,\uf_2\in\Qc[\Bb]$ and $\uf_{11},\uf_{21}\not\in\Qc[\Bb]$). See Figure \ref{fig:couples_cn}.
\end{itemize}
\item For (CL) blocks we can define a new couple $\Qc^{\mathrm{sp}}$ by removing all nodes $\mf\in\Qc[\Bb]\backslash\{ \uf_1\}$, and turning $\uf_{11},\uf_{21}$ and $\uf_{22}$ into the three new children of $\uf_1$ with corresponding subtrees attached; here the position of $\uf_{11}$ as a child of $\uf_1$ remains the same as in $\Qc$, and the positions of $\uf_{21}$ and $\uf_{22}$ as children of $\uf_1$ are determined by their signs. Then, the molecule $\Mb^{\mathrm{sp}}=\Mb(\Qc^{\mathrm{sp}})$ is formed from $\Mb$ by merging all the atoms in $\Bb$ (including two joints) into one single atom. We call this operation going from $\Qc$ to $\Qc^{\mathrm{sp}}$ \emph{splicing}.
\item For (CN) blocks, if we remove from $\Mb(\Qc)$ \emph{any set of disjoint (CN) blocks} in $\Qc$, where by removing a block we mean removing all bonds $\ell\in\Bb$, then the resulting molecule is still \emph{connected} (though it no longer comes from a couple). This remains true if we remove also a (CL) block, provided \emph{that both joints of this (CL) block have degree $3$. Note that there is at most one such block due to Proposition \ref{moleproperty}; for simplicity we will call it a \emph{root block}.}
\end{enumerate}
\end{prop}
\begin{figure}[h!]
\includegraphics[scale=0.4]{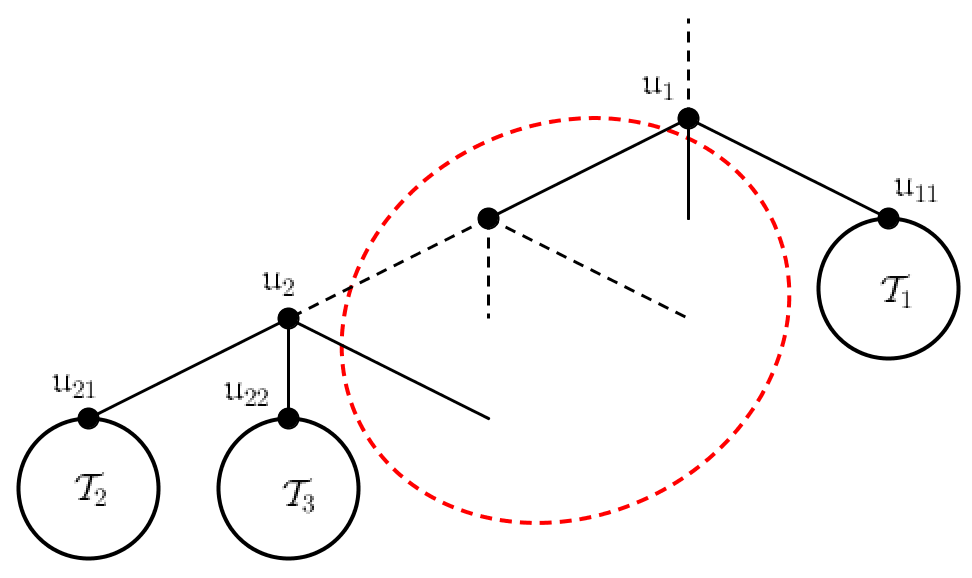}
\caption[Block of type (CL)]{A block of type (CL), as in Proposition \ref{block_clcn}, viewed in the couple. Here $\uf_{11}$ is the right (or left) child of $\uf_1$, $\uf_{21}$ and $\uf_{22}$ have signs $+$ and $-$ respectivly}, and all the leaves in the red circle are completely paired.
\label{fig:couples_cl}
\end{figure}
\begin{figure}[h!]
\includegraphics[scale=0.4]{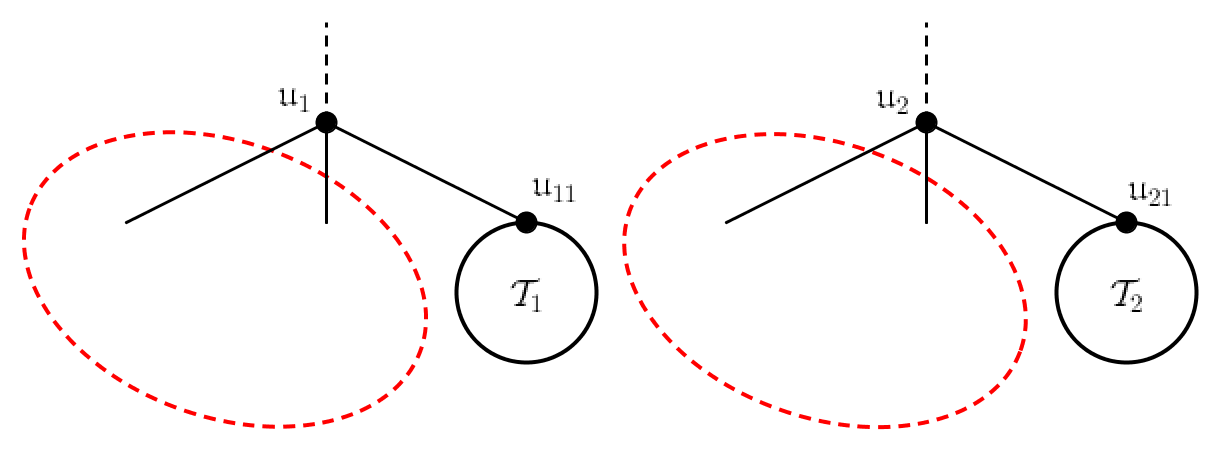}
\caption[Block of type (CN)]{A block of type (CN), as in Proposition \ref{block_clcn}, viewed in the couple. Here $\uf_{11}$ is the right (or left) child of $\uf_1$, $\uf_{21}$ is the right (or left) child of $\uf_2$, and all the leaves in the two red circles are completely paired.}
\label{fig:couples_cn}
\end{figure}
\begin{proof} For any $v\in\Bb\backslash\{v_1,v_2\}$, there is a unique bond $\ell\sim v$ such that $\mf(v,\ell)=v$; let $v^+$ be the other endpoint of $\ell$, then $\nf(v^+)$ is just the parent of $\nf(v)$ in $\Qc$. Consider the path $v\to v^+\to v^{++}\to\cdots$, then it either stays in $\Bb$ or reaches the joints $v_1$ or $v_2$ at some point, due to the structure of $\Bb$. However, if it stays in $\Bb$ then eventually it will reach one of the roots of $\Qc$, which is impossible because the roots have degree $3$ as atoms. Therefore it must reach $v_1$ or $v_2$, which means that for any atom $v\in\Bb$, $\nf(v)$ must be \emph{a descendant of either $\uf_1$ or $\uf_2$}.

Now consider $v_1$, there are two possibilities: (a) there is an atom $v_1^+\in\Bb$ such that $\nf(v_1^+)$ is the parent of $\uf_1$ (as explained above), or (b) there is no such atom $v_1^+\in\Bb$. For $v_2$ there are similarly these two possibilities. If case (a) holds for $v_2$ then by the same proof above, we know that $\uf_2$ is a descendant of $\uf_1$; since $\uf_1$ and $\uf_2$ cannot be a descendant of each other, by symmetry we have only two cases: either (a) holds for $v_2$ and (b) holds for $v_1$, or (b) holds for both $v_1$ and $v_2$. Below we define $\ell_1$ and $\ell_2$ as the two bonds connecting $v_1$ to atoms in $\Bb$, and similarly define $\ell_3$ and $\ell_4$ corresponding to $v_2$.

(1) Suppose (a) holds for $v_2$ and (b) holds for $v_1$. In particular $\uf_2$ is a descendant of $\uf_1$, and $\mf(v_1,\ell_j)\,(j\in\{1,2\})$ are two children of $\uf_1$; let $\uf_{11}$ be the other child of $\uf_1$. Similarly, $\mf(v_2,\ell_j)\,(j\in\{3,4\})$ are $\uf_2$ and one child of $\uf_2$, let $\uf_{21}$ and $\uf_{22}$ be the two other children of $\uf_2$. Clearly $\uf_{11}$ must have the same sign as $\uf_1$ and $\uf_{21}$ must have opposite sign with $\uf_{22}$, because $\ell_1$ and $\ell_2$ have opposite directions, and the same for $\ell_3$ and $\ell_4$. We will assume $\uf_{21}$ has sign $+$ and $\uf_{22}$ has sign $-$.

Now we claim that $v\in\Bb\backslash\{v_1\}$ if and only if $\nf(v)\in\Qc[\Bb]\backslash\{\uf_1\}$. In fact, if $v\in\Bb\backslash\{v_1\}$ then first $\nf(v)$ is a descendant of $\uf_1$ as shown above; second, consider the path $\nf(v)\to \nf(v^+)\to \nf(v^{++})\to\cdots\to \uf_1$, then the node immediately before $\uf_1$ must be $\nf(w)$ for some $w\in\Bb\backslash\{v_1\}$ and thus cannot be $\uf_{11}$ by definition, hence $\nf(v)$ is not a descendant of $\uf_{11}$; third, if the above path contains $\uf_2$, then the node immediately before $\uf_2$ must be $\nf(w)$ for some $w\in\Bb\backslash\{v_1,v_2\}$ and thus cannot be $\uf_{21}$ or $\uf_{22}$ by definition, hence $\nf(v)$ is not a descendant of $\uf_{21}$ or $\uf_{22}$ either.

Conversely, if $\nf(v)\neq \uf_1$ is a descendant of $\uf_1$ but not of $\uf_{11},\uf_{21}$ or $\uf_{22}$, then the path $\nf(v)\to \nf(v^+)\to \nf(v^{++})\to\cdots$ must end at $\uf_1$, and the node immediately before $\uf_1$ must not be $\uf_{11}$. Thus this node must be $\nf(w)$ for some $w\in\Bb\backslash\{v_1\}$, and the $v^{+\cdots +}$ atoms involved in this path must all be in $\Bb$ unless this path contains $\uf_2$. But if $\uf_2$ belongs to this path, then the node immediately before it must not be $\uf_{21}$ or $\uf_{22}$, so it must also be $\nf(w)$ for some $w\in\Bb\backslash\{v_1,v_2\}$, and again all the $v^{+\cdots +}$ atoms involved in this path must be in $\Bb$. In any case we have $v\in\Bb$, so our claim is true.

Now with the above claim, it is easy to see that all the leaves in $\Qc[\Bb]$ must be completely paired, and these leaf pairs exactly correspond to all LP bonds in $\Bb$. It is also clear that, merging $\Bb$ to a single atom corresponds to removing all nodes $\mf\in\Qc[\Bb]\backslash\{ \uf_1\}$, and the resulting molecule is exactly $\Mb(\Qc^{\mathrm{sp}})$ for the resulting couple $\Qc^{\mathrm{sp}}$.

(2) Suppose (b) holds for both $v_1$ and $v_2$. In particular $\mf(v_1,\ell_j)\,(j\in\{1,2\})$ are two children of $\uf_1$, let the other child of $\uf_1$ be $\uf_{11}$. Similarly define $\uf_{21}$, then $\uf_{11}$ must have the same sign as $\uf_1$ and $\uf_{21}$ has the same sign as $\uf_2$, again due to the directions of the bonds $\ell_j$. Moreover, if $\uf_2$ is a descendant of $\uf_1$, then in the path $\nf(v_2)\to\nf(v_2^+)\to\cdots\to \uf_1$, the second node does not belong to $\Bb$, and neither does any subsequent terms; thus the node immediately before $\uf_1$ \emph{cannot} be $\nf(w)$ for any $w\in\Bb$, so it must be $\uf_{11}$, which means that $\uf_2$ is a descendant of $\uf_{11}$ (actually $\uf_2$ also cannot equal $\uf_{11}$ because otherwise we would have an extra bond between $v_1$ and $v_2$, turning $\Bb$ into a hyper-block). Now, by arguing similarly as in (1) we can show that $v\in\Bb\backslash\{v_1,v_2\}$ if and only if $\nf(v)\in\Qc[\Bb]\backslash\{\uf_1,\uf_2\}$. This easily implies that all the leaves in $\Qc[\Bb]$ are completely paired.

Next we prove the preservation of connectivity after removing any set of disjoint type (CN) blocks. In fact, we may define the molecule $\Mb(\widetilde{\Qc})$ for generalized couples $\widetilde{\Qc}$ formed by two \emph{arbitrary trees} which are not necessarily ternary trees (plus that we only keep the pairing structure but ignore the signs of nodes and directions of bonds), similar to Definition \ref{defmole}.

In this regard, removing a type (CN) block amounts to removing all nodes $\mf\in\Qc[\Bb]\backslash\{\uf_1,\uf_2\}$. What remains is a generalized couple formed by two trees in which $\uf_{11}$ is the only child of $\uf_1$ and $\uf_{21}$ is the only child of $\uf_2$ (the fact that $\uf_1$ and $\uf_2$ each has only one child in the new generalized couple, corresponds to the fact that $v_1$ and $v_2$ each has degree $2$ in the new molecule). This can be extended to the removal of multiple disjoint type (CN) blocks, and the resulting molecule is $\Mb(\widetilde{\Qc})$ where $\widetilde{\Qc}$ is a generalized couple formed by two trees such that each branching node has either $1$ or $3$ children. However, the corresponding molecule $\Mb(\widetilde{\Qc})$ is still connected, because each node can be connected to the root of its tree by using PC bonds, and there exists at least one LP bond between the two trees since the number of leaves in each tree is odd. This completes the proof.

Finally, let $\Bb$ be a root (CL) block, i.e. both joints of $\Bb$ has degree $3$. Then, in the notation of Proposition \ref{block_clcn}, we must have (up to symmetry) that $\uf_1$ is the root of one tree in $\Qc$, and a child of $\uf_2$ (say $\uf_{21}$) is paired with the root of the other tree as leaves. Thus we have a couple $\widetilde{\Qc}$ rooted at $u_{11}$ and $u_{22}$, and removing $\Bb$ reduces $\Mb(\Qc)$ to a molecule $\widetilde{\Mb}$ which equals $\Mb(\widetilde{\Qc})$ plus two extra single bonds. Clearly $\widetilde{\Mb}$ is connected, as is the result of removing from it any number of (CN) blocks in $\Mb(\widetilde{\Qc})$. This completes the proof.
\end{proof}
\begin{cor}\label{blockchainprop} Let $\Qc$ be a couple and $\Bb\subset\Mb(\Qc)$ be a block or hyper-block that is concatenated by at least two blocks $\Bb_j\,(1\leq j\leq m)$ as in Definition \ref{defblock}, where $m\geq 2$. Then at most one $\Bb_j$ can be a (CN) block. If $\Bb$ is a block and all $\Bb_j$ are (CL) blocks, then $\Bb$ is a (CL) block. If $\Bb$ is a block and there is one (CN) block $\Bb_j$, then after doing splicing at all other (CL) blocks, this $\Bb$ becomes a single (CN) block $\Bb_j$.
\end{cor}
\begin{proof} Let the joints of $\Bb_j$ be $v_j$ and $v_{j+1}$ for $1\leq j\leq m$. Recall the possibilities (a) and (b) defined in the proof of Proposition \ref{block_clcn}, which are stated for any joint $v$ of any block $\Bb_j\subset\Mb(\Qc)$. If some $\Bb_j$ is a (CN) block, then as in the proof of Proposition \ref{block_clcn}, (b) must happen for both joints $v_j$ and $v_{j+1}$ relative to the block $\Bb_j$. Thus (a) must happen for the joint $v_j$ relative to $\Bb_{j-1}$, and $\Bb_{j-1}$ is a (CL) block. Moreover, (b) must happen for $v_{j-1}$ relative to $\Bb_{j-1}$, and hence (a) must happen for $v_{j-1}$ relative to $\Bb_{j-2}$, and so on. Of course we can also start with $v_{j+1}$ and proceed with $\Bb_{j+1}$ etc., and altogether we know that all blocks other than $\Bb_j$ must be (CL) blocks. If $\Bb$ is a block, then after we splice at all the other (CL) blocks, this $\Bb_j$ should remain unperturbed, as a (CN) block.

If $\Bb$ is a block and all $\Bb_j$ are (CL) blocks, then for each block $\Bb_j$, (a) must happen at one of its joints, say $v_j$ (if it is $v_{j+1}$ then the proof is the same by going in the other direction). Then (b) must happen for the joint $v_j$ relative to $\Bb_{j-1}$, and (a) must happen for $v_{j-1}$ relative to $\Bb_{j-1}$ and so on. In the end (a) must happen for $v_1$ relative to $\Bb_1$ (and hence relative to $\Bb$), so $\Bb$ is a (CL) block.
\end{proof}
\section{Vines and twists}\label{secvine}
\subsection{Vines}\label{subsecvine} We are now ready to introduce the notion of \emph{vines} which are the special kind of blocks that are of fundamental importance in our proof.
\begin{df}[Vines]\label{defvine}\emph{Vines}\footnote{The nomenclature comes from the shapes of the blocks drawn in Figure \ref{fig:vines}.} are defined as the blocks (I)--(VIII) drawn in Figure \ref{fig:vines}. We also define the notion of \emph{ladders} as drawn in Figure \ref{fig:vines}. For each ladder we also require that each pair of two parallel single bonds must have opposite directions, and define its \emph{length} to be the number of double bonds in it minus one. We refer to vines (I)--(II) as \emph{bad vines}, vines (III)--(VIII) as \emph{normal vines}. Note that $\sigma(\Vb)=0$ for all vines $\Vb$ except vines (V) and vines (I) (see Definition \ref{defblock}), for which $\sigma(\Vb)=1$ and $\sigma(\Vb)=2$ respectively.

Define \emph{hyper-vines} (or HV for short) to be the hyper-blocks that are adjoints of vines, as in Definition \ref{defblock}. We also define \emph{vine-chains} (or VC), resp. \emph{hyper-vine-chains} (or HVC), to be the blocks, resp. hyper-blocks, that are formed by concatenating finitely many vines as in Definition \ref{defblock} (these vines are called \emph{ingredients}). Note that a single vine is viewed as a VC, but an HV is not viewed as an HVC. It is easy to verify that assumptions (i) and (ii) in Lemma \ref{disjointlem} hold for any HV, VC or HVC. For simplicity, we will refer to any HV, VC or HVC as \emph{vine-like objects}.

Note that, if the molecule $\Mb=\Mb(\Qc)$ comes from a couple, then any vine could be a (CL) or (CN) vine depending on whether it is a (CL) or (CN) block.
\end{df}
  \begin{figure}[h!]
  \includegraphics[scale=.13]{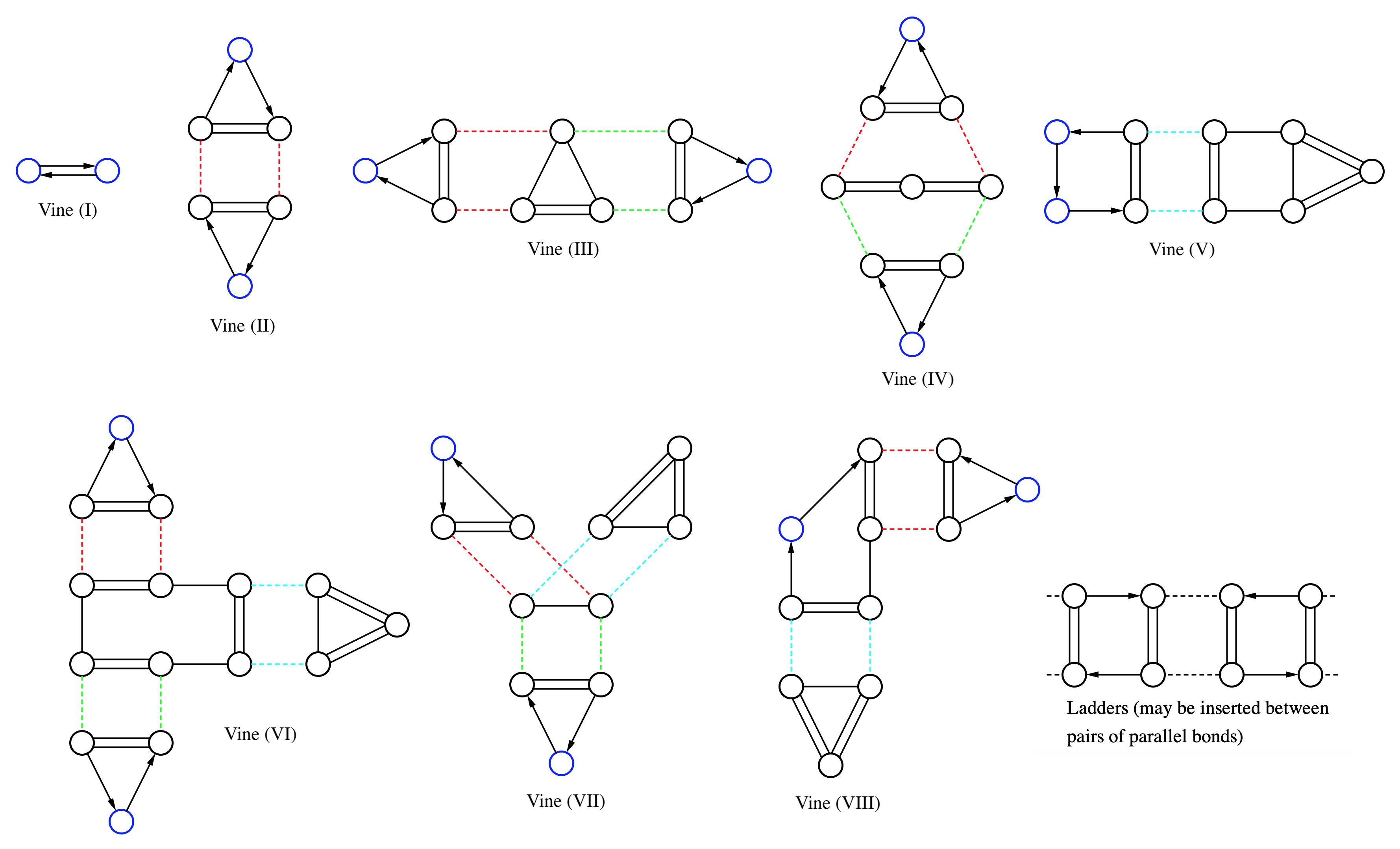}
  \caption[Block families (I)--(V)]{Vines (I)--(VIII). For conventions in this illustration see Remark \ref{remilia}.}
  \label{fig:vines}
\end{figure} 
\begin{rem}\label{remilia} We make a few remarks about the illustrations in Figure \ref{fig:vines}:
\begin{enumerate}[{(a)}]
\item In each vine, the two joints are colored blue; the bonds at the joints are drawn with directions, indicating that the two bonds at each joint must have opposite directions.
\item The other bonds are drawn without directions, meaning they are arbitrary, as long as requirements of a molecule are met (each non-joint atom has in-degree $2$ and out-degree $2$).
\item In each vine, we may insert a ladder between each pair of parallel bonds that are drawn as dashed lines (distinguished by different colors).
\item In the ladder, some bonds are drawn with directions, indicating that each pair of two parallel single bonds must have opposite directions.
\end{enumerate}
\end{rem}
\subsubsection{Bad (CL) vines in a couple}\label{propfamily} We need to study the relative position of (CL) vines (I), and the part of (CL) vines (II) near one of its joints, in a couple.
\begin{prop}\label{molecpl} Consider a (CL) vine $\Vb\subset\Mb(\Qc)$ with joints $v_1$ and $v_2$. Let $\uf_j=\nf(v_j)$, by Proposition \ref{block_clcn} we may assume $\uf_2$ is a descendant of $\uf_1$, and also specify two children $\uf_{21}$ and $\uf_{22}$ of $\uf_2$ that have signs $+$ and $-$ respectively; let $\uf_{23}$ be the other child of $\uf_2$, note that $\uf_{23}$ has the same sign as $\uf_2$.
\begin{enumerate}[{(1)}]
\item If $\Vb$ is vine (II), then $v_2$ is connected to two atoms $v_3$ and $v_4$ by single bonds, while $v_3$ and $v_4$ are connected by a double bond. Let $\uf_j=\nf(v_j)$, then (up to symmetry) exactly one of the following five scenarios happens. See Figure \ref{fig:block_mole}.
\begin{enumerate}[{(a)}]
\item Vine (II-a): $\uf_2$ is a child of $\uf_4$, and $\uf_{23}$ is paired with one child $\uf_0$ of $\uf_3$ as leaves, and the other two children of $\uf_4$ are paired with the other two children of $\uf_3$ as leaves. Here neither $\uf_3$ nor $\uf_4$ is a descendant of the other, but they have a common ancestor, namely $\uf_1$.

\item Vine (II-b): $\uf_2$ is a child of $\uf_3$, and $\uf_{23}$ is paired with one child $\uf_0$ of $\uf_4$ as leaves, and the other two children of $\uf_3$ are paired with the other two children of $\uf_4$ as leaves. Here neither $\uf_4$ nor $\uf_3$ is a descendant of the other, but they have a common ancestor, namely $\uf_1$.

\item Vine (II-c): $\uf_4$ is a child of $\uf_3$ and $\uf_2$ is a child of $\uf_4$. One of the the other two children of $\uf_3$ is paired with one of the other children of $\uf_4$ as leaves, and the remaining child $\uf_0$ of $\uf_3$ is paired with $\uf_{23}$ as leaves. Here $\uf_3$ is a descendant of $\uf_1$.

\item Vine (II-d): $\uf_2$ and $\uf_4$ are two children of $\uf_3$, and $\uf_{23}$ is paired with one child $\uf_0$ of $\uf_4$ as leaves, and the remaining child of $\uf_3$ is paired with another child of $\uf_4$ as leaves. Here $\uf_3$ is a descendant of $\uf_1$.

\item Vine (II-e): $\uf_2$ is a child of $\uf_3$, and $\uf_4=\uf_{23}$. The other two children of $\uf_3$ are paired with two of the children of $\uf_4$ as leaves. Here $\uf_3$ is a descendant of $\uf_1$.
\end{enumerate}

\item If $\Vb$ is vine (I), then exactly one of the following two scenarios happens. See Figure \ref{fig:block_mole}.
\begin{enumerate}[{(a)}]
\item Vine (I-a): $\uf_2$ is the left or right child of $\uf_1$, and $\uf_{23}$ is paired to the middle child $\uf_0$ of $\uf_1$ as leaves.
\item Vine (I-b): $\uf_2$ is the middle of $\uf_1$, and $\uf_{23}$ is paired to the left or right child $\uf_0$ of $\uf_1$ as leaves.
\end{enumerate}
\end{enumerate}
\begin{figure}[h!]
\includegraphics[scale=0.2]{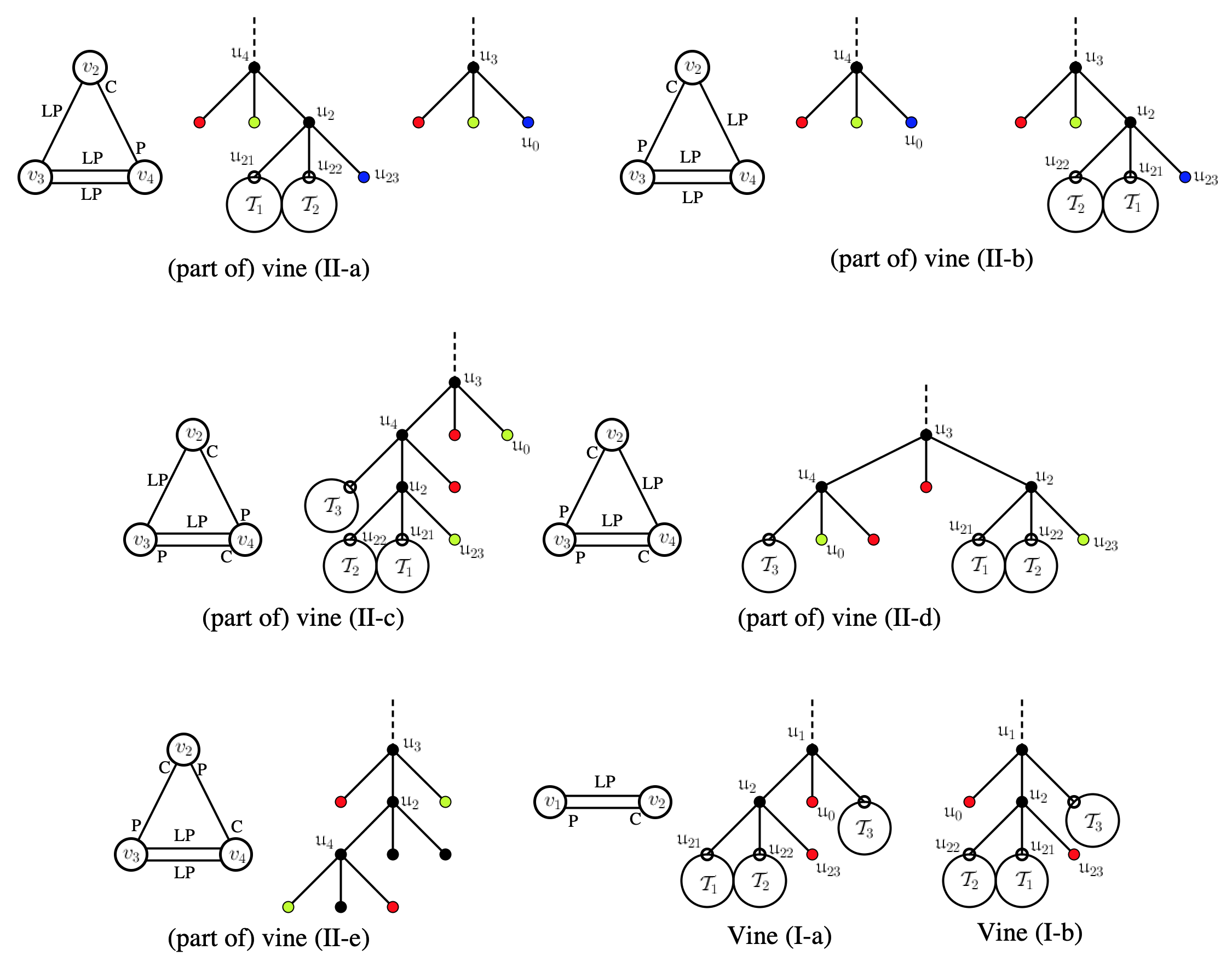}
\caption[Block families]{Vines (II-a)--(II-e) and (I-a)--(I-b) in couples. For conventions in this illustration see Remark \ref{molecplrem}.}
\label{fig:block_mole}
\end{figure}
For simplicity, below we will call a (CL) vine $\Vb$ \emph{core} if it is bad and not vine (II-e), and \emph{non-core} if it is normal or vine (II-e).
\end{prop}
\begin{proof} We examine the labels of each bond between $(v_2,v_3,v_4)$ using Definition \ref{defmole}. Note that, since $\Vb$ is a (CL) vine, one of the bonds connecting $v_2$ to $v_3$ and $v_4$ must be labeled PC with $v_2$ labeled by C; moreover, some configurations are not possible (for example, the two bonds between $v_3$ and $v_4$ cannot both be labeled PC) due to basic properties of the couple $\Qc$ (for example two branching nodes cannot be the parent of each other). By a simple enumeration and symmetry, we find that there is only one case for vine (I) and five cases for vine (II), as drawn in Figure \ref{fig:block_mole}.

For each case of labels, one can apply the definition of the label to deduce that the relation between $(\uf_2,\uf_3,\uf_4)$ has to be as in the corresponding figures. For example, for vine (II-a), by the PC bond between $v_2$ and $v_4$, we know that $\uf_2$ must be a child of $\uf_4$. By the two LP bonds between $v_3$ and $v_4$, we know that the two other children of $\uf_4$ must be paired with two children of $\uf_3$ as leaves. By the LP bond between $v_2$ and $v_3$, we know that the remaining child of $\uf_3$ must be paired to a child of $\uf_2$ as leaves, and this child of $\uf_2$ must be $\uf_{23}$, which is drawn in blue. Clearly neither $\uf_3$ nor $\uf_4$ can be descendant of the other, but the have a common ancestor $\uf_1$, so we arrive at the illustration of Vine (II-a) in Figure \ref{fig:block_mole}. The other cases are treated similarly.
\end{proof}
\begin{rem}\label{molecplrem} We make a few remarks about the illustrations in Figure \ref{fig:block_mole}.
\begin{enumerate}[{(a)}]
\item The labels of the bonds in the molecule are indicated as in Definition \ref{defmole}. Note that cases (II-a) and (II-b) are actually symmetric, but we prefer to state them as two cases for convenience of showing the cancellation between them, see Definition \ref{twist}.
\item The relevant positions of nodes may vary (for example, for vines (II-a) and (II-b) in Figure \ref{fig:block_mole}, $\uf_2$ might be the middle child of $\uf_3$ or $\uf_4$ instead of the right child), but the descriptions in Proposition \ref{molecpl} must be met (for example the blue child of $\uf_2$, which is $\uf_{23}$, must have the same sign as $\uf_2$).
\item The nodes represented by hollow dots instead of solid dots are called \emph{free children}. These include $\uf_{21}$ and $\uf_{22}$ (\emph{except} vine (II-e)), and either a child of $\uf_4$ for vines (II-c) and (II-d), or a child of $\uf_1$ for vines (I-a) and (I-b).
\item For each free child we also draw the subtree rooted at it, and indicate it by suitable $\Tc_j$, as drawn in Figure \ref{fig:block_mole}. These are used in Definition \ref{twist} below.
\item Note that the distinction between vines (I)--(VIII) only involves the structure of the molecule $\Mb$ (as a directed graph), but the distinction between (CL) and (CN) vines, as well as families (II-a)--(II-e) etc., is intrinsic to the structure of the \emph{couple} $\Qc$; for instance, it does not make sense to talk about (CL) vines or vines (II-e) if $\Mb$ does not have the form $\Mb(\Qc)$.
\end{enumerate}
\end{rem}
\subsection{Twists} By exploiting the structure of bad (CL) vines in a couple, as described in Proposition \ref{molecpl}, we can define the operation of \emph{twisting}, which captures the cancellation between such vines.
\begin{df}[Twists]\label{twist} Let $\Qc$ be a given couple with the corresponding molecule $\Mb(\Qc)$.
\begin{enumerate}[{(1)}]
\item Given a core (CL) vine $\Vb\subset\Mb(\Qc)$ as in Proposition \ref{molecpl}, let $v_j$ and $\uf_j$ for $1\leq j\leq 4$ be as in that proposition. Then, we shall define a new couple $\Qc'$, which we call a \emph{unit twist} of $\Qc$, as follows.

First, in $\Qc'$, let any possible parent-child relation, as well as any possible children pairings, between $\uf_3$ and $\uf_4$, be exactly the same as in $\Qc$. Next, let the structure of $\Qc'$ \emph{excluding the subtrees rooted at $\uf_3$ and $\uf_4$ (or $\uf_1$ for vine (I))}, be exactly the same as $\Qc$. Moreover, consider the \emph{free children} in Figure \ref{fig:block_mole} (as in Remark \ref{molecplrem}); we require that the positions of the two free children $\uf_{21}$ and $\uf_{22}$ (as children of $\uf_2$), as well as the positions of the two subtrees (namely $\Tc_1$ and $\Tc_2$) rooted at them, be \emph{switched}\footnote{They are switched because the sign of $\uf_2$ is changed (see Remark \ref{twistexplain}); if we locate $\uf_{21}$ as the child of $\uf_{2}$ other than $\uf_{23}$ that has sign $+$ (and same for $\uf_{22}$), then this remains the same for both couples.} in $\Qc'$ compared to $\Qc$. For the other free child (if it exists), we require that its position (as a child of $\uf_1$ or $\uf_4$) and the subtree (namely $\Tc_3$) rooted at it, be exactly the same in $\Qc'$ as in $\Qc$. Then, it is easy to see that there are exactly two options to insert $\uf_2$, one as a child of $\uf_3$, and the other as a child of $\uf_4$ (for vine (I), the two options are children of $\uf_1$ that has the same or opposite sign with $\uf_1$). One of these two choices leads to $\Qc$, and we define the couple given by the other choice as $\Qc'$. Clearly $\Qc'$ is prime iff $\Qc$ is.
\item In the same way, start with a collection of (CL) vines $\Vb_j\subset\Mb(\Qc)\,(0\leq j\leq q-1)$, such that any two are either disjoint or only share one common joint and no other common atom (i.e. the union of all $\Vb_j$ equals the disjoint union of some VC and HVC). Then, we call any block $\Qc'$ a \emph{twist} of $\Qc$, if $\Qc'$ can be obtained from $\Qc$ by performing the unit twist operation at a subset of these blocks, which only contains core vines. In particular, for any given $\Qc$ and $\Vb_j$, the number of possible twists is a power of two, and at most $2^q$.
\end{enumerate}
\end{df}
Since the notion of twisting is of vital importance in our proof (especially in Section \ref{reduct1}), we will make several remarks below explaining Definition \ref{twist} in more detail.
\begin{rem}\label{explaintwist} We discuss an example of the (unit) twist operation in Definition \ref{twist}. Suppose $\Vb$ is vine (II-c) or (II-d) in Figure \ref{fig:block_mole}. Then we have that:
\begin{enumerate}[{(a)}]
\item The node $\uf_4$ is the left child of $\uf_3$, and the middle child of $\uf_3$ is paired with the right child of $\uf_4$ as leaves.
\item The left child of $\uf_4$ is a free child with subtree $\Tc_3$. The left and middle children of $\uf_2$ are the two free children $\uf_{21}$ and $\uf_{22}$ (or $\uf_{22}$ and $\uf_{21}$), with the subtrees rooted at $\uf_{21}$ and $\uf_{22}$ being $\Tc_1$ and $\Tc_2$ respectively.
\item $\uf_2$ is a child of $\uf_3$ (or $\uf_4$), and the right child $\uf_{23}$ of $\uf_2$ is paired to a child of $\uf_4$ (or $\uf_3$) as leaves.
\end{enumerate}

Now by Definition \ref{twist}, all these properties must hold in both $\Qc$ and $\Qc'$; also the structure of $\Qc$ and $\Qc'$, excluding the subtree rooted at $\uf_3$, must be the same. This leaves only two possibilities: either $\uf_2$ is \emph{the middle child} of $\uf_4$ and $\uf_{23}$ is paired to \emph{the right child} of $\uf_3$ as leaves, or $\uf_2$ is \emph{the right child} of $\uf_3$ and $\uf_{23}$ is paired to \emph{the middle child} of $\uf_4$ as leaves. These are exactly vines (II-c) and (II-d) in in Proposition \ref{molecpl}. Note that for vines (II-c) $\uf_{22}$ is the left child of $\uf_2$ and $\uf_{21}$ is the middle child, while for vines (II-d) $\uf_{21}$ is the left child and $\uf_{22}$ is the middle child, which is consistent with the description in Definition \ref{twist}.

In the same way, we can see that performing one unit twist operation \emph{exactly switches vines (I-a), (II-a), (II-c) vines with vines (I-b), (II-b), (II-d) vines}, respectively.
\end{rem}
\begin{rem}\label{twistexplain} Throughout the proof below, for any fixed (CL) vine $\Vb$, we always adopt the notations $(\uf_1,\uf_2,\uf_{11},\uf_{21},\uf_{22})$ as in Proposition \ref{block_clcn}; for bad (CL) vines we also adopt the notations $(\uf_3,\uf_4,\uf_{23},\uf_0)$ as in Proposition \ref{molecpl}, whenever applicable. The following useful facts are easily verified from Definition \ref{twist}. They are stated for unit twists but can be extended to general twists.
\begin{enumerate}[{(a)}]
\item Let $\Qc$ and $\Qc'$ be unit twists of each other at a bad (CL) vine $\Vb\subset\Mb(\Qc)$, then $\Mb(\Qc)$ and $\Mb(\Qc')$ are the same as directed graphs. They also have the same labelings of bonds, except at the atom $v_2$ (see Figure  \ref{fig:block_mole}).
\item Continuing (a), the only difference at $v_2$ is that the labels of the two bonds connecting $v_2$ to atoms in $\Vb$ are switched (one label is PC with $v_2$ labeled C and the other label is LP).
\item Moreover, if we do splicing (as defined in Proposition \ref{block_clcn}) for $\Qc$ and the (CL) vine $\Vb$, or for $\Qc'$ and the same (CL) vine $\Vb$ (as shown in (a) above), then the two resulting \emph{couples}, defined as $\Qc^{\mathrm{sp}}$ and $(\Qc')^{\mathrm{sp}}$, are exactly the same.
\item Finally, the values of $\zeta_{\uf_j}$ for any branching node $\uf_j\,(j\neq 2)$ are the same for $\Qc$ and $\Qc'$, while the values of $\zeta_{\uf_2}$ are the opposite for $\Qc$ and $\Qc'$.
\end{enumerate}
\end{rem}
\begin{figure}[h!]
\includegraphics[scale=0.4]{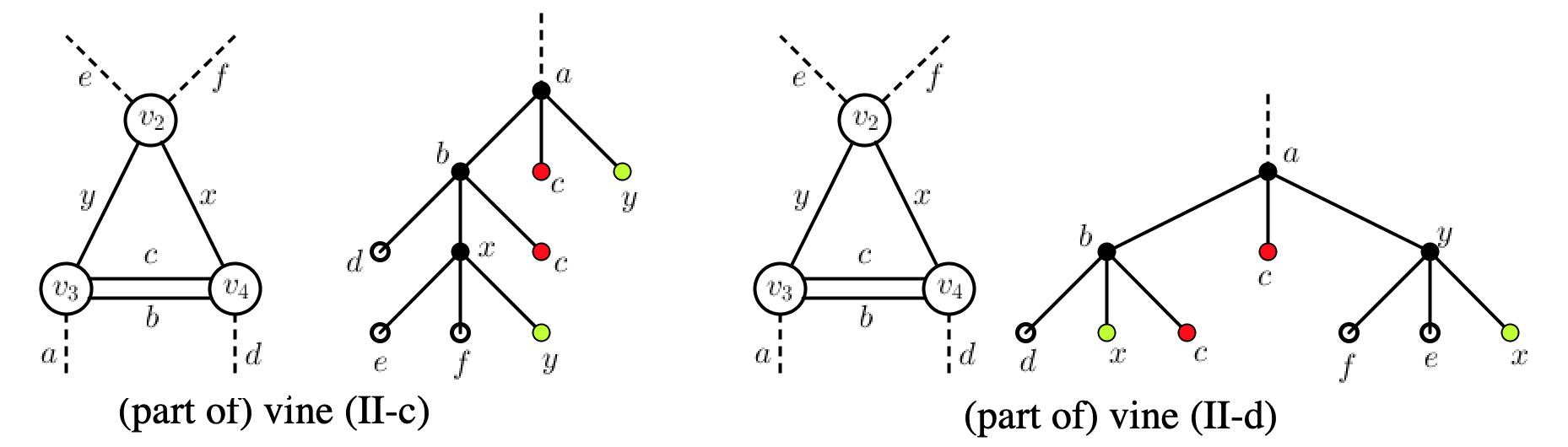}
\caption{Decorations for two couples with vine (II-c) and vine (II-d) which are unit twists of each other, see Remark \ref{dectwist}. Note that compared to Figure \ref{fig:block_mole}, the values of $k_{\mf}$ are the same for each $\mf$, except that the values of $k_{\uf_2}$ and $k_{\uf_{23}}$ are switched.}
\label{fig:twist_dec}
\end{figure}
\begin{rem}
\label{dectwist} We make another simple observation regarding decorations of couples and their twists. Let $\Qc$ be a couple and $\Qc'$ be formed from $\Qc$ by a unit twist (this is readily generalized to arbitrary twists). Then, the $k$-decorations of $\Qc$ are in one-to-one correspondence with $k$-decorations of $\Qc'$, where the values of $k_\mf$ for any branching node or leaf $\mf$ are the same in both cases, but one \emph{switches} the values of $k_{\uf_2}$ and $k_{\uf_{23}}$ in both cases, see Figure \ref{fig:twist_dec}.

Let $\Qc^{\mathrm{sp}}=(\Qc')^{\mathrm{sp}}$ be the couple formed from $\Qc$ (or $\Qc'$) by doing splicing, then for any $k$-decoration of $\Qc$ and the corresponding $k$-decoration of $\Qc'$ defined above, the decorations of $\Qc^{\mathrm{sp}}$ \emph{inherited} from them are the same; here inheriting means that the value of $k_\mf$ is kept the same for any $\mf$, whether it is viewed as a node of $\Qc$ or $\Qc^{\mathrm{sp}}$. This notion also applies to any subset of couples or molecules, or when a couple or molecule is reduced by repeated splicing (or merging) of vines, and will be used frequently below.
\end{rem}
\begin{df}\label{twistrep} Note that the unit twist operation in Definition \ref{twist} only changes the structure of the set $\Qc[\Vb]$ defined in Proposition \ref{block_clcn} (which is part of a ternary tree, obtained by taking the subtree rooted at $\uf_1$ and removing the subtrees rooted at $\uf_{11}$, $\uf_{21}$ and $\uf_{22}$), and does not affect the rest of the couple $\Qc$, so we can view it as a unit twist for $\Qc[\Vb]$.

We may define another operation on $\Qc[\Vb]$ which we call \emph{flipping}, where we flip the signs of $\uf_{1}$ (and $\uf_{11}$), and switch the two subtrees rooted at the two other children nodes of $\uf_1$, together with leaf pairings, see Figure \ref{fig:flip}. Note that unit twisting is an operation on $\Qc[\Vb]$, defined only for core (CL) vines $\Vb$ (Proposition \ref{molecpl}), that can be canonically extended to the couple. Flipping, however, is an operation on $\Qc[\Vb]$, defined for all (CL) vines $\Vb$, that in general cannot be canonically extended to the couple. However, \emph{if $\Vb$ is concatenated with another (CL) vine $\Vb_1$ above it}, then flipping at $\Vb$ can be extended to a couple operation, which is \emph{unit twisting} at the vine $\Vb_1$.
\begin{figure}[h!]
\includegraphics[scale=0.4]{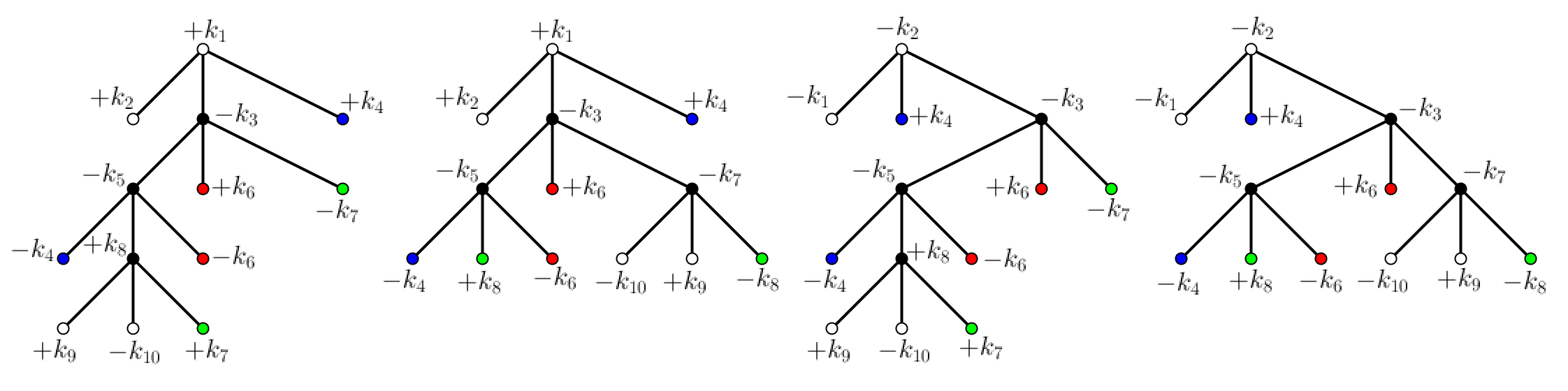}
\caption{Starting from the left couple structure in Figure \ref{fig:vinescancel}, from left to right: original, twisted, flipped, twisted and flipped (Definition \ref{twistrep}). The signs of nodes and corresponding decorations are included; note that the values of $k_{\uf_1}$ and $k_{\uf_{11}}$ are switched after flipping. the nodes $\uf_{1},\uf_{11},\uf_{21},\uf_{22}$ are indicated by hollow dots. The values of $(\mathtt{sgn},\mathtt{ind})$ are $(+,+)$, $(+,-)$, $(-,+)$ and $(-,-)$.}
\label{fig:flip}
\end{figure}

Now, given a (CL) vine $\Vb$, there are $O(C^n)$ possible structures for $\Qc[\Vb]$, where $n$ is the number of branching nodes in $\Qc[\Vb]\backslash\{\uf_1\}$. We define two of them to be equivalent, if one can be formed from the other by flipping, and (if $\Vb$ is a core (CL) vine) unit twisting. We call each equivalence class a \emph{code}, denoted by $\mathtt{cod}$. Note that for a core (CL) vine $\Vb$, each code contains exactly four elements, as shown in Figure \ref{fig:flip}; they are uniquely determined by the signs $\mathtt{sgn}:=\zeta_{\uf_1}$ and $\mathtt{ind}:=\zeta_{\uf_2}$. For other (CL) vines $\Vb$, each code contains exactly two elements; they are uniquely determined by $\mathtt{sgn}$, and the value of $\mathtt{ind}$ is determined by the code.

Let $\Qc$ be a couple with a core (CL) vine $\Vb\subset\Mb(\Qc)$, then by Definition \ref{twist} and the above discussions, we know that $\Qc$ is in one-to-one correspondence with the quadruple $(\Qc^{\mathrm{sp}}, \mathtt{cod}, \nf, \mathtt{ind})$, where $\mathtt{cod}$ and $\mathtt{ind}$ are as above, and $\nf$ is the branching node in $\Qc^{\mathrm{sp}}$ that corresponds to the $\uf_1$ node of $\Qc[\Vb]$ in $\Qc$ (note that $\mathtt{sgn}$ is determined by $\Qc_{\mathrm{sp}}$ and $\nf$). We shall write this as $\Qc\leftrightarrow (\Qc^{\mathrm{sp}}, \mathtt{cod}, \nf, \mathtt{ind})$, and note that making a unit twist for $\Qc$ corresponds to changing the value of $\mathtt{ind}$ only. If $\Vb$ is non-core then $\Qc$ is still uniquely determined by the above quadruple, except that $\mathtt{ind}$ is now determined by $\mathtt{cod}$.
\end{df}
\subsection{Full twists} Finally, we extend the notion of twists to general couples by adding back the regular couples $\Qc^{(\lf,\lf')}$ and regular trees $\Tc^{(\mf)}$ in Proposition \ref{skeleton}.
\begin{df}[Full twists]\label{twistgen} Let $\Qc$ be a couple, $\Qc_{\mathrm{sk}}$ be its skeleton, and write $\Qc\sim (\Qc_{\mathrm{sk}},\As)$ as in Proposition \ref{skeleton}.
\begin{enumerate}[{(1)}]
\item Fix one core (CL) vine $\Vb\subset\Mb(\Qc_{\mathrm{sk}})$. Now consider another couple $\Qc'$, such that $(\Qc')_{\mathrm{sk}}$ either equals $\Qc_{\mathrm{sk}}$, or is a unit twist of it as in Definition \ref{twist}. In the unit twist case, consider the branching node $\uf_2$ (in the notation of Proposition \ref{molecpl}), which occurs in both $\Qc_{\mathrm{sk}}$ and $(\Qc')_{\mathrm{sk}}$. Consider also the leaf pair $(\uf_{23},\uf_0)$ involving the non-free child $\uf_{23}$ of $\uf_2$, which again occurs in both $\Qc_{\mathrm{sk}}$ and $(\Qc')_{\mathrm{sk}}$, see Figure \ref{fig:block_mole}. Apart from these, any other branching node and leaf pair is exactly the same in the two couples $\Qc_{\mathrm{sk}}$ and $(\Qc')_{\mathrm{sk}}$.

Now, if $(\Qc')_{\mathrm{sk}}$ is a unit twist of $\Qc_{\mathrm{sk}}$, we say $\Qc'\sim((\Qc')_{\mathrm{sk}},\As')$ is a \emph{full unit twist} of $\Qc$, if (i) for any branching node $\mf\neq\uf_2$ (or any leaf pair $(\lf,\lf')\neq(\uf_{23},\uf_0)$), the regular trees $\Tc^{(\mf)}\in \As$ and $(\Tc')^{(\mf)}\in\As'$ (or the regular couples $\Qc^{(\lf,\lf')}\in\As$ and $(\Qc')^{(\lf,\lf')}\in\As'$) are the same, and (ii) the regular trees $\Tc^{(\uf_2)}\in\As$ and $(\Tc')^{(\uf_2)}\in\As'$, and the regular couples $\Qc^{(\uf_{23},\uf_0)}\in\As$ and $(\Qc')^{(\uf_{23},\uf_0)}\in\As'$, satisfy that $n(\Tc^{(\uf_2)})+n(\Qc^{(\uf_{23},\uf_0)})=n((\Tc')^{(\uf_2)})+n((\Qc')^{(\uf_{23},\uf_0)})$. If $(\Qc')_{\mathrm{sk}}=\Qc_{\mathrm{sk}}$, the definition is modified in the obvious way.
\item In general, for any set of (CL) vines $\Vb_j\subset\Mb(\Qc_{\mathrm{sk}})$ as in Definition \ref{twist}, such that any two are either disjoint or only share one common joint and no other common atom, we define $\Qc'$ to be a \emph{full twist} of $\Qc$, if $\Qc'$ can be constructed from $\Qc$ by performing some full unit twist at each $\Vb_j$ that is core.
\end{enumerate}
\end{df}
\begin{rem}\label{fulltwistrep} Let $\Qc_0$ be a couple and $\As$ be a collection of regular couples and regular trees as in Proposition \ref{skeleton}. Suppose $\Vb\subset\Mb(\Qc_{0})$ is a (CL) vine, let $\Qc^{\mathrm{sp}}$ be the result of splicing $\Qc_{0}$ at $\Vb$, then we have $\Qc_{0}\leftrightarrow(\Qc^{\mathrm{sp}}, \texttt{cod},\mathfrak{n},\texttt{ind})$, with notations as in Definition \ref{twistrep}. Therefore, we also have the one-to-one correspondence $(\Qc_{0},\As)\leftrightarrow(\Qc^{\mathrm{sp}}, \texttt{cod},\mathfrak{n},\texttt{ind},\Bs,\As^{\mathrm{sp}})$, where $\Bs$ is the sub-collection of $\As$ that involves $\Qc^{(\lf,\lf')}$ and $\Tc^{(\mf)}$ at nodes $\lf,\mf\in\Qc_{0}[\Vb]\backslash\{\uf_1\}$, and $\As^{\mathrm{sp}}$ is the corresponding $\As$ collection for $\Qc^{\mathrm{sp}}$, which is the same as the sub-collection $\As\backslash\Bs$ that involves $\Qc^{(\lf,\lf')}$ and $\Tc^{(\mf)}$ at nodes $\lf,\mf\not\in\Qc_{0}[\Vb]\backslash\{\uf_1\}$.

If $\Qc$ is a couple with skeleton $\Qc_{\mathrm{sk}}$, the we have $\Qc\sim(\Qc_{\mathrm{sk}},\As)$ by Proposition \ref{skeleton}, and $(\Qc_{\mathrm{sk}},\As)$ is also in one-to-one correspondence with the sextuple as above. If $\Qc$ runs over all full unit twists of a given couple at a core (CL) vine $\Vb$, then only $\mathtt{ind}$ and $\Bs$ in this sextuple may vary; in fact $\mathtt{ind}$ takes values in $\{\pm\}$, while $\Bs$ runs over the collections of $\Qc^{(\lf,\lf')}$ and $\Tc^{(\mf)}$ for $\lf,\mf\in\Qc_{\mathrm{sk}}[\Vb]\backslash\{\uf_1\}$, such that $\Qc^{(\lf,\lf')}$ and $\Tc^{(\mf)}$ are fixed when $\mf\neq\uf_2$ and $(\lf,\lf')\neq (\uf_{23},\uf_0)$, and that $n(\Qc^{(\uf_{23},\uf_0)})+n(\Tc^{(\uf_2)})$ takes a fixed value.
\end{rem}

\section{$\Kc_\Qc$ estimates for regular couples and regular trees}\label{regular}
\subsection{Asymptotics and cancellations for $\Kc_\Qc$} In this subsection we study expressions $\Kc_\Qc$ associated with regular couples $\Qc$ (as well as similar expressions for regular trees). The main results are stated as follows.
\begin{prop}\label{regcpltreeasymp} Let $\Qc$ be a regular couple of order $2n$, and $\Kc_\Qc(t,s,k)$ be defined as in (\ref{defkq}). Then, $\Kc_\Qc(t,s,k)$ extends as a smooth function in $k$ and admits  the decomposition $\Kc_\Qc=(\Kc_\Qc)_{\mathrm{app}}+\Rs$, where the remainder $\Rs$ satisfies the bound
\begin{equation}\label{remainderbd}\sup_{|\rho|\leq 40d}\|\partial_k^\rho \Rs\|_{X_{\mathrm{loc}}^{\eta,40d}}\lesssim(C^+\delta)^nL^{-\gamma_1+2\eta},
\end{equation} where $\gamma_1$ is defined as in (\ref{othergamma}). The main term $(\Kc_\Qc)_{\mathrm{app}}$ equals the sum of at most $2^n$ terms of form
\begin{equation}\label{kqterms}(\Kc_\Qc)_{\mathrm{app}}(t,s,k)=\sum \delta^n\cdot\Jc(t,s)\cdot\Mc(k),\end{equation} where each of these terms satisfies the estimate
\begin{equation}\label{kqtermsest}\|\Jc\|_{X_{\mathrm{loc}}^{1-\eta}}\leq (C^+)^n,\quad \sup_{|\rho|\leq 40d}|\partial_k^\rho \Mc(k)|\lesssim (C^+)^n\langle k\rangle^{-40d}.
\end{equation} In particular we also have
\begin{equation}\label{kqbd}\sup_{|\rho|\leq 40d}\|\partial_k^\rho\Kc_\Qc\|_{X_{\mathrm{loc}}^{\eta,40d}}\lesssim(C^+\delta)^n.
\end{equation}

Now let $\Tc$ be a regular tree of order $2n$. For $0\leq s<t\leq 1$, define $\Kc_\Tc^*=\Kc_\Tc^*(t,s,k)$ in the same way as (\ref{defkq}) but with a few differences: the $k$-decoration $\Es$ is replaced by $\Ds$, the product $\prod_{\lf}$ is only taken over leaves $\lf$ of $+$ sign \emph{different} from the lone leaf, and the domain $\Ec$ is replaced by
\begin{equation}\label{defdomaind*}\Dc^*=\big\{t[\Nc]:t_{(\lf_*)^p}>s;\,\,0<t_{\nf'}<t_\nf<t,\mathrm{\ whenever\ }\nf'\mathrm{\ is\ a\ child\ node\ of\ }\nf\big\},
\end{equation} where $(\lf_*)^p$ is the parent node of the lone leaf $\lf_*$. Then we have the decomposition $\Kc_\Tc^*=(\Kc_\Tc^*)_{\mathrm{app}}+\Rs^*$, and the main term $(\Kc_\Tc^*)_{\mathrm{app}}$ equals the sum of at most $2^n$ terms of form \begin{equation}\label{kt*terms}(\Kc_\Tc^*)_{\mathrm{app}}(t,s,k)=\sum \delta^n\cdot\Jc^*(t,s)\cdot\Mc^*(k).\end{equation} The bounds satisfied by $\Rs^*$, $\Jc^*$, $\Mc^*$ and $\Kc_\Tc^*$ are the same as in (\ref{remainderbd}), (\ref{kqtermsest}) and (\ref{kqbd}) above, except that the norm $X_{\mathrm{loc}}^{\eta,40d}$ is replaced by $X_{\mathrm{loc}}^{\eta,0}$, and the factor $\langle k\rangle^{-40d}$ on the right hand side of (\ref{kqtermsest}) is replaced by $1$. Finally, we have the simple identities (also for the $(\cdots)_{\mathrm{app}}$ variants)
\begin{equation}\label{conjugatekq}
\Kc_{\overline{\Qc}}(t,s,k)=\overline{\Kc_\Qc(s,t,k)},\quad\Kc_{\overline{\Tc}}^* (t,s,k)=\overline{\Kc_\Tc^*(t,s,k)}.
\end{equation}
\end{prop}
\begin{prop}\label{regcpltreesum} For any regular couple $\Qc$, let $(\Kc_\Qc)_{\mathrm{app}}$ be defined as in (\ref{kqterms}). Then for any $0\leq t\leq 1$, we have
\begin{equation}\label{matchn}\sum_{n(\Qc)=2n}(\Kc_\Qc)_{\mathrm{app}}(t,t,k)=\Mc_n(t,k),
\end{equation} where the summation is taken over all regular couples $\Qc$ of order $2n$, and the right hand side is defined in (\ref{wketaylor2}).
\end{prop}
\begin{prop}\label{regcpltreecancel} Recall $(\Kc_\Qc)_{\mathrm{app}}$ defined in (\ref{kqterms}) and $(\Kc_\Tc^*)_{\mathrm{app}}$ defined in (\ref{kt*terms}). Then for any $0\leq s<t\leq 1$, we have that
\begin{equation}\label{realcancel}\sum_{n(\Qc)+n(\Tc)=2n}(\Kc_\Qc)_{\mathrm{app}}(t,s,k)\cdot\overline{(\Kc_\Tc^*)_{\mathrm{app}}(t,s,k)},
\end{equation} where the sum is taken over all regular couples $\Qc$ and regular trees $\Tc$ with $+$ sign that have total order $2n$, is a \emph{real valued} function of $(t,s,k)$.
\end{prop}
\subsection{Proof of Propositions \ref{regcpltreeasymp} and \ref{regcpltreecancel}} In this subsection we will prove Propositions \ref{regcpltreeasymp} and \ref{regcpltreecancel}. Note that Proposition \ref{regcpltreesum} follows from the exact same calculations in Subsection 7.4 of \cite{DH21}, so we don't repeat it here.
\subsubsection{Proof of Proposition \ref{regcpltreeasymp}} First, (\ref{conjugatekq}) is obvious by definition in (\ref{defkq}). The proof of other results goes along the same lines as the parallel results in \cite{DH21} (Propositions 6.7, 7.4--7.7). In fact, the only difference between the results here and in \cite{DH21} is an improved bound for the remainder in \eqref{remainderbd}, as well as the bound for an improved norm $\|\Jc(t,s)\|_{X^{1-\eta}}$ in \eqref{kqtermsest} compared to the estimate on the $X^{\frac19}$ norm in \cite{DH21} (the results for $\Kc_\Tc^*$ are similar). We shall explain below how these estimates follow by combining the analysis in \cite{DH21} with new estimates improving Section 6 of \cite{DH21}. We remark that the analysis in Sections 5 and 7 of \cite{DH21} is independent of the chosen scaling law, and as such carries over to the current setting. The analysis in Section 6 of \cite{DH21} is dependent on the scaling law, and we will provide the corresponding alternative bounds here, which carries over to the full range of scaling laws and provides the improved estimate on the remainder. 

As such, we only need to prove \eqref{remainderbd} and the first part of \eqref{kqtermsest}. The latter follows from the fact that $(\Kc_\Qc)_{\mathrm{app}}$ is zero unless $\Qc$ is a dominant couple (cf. Proposition 7.4 of \cite{DH21}), and in which case $\Jc_\Qc(t,s)$ is an explicit homogeneous polynomial in the variables $t,s$ and $\min(t,s)$. The desired $X^{1-\eta}$ bound then follows. We now focus on the remainder estimate in \eqref{remainderbd}.

We start by recalling some notation for regular couples: From equation \eqref{defkq}, we may write
$$\Kc_\Qc(t,s,k)=\bigg(\frac{\delta}{2L^{d-\gamma}}\bigg)^n\zeta(\Qc)\sum_\Es\epsilon_\Es\cdot\Bc_\Qc(t, s, \delta L^{2\gamma}\Omega[\Nc])\cdot\prod_{\lf\in\Lc}^{(+)}n_{\mathrm{in}}(k_\lf),
$$
where  
\begin{equation}\label{defcoefb2}\Bc_\Qc(t,s,\alpha[\Nc])=\int_\Ec\prod_{\nf\in\Nc}e^{\zeta_\nf \pi i\alpha_\nf t_\nf}\,\mathrm{d}t_\nf,
\end{equation}
Here, $\Ec$ is defined in \eqref{defdomaine}. Given a regular couple $\Qc$, a natural pairing exists between the branching nodes in $\Nc$ (cf. Proposition 4.3 and 4.8 in \cite{DH21}). We shall fix a choice of $\Nc^{ch}\subset \Nc$ (here $ch$ for ``choice''), which contains exactly one branching node in each pair. As such, for any decoration $\Es$ of $\Qc$, we must have $\zeta_{\nf'}\Omega_{\nf'}=-\zeta_\nf\Omega_\nf$ for any pair $\{\nf,\nf'\}$ of branching nodes. This allows us to define $\widetilde{\Bc}_\Qc=\widetilde{\Bc}_\Qc(t,s,\alpha[\Nc^{ch}])$ by
\begin{equation}\label{deftildeb}\widetilde{\Bc}_\Qc(t,s,\alpha[\Nc^{ch}])=\Bc_\Qc(t,s,\alpha[\Nc]),\end{equation} assuming that $\alpha[\Nc\backslash\Nc^{ch}]$ is defined such that $\zeta_{\nf'}\alpha_{\nf'}=-\zeta_\nf\alpha_\nf$ for each pair $\{\nf,\nf'\}$.

Now Proposition 5.1 of \cite{DH21} shows that, if $\Qc$ is a regular couple of order $2n$, then the function $\widetilde{\Bc}_\Qc\big(t,s,\alpha[\Nc^{ch}]\big)$ is the sum of at most $2^{n}$ terms. For each term there exists a subset $Z\subset\Nc^{ch}$, such that this term has form
\begin{equation}\label{maincoef1}
\prod_{\nf\in Z}\frac{\chi_\infty(\alpha_\nf)}{\zeta_\nf\pi i\alpha_\nf}\cdot\int_{\Rb^2}\Cc\big(\lambda_1,\lambda_2,\alpha[\Nc^{ch}\backslash Z]\big)e^{\pi i(\lambda_1t+\lambda_2s)}\,\mathrm{d}\lambda_1\mathrm{d}\lambda_2
\end{equation} for $t,s\in[0,1]$, where $\chi_\infty$ is as in Section \ref{norms}.
In (\ref{maincoef1}) the function $\Cc$ satisfies the estimate
\begin{equation}\label{maincoef2}\int \langle \max (\lambda_1,\lambda_2)\rangle^{1-\frac{\eta}{8}}\big|\partial_\alpha^\rho\Cc\big(\lambda_1,\lambda_2,\alpha[\Nc^{ch}\backslash Z]\big)\big|\,\mathrm{d}\alpha[\Nc^{ch}\backslash Z]\mathrm{d}\lambda_1\mathrm{d}\lambda_2\leq C^n(2|\rho|)!
\end{equation} for any multi-index $\rho$, as well as
\begin{equation}\label{maincoef2.5}\int \langle \max (\lambda_1,\lambda_2)\rangle^{\frac{\eta}{4}}\cdot\max_{\nf\in \Nc^{ch}\backslash Z}\langle \alpha_\nf\rangle^{1-\frac{\eta}{2}}\big|\Cc\big(\lambda_1,\lambda_2,\alpha[\Nc^{ch}\backslash Z]\big)\big|\,\mathrm{d}\alpha[\Nc^{ch}\backslash Z]\mathrm{d}\lambda_1\mathrm{d}\lambda_2\leq C^n.
\end{equation}We will denote the $(\lambda_1,\lambda_2)$ integral in (\ref{maincoef1}) by $\widetilde{\Bc}_{\Qc,Z}=\widetilde{\Bc}_{\Qc,Z}(t,s,\alpha[\Nc^{ch}\backslash Z])$, so we have \begin{equation}\label{maincoef3}\widetilde{\Bc}_\Qc(t,s,\alpha[\Nc^{ch}])=\sum_{Z\subset\Nc^{ch}}\prod_{\nf\in Z}\frac{\chi_\infty(\alpha_\nf)}{\zeta_\nf\pi i\alpha_\nf}\cdot \widetilde{\Bc}_{\Qc,Z}(t,s,\alpha[\Nc^{ch}\backslash Z]).\end{equation}
We should remark here that the powers on the weights $\langle \max (\lambda_1,\lambda_2)\rangle$ and $\max_{\nf\in \Nc^{ch}\backslash Z}\langle \alpha_\nf\rangle$ are stated differently in Proposition 5.1 of \cite{DH21}, but a careful inspection of the proof in Section 5 of \cite{DH21} shows that the bounds in \eqref{maincoef2} and \eqref{maincoef2.5} actually hold as well. In fact, this relies on the fact that the corresponding integrals appearing in the proof of Proposition 5.1 of \cite{DH21} have a form like
\[\begin{split}&\int_\Rb \langle \xi\rangle^{p}\bigg|\frac{\chi_\infty(\zeta_1+\epsilon\alpha_{1})}{\zeta_1+\epsilon\alpha_{1}}\bigg|\cdot\bigg|\frac{\chi_\infty(\zeta_2+\epsilon\alpha_{1})}{\zeta_2+\epsilon\alpha_{1}}\bigg|\,\mathrm{d}\alpha_{1}\qquad\textrm{or}\\ 
&\int_{\Rb^2}\langle \xi\rangle^p \bigg|\frac{\chi_\infty(\zeta_1+\epsilon_2\alpha_{2}-\epsilon_1\alpha_1)}{\zeta_1+\epsilon_2\alpha_{2}-\epsilon_1\alpha_1}\bigg|\cdot\bigg|\frac{\chi_\infty(\zeta_2+\epsilon_1\alpha_1)}{\zeta_2+\epsilon_1\alpha_1}\bigg|\cdot\bigg|\frac{\chi_\infty(\zeta_3+\epsilon_2\alpha_{2})}{\zeta_3+\epsilon_2\alpha_{2}}\bigg|\,\mathrm{d}\alpha_1\mathrm{d}\alpha_{2}
\end{split}
\]
which are bounded \emph{for all} $0<p<1$; here $\epsilon,\epsilon_j\in\{\pm1\}$ and $\xi$ is one of the denominators ($\zeta_1+\epsilon\alpha_{1}$ or $\zeta_2+\epsilon\alpha_{1}$ in the first expression, and $\zeta_1+\epsilon_2\alpha_{2}-\epsilon_1\alpha_1$ or $\zeta_2+\epsilon_1\alpha_1$ or $\zeta_3+\epsilon_2\alpha_{2}$ in the second).

As a consequence, we can write
\begin{equation}\label{KQZdef}
\begin{split}
\Kc_{\Qc}(t, s, k)&=\sum_{Z\subset \Nc^{ch}} \Kc_{\Qc, Z}(t, s, k),\\
\Kc_{\Qc, Z}(t, s, k)&=\bigg(\frac{\delta}{2L^{d-\gamma}}\bigg)^n\zeta(\Qc)\sum_\Es\epsilon_\Es\cdot\prod_{\nf\in Z}\frac{\chi_\infty(\delta L^{2\gamma}\Omega_\nf)}{\zeta_\nf\pi i\cdot\delta L^{2\gamma}\Omega_\nf}\cdot \widetilde{\Bc}_{\Qc,Z}(t,s,\delta L^{2\gamma}\Omega[\Nc^{ch}\backslash Z])\cdot\prod_{\lf\in\Lc}^{(+)}n_{\mathrm{in}}(k_\lf),
\end{split}
\end{equation} The remainder $\Rs$ appears upon approximating the sum in $\Kc_{\Qc, Z}(t, s, k)$ with an integral. This approximation, along with the estimates on the remainder $\Rs$, is done in Section 6 of \cite{DH21} (See Propositions 6.1 and 6.7). It is here that the argument becomes dependent on the chosen scaling law. Indeed, we will provide below the alternative propositions that will allow to prove the claimed bound for $\Rs$ in \eqref{remainderbd}.

\begin{prop}\label{approxnt} Consider the following expression
\begin{equation}\label{defi}I:=\sum_{(x_1,\cdots,x_n)}\sum_{(y_1,\cdots,y_n)}W(x_1,\cdots,x_n,y_1,\cdots,y_n)\cdot\Psi(L^{2\gamma} \delta\langle x_1,y_1\rangle,\cdots,L^{2\gamma} \delta\langle x_n,y_n\rangle)\end{equation} where $(x_1,\cdots,x_n,y_1,\cdots,y_n)\in(\Zb_L^d)^{2n}$. Assume there is a (strict) partial ordering $\prec$ on $\{1,\cdots,n\}$, and that the followings hold for the functions $W$ and $\Psi$:

(1) The function $W$ satisfies the bound (here $\widehat{W}$ denotes the Fourier transform in $(\Rb^d)^{2n}$)
\begin{equation}\label{propertyw1}\|\widehat{W}\|_{L^1}+\|\widehat{\partial W}\|_{L^1}\leq (C^+)^n.
\end{equation}

(2) This $W$ is supported in the set\begin{equation}\label{propertyw2}
E:=\big\{(x_1,\cdots,x_n,y_1,\cdots,y_n):|\widetilde{x_j}-a_j|,\,|\widetilde{y_j}-b_j|\leq \lambda_j,\,\forall 1\leq j\leq n\big\},
\end{equation} where $1\leq\lambda_j\leq (\log L)^4$ are constants, $a_j$ and $b_j$ are constant vectors. Each $\widetilde{x_j}$ is a linear function that equals either $x_j$, or $x_j\pm x_{j'}$ or $x_j\pm y_{j'}$ for some $j'\prec j$, similarly each $\widetilde{y_j}$ equals either $y_j$, or $y_j\pm x_{j''}$ or $y_j\pm y_{j''}$ for some $j''\prec j$.

(3) For some set $J\subset\{1,\cdots,n\}$, the function $\Psi$ has the expression
\begin{equation}\label{propertypsi1}\Psi(\Omega_1,\cdots,\Omega_n)=\prod_{j\in J}\frac{\chi_\infty(\Omega_j)}{\Omega_j}\cdot\Psi_1(\Omega[J^c]),
\end{equation} where $\chi_\infty$ is as in Section \ref{norms}, and for any $\rho$ we have
\begin{equation}\label{propertypsi2}
\|\partial^\rho\Psi_1\|_{L^1}\leq C^n(4|\rho|)!,\,\, \big\|\max_{j\in J^c}\langle \Omega_j\rangle^{1-\frac{\eta}{2}}\cdot\Psi_1\big\|_{L^1}\leq C^n.
\end{equation}

Assume $n\leq (\log L)^3$. Then we have
\begin{multline}\label{conclusion1}\bigg|I-L^{2dn}\int_{(\Rb^d)^{2n}}W(x_1,\cdots,x_n,y_1,\cdots,y_n)\cdot\Psi(L^{2\gamma}  \delta \langle x_1,y_1\rangle_\beta,\cdots,L^{2\gamma}  \delta\langle x_n,y_n\rangle_\beta)\,\mathrm{d}x_1\cdots\mathrm{d}x_n\mathrm{d}y_1\cdots\mathrm{d}y_n\bigg|\\\leq (\lambda_1\cdots\lambda_n)^C(C^+L^{2d-2\gamma}\delta^{-1})^{n} L^{-\gamma_1+\eta}.
\end{multline} Moreover, defining
\begin{multline}\label{intappr}I_{\mathrm{app}}=(L^{2d-2\gamma}\delta^{-1})^{n}\int\Psi_1\mathrm{d}\Omega[J^c]\cdot\int_{(\Rb^d)^{2n}}W(x_1,\cdots,x_n,y_1,\cdots,y_n)\\\times\prod_{j\in J}\frac{1}{\langle x_j,y_j\rangle_\beta}\prod_{j\not\in J}\dirac(\langle x_j,y_j\rangle_\beta)\mathrm{d}x_1\cdots\mathrm{d}x_n\mathrm{d}y_1\cdots\mathrm{d}y_n,
\end{multline} where the singularities $1/\langle x_j,y_j\rangle_\beta$ are treated using the Cauchy principal value, we have
\begin{equation}\label{conclusion2}|I_{\mathrm{app}}|\leq (\lambda_1\cdots\lambda_n)^C(C^+L^{2d-2\gamma}\delta^{-1})^{n},\quad |I-I_{\mathrm{app}}|\leq(\lambda_1\cdots\lambda_n)^C(C^+L^{2d-2\gamma}\delta^{-1})^{n} L^{-\gamma_1+\eta}.
\end{equation}
\end{prop}

We shall prove this proposition by relying on a series of Lemmas formulated and proved below. In what follows, we assume that $1\leq \lambda \leq (\log L)^4$ and use the notation $e(z)=e^{2\pi i z}$. The first lemma replaces Lemma 6.2 of \cite{DH21}.

\begin{lem}\label{NTSP}
	Suppose $\Phi:\Rb\times \Rb^d \times \Rb^d \to \Cb$ is a function satisfying the bounds
	\begin{equation}\label{NTphibounds}
	\sup_{s, x, y}|\partial_x^\alpha \partial_y^\beta \Phi(s,x, y)|\leq D
	\end{equation}
for all multi-indices $|\alpha|, |\beta|\leq 10d$. Then we have:

(1) The following bound
		\begin{equation}\label{NTSPest0}
		\int_{\Rb}\left|\int_{\Rb^{2d}}	\chi_0\big(\frac{x-a}{\lambda}\big)\chi_0\big(\frac{y-b}{\lambda}\big)\Phi(s,x, y)\cdot e(\xi \cdot x+\zeta\cdot y+s\langle x, y\rangle) \, \mathrm{d}x \mathrm{d}y \right|\,\mathrm{d}s \lesssim D\lambda^{2d}
		\end{equation}
holds uniformly in $(\xi, \zeta, a, b)\in \Rb^{4d}$.	

(2) Suppose, in addition, that $\Phi$ is supported on the set $|s|\lesssim L$. Then, there holds
\begin{equation}\label{NTSPest}
\begin{split}
 \int_{\Rb} \bigg|\sum_{0\neq(g,h)\in \Zb^{2d}}\int_{\Rb^{2d}} \chi_0\big(\frac{x-a}{\lambda}\big)\chi_0\big(\frac{y-b}{\lambda}\big)\Phi(s, x, y)\cdot e[(Lg+\xi) \cdot x +(Lh +\zeta)\cdot y+s\langle x, y\rangle]\bigg|\\
\lesssim D\lambda^{2d} L^{-1+2\eta} (1+|\xi|+|\zeta|).
\end{split}
\end{equation}
uniformly in $(a, b)\in \Rb^{2d}$. In particular, we have
\begin{equation}
		\int_{\Rb}\bigg|\sum_{(g,h) \in \Zb^{2d}}  \int_{\Rb^{2d}} \chi_0\big(\frac{x-a}{\lambda}\big)\chi_0\big(\frac{y-b}{\lambda}\big)\Phi(s, x, y)\cdot e[(Lg+\xi) \cdot x +(Lh +\zeta)\cdot y+s\langle x, y\rangle]\bigg|
		\lesssim D\lambda^{2d}\label{NTSPest2}
\end{equation}
		uniformly in $(\xi, \zeta,a,b)\in \Rb^{4d}$.
\end{lem}

\begin{proof}
With no loss of generality we assume $D=1$. Part (1) follows easily by translating in $x$ and $y$, and then applying the stationary phase lemma to estimate the $(x, y)$ integral by $\langle s\rangle^{-d}\lambda^{2d}$. To prove part (2), we define $Q=L^{\eta}(\lambda+|s|)$, and split the left-hand side of \eqref{NTSPest} into two parts $A$ and $B$ defined as follows:

$\bullet$ \emph{Part A: where $|Lg+\xi+sb|\gtrsim Q$ or $|Lh+\zeta+sa|\gtrsim Q$}. For this part, we will fix the variables $(g, h, s)$ and integrate by parts in $(x,y)$ as follows: Defining $z=x+\frac{Lh+\zeta}{s}$, $w=y+\frac{Lg+\xi}{s}$, $a'=a+\frac{Lh+\zeta}{s}$, and $b'=b+\frac{Lg+\xi}{s}$, then up to a unimodular constant independent of the variables, we can write the integral in $(x,y)$ as 
\begin{multline*}
\int_{\Rb^{2d}} \chi_0\big(\frac{z-a'}{\lambda}\big)\chi_0\big(\frac{w-b'}{\lambda}\big)\Phi\bigg(s, z-\frac{Lh+\zeta}{s}, w-\frac{Lg+\xi}{s}\bigg)e\big(s (z\cdot w)\big) \, \mathrm{d}z\mathrm{d}w\\
=e(s a'\cdot b')\int_{\Rb^{2d}} \chi_0\big(\frac{u}{\lambda}\big)\chi_0\big(\frac{v}{\lambda}\big)\Phi\bigg(s, u+a'-\frac{Lh+\zeta}{s}, v+b'-\frac{Lg+\xi}{s}\bigg)e\big(s (u\cdot v)\big)e\big(s a'\cdot v+sb'\cdot u\big) \, \mathrm{d}u\mathrm{d}v.
\end{multline*}
Integrating by parts many times, using the last oscillatory factor and the fact that that either $|sa'|$ or $|sb'|\geq Q=L^\eta(\lambda+|s|)$, we obtain that the contribution of Part A is better than acceptable.

$\bullet$ \emph{Part B: where $|Lg+\xi+sb|$ and $|Lh+\zeta+sa|\ll Q$.} Recall that since $|s|\lesssim L$ and $\lambda\lesssim L^\eta$, we have that $Q\lesssim L^{1+\eta}$, and hence for fixed $s$ the number of choices of $(g,h)$ in the sum in this part B is $L^{O(\eta)}$. We further split this part into two: Part B1 in which the $s$ integral is over $|s|\leq L^{\frac56}$ and part B2 for which $L^{\frac56}\lesssim |s|\lesssim L$. For B2, we simply use stationary phase again in the $(x,y)$ integral to estimate it by $\lambda^{2d} |s|^{-d}$, so the contribution of this part is bounded by
$$
|(\textrm{contribution\ of\ B2})|\lesssim \lambda^{2d}L^{O(\eta)}\int_{L^{\frac56}\lesssim |s|\lesssim L}|s|^{-d}\mathrm{d}s\lesssim \lambda^{2d}L^{C\eta -\frac{5}{6}(d-1)}\ll \lambda^{2d} L^{-1}.
$$
Moving back to $B1$, we notice that in this case $Q\leq L^{\eta+\frac{5}{6}}\ll L$, which implies that, for each fixed $s$, there is at most one element in the whole sum over $(g, h)\neq 0$. Now we split the $s$ integral into dyadic pieces, and rewrite the contribution of B1 as
$$
\sum_{K\in 2^\Nb,K\leq L^{\frac56}} \int_{|s|\sim K}\bigg|\sum_{(g, h)\neq 0}\int_{\Rb^{2d}}\chi_0\big(\frac{x-a}{\lambda}\big)\chi_0\big(\frac{y-b}{\lambda}\big)\Phi(s, x, y)\cdot e[(Lg+\xi) \cdot x +(Lh +\zeta)\cdot y+s\langle x, y\rangle]\bigg|.
$$
Let $Y=\max(\langle a\rangle, \langle b\rangle)$. For the part of the dyadic sum with $K\ll LY^{-1}$, notice that since $\max(|Lg+\xi+sb|,|Lh+\zeta+sa|)\ll Q\ll L$ and $(g, h) \neq 0$, we must have that $|\xi|+|\zeta|\gtrsim L$. Therefore, if we just use stationary phase again in the $(x,y)$ integral, we can estimate the sum over $K$ by
$$
\sum_{K\in 2^\Nb,K\leq L^{\frac56}} \int_{|s|\sim K\ll LY^{-1}}\lambda^{2d}\langle s\rangle^{-d} \mathrm{d}s\lesssim \lambda^{2d}\lesssim \lambda^{2d}L^{-1}(|\xi|+|\zeta|),
$$
as needed. This leaves us with the sum over $K\gtrsim LY^{-1}$. Assume with no loss of generality that $|a^1|\sim Y\gg 1$. Since $|Lh+\zeta+sa|\ll Q$, we have that $\{\frac{\zeta^1+s a^1}{L}\} \ll QL^{-1}$ where $\{\cdot\}$ denotes the fractional part. Since we also have that $\frac{\zeta^1+s a^1}{L}$ belongs to an interval of size $\frac{KY}{L}$, it follows that $\frac{\zeta^1+s a^1}{L}$ belongs to a set of measure $\lesssim \frac{KYQ}{L^2}$. Hence, $s$ belongs to a set of measure $\lesssim \frac{KQ}{L}$. As a result, recall that for fixed $s$ the the sum over $(g,h)$ has at most one element and the integral in $(x,y)$ is $\lesssim \lambda^{2d}K^{-d}$ by stationary phase, we can estimate
$$
\int_{|s|\sim K}\bigg|\sum_{(g, h)\neq 0}\int_{\Rb^{2d}}\chi_0\big(\frac{x-a}{\lambda}\big)\chi_0\big(\frac{y-b}{\lambda}\big)\Phi(s, x, y)\cdot e[(Lg+\xi) \cdot x +(Lh +\zeta)\cdot y+s\langle x, y\rangle]\bigg|\lesssim \lambda^{2d}\frac{KQ}{L} K^{-d},
$$
which sums in $K$ to give,
$$
\sum_{K\in 2^\Nb,K\leq L^{\frac56}} \lambda^{2d} L^\eta (\lambda +K){K^{-d+1}}{L}^{-1}\lesssim \lambda^{2d} L^{-1+2\eta}.
$$ This finishes the proof.
\end{proof}

The following lemma replaces Lemma 6.4 of \cite{DH21}.

\begin{lem}\label{lem:minorarcs}
	Suppose that $\Phi(s,x,y): \Rb\times \Rb^d \times \Rb^d\to \Cb$ is a function satisfying (\ref{NTphibounds}).

(1) If $\Phi$ is supported on $|s|< L^{2\gamma}$, then the following bound holds uniformly in $(a, b, \xi, \zeta)\in \Rb^{4d}$:
\begin{equation}\label{NTlem21}
\int_{\Rb}\bigg|\sum_{(x, y)\in \Zb^{2d}_L} \Phi(s, x, y) \chi_0\big(\frac{x-a}{\lambda}\big)\chi_0\big(\frac{y-b}{\lambda}\big)e(x\cdot \xi+y\cdot \zeta+s \langle x, y\rangle)\bigg|\mathrm{d}s \lesssim D \lambda^{4d} L^{2d}.
\end{equation}

(2) If $\Phi(s,x,y)$ is supported on the set $L\lesssim |s|$, then the following improved estimate holds uniformly in $(a, b, \xi,\zeta)\in\Rb^{4d}$: for $P>\eta^{-2}$,
\begin{equation}\label{NTlem22}
\int_{\Rb}\big\langle\frac{s}{\delta L^{2\gamma}}\big\rangle^{-P}\bigg|\sum_{(x, y)\in \Zb^{2d}_L} \Phi(s, x, y) \chi_0\big(\frac{x-a}{\lambda}\big)\chi_0\big(\frac{y-b}{\lambda}\big)e(x\cdot \xi+y\cdot \zeta+s \langle x, y\rangle)\bigg|\mathrm{d}s \lesssim D \lambda^{4d} L^{2d-2(d-1)(1-\gamma)+\eta}.
\end{equation}	
\end{lem}	

\begin{proof}
Part (1) is implied by Lemma 6.4 of \cite{DH21} (which even extends to $\Phi$ supported on $|s|< L^2$. So we focus on part (2); if $\gamma<1/2$ then (\ref{NTlem22}) is trivial due to the factor $\langle s\delta^{-1}L^{-2\gamma}\rangle^{-P}$, so we will assume $\gamma\geq 1/2$. Recall that, $\langle x,y\rangle=\sum_{j=1}^d x^j y^j$ where $x^j, y^j\in \Zb_L$. We make the change of variables 
	$$
	L^{-1}p^j={x^j+y^j}, \qquad L^{-1}q^j={x^j-y^j}, \qquad p^j\equiv q^j \pmod 2.
	$$
The sum in $(x^j,y^j)\in \Zb_L^{2}$ then becomes the linear combination of four sums, which are taken over $(p^j, q^j)\in \Zb^{2}$, or $(p^j, q^j)\in 2\Zb\times\Zb$, or $(p^j, q^j)\in \Zb\times 2\Zb$, or $(p^j, q^j)\in (2\Zb)^2$. We will only consider the first sum, and it will be obvious from the proof that the other sums are estimated similarly. Define \[ \Upsilon (s,z, w)=\Phi\big(s, \frac{z+w}{2},\frac{z-w}{2}\big)\chi_0\big(\frac{z+w-2a}{2\lambda}\big)\chi_0\big(\frac{z-w-2b}{2\lambda}\big),\] which has all derivatives in $(z,w)$ up to order $10d$ uniformly bounded, and is supported in the set $\{g^j\leq Lz^j\leq g^j+2\lambda L,\,h^j\leq Lw^j\leq h^j+2\lambda L\}$, where $(g^j,h^j)\in\Zb^2$ are determined by $(a,b)$.

Now, by possibly redefining $(s,\xi,\zeta)$, we need to show that the function 
	\begin{align*}
	 B(\xi, \zeta)&=\int_{\Rb}\langle s/\delta L^{2\gamma}\rangle^{-P} \bigg|\sum_{(p, q)\in \Zb^{2d}} \Upsilon\big(s, pL^{-1}, qL^{-1}\big)e\big[sL^{-2}(|p|^2-|q|^2)+p\cdot  \xi+y\cdot \zeta\big]\bigg| \, \mathrm{d}s\\
	&=\int_{\Rb} \langle s/\delta L^{2\gamma}\rangle^{-P}\bigg|\sum_{(p, q)\in \Zb^{2d}} \Upsilon\big(s,pL^{-1}, qL^{-1}\big)\prod_{j=1}^d e\big[sL^{-2} (p^j)^2 +p^j \xi^j]\cdot e[-sL^{-2} (q^j)^2 +q^j \zeta^j\big]\bigg| \, \mathrm{d}s
	\end{align*}
satisfies the bounds in \eqref{NTlem22} when $\Upsilon$ is supported on $|s|\gtrsim L$, noting that in the above sum we must have $p^j\in[g^j,g^j+20\lambda L]$ and $q^j\in[h^j,h^j+20\lambda L]$. Now, recall the Gauss sums $G_h(s,r,n)$ and $G_h(s,r,x)$ defined by 
\begin{equation}\label{huagausssum}G_h(s, r,n)=\sum_{p=h}^{h+n}e(s p^2+rp), n \in \Nb; \qquad \mathrm{and} \qquad G_h(s, r,x)=G_h(s, r, \lfloor x\rfloor), x\in \Rb_+,\end{equation} where $\lfloor x\rfloor$ is the floor function, and notice that since $\partial_x G_h(s, r; x)=\sum_{p\in \Nb}e(s (h+p)^2+r(h+p)) \dirac(x-p)$, we can write
	\begin{align*}
	 B(\xi, \zeta)&=\int_{\Rb} \langle s/\delta L^{2\gamma}\rangle^{-P}\bigg|\int_{(u, v)\in \Rb_+^{2d}} \Upsilon\big(s, (u+g)L^{-1}, (v+h)L^{-1}\big)\prod_{j=1}^d \partial_{u^j}G_{g^j}( sL^{-2}, \xi^j , u^j) \\&\qquad\qquad\qquad\qquad\qquad\qquad\qquad\qquad\qquad\qquad\qquad\qquad\,\,\times\partial_{v^j}G_{h^j}(- sL^{-2}, \zeta^j, v^j) \, \mathrm{d}u \mathrm{d}v\bigg| \, \mathrm{d}s\\
	&\leq L^{-2d}\int_{\Rb}\langle s/\delta L^{2\gamma}\rangle^{-P}\int_{(u, v)\in \Rb_+^{2d}} \big| D^\alpha\Upsilon\big(s,(u+g)L^{-1}, (v+h)L^{-1}\big)\big|\\&\qquad\qquad\qquad\qquad\qquad\qquad\qquad\qquad\times\prod_{j=1}^d\big|G_{g^j}( sL^{-2}, \xi^j , u^j) G_{h^j}(- sL^{-2}, \zeta^j, v^j)\big| \, \mathrm{d}u\mathrm{d}v\mathrm{d}s,
	\end{align*}
	where $g=(g^1,\cdots,g^d)$ etc., and $D^\alpha \Upsilon$ is obtained from $ \Upsilon$ by taking one derivative in each of the variables $u_j, v_j$ (and hence has the same support properties). By rescaling in $s$, we obtain that 
	
\begin{align*}
	 B(\xi, \zeta)&\lesssim L^{-2d+2}\int_{(u,v) \in \Omega}\int_{|s|\geq L^{-1}} \langle s\delta^{-1} L^{2-2\gamma}\rangle^{-P}\prod_{j=1}^d \bigg|G_{g^j}( s, \xi^j , u^j) G_{h^j}(- s, \zeta^j, v^j) \bigg| \, ds \mathrm{d}u \mathrm{d}v\\
	\end{align*}
where $\Omega$ is a set in $\Rb^{2d}$ of volume $\lesssim (\lambda L)^{2d}$. We will be done then once we show that, uniformly in $|u|, |v|\lesssim \lambda L$, there holds
$$
 \int_{|s|\geq L^{-1}} \langle s\delta^{-1} L^{2-2\gamma}\rangle^{-P}\prod_{j=1}^d \bigg|G_{g^j}( s, \xi^j , u^j) G_{h^j}(- s, \zeta^j, v^j) \bigg| \, ds \lesssim 
\lambda^{2d} L^{2d-2} L^{-(d-1)(2-2\gamma)+\eta}. 
$$
First note that using the trivial bound $|G_{h_j}|\lesssim L$, and by our choice of $P\geq \eta^{-2}$, it is enough to show that
$$
 \int_{L^{-1} \leq |s|\lesssim  L^{2\gamma-2+\eta/(10d)}} \prod_{j=1}^d \bigg|G_{g^j}( s, \xi^j , u^j) G_{h^j}(- s, \zeta^j, v^j) \bigg| \, ds \lesssim\lambda^{2d} L^{2d-2} L^{-(d-1)(2-2\gamma)+\eta}, 
$$
which is implied by proving that, uniformly in $(g, \xi, u)$ satisfying $0<u\lesssim \lambda L$, there holds
\begin{equation}\label{refinedHua}
 \int_{L^{-1} \leq |s|\lesssim   L^{2\gamma-2+\eta/(10d)}} \bigg|G_{g}( s, \xi , u) \bigg|^{2d} \, ds \lesssim \lambda^{2d} L^{2d-2} L^{-(d-1)(2-2\gamma)+\eta}.
\end{equation} 

For this, let $I$ be the interval of integration above, and let us assume that $u \in [K, 2K)$ for some dyadic integer $K\lesssim \lambda L$. We will rely on the Gauss lemma for Gauss sums in \eqref{huagausssum} which gives that if $|s-\frac{a}{q}|\leq \frac{1}{qK}$ for some $0\leq a <q\leq K$ with $\gcd(a, q)=1$, then
 $$
|G_g(s, \xi, u)|\lesssim \frac{K\log L}{\sqrt q \left(1+K|s-\frac{a}{q}|^{1/2}\right)}. 
$$
Here we recall that $s\in [0,1]$ can be approximated by such rational number in the above fashion by Dirichlet's Approximation Theorem. Let us assume that $q\in [B, 2B)$ for some dyadic integer $B$ with $1\leq B \leq K$ and that $s\in I_{a/q, n, B}:=\frac{a}{q}\pm [\frac{1}{2^{n}BK}, \frac{1}{2^{n-1}BK})$ with $1\leq 2^n\lesssim \frac{K}{B}$ (with the obvious modification for $2^n\sim\frac{K}{B}$), so that 
$$
I=\bigcup_{B\leq K} \bigcup_{1\leq 2^n \lesssim \frac{K}{B}} \bigcup_{q\in [B, 2B)} \bigcup_{0\leq a < q,\gcd(a, q)=1}I_{a/q, n, B}.
$$
First note that if $a=0$, then $|G_g|\lesssim |s|^{-1/2}\log L\leq L^{\frac{1}{2}}\log L$ on $I_{0, n, B}$, which gives a bound that is much better than \eqref{refinedHua}. Otherwise, we have $|G_g(s, \xi, u)|\lesssim K^{\frac{1}{2}}2^{n/2}\log L$ on $I_{a/q, n, B}$. Also, note that since $s\lesssim L^{2\gamma-2+\eta/(10d)}$ on $I$, we must have that $B\sim q \gtrsim L^{2-2\gamma-\eta/(10d)}$ for any $I_{a/q, n, B}$ with $a\neq 0$, as well as $a\leq qL^{2\gamma-2+\eta/(10d)}$. As a result, we obtain that
\begin{align*}
\int_I \bigg|G_{g}( s, \xi , u) \bigg|^{2d} &\lesssim L^{d-1+\eta}+\sum_{L^{2-2\gamma-\eta/(10d)}\leq B \leq K}\sum_{1\leq 2^n \lesssim \frac{K}{B}} \sum_{q\in [B, 2B)} \sum_{1\leq a < q L^{2\gamma-2+\eta/(10d)},\gcd(a, q)=1}\frac{(K2^n\log L)^d}{2^n BK}\\
&\lesssim L^{d-1+\eta}+L^{2\gamma-2+\eta/(10d)}\sum_{L^{2-2\gamma-\eta/(10d)}\leq B \leq K}\sum_{1\leq 2^n \lesssim KB^{-1}} B K^{d-1}2^{n(d-1)}(\log L)^d
\\&\lesssim L^{d-1+\eta}+K^{2(d-1)}L^{(d-1)(2\gamma-2)+\eta},
\end{align*}
which gives the result since $K \leq \lambda L$.
\end{proof}
\begin{lem}\label{NTintlem}
Suppose $\Phi(x,y)$ satisfyies \eqref{NTphibounds} without $s$. Let $\Omega(x,y)=\langle x, y\rangle$ and $\mu:=\delta L^{2\gamma}$.

(1) Suppose $\psi$ is a function such that $\|\psi\|_{L^1(\Rb)}\leq D$, then
	\begin{equation}\label{NTintlem1}
	\left|\mu^{-1}\int_{\Rb^{2d}}\psi(\mu \Omega) \chi_0\big(\frac{x-a}{\lambda}\big)\chi_0\big(\frac{y-b}{\lambda}\big)\Phi(x,y) e(x\cdot\xi+y\cdot \zeta) \, \mathrm{d}x\mathrm{d}y\right|\lesssim D \lambda^{2d}
	\end{equation}
	uniformly in $(a, b, \xi, \zeta)\in \Rb^{4d}$. The same holds if $\psi(\mu\Omega)\Phi(x,y)$ is replaced by $\Psi(\mu\Omega,x,y)$ where $\Psi=\Psi(u,x,y)$ satisfies $\big\|\sup_{x,y}|\partial_x^\alpha\partial_y^\beta\Psi|\big\|_{L_u^1}\leq D$ for all multi-indices $|\alpha|,|\beta|\leq 10d$.
	
(2) Suppose further that $\| \langle y\rangle^{1-\eta} \psi\|_{L^1(\Rb)}\leq D$, then 
	\begin{align}
	&\left|\mu^{-1}\int_{\Rb^{2d}}\psi(\mu \Omega) \chi_0\big(\frac{x-a}{\lambda}\big)\chi_0\big(\frac{y-b}{\lambda}\big) \Phi(x,y)e(x\cdot\xi+y\cdot \zeta) \, \mathrm{d}x\mathrm{d}y-\right.\label{NTintlem2}\\
	&\qquad \qquad \left.\left(\int \psi\right)\int_{\Rb^{2d}}\dirac(\Omega) \chi_0\big(\frac{x-a}{\lambda}\big)\chi_0\big(\frac{y-b}{\lambda}\big)\Phi(x,y) e(x\cdot\xi+y\cdot \zeta) \,\mathrm{d}x\mathrm{d}y\right|\lesssim D \lambda^{2d}\mu^{1-2\eta}(1+|\xi|+|\zeta|),\nonumber
	\end{align}
		uniformly in $(a, b)\in \Rb^{2d}$.
		\end{lem}	
\begin{proof}
The proof is the same as that of Lemma 6.5 of \cite{DH21} (with the weight in (2) replaced by $\langle y\rangle^{1-\eta}$ which does not affect the proof). We omit the details. \end{proof}
\begin{proof}[Proof of Proposition \ref{approxnt}] The proof follows now exactly as the proof of Proposition 6.1 in \cite{DH21} but using Lemma \ref{NTSP}, \ref{lem:minorarcs}, and \ref{NTintlem} to replace Lemma 6.2, 6.4, and 6.5 respectively in \cite{DH21}. 
\end{proof}
With Proposition \ref{approxnt} in hand, we can apply it to the sum in \eqref{KQZdef} exactly as is done in the proof of Proposition 6.7 of \cite{DH21}: First, for any $\nf\in\Nc^{ch}$ we define $x_\nf=k_{\nf_1}-k_\nf$ and $y_{\nf}=k_\nf-k_{\nf_3}$, so we have $\Omega_\nf=2( x_\nf \cdot y_\nf)$ by (\ref{res}). As explained in \cite{DH21}, the linear mapping
\begin{equation*}(x_\nf,y_\nf)_{\nf\in\Nc^{ch}}\leftrightarrow (k_{\lf_1},\cdots, k_{\lf_{2n}})\end{equation*} is volume preserving and preserves the lattice $(\Zb_L^d)^{2n}$, where $(k_{\lf_j})$ are the decorations of some $2n$ leaf pairs (out of the $2n+1$ pairs in total) in the $k$-decoration $\Es$. We can then rewrite the sum in \eqref{KQZdef} as 
\begin{equation}\label{newsum1}\sum_{(x_\nf,y_\nf):\nf\in\Nc^{ch}}\epsilon\cdot\prod_{\nf\in Z}\frac{\chi_\infty(2\delta L^{2\gamma}( x_\nf\cdot y_\nf))}{2\delta L^{2\gamma}( x_\nf \cdot y_\nf)}\cdot\widetilde{\Bc}_{\Qc,Z}(t,s,2\delta L^{2\gamma}( x_\nf \cdot y_\nf):\nf\in\Nc^{ch}\backslash Z)\cdot W(x[\Nc^{ch}],y[\Nc^{ch}]),\end{equation} where $\epsilon=\epsilon_\Es$ (which depends only on $(x_\nf,y_\nf)$) and
$W(x[\Nc^{ch}],y[\Nc^{ch}])=\prod_{j=1}^{2n+1}n_{\mathrm{in}}(k_{\lf_j})$, with each $k_{\lf_j}$ and $k_{2n+1}:=\pm k\pm k_{\lf_{2m+1}}\cdots\pm k_{\lf_{2n}}$ equaling $k$ plus some linear combination of $(x_\nf,y_\nf)$.

The first thing to notice from \eqref{newsum1} is that $\Kc_{\Qc, Z}$ is clearly smooth in $k$, as any derivative in $k$ falls on the $W$ function. This allows to write \eqref{newsum1}, as well as its derivatives in $k$ up to order $40d$, in the form \eqref{defi} as explained in details in the proof of Proposition 6.7 of \cite{DH21}, which allows us to apply Proposition \ref{approxnt}. Consequently, we obtain that if $\Qc$ be a regular couple of order $2n$ where $n\leq N^3$, then we have $\Kc_\Qc(t,s,k)=\sum_{Z\subset \Nc^{ch}}\Kc_{\Qc,Z}(t,s,k)$, where 
\begin{equation}\label{KQZapprox}
\begin{split}
\Kc_{\Qc,Z}(t,s,k)&=\left(\Kc_{\Qc,Z}\right)_{\textrm{app}}(t,s,k)+\Rs_{\Qc, Z},\\
\left(\Kc_{\Qc,Z}\right)_{\textrm{app}}(t,s,k)&=2^{-2n}\delta^n\zeta(\Qc)\prod_{\nf\in Z}\frac{1}{\zeta_\nf \pi i}\cdot\Jc \widetilde B_{\Qc, Z}(t, s)\cdot\Mc_{\Qc,Z}(k),\qquad \textrm{where}\\
\Jc \widetilde B_{\Qc, Z}(t, s)&=\int\widetilde{\Bc}_{\Qc,Z}\big(t,s,\alpha[\Nc^{ch}\backslash Z]\big)\,\mathrm{d}\alpha[\Nc^{ch}\backslash Z].
\end{split}
\end{equation} Here the error term $\Rs_{\Qc, Z}$ satisfies \eqref{remainderbd} for each $(\Qc, Z)$, by examining the gain of power in each error term occurring in Lemmas \ref{NTSP}, \ref{lem:minorarcs}, and \ref{NTintlem}, and noticing that $\gamma_1:=\min(2\gamma,1,2(d-1)(1-\gamma))$. The exact expression $\Mc_{\Qc,Z}(k)$, which is provided in Proposition 6.7 of \cite{DH21}, is not needed here, but we shall need some of its properties which are recalled below. We refer the reader to \cite{DH21} for the complete details, which are the same in our case here. This finishes the proof of Proposition \ref{regcpltreeasymp} (for couples, but the results for regular trees $\Tc$ are proved similarly using the fact that a regular tree forms a regular couple with the trivial tree, see Proposition 6.10 of \cite{DH21}).
\subsubsection{Proof of Proposition \ref{regcpltreecancel}} We start by recalling some facts concerning the structure of regular and dominant couples from Section 4 of \cite{DH21}.  A \emph{regular chain} is a saturated paired tree, obtained by repeatedly applying operation $B$ (attaching one of the mini trees in Figure \ref{fig:minitree} at either a branching node or the lone leaf, as described in Definition \ref{defmini}), starting from the trivial tree $\bullet$. A \emph{regular double chain} is a couple consisting of two regular chains (where, of course, the lone leaves of the two regular chains are paired). It can also be obtained by repeatedly applying operation $B$ at either a branching node or a lone leaf, starting from the trivial couple $\times$. The order of a regular chain $\Tc$ is always an even number $2m$. The $2m$ branching nodes are naturally ordered by parent-child relation; denote them by $\nf_j\,(1\leq j\leq 2m)$ from top to bottom. A \emph{dominant chain} is a special case of regular chain in which the application of operation $B$ is only done at the lone leaves. In this case, we can group the branching nodes $\nf \in \Nc$ from top to bottom as pairs $(\nf_{2j-1}, \nf_{2j})$, $1\leq j \leq m$, which are exactly the branching nodes of the mini tree attached in the $j$-th application of operation $B$. Note that the signs of the nodes $\nf_{2j-1}$ are all the same. 
\begin{figure}[h!]
  \includegraphics[scale=.4]{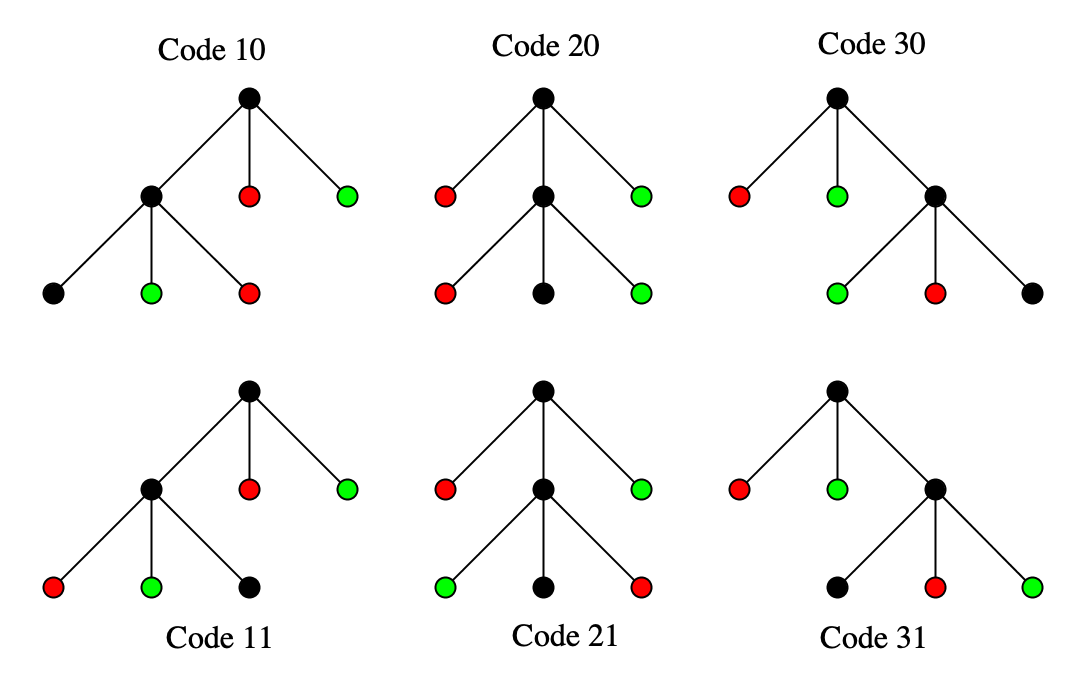}
  \caption{Six possibilities of mini trees (Definition \ref{defmini}).}
  \label{fig:minitree}
\end{figure}

The structure theorem of regular couples states that, for any nontrivial regular couple $\Qc\neq \times$, there exists a regular couple $\Qc_0\neq\times$ which is either a $(1,1)$-mini couple or a regular double chain, such that $\Qc$ is formed by replacing each leaf pair in $\Qc_0$ with a regular couple (cf. Proposition 4.8 of \cite{DH21}). In the first case, we call the couple type 1 and in the second we call it type 2. For type 2 regular couples $\Qc$, if we we require that the couple replacing the lone leaf pair of $\Qc_0$ is trivial or has type 1, then this $\Qc_0$ is unique (cf. Proposition 4.10 of \cite{DH21}). 

\emph{Dominant couples} are a special class of regular couples defined inductively as follows. First the trivial couple $\times$ is dominant. Suppose $\Qc\neq\times$, let $\Qc_0$ be uniquely determined as explained above, and let $\Qc_k\,(k\geq 1)$ be the regular couples in $\Qc$ replacing leaf pairs in $\Qc_0$. Then we define $\Qc$ to be dominant, if (i) $\Qc_0$ is either a $(1,1)$-mini couple or a regular double chain formed by two \emph{dominant} chains, and (ii) each regular couple $\Qc_k$ is dominant. For a dominant couple of type 2, we enumerate the couples $\Qc_k$ replacing leaf pairs as follows: $\Qc_0$ is formed of two dominant chains $\Tc_0^+$ and $\Tc_0^-$; each of $\Tc_0^\pm$ is formed by attaching $m^\pm$ mini trees (from root to lone-leaf) at the nodes $\nf_{2j-1}$ ($1\leq j \leq m^\pm$). Look at the $j$-th mini tree in $\Tc_0^\pm$, we define the dominant couple replacing the pair of \emph{red} leaves in Figure \ref{fig:minitree} by $\Qc_{j,+,1}$, and define the regular couple in $\Qc$ replacing the pair of \emph{green} leaves in Figure \ref{fig:minitree} by $\Qc_{j,+,2}$. Then, for the dominant couple $\Qc$, we have
\begin{equation}\label{newn*}\Nc=\bigg(\bigcup_{j,\epsilon,\iota}\Nc_{j,\epsilon,\iota}\bigg)\cup\Nc_{lp}\cup \big\{\nf_1^+,\cdots,\nf_{2m^+}^+\big\}\cup\big\{\nf_1^-,\cdots,\nf_{2m^-}^-\big\}
\end{equation} and we define
\begin{equation}\label{newnch}\Nc^{ch}=\bigg(\bigcup_{j,\epsilon,\iota}\Nc_{j,\epsilon,\iota}^{ch}\bigg)\cup\Nc_{lp}^{ch}\cup\big\{\nf_{2j-1}^+:1\leq j \leq m^+\big\}\cup\big\{\nf_{2j-1}^-:1\leq j \leq m^-\big\}.
\end{equation} Here in (\ref{newn*}) and (\ref{newnch}), the couples $\Qc_{j,\epsilon,\iota}$, where $\epsilon\in\{\pm\}$ and $\iota\in\{1,2\}$, are the ones described above, and $\Nc_{j,\epsilon,\iota}$ (and $\Nc_{j,\epsilon,\iota}^{ch}$) are defined correspondingly; similarly for $\Qc_{lp}$, $\Nc_{lp}$ and $\Nc_{lp}^{ch}$. The notation for a dominant couple of type 1 is much simpler as we denote by $\Qc_1, \Qc_2, \Qc_3$ the three couples replacing the three leaf pairs (from left to right) of a $(1,1)$-mini couple. Here, we set $\Nc^{ch}=\Nc_1^{ch}\cup\Nc_2^{ch}\cup\Nc_3^{ch}\cup\{\rf\}$ where $\rf$ is the root with $+$ sign.
  
Motivated by \eqref{KQZdef} and still recalling \cite{DH21}, we define below the notion of \emph{special subsets} $Z$ of $\Nc^{ch}$ for a dominant couple $\Qc$. The pair $\Qs:=(\Qc,Z)$ will be called an \emph{enhanced dominant couple}, on which which we also define an equivalence relation $\sim$ between two enhanced dominant couples $\Qs=(\Qc, Z)$ and $\Qs'=(\Qc', Z')$, as follows. First $\varnothing$ is a special subset and the enhanced trivial couple $(\times,\varnothing)$ is only equivalent to itself.

Next, if $\Qs$ is a dominant couple of type 1, then $Z$ is special if and only if $Z=Z_1\cup Z_2\cup Z_3$ (i.e. $\rf$ is \emph{not} in $Z$) where $Z_j\subset \Nc_j^{ch}$ is special. If we denote by $\Qs_j=(\Qc_j,Z_j)$ the three enhanced dominant couples defining $\Qs$, and similarly for $\Qs'$, we say $\Qs\sim\Qs'$ if and only if $\Qs_j\sim\Qs_j'$ for $1\leq j\leq 3$.

Now let $\Qs$ and $\Qs'$ be as before, but suppose $\Qc$ and $\Qc'$ have type $2$. Let $\Qc_0$ be associated with $\Qc$ as explained above, and similarly for $\Qc'$ (same for the other objects appearing below). Suppose the two regular chains of $\Qc_0$ have order $2m^+$ and $2m^-$ respectively, and let the branching nodes in $\Qc_0$ be $\nf_a^\pm(1\leq a\leq 2m^\pm)$. In the construction of $\Qc_0$, at each node $\nf^{\pm}_{2j-1}$ ($1\leq j \leq n$) one of the six mini trees (Figure \ref{fig:minitree}) is attached, and $\nf_{2j}^\pm$ is the other branching node of this mini tree. We define a set $Z\subset \Nc^{ch}$ to be special if and only if
\begin{equation}\label{unionz}Z=\bigg(\bigcup_{j,\epsilon,\iota}Z_{j,\epsilon,\iota}\bigg)\cup Z_{lp}\cup\big\{\nf_{2j-1}^+:j\in Z^+\big\}\cup\big\{\nf_{2j-1}^-:j\in Z^-\big\}
\end{equation} for some special subsets $Z_{j,\epsilon,\iota}\subset \Nc_{j,\epsilon,\iota}^{ch}$ and $Z_{lp}\subset\Nc_{lp}^{ch}$, and some subsets $Z^\pm\subset \{1,\cdots,m^\pm\}$. Similar representations are defined for $\Qs'$. For $\epsilon\in\{\pm\}$ and each $1\leq j\leq m^\epsilon$, consider the tuple $(\mathtt{I}_{j,\epsilon},\mathtt{c}_{j,\epsilon},\Xs_{j,\epsilon,1},\Xs_{j,\epsilon,2})$. Here $\mathtt{I}_{j,\epsilon}=1$ if $j\in Z^\epsilon$ and $\mathtt{I}_{j,\epsilon}=0$ otherwise, $\mathtt{c}_{j,\epsilon}\in\{1,2,3\}$ is the \emph{first digit of} the code of the mini tree attached at the node $\nf_{2j-1}$. Moreover $\Xs_{j,\epsilon,\iota}$ is the equivalence class of the enhanced dominant couple $\Qs_{j,\epsilon,\iota}=(\Qc_{j,\epsilon,\iota},Z_{j,\epsilon,\iota})$ for $\iota\in\{1,2\}$, and let $\Ys$ be the equivalence class of the enhanced dominant couple $\Qs_{lp}=(\Qc_{lp},Z_{lp})$.

We now define $\Qs\sim\Qs'$, if and only if (i) $m^++m^-=(m^+)'+(m^-)'$, and (ii) the tuples coming from $\Qc_0$ (there are total $m^++m^-$ of them) form \emph{a permutation of} the corresponding tuples coming from $\Qc_0'$ (there are total $(m^+)'+(m^-)'$ of them), and (iii) $\Ys=\Ys'$. Finally, note that if $\Qs=(\Qc,Z)$ and $\Qs'=(\Qc',Z')$ are equivalent then $n(\Qc)=n(\Qc')$ and $|Z|=|Z'|$. When $\Qs\sim\Qs'$ with $Z=Z'=\varnothing$, we also say that $\Qc\sim\Qc'$.

Similarly, we can define the notions of dominant trees $\Tc$, special subsets $Z$, enhanced dominant trees $\Ts:=(\Tc, Z)$, and equivalence relations among them, similar to type 2 dominant couples above, except that there there is only one dominant chain, and so there is no lone pair couple $\Qc_{lp}$. In particular, if we denote by $\Tc_0$ the dominant chain of order $2m$ such that $\Tc$ is obtained by replacing leaf pairs of $\Tc_0$ by dominant couples, then the equivalence class of the enhanced dominant tree $\Ts$ is determined by specifying $m$ tuples $(\mathtt{I}_{j},\mathtt{c}_{j},\Xs_{j,1},\Xs_{j,2})$ exactly as defined above.

With this notation in hand, we can recount the main results in Section 7 of \cite{DH21}, which carry over verbatum to our setting here:
\begin{itemize}
\item If $\Qc$ is a regular, but not dominant couple, then $\Jc \widetilde \Bc_{\Qc, Z}(t, s)=0$. This means that the sum in \eqref{realcancel} is only over dominant couples. This is Proposition 7.4 in \cite{DH21}. The same holds for the sum over the regular trees $\Tc$ in \eqref{realcancel}, which is only over dominant trees.
\item Let $\Qc$ be a dominant couple. Then $\Kc_{\Qc}(t, s,k)=\sum_{Z}\Kc_{\Qc, Z}(t,s,k)$ where the sum is over special subsets $Z\subset\Qc^{ch}$ as defined above and $\Kc_{\Qc, Z}(t,s,k)$ is defined in \eqref{KQZapprox}. The function $\Jc\widetilde{\Bc}_{\Qc,Z}(t,s)$ is independent of $Z$ and may be denoted $\Jc\widetilde{\Bc}_{\Qc}(t,s)$. Moreover, these functions satisfy some explicit recurrence relation, described as follows. First $\Jc\widetilde{\Bc}_\Qc(t,s)\equiv 1$ for the trivial couple (Proposition 7.5 of \cite{DH21}). If $\Qc$ has type $1$, then it is formed from the $(1,1)$-mini couple by replacing its three leaf pairs by dominant couples $\Qc_j\,(1\leq j\leq 3)$. In this case, we have
\begin{equation}
\label{recurtype1}
\Jc\widetilde{\Bc}_{\Qc}(t,s)=2\int_0^{\min(t,s)}\prod_{j=1}^3\Jc\widetilde{\Bc}_{\Qc_j}(\tau,\tau)\,\mathrm{d}\tau.
\end{equation} In particular $\Jc\widetilde{\Bc}_{\Qc}=\Jc\widetilde{\Bc}_{\Qc}(\min(t,s))$ is a function of $\min(t,s)$ for type $1$ dominant couples $\Qc$. Finally, if $\Qc$ has type $2$, then $\Qc$ is formed from a regular double chain $\Qc_0$, which consists of two dominant chains, by replacing each leaf pair in $\Qc_0$ with a dominant couple. Using the notations described above for the structure of $\Qc$ and $\Qc_0$ in this case, we have
\begin{multline}
\label{recurtype2}
\Jc\widetilde{\Bc}_{\Qc}(t,s)=\int_{t>t_1>\cdots>t_{m^+}>0}\int_{s>s_1>\cdots >s_{m^-}>0}\prod_{j=1}^{m^+}\Jc\widetilde{\Bc}_{\Qc_{j,+,1}}(t_j,t_j)\Jc\widetilde{\Bc}_{\Qc_{j,+,2}}(t_j,t_j)\\\times\prod_{j=1}^{m^-}\Jc\widetilde{\Bc}_{\Qc_{j,-,1}}(s_j,s_j)\Jc\widetilde{\Bc}_{\Qc_{j,-,2}}(s_j,s_j)\cdot\Jc\widetilde{\Bc}_{\Qc_{lp}}(\min(t_{m^+},s_{m^-}))\prod_{j=1}^{m^+}\mathrm{d}t_j\prod_{j=1}^{m^-}\mathrm{d}s_j.
\end{multline} Here we understand that $t_0=t$ and $s_0=s$.

\item Let $\Qs=(\Qc,Z)$ be an enhanced dominant couple. Let $\Mc_{\Qs}(k)=\Mc_{\Qc,Z}(k)$ be defined as in (\ref{KQZapprox}). Then, the expression $\Mc_\Qs(k)$ is real-valued and depends only on the equivalence class $\Xs$ of $\Qs$, so we can denote it by $\Mc_\Xs(k)$, (cf. Proposition 7.7 of \cite{DH21}).

\item Similarly, for a dominant tree $\Tc$, and with the notation described above, we have
\begin{align}
(\Kc^*_{\Tc})_{\mathrm{app}}(t,s,k)&=\sum_{Z\textrm{\ special}}(\Kc^*_{\Tc,Z})_{\mathrm{app}}(t,s,k) \nonumber \\
(\Kc^*_{\Ts})_{\mathrm{app}}(t,s,k)&=(\Kc^*_{\Tc,Z})_{\mathrm{app}}(t,s,k)=2^{-2n}\delta^n\zeta(\Tc)\prod_{\nf\in Z}\frac{1}{\zeta_\nf \pi i}\cdot\Jc \widetilde B^*_{\Tc}(t, s)\cdot\Mc_{\Ts}^*(k),\nonumber\\
\Jc \widetilde{\Bc}^*_{\Tc}(t,s)&=\int_{t>t_1>\cdots>t_{m^*}>s}\prod_{j=1}^{m^*}\Jc\widetilde{\Bc}_{\Qc_{j,+,1}}(t_j,t_j)\Jc\widetilde{\Bc}_{\Qc_{j,+,2}}(t_j,t_j)\prod_{j=1}^{m^*}\mathrm{d}t_j, \label{recurtypetree}
\end{align}
and $\Mc_{\Ts}^*(k)$ is real-valued and depends only on the equivalence class of the enhanced dominant tree $(\Tc, Z)$. 
\end{itemize}

We are now finally ready to give the proof of \eqref{realcancel}. In fact, we shall split this sum into subsets and show that the sum of each subset is real. To define these subsets, we first notice that 
\begin{equation}\label{realcancel2}
\sum_{n(\Qc)+n(\Tc)=2n}(\Kc_\Qc)_{\mathrm{app}}(t,s,k)\cdot\overline{(\Kc^*_\Tc)_{\mathrm{app}}(t,s,k)}=\sum_{n(\Qs)+n(\Ts)=2n}
(\Kc_{\Qs})_{\mathrm{app}}(t,s,k)\cdot\overline{(\Kc^*_{\Ts})_{\mathrm{app}}(t,s,k)}
\end{equation}
 where the sum is now over enhanced dominant couples $\Qs$ of order $n(\Qs):=n(\Qc)$ and enhanced dominant trees $\Ts$ of order $n(\Ts):=n(\Tc)$. Recall that each $\Qs$ belongs to an equivalence class that is completely determined by specifying $m^++m^-$ tuples $(\mathtt{I}_{j,\epsilon},\mathtt{c}_{j,\epsilon},\Xs_{j,\epsilon,1},\Xs_{j,\epsilon,2})$ ($\epsilon=\pm$) and an equivalence class $\Ys$ for the enhanced dominant couple $\Qs_{lp}=(\Qc_{lp}, Z_{lp})$. Note that $\Qs$ is of type 1 if and only if $m^++m^-=0$, in which case $\Ys$ is the equivalence class of the enhanced dominant couple $\Qs$. Similarly, $\Ts$ belongs to an equivalence class of enhanced dominant trees that is completely determined by specifying $m^*$ tuples $(\mathtt{I}_{j,*},\mathtt{c}_{j,*},\Xs_{j,*,1},\Xs_{j,*,2})$.

Suppose we fix the value $n=m^++m^-+m^*$, fix a collection of $n$ tuples $(\mathtt{I}_{j},\mathtt{c}_{j},\Xs_{j,1},\Xs_{j,2})$, and fix an equivalence class $\Ys$, and then sum in \eqref{realcancel2} \emph{only for $(\Qs, \Ts)$ that belong to equivalence classes formed from this collection}. We will show that this sum is real valued, which then completes the proof. Denote by $\boldsymbol{\Af}$ all possible (enhanced dominant) couple-tree pairs $(\Qs, \Ts)$ that belong to equivalence classes formed from this given collection. Note that for any $(\Qs, \Ts)\in \boldsymbol{\Af}$, the total order $n(\Qc)+n(\Tc)$, the sum $|Z|+|Z^*|$ of the cardinalities of the special subsets $Z$ for $\Qc$ and $Z^*$ for $\Tc$, and the product $\zeta(\Qc) \zeta(\Tc)$ are all the same. Moreover, the product $\Mc_{\Qs}(k) \Mc_{\Ts}^*(k)$ is also the same for all enhanced dominant $(\Qs, \Ts) \in \boldsymbol{\Af}$ since it is a product of factors determined by the tuples $(\mathtt{I}_{j},\mathtt{c}_{j},\Xs_{j,1},\Xs_{j,2})$ and the equivalence class $\Ys$ (see Proposition 7.9 of \cite{DH21}). Hence, for some real-valued function $\Cf(k)$ we have
\begin{equation}\label{realcancel4}
\sum_{(\Qs, \Ts) \in \boldsymbol{\Af}} (\Kc_{\Qs})_{\mathrm{app}}(t,s,k)\cdot\overline{(\Kc^*_{\Ts})_{\mathrm{app}}(t,s,k)}=\Cf(k)\sum_{(\Qs, \Ts) \in \boldsymbol{\Af}}\prod_{\nf \in Z}\frac{1}{\zeta_\nf \pi i}\prod_{\nf \in Z^*}\Jc \widetilde B_{\Qc}(t,s)\frac{-1}{\zeta_\nf \pi i}  \Jc \widetilde B^*_{\Tc}(t,s).
\end{equation}
We will show that this quantity is real, by showing that the sum vanishes unless $Z\cup Z^*$ is empty, in which case reality follows from that the evident reality of $\Jc \widetilde B_{\Qc}(t,s)$ and $\Jc \widetilde B^*_{\Tc}(t,s)$. To see this, let us assume that $|Z|+ |Z^*|\neq 0$ for $(\Qs, \Ts) \in \boldsymbol{\Af}$ (recall that this value is the same for all $(\Qs, \Ts) \in \boldsymbol{\Af}$), and consider the sum on the right hand side of \eqref{realcancel4}. We first note that from Proposition 7.8 of \cite{DH21}, if $\Xs$ is an equivalence class of enhanced dominant couples $(\widetilde \Qc, \widetilde Z)$, with $\widetilde{Z}\neq \varnothing$, then 
\begin{equation}
\Gc_{\Xs}:=\sum_{\widetilde \Qs=(\widetilde \Qc,\widetilde Z)\in\Xs}\bigg(\prod_{\nf\in \widetilde Z}\frac{1}{\zeta_\nf\pi i}\bigg)\cdot \Jc\widetilde{\Bc}_{\widetilde \Qc}(t,t)=0.
\end{equation}
As a result, by \eqref{recurtype1}--\eqref{recurtypetree}, the sum in \eqref{realcancel4} is a linear combination of factors of the form $\Gc_{\Xs_{j,1}}\Gc_{\Xs_{j, 2}}\Gc_{\Ys}$ (possibly with different arguments/variables), hence it vanishes unless all the equivalence classes $\Xs_{j, 1}, \Xs_{j_2}$, and $\Ys$ have empty special subsets (recall that $|Z|$ is constant on an equivalence class). Moreover, since $|Z|+|Z^*|$ is constant over $\boldsymbol{\Af}$, we may replace the factors $\frac{1}{\zeta_\nf \pi i}$ in \eqref{realcancel4} by $\zeta_n$ and prove the vanishing of the resulting expression 
\begin{equation}\label{realcancel3}
\Gc_{\boldsymbol{\Af}}(t,s)=\sum_{(\Qs, \Ts) \in \boldsymbol{\Af}}\prod_{\nf \in Z}\zeta_\nf \Jc \widetilde B_{\Qc}(t,s) \prod_{\nf \in Z^*}(-\zeta_\nf) \Jc \widetilde B^*_{\Tc}(t,s).
\end{equation}

If $(\Qs, \Ts) \in \boldsymbol{\Af}$, then we have a decomposition $n=m^++m^-+m^*$ where $2m^+, 2m^-, 2m^*$ are the orders of three dominant chains $\Tc_0^+, \Tc_0^-, \Tc_0^*$ associated to the dominant couple $\Qc$ and dominant tree $\Tc$, as well as a division of the $m$ tuples $(\mathtt{I}_j,\mathtt{c}_j,\Xs_{j,1},\Xs_{j,2})$ into three groups: one with $m^+$ elements denoted by $(\mathtt{I}_{j,+},\mathtt{c}_{j,+},\Xs_{j,+,1},\Xs_{j,+,2})$ ($1\leq j \leq m^+$), one with $m^-$ elements $(\mathtt{I}_{j,-},\mathtt{c}_{j,-},\Xs_{j,-,1},\Xs_{j,-,2})$ ($1\leq j\leq m^-)$, and one with $m^*$ elements $(\mathtt{I}_{j,*},\mathtt{c}_{j,*},\Xs_{j,*,1},\Xs_{j,*,2})$ ($1\leq j \leq m^*$). Moreover, since $\mathtt{c}_{j,\epsilon}$ are just the \emph{first digits} of the codes of the mini trees appearing in the structure of $\Qc$ and $\Tc$, the corresponding \emph{second digits} can be arbitrary (and $\widetilde{\Bc}_\Qc$ and $\Bc_\Tc$ do not depend on this second digit) which results in a $2^n$ factor. Putting together, if we sum over all possible $(\Qs, \Ts) \in \boldsymbol{\Af}$---which means summing over all possible decompositions of $n=m^++m^-+m^*$ and permutations of the tuples, and then summing over all possible $\Qc_{j,\epsilon,\iota}$ and $\Qc_{lp}$---we would get
\begin{equation} \label{splitformula0}
\begin{aligned}
\Gc_{\boldsymbol \Af}(t)&=2^n\sum_{m^++m^-+m^*=n}\sum_{(\As_1,\cdots, \As_{m^+},\Bs_1,\cdots,\Bs_{m^-}, \Cs_1, \cdots, \Cs_{m^*})}\int_{t>t_1>\cdots >t_{m^+}>0}\int_{s>s_1>\cdots>s_{m^-}>0}\int_{t>u_1>\cdots >u_{m^*}>s}\\&\times \prod_{j=1}^{m^-}(-1)^{\mathtt{I}_{j,-}'} \prod_{j=1}^{m^*}(-1)^{\mathtt{I}_{j,*}'}\prod_{j=1}^{m^+}\Ms(\As_j)(t_j)\prod_{j=1}^{m^-}\Ms(\Bs_j)(s_j)\prod_{j=1}^{m^*}\Ms(\Cs_j)(u_j)\cdot \Gc_{\Ys}(\min(t_{m^+},s_{m^-}))\\
&\qquad\,\, \qquad\qquad \qquad\qquad \qquad\qquad \qquad \qquad \qquad\qquad \qquad\times\mathrm{d}t_1\cdots\mathrm{d}t_{m^+}\mathrm{d}s_1\cdots\mathrm{d}s_{m^-}\mathrm{d}u_1\cdots\mathrm{d}u_{m^*}.
\end{aligned}
\end{equation} Here in (\ref{splitformula0}) the sum is taken over all permutations $(\As_1,\cdots, \As_{m^+},\Bs_1,\cdots,\Bs_{m^-}, \Cs_1, \ldots, \Cs_{m^*})$ of the tuples $(\mathtt{I}_j,\mathtt{c}_j,\Xs_{j,1},\Xs_{j,2})$. Moreover $\mathtt{I}_{j,-}'$ and $\mathtt{I}_{j, *}'$ represent the first (``$\mathtt{I}$") component of $\Bs_j$ and $\Cs_j$, the function $\Ms(\As_j)$ equals $\Gc_{\Xs_{j,+,1}'}\cdot\Gc_{\Xs_{j,+,2}'}$ where $(\Xs_{j,+,1}',\Xs_{j,+,2}')$ represents the last two  (``$\Xs$") components of $\As_j$, and $\Ms(\Bs_j)$ and $\Ms(\Cs_j)$ are defined similarly. Arranging all the time variables in decreasing order, and renaming them as $t>v_1>v_2 >\ldots> v_n>0$, we can write 
\begin{align*}
\Gc_{\boldsymbol \Af}(t)&=2^n\sum_{(\Ds_1, \ldots, \Ds_n)}\int_{t>v_1>\ldots>v_n>0} \prod_{j=1}^n\Ms(\Ds_j)(v_j) \Gc_{\Ys}\left(\min(v_n, s) \right)\\
&\qquad \sum_{A_1\cup A_2\cup A_3=[1,n]} \prod_{j\in A_2}(-1)^{\mathtt{I}_{j}'}\mathbf{1}_{v_j<s} \prod_{j\in A_3}(-1)^{\mathtt{I}_{j}'}\mathbf{1}_{v_j>s} \quad \mathrm{d}v_1\cdots\mathrm{d}v_{n},
\end{align*}
where the first sum is over all permutations $(\Ds_1, \ldots, \Ds_n)$ of the tuples $(\mathtt{I}_j,\mathtt{c}_j,\Xs_{j,1},\Xs_{j,2})$, and the second sum is over all partitions of the set $[1,n]=\{1, 2, \ldots, n\}$ into three subsets $A_1, A_2,A_3$, and $\mathtt{I}_j'$ is the first (``$\mathtt{I}$") component of $\Ds_j$. Now, notice that
$$
\sum_{(A_1,A_2,A_3)} \prod_{j\in A_2}(-1)^{\mathtt{I}_{j}'}\mathbf{1}_{v_j<s} \prod_{j\in A_3}(-1)^{\mathtt{I}_{j}'}\mathbf{1}_{v_j>s}=\prod_{j=1}^n\left(1+(-1)^{\mathtt{I}_{j}'}\mathbf{1}_{v_j<s}+(-1)^{\mathtt{I}_{j}'}\mathbf{1}_{v_j>s}\right)=\prod_{j=1}^n\left(1+(-1)^{\mathtt{I}_{j}'}\right),
$$ if $Z\cup Z^*\neq\varnothing$, then at least one of the $\mathtt{I}_{j}'=1$, so the above product vanishes and hence $\Gc_{\boldsymbol \Af}(t)$ also vanishes. This finishes the proof.
\subsection{Reduction to prime couples}\label{primered} Using Proposition \ref{regcpltreeasymp}, we can reduce the expression $\Kc_\Qc$ for any couple $\Qc$, defined in (\ref{defkq}), to a similar expression associated with the skeleton $\Qc_{\mathrm{sk}}$ of $\Qc$, in the same way as in Section 8.1 of \cite{DH21}.

Recall from (\ref{defkq}) that
\begin{equation}\label{bigformula}\Kc_\Qc(t,t,k)=\bigg(\frac{\delta}{2L^{d-\gamma}}\bigg)^n\zeta(\Qc)\sum_{\Es}\int_\Ec\epsilon_\Es\prod_{\nf\in \Nc} e^{\zeta_\nf\pi i\cdot\delta L^{2\gamma}\Omega_\nf t_\nf}\,\mathrm{d}t_\nf{\prod_{\lf\in\Lc}^{(+)}n_{\mathrm{in}}(k_\lf)},
\end{equation}  where $n$ is the order of $\Qc$, $\Ec$ is the domain defined in (\ref{defdomaine}), $\Es$ is a $k$-decoration and other objects are defined as before, all associated to the couple $\Qc$. By repeating the arguments in Subsection 8.1 of \cite{DH21}, using also the notation $\As$ in Proposition \ref{skeleton}, we get
\begin{multline}\label{bigformula2}\Kc_\Qc(t,t,k)=\Kc_{(\Qc_{\mathrm{sk}},\As)}(t,t,k):=\bigg(\frac{\delta}{2L^{d-\gamma}}\bigg)^{n_0}\zeta(\Qc_{\mathrm{sk}})\sum_{\Es_{\mathrm{sk}}}\int_{\Ec_{\mathrm{sk}}}\epsilon_{\Es_{\mathrm{sk}}}\prod_{\nf\in \Nc_{\mathrm{sk}}} e^{\zeta_\nf\pi i\cdot\delta L^{2\gamma}\Omega_\nf t_\nf}\,\mathrm{d}t_\nf\\\times{\prod_{\lf\in\Lc_{\mathrm{sk}}}^{(+)}\Kc_{\Qc^{(\lf,\lf')}}(t_{\lf^p},t_{(\lf')^p},k_\lf)}\prod_{\mf\in\Nc_{\mathrm{sk}}}\Kc_{\Tc^{(\mf)}}^*(t_{\mf^p},t_{\mf},k_\mf),
\end{multline} where $n_0$ is the order of $\Qc_{\mathrm{sk}}$, $\Ec_{\mathrm{sk}}$ is the domain defined in (\ref{defdomaine}) but with $s=t$, and $\Es_{\mathrm{sk}}$ is a $k$-decoration, the other objects are as before but associated to the couple $\Qc_{\mathrm{sk}}$. Here in (\ref{bigformula2}), the first product is taken over all leaf pairs $(\lf,\lf')$ where $\lf$ has sign $+$, the second product is taken over all branching nodes $\mf$, and $\mf^p$ represents the parent of $\mf$ (if $\mf$ is the root of a tree then $t_{\mf^p}$ should be replaced by $t$).

Note that, using Proposition \ref{regcpltreeasymp}, we can decompose
\begin{equation}\label{inputdecomp}\Kc_{\Qc^{(\lf,\lf')}}=(\Kc_{\Qc^{(\lf,\lf')}})_{\mathrm{app}}+\Rs_{\Qc^{(\lf,\lf')}},\quad \Kc_{\Tc^{(\mf)}}^*=(\Kc_{\Tc^{(\mf)}}^*)_{\mathrm{app}}+\Rs_{\Tc^{(\mf)}}^*,
\end{equation} where the leading term $(\Kc_{\Qc^{(\lf,\lf')}})_{\mathrm{app}}$ and the remainder $\Rs_{\Qc^{(\lf,\lf')}}$ satisfy the bounds (\ref{remainderbd})--(\ref{kqbd}), and similarly $(\Kc_{\Tc^{(\mf)}}^*)_{\mathrm{app}}$ and $\Rs_{\Tc^{(\mf)}}^*$ satisfy the slightly modified bounds as in Proposition \ref{regcpltreeasymp}. The leading terms also satisfy the exact equalities (\ref{conjugatekq})--(\ref{realcancel}).

\section{$\Kc_\Qc$ estimates for vines}\label{funcgroup}
\subsection{Two general estimates for sum-integrals} Before getting to the needed estimates for vines, we first state two general results about expressions that are sums of time-oscillatory integrals.
\begin{lem}\label{sumintest1} Fix $\gamma\leq \frac45-\eta$. For $m\geq 1$, consider the output variables $e,f,g,h,x_0\in\Zb_L^d$ and $t\in[0,1]$, and parameters $\lambda_0,\lambda_j,\mu_j\in\Rb\,(1\leq j\leq m)$. Assume each of $(e,f,g,h,x_0)$ is restricted to a fixed unit ball, and $e-f=h-g:=r\neq 0$ and $|r|\lesssim L^{-\gamma-\eta}$. Let the input variables be $x_j,y_j\in\Zb_L^d\,(1\leq j\leq m)$ and $t_0,t_j,s_j\in[0,1]\,(1\leq j\leq m)$; denote ${\textbf{x}}=(x_0,\cdots,x_m,y_1,\cdots,y_m)$.

Suppose there exist alternative variables $k_j\,(1\leq j\leq 2m+1)$ and $\ell_j\,(1\leq j\leq 2m+1)$, where we write ${\textbf{k}}=(k_1,\cdots,k_{2m+1})$ and $\boldsymbol{\ell}=(\ell_1,\cdots,\ell_{2m+1})$, such that: (i) we have ${\textbf{k}}=T_1{\textbf{x}}+{\textbf{h}}_1$ and $\boldsymbol{\ell}=T_2{\textbf{x}}+{\textbf{h}}_2$ for some matrices $T_j$ and some constant vectors ${\textbf{h}}_j$ depending only on $(e,f,g,h)$, such that all coefficients of $T_1,\,T_1^{-1},\,T_2$ are integers $\lesssim 1$; (ii) for each $1\leq p\leq 2m+1$ there exists $q<p$ such that $\ell_{p}\pm\ell_{q}$ is an integer linear combination of $(k_j)$ and $(e,f,g,h)$ with absolute value sum $\lesssim \Lambda_p$, such that $\Lambda_1\cdots \Lambda_{2m+1}\lesssim C^m$; (iii) any component of the vector ${\textbf{x}}$ (such as $x_j$ or $y_j$) is the sum or difference of two variables, each which is a component of the vector ${\textbf{k}}$ or $\boldsymbol{\ell}$ or $(e,f,g,h)$. Now define\begin{multline}\label{sumintI1}\Ic=\Ic(x_0,e,f,g,h,t):=\sum_{(x_j,y_j):1\leq j\leq m}\prod_{j=1}^{2m+1}\Kc_j(k_j)\prod_{j=1}^{2m+1}\Kc_j^*(\ell_j)\\\times \int_{\Dc} \prod_{j=1}^m e^{\pi i \cdot\delta L^{2\gamma}(t_j-s_j)(x_j\cdot y_j)}\cdot\prod_{j=0}^m e^{\pi i\cdot\delta L^{2\gamma}t_j(r\cdot \zeta_j)}\cdot\prod_{j=0}^me^{\pi i\lambda_j t_j}\prod_{j=1}^m e^{\pi i\mu_js_j}\prod_{j=1}^m\mathrm{d}t_j\mathrm{d}s_j \cdot\mathrm{d}t_0.
\end{multline} Here in (\ref{sumintI1}), $\Dc$ is the domain of $(t_0,t_j,s_j)$ defined by the following conditions: (i) $t_j,s_j>0$ for any $j$, and $0<t_0<t$; (ii) any fixed collection of inequalities of form $s_i<s_j$, $t_i<t_j$, $s_i<t_j$ or $t_i<s_j$. For each $j$, we allow $\Kc_j$ and $\Kc_j^*$ to depend on the parameters $(\lambda_0, \lambda_j, \mu_j)_{1\leq j \leq m}$ and assume
\begin{equation}
\label{propertyKj}\sup_{|\alpha|\leq 40d}\sup_{k}\langle k\rangle^{40d}|\partial^\alpha\Kc_j(k)|\lesssim 1,\qquad \sup_{|\alpha|\leq 40d}\sup_{\ell}|\partial^\alpha\Kc_j^*(\ell)|\lesssim 1.
\end{equation} 
Moreover $\zeta_j=a_jx_j+b_jy_j+c_jr$ for $a_j,b_j,c_j\in\{0,\pm1\}$, and if $j=0$, then $y_0$ should be replaced by one of $(e,f,g,h)$. Finally, assume one of the followings:
\begin{enumerate}[{(a)}]
\item The conditions in the definition of $\Dc$ includes either $s_1<t_0<t_1$ or $t_1<t_0<s_1$,
\item The summand-integrand in $\Ic$ contains an extra factor  of $r$, in addition to (\ref{sumintI1}),
\item The summand-integrand in $\Ic$ contains an extra factor 
 of $|t_1-s_1|^{1-\eta}$.
\item For at least one of the $1\leq j\leq 2m+1$, one of the two bounds in \eqref{propertyKj} is replaced by $L^{-\gamma}$
\end{enumerate}
Then, uniformly in $(\lambda_0,\lambda_j,\mu_j)$ and in the choice of the unit balls containing $(e,f,g,h,x_0)$ described above, and with the norm taken in $(x_0,e,f,g,h,t)$, we have that
\begin{equation}\label{FourierboundI}\|\Ic\|_{X_{\mathrm{loc}}^{2\eta^2,0}}\lesssim (C^+\delta^{-1})^{m}L^{2m(d-\gamma)-\gamma-\eta^2}.
\end{equation}
\end{lem}
\begin{proof}
\underline{\emph{Step 1: Preparations.}} Using the $\langle k\rangle^{-40d}$ decay of $\Kc_j(k)$, we can localize $k_j$ into sets of the form $|k_j-a_j^*|\leq 1$, where $a_j^*\in \Zb^d$ and sum the obtained estimate in $a_j^*$ at the end. Once this set of $a_j^*$ is fixed, we notice that there is $O(\Lambda_1^d)$ choices for the integer part of the coordinates of $\ell_1$. Similarly, once we fix all those integer parts, there are $O(\Lambda_2^d)$ choices for the integer part of the coordinates of $\ell_2$, and so on. As a result, at the expense of a multiplicative factor of $(\Lambda_1 \ldots \Lambda_{2m+1})^d\lesssim C^m$, we may fix points $b_j^*\in \Zb^d$ so that $|\ell_j-b_j^*|\leq 1$. Consequently, since each $x_j,y_j\,(j\geq 1)$ is the difference of some $k_j$ and $\ell_j$, we may assume that for some fixed $a_j, b_j$ we have $|x_j-a_j|\leq 1$ and $|y_j-b_j|\leq 1$ (where $a_j$ and $b_j$ may depend on the unit balls containing $(e,f,g,h,x_0)$). Finally, if $(e,f,g,h,x_0)$ is fixed, we may express one of the $k_j$ (for example $k_{2m+1}$) as an affine linear combination of the others. This allows us to define the change of variables $
(x_1,\cdots,x_m, y_1,\cdots,y_m) \leftrightarrow (k_1, \ldots, k_{2m})$ which is affine and volume preserving, and preserves the lattice $\Zb^d$ (at most up to an absolute constant).

Now set $W(x_1,\cdots,x_m, y_1,\cdots,y_m):=\prod_{j=1}^{2m+1}(\Kc_j(k_j)\Kc_j^*(\ell_j))$, which is a function that satisfies the same conditions \eqref{propertyw1} and \eqref{propertyw2} (with $\widetilde x_j=x_j$ and $\widetilde y_j=y_j$) of Proposition \ref{approxnt} in the variables $(x_j, y_j)_{1\leq j\leq m}$. In fact, \eqref{propertyw1} holds if $W$ is regarded as a function of $(k_j)_{1\leq j \leq 2m}$, and thus it also holds for $(x_j, y_j)_{1\leq j\leq m}$ due to the properties of the affine linear transform. As such, by expanding $W$ in terms of its Fourier transform, we can replace this function by $\prod_{j=1}^m e^{\pi i \xi_j \cdot x_j} e^{\pi i \rho_j \cdot y_j}$, with one extra weight $\langle\xi_j\rangle^{-1}$ or $\langle \rho_j\rangle^{-1}$ if needed. We thus reduce to
\begin{equation}\label{IexpProof1}
\begin{aligned}
\Ic&=\int_0^t\Jc\,\mathrm{d}t_0,\qquad\mathrm{where}\\
\Jc&=\int_{\Dc_1}e^{\pi i\cdot\delta L^{2\gamma}t_0(r\cdot\zeta_0)}\cdot\prod_{j=0}^me^{\pi i\lambda_j t_j}\prod_{j=1}^m e^{\pi i\mu_js_j}\prod_{j=1}^m\mathrm{d}t_j\mathrm{d}s_j\\
&\times\prod_{j=1}^m\sum_{(x_j,y_j)}\chi_0(x_j-a_j)\chi_0(y_j-b_j)e^{\pi i\cdot\delta L^{2\gamma}(t_j-s_j)(x_j\cdot y_j)}\cdot e^{\pi i\cdot\delta L^{2\gamma}t_j(r\cdot\zeta_j)}\cdot e^{\pi i(\xi_j\cdot x_j+\rho_j\cdot y_j)}
\end{aligned}
\end{equation} 
 for some set $\Dc_1$ depending on $t_0$. Below we first study $\Jc$, which we will bound uniformly in all variables and parameters $(\lambda_0,\lambda_j,\mu_j,e,f,g,h,x_0,t_0)$.

\medskip 

\underline{\emph{Step 2: Continuous approximation.}} We shall neglect all the exponential factors in $(t_j,s_j)$ outside the sums in $(x_j, y_j)$ in (\ref{IexpProof1}) and take absolute value outside this sum. This allows to enlarge the domain $\Dc_1$ (but still keeping the restriction in (a) if applicable) and factorize the whole expression $\Jc$ into a product of $m$ terms $\Mc_j$ for $1\leq j\leq m$. We will focus on the $\Mc_1$ factor, as others are similar and easier. This factor then reads
\begin{equation*}
\Mc_1:=\int_{\Dc_2}\bigg|\sum_{(x_1, y_1) \in \Zb^{2d}_L}\chi_0({x_1-a_1})\chi_0({y_1-b_1})e^{\pi i \left[x_1\cdot\xi_1+y_1\cdot \rho_1+ \delta L^{2\gamma}(t_1-s_1)(x_1\cdot y_1)+\delta L^{2\gamma}t_1r\cdot \zeta_1\right]}\bigg|\,\mathrm{d}t_1\mathrm{d}s_1
\end{equation*}
where $\Dc_2=[0,1]^2$ with the extra restriction $s_1<t_0<t_1$ or $t_1<t_0<s_1$ in (a) if applicable. We may change the variables $(t_1, s_1)\to(t_1, u_1)$ where $u_1=\delta L^{2\gamma}(t_1-s_1)$ and split the integral into two regions: where $|u_1|> L$ and where $|u_1|\leq L$. Denote the contributions of those two regions by $A$ and $B$ respectively. 

For term $A$, since $L\leq |u_1| \lesssim L^{2\gamma}$, we have
\begin{align*}
|A|\lesssim& (\delta L^{2\gamma})^{-1}\int_0^1\int_{L}^{\delta L^{2\gamma}}\left|\sum_{(x_1, y_1) \in \Zb^{2d}_L}\chi_0({x_1-a_1})\chi_0({y_1-b_1})e^{\pi i \left[x_1\cdot\xi_1+y_1\cdot \rho_1+ u_1(x_1\cdot y_1)+\delta L^{2\gamma}t_1r\cdot \zeta_1\right]}\right|\mathrm{d}u_1\mathrm{d}t_1\\
\lesssim& (\delta L^{2\gamma})^{-1} L^{2d-2(d-1)(1-\gamma)+\eta}\lesssim (\delta L^{2\gamma})^{-1} L^{2d-\gamma-\eta},
\end{align*}
where we have used Lemma \ref{lem:minorarcs} and the fact that $\gamma<\frac{4}{5}-\eta$.

As for $B$, we do Poisson summation in the variables $(x_1, y_1)$ and obtain that $B\leq B_1+B_2$, where
\begin{align*}
B_1&=L^{2d}(\delta L^{2\gamma})^{-1}\int_{\Rb^2}\bigg|\sum_{(g, h)\in \Zb^{2d}\setminus \{0\}}\int_{\Rb^d\times \Rb^{d}}  \Phi(t_1, u_1, x-a_1, y-b_1) \\
&\qquad \qquad \qquad \qquad \qquad \qquad \times e^{\pi i \left[x_1\cdot(\xi_1-Lg)+y_1\cdot (\rho_1-Lh)+ u_1(x_1\cdot y_1)+\delta L^{2\gamma}t_1r\cdot \zeta_1\right]}\,\mathrm{d}x\mathrm{d}y\bigg|\,\mathrm{d}t_1\mathrm{d}u_1,\\
B_2&=L^{2d}(\delta L^{2\gamma})^{-1}\int_{\Rb^2} \bigg|\int_{\Rb^d\times \Rb^{d}}\Phi(t_1, u_1, x-a_1, y-b_1) e^{\pi i \left[x_1\cdot\xi_1+y_1\cdot \rho_1+ u_1(x_1\cdot y_1)+\delta L^{2\gamma}t_1r\cdot \zeta_1\right]} \,\mathrm{d}x\mathrm{d}y\bigg|\,\mathrm{d}t_1\mathrm{d}u_1.
\end{align*}
Here $\Phi(t_1, u_1, x-a_1, y-b_1)= \mathbf 1_{\Dc_2}\left(t_1, t_1- (\delta L^{2\gamma})^{-1}u_1\right)\mathbf 1_{|u_1|\leq L}\cdot \chi_0({x_1-a_1})\chi_0({y_1-b_1})$ is a function satisfying \eqref{NTphibounds} uniformly in $t_1$. Note that $B_1$ is bounded using Lemma \ref{NTSP} by
$$
L^{2d}(\delta L^{2\gamma})^{-1} (1+|\xi_1+\epsilon\cdot\delta L^{2\gamma}r|+|\rho_1+\epsilon'\cdot\delta L^{2\gamma}r|) \cdot L^{-1+2\eta}
$$
where we integrate in $t_1$ trivially. Moreover, due to the proof of Lemma \ref{NTSP}, the bound can be improved to $L^{2d}(\delta L^{2\gamma})^{-1}\cdot L^{-1+2\eta}$, \emph{unless} $|\xi_1+\epsilon\cdot\delta L^{2\gamma}r|+|\rho_1+\epsilon'\cdot\delta L^{2\gamma}r|\gtrsim L$. In the latter case, since $|r|\lesssim L^{-\gamma-\eta}$, we must have $|\xi_1|+|\rho_1|\gtrsim L$. Therefore we can gain the power $L^{-1+2\eta}$ using the $\langle\xi_1\rangle^{-1}$ or $\langle \rho_1\rangle^{-1}$ as above.

Now we are left with $B_2$. By stationary phase, the integral in $(x,y)$ is bounded by $\langle u_1\rangle^{-d}$, so trivially $|B_2|\lesssim \delta^{-1}L^{2(d-\gamma)}$ if without extra restrictions (this applies for example to $\Mc_j$ for $j>1$). For $j=1$ we have one of the assumptions (a)--(c) available, which will lead to further power gains. In fact, since $|r|\lesssim L^{-\gamma-\eta}$, case (b) already implies the gain $L^{-\gamma-\eta}$ which suffices for (\ref{FourierboundI}); for (c), the extra factor $|t_1-s_1|^{1-\eta}= (\delta L^{2\gamma})^{-1+\eta} |u_1|^{1-\eta}$ can be transformed to a gain of $(\delta L^{2\gamma})^{-1+\eta}\ll L^{-\gamma-\eta}$ since $\langle u_1\rangle^{-d}|u_1|^{1-\eta}$ is still integrable in $u_1$. Finally, in case (a), we have that $s_1\leq t_0\leq t_1$ or $t_1\leq t_0 \leq s_1$, which means that $|t_1-t_0|\leq |t_1-s_1|=(\delta L^{2\gamma})^{-1} |u_1|$. Thus, integrating in $t_1$ with fixed $u_1$ gives an extra factor of $(\delta L^{2\gamma})^{-1} |u_1|$ in estimating $B_2$, which also gives the gain $(\delta L^{2\gamma})^{-1}$.

Summing up, in all cases we have proved that $|\Mc_1|\lesssim \delta^{-1}L^{2(d-\gamma)-\gamma-\eta}$ and $|\Mc_j|\lesssim \delta^{-1}L^{2(d-\gamma)}$ for $j>1$, which then implies $|\Jc|\lesssim (C^+\delta^{-1})^m L^{2m(d-\gamma)-\gamma-\eta}$ uniformly in $(e,f,g,h,x_0)$ and $t_0$.

\medskip

\underline{\emph{Step 3: Going from $\Jc$ to $\Ic$.}} Recall that $\Ic$ is defined as in (\ref{IexpProof1}) and $\Jc=\Jc(t_0)$ is function valued in a Banach space $\Xf:=L_{(e,f,g,h,x_0)}^\infty$ with $\|\Jc\|_{L^\infty}\lesssim (C^+\delta^{-1})^m L^{2m(d-\gamma)-\gamma-\eta}:=\Ac$. Now it suffices to prove that $\|\Ic\|_{X_{\mathrm{loc}}^{2\eta^2,0}}\lesssim L^{100d\eta^2}\Ac$. Note that, if we insert suitable time cutoff functions to the definition of $\Ic$ in (\ref{IexpProof1}), then a simple integration by parts with (\ref{IexpProof1}) gives
\begin{equation}\label{JtoIbound}\|\widehat{I}(\xi)\|_{\Xf}\lesssim \langle \xi\rangle^{-1}\bigg\|\int_\Rb\chi(t_0) \Jc(t_0)e^{i\xi t_0}\,\mathrm{d}t_0\bigg\|_{\Xf}+\langle \xi\rangle^{-2}\|\Jc\|_{L^\infty}\lesssim \langle \xi\rangle^{-1}\|\Jc\|_{L^\infty},\end{equation} which immediately implies (\ref{FourierboundI}) if we restrict to $|\xi|\leq L^{50d}$.

Now if instead $|\xi|\geq L^{50d}$, then we may replace the $L_{(e,f,g,h,x_0)}^\infty$ norm by the $L_{(e,f,g,h,x_0)}^1$ norm, and then sum over $(e,f,g,h,x_0)$ (which has $\lesssim L^{5d}$ choices due to the support assumption). Then we are dealing with the expression $\Ic$ (or $\Jc$) with the values of $(e,f,g,h,x_0)$ fixed, so the value of $\alpha:=\pi (\delta L^{2\gamma}(r\cdot\zeta_0)+\lambda_0)$ is also fixed. Integrating by parts in $t_0$ in (\ref{JtoIbound}) again, we get
\begin{equation}\label{JtoIbound2}\|\widehat{I}(\xi)\|_{\Xf}\lesssim\langle \xi\rangle^{-1}(\langle \xi-\alpha\rangle^{-1}+\langle \xi\rangle^{-1})\cdot\|\widetilde{\Jc}\|_{L^\infty},\end{equation} where $\widetilde{\Jc}=(\partial_{t_0}-i\alpha)\Jc$. By (\ref{IexpProof1}), this $\widetilde{\Jc}$ is defined similar to $\Jc$, but without the $e^{\pi i\cdot\delta L^{2\gamma}t_0(r\cdot\zeta_0)}$ and $e^{\pi i\lambda_0t_0}$ factors and with one of the $t_j$ or $s_j$ variables fixed. Going back to the argument in Step 2, we can then estimate at most one $\Mc_j$ trivially (say $|\Mc_j|\lesssim L^{2d}$) while the other $\Mc_j$ are still bounded as above, and gain a power $\langle \xi\rangle^{-1/2}\lesssim L^{-25d}$ from (\ref{JtoIbound2}), which completes the proof.
\end{proof}
\begin{lem}\label{sumintest2} Fix $\gamma\leq \frac12$ and $m\geq 0$, consider the same setting as in Lemma \ref{sumintest1}, but with the following differences. First, there are three more vector variables which we call $(u_1,u_2,u_3)$ in addition to $(x_0,x_j,y_j)$; accordingly, there are three more time variables $(\tau_1,\tau_2,\tau_3)$ in addition to $(t_0,t_j,s_j)$ and three more parameters $(\sigma_1,\sigma_2,\sigma_3)$ in addition to $(\lambda_0,\lambda_j,\mu_j)$. There are also three more alternative variables $(k_{2m+2},k_{2m+3},k_{2m+4})$ in additional to $(k_j)$ and $(\ell_{2m+2},\ell_{2m+3},\ell_{2m+4})$ in addition to $(\ell_j)$. Let the definitions of ${\textbf{x}}$, ${\textbf{k}}$ and $\boldsymbol{\ell}$ include these extra variables, then they satisfy the same assumptions as in Lemma \ref{sumintest1}.

Then $\Ic$ is defined as in (\ref{sumintI1}), but we also sum over the new $u_j$ variables and integrate over the new $\tau_j$ variables. The domain $\Dc$ defined in the same way as in Lemma \ref{sumintest1} (except we include the new $\tau_j$ variables). The summand-integrand in $\Ic$ is the same as in (\ref{sumintI1}), and the bounds in (\ref{propertyKj}) are also the same, except that (i) we also include the new $k_j$ and $\ell_j$ variables and the corresponding $\Kc_j$ and $\Kc_j^*$ functions that also satisfy (\ref{propertyKj}), and (ii) we also include the extra factors \[e^{\pi i\cdot \delta L^{2\gamma}(\tau_1-\tau_3)(u_1\cdot u_2)},\quad e^{\pi i\cdot \delta L^{2\gamma}(\tau_2-\tau_3)\Lambda},\quad e^{\pi i\cdot \delta L^{2\gamma}\tau_3(r\cdot\xi)},\quad e^{\pi i\sigma_j\tau_j}\,(1\leq j\leq 3),\] where $\Lambda\in\{u_1\cdot u_3,u_3\cdot (u_1+u_2-u_3)\}$, and $\xi=au_1+bu_2+cu_3+dr$ with $a,b,c,d\in\{0,\pm1\}$. Finally, we do not require one of the scenarios (a)--(c) to happen as in Lemma \ref{sumintest1}. 

Then, uniformly in $(\lambda_0,\lambda_j,\mu_j,\sigma_j)$ and in the choice of the unit balls containing $(e,f,g,h,x_0)$ described above, and with the norm taken in $(x_0,e,f,g,h,t)$, we have that
\begin{equation}\label{FourierboundI2}\|\Ic\|_{X_{\mathrm{loc}}^{2\eta^5,0}}\lesssim (C^+\delta^{-1})^{m+2}L^{(2m+4)(d-\gamma)-d+\eta^4}.
\end{equation}
\end{lem}
\begin{proof}We perform the same reduction procedure as in the proof of Lemma \ref{sumintest1}. By restricting each $x_j,y_j$ and $u_j$ to a unit ball, expanding $W$ (which is the product of all the $\Kc_j$ and $\Kc_j^*$ factors) using Fourier integral, writing $\Ic$ as an integral of $\Jc$ in $t_0$, and enlarging the time integration domain, we can reduce to $m+1$ expressions, which we denote by $\Mc_j\,(1\leq j\leq m)$ and $\Mc_*$. In fact, the expressions $\Mc_j$ are exactly the same as in the proof of Lemma \ref{sumintest1}, and satisfy the same bounds $|\Mc_j|\lesssim \delta^{-1}L^{2(d-\gamma)}$; therefore it suffices to study $\Mc_*$, which has the expression
\begin{multline}\Mc_*:=\int_{[0,1]^3}\bigg|\sum_{(u_1,u_2,u_3)\in\Zb_L^{3d}}\chi_0(u_1-c_1)\chi_0(u_2-c_2)\chi_0(u_3-c_3)\\
\times e^{\pi i\left[u_1\cdot\nu_1+u_2\cdot\nu_2+u_3\cdot\nu_3+\delta L^{2\gamma}(\tau_1-\tau_3)(u_1\cdot u_2)+\delta L^{2\gamma}(\tau_2-\tau_3)\Lambda+\delta L^{2\gamma}\tau_3(r\cdot\xi)\right]}\bigg|\,\mathrm{d}\tau_1\mathrm{d}\tau_2\mathrm{d}\tau_3,
\end{multline} with fixed vectors $\nu_j$. By integrating trivially in $\tau_3$, defining $\theta_j=\delta L^{2\gamma}(\tau_j-\tau_3)$ for $1\leq j\leq 2$ with fixed $\tau_3$, and applying Poisson summation as in the proof of Lemma \ref{sumintest1} above, we can reduce to
\begin{multline}\label{IexpProof3}
\Mc_*\leq(\delta L^{2\gamma})^{-2}L^{3d}\sup_{\tau_3}\int_{|\theta_j|\lesssim \delta L^{2\gamma}}\bigg|\sum_{(f_1, f_2, f_3)\in \Zb^{3d}} \int_{\Rb^{3d}}\prod_{j=1}^3\chi(u_j-c_j)e^{\pi i u_j \cdot(\nu_j-Lf_j)}\\
\times e^{\pi i\cdot \left[\theta_1(u_1\cdot u_2)+\theta_2\Lambda+\rho\cdot\xi\right]}\,\mathrm{d}u_1\mathrm{d}u_2\mathrm{d}u_3\bigg|\,\mathrm{d}\theta_1\mathrm{d}\theta_2,\end{multline}
where $\rho:=\delta L^{2\gamma}\tau_3 \cdot r$. To control this last expression, note that $|\theta_1|, |\theta_2|\lesssim L$ due to the assumption $\gamma\leq \frac12$. For fixed $(\theta_1,\theta_2)$, the phase function
$$
\Phi(u_1, u_2, u_3)=\sum_{j=1}^3 u_j \cdot(\nu_j-Lf_j) + \theta_1(u_1\cdot u_2)+\theta_2\Lambda+\rho\cdot\xi
$$
satisfies $|\nabla_{u_j}\Phi|\geq |Lf_j-C_j|-O(L)$ for some fixed $C_j$ (which may depend on $\theta_j$), therefore for all but $O(1)$ values of $f_j$, the integral in $(u_1,u_2,u_3)$ can be controlled trivially by integrating by parts. For these $O(1)$ values of $f_j$, a simple stationary phase argument yields that the integral in $(u_1,u_2,u_3)$ integral is bounded by $\min(\langle \theta_1\rangle^{-d},\langle \theta_2\rangle^{-d})$, which is integrable in $(\theta_1,\theta_2)$ for $d\geq 3$. This implies that $\Mc_*\lesssim \delta^{-2}L^{3d-4\gamma}$, and combining with the bounds for other $\Mc_j$ implies that
\[|\Jc|\lesssim (C^+\delta^{-1}L^{2(d-\gamma)})^m\cdot C^+\delta^{-2}L^{3m-4\gamma}.\] Then, repeating the last part of the proof in Lemma \ref{sumintest1} and noticing that we are allowed to lose $L^{\eta^4}$ here, we can easily deduce (\ref{FourierboundI2}).
\end{proof}
\begin{cor}\label{sumintest3} Fix $\gamma\leq 1/2$. Consider the following setting, which is basically a ``concatenation" of Lemma \ref{sumintest2}: the output variables are denoted $(e,f,g,h,x_0^0,t)$, with $(e,f,g,h,x_0^0)$ each in a unit fixed ball. The input variables are $(x_0^q,x_j^q,y_j^q,u_1^q,u_2^q,u_3^q)$ and $(t_0^q,t_j^q,s_j^q,\tau_1^q,\tau_2^q,\tau_3^q)$ where $0\leq q<Q$ and $1\leq j\leq m_q$ (but excluding $x_0^0$), with $m_0+\cdots +m_Q=m$. The parameters are $(\lambda_0^q,\lambda_j^q,\mu_j^q,\sigma_1^q,\sigma_2^q,\sigma_3^q)$ as above. The alternative variables are ${\textbf{k}}=(k_1,\cdots,k_{2m+4Q})$ and $\boldsymbol{\ell}=(\ell_1,\cdots,\ell_{2m+4Q})$, and ${\textbf{x}}=(x_0^0,x_0^q, x_j^q, y_j^q, u_1^q, u_2^q, u_3^q)_{0\leq q<Q,1\leq j\leq m_q}$ (including $x_0^0$) satisfy the same properties as in Lemma \ref{sumintest2} and Lemma \ref{sumintest1}), with $\Lambda_1\cdots\Lambda_{2m+4Q}\lesssim C^{m+Q}$ in assumption (ii) of Lemma \ref{sumintest1}. The expression $\Ic$ is defined as in (\ref{sumintI1}) with summation and integration in all input variables (vector and time, see above). The domain $\Dc$ is defined in the same way but involves all the time variables, and the summand-integrand in $\Ic$ includes the following factors:
\begin{itemize}
\item Functions $\Kc_j(k_j)$ and $\Kc_j^*(\ell_j)$ for $1\leq j\leq 2m+4Q$ that each satisfies (\ref{propertyKj});
\item All the factors $e^{\pi i(\cdots)}$ occurring in Lemma \ref{sumintest1} and \ref{sumintest2} \emph{except $e^{\pi i\cdot\delta L^{2\gamma}t_0(r\cdot \zeta_0)}$}, for all $0\leq q<Q$ (such as $e^{\pi i\cdot\delta L^{2\gamma}(t_j^q-s_j^q)(x_j^q\cdot y_j^q)}$, $e^{\pi i(\lambda_j^qt_j^q+\mu_j^qs_j^q)}$, $e^{\pi i\cdot \delta L^{2\gamma}(\tau_1^q-\tau_3^q)(u_1^q\cdot u_2^q)}$, $e^{\pi i\cdot\delta L^{2\gamma}t_j^q(r\cdot \zeta_j^q)}$ for $j\neq 0$);
\item Extra factors of $e^{\pi i\cdot\delta L^{2\gamma}t_0^q(r\cdot\zeta_0^q)}$ for $0\leq q<Q$, where $\zeta_0^0$ equals $x_0^0$ plus or minus a vector in $(e,f,g,h)$, and $\zeta_0^q\,(1\leq q<Q)$ is an arbitrary linear combination of the $(x,y,u)$ variables and $(e,f,g,h)$.
\end{itemize} Then, uniformly in $(\lambda,\mu,\sigma)$ parameters and in the choice of the unit balls containing $(e,f,g,h,x_0^0)$ described above, and with the norm taken in $(x_0^0,e,f,g,h,t)$, we have that
\begin{equation}\label{FourierboundI3}\|\Ic\|_{X_{\mathrm{loc}}^{2\eta^5,0}}\lesssim (C^+\delta^{-1})^{m+2Q}L^{(2m+4Q)(d-\gamma)-d+\eta^4}.
\end{equation}
\end{cor}
\begin{proof} The proof is almost the same as Lemma \ref{sumintest2}. We write $\Ic$ as an integral of $\Jc$ in $t_0^0$, and use the same arguments as in the proof of Lemma \ref{sumintest1} and \ref{sumintest2} to reduce the $X_{\mathrm{loc}^{2\eta^5,0}}$ norm of $\Ic$ to the $L^\infty$ norm of $\Jc$ with $L^{\eta^4}$ loss. Then, in estimating $\Jc$ we may make the same reductions, sum in $x_0^q$ variables an integrate in $t_0^q$ variables trivially, and then reduce to the same $\Mc_j^q$ and $\Mc_*^q$ quantities as occurred in the proof of Lemma \ref{sumintest2}. These quantities are then estimated in the same way, noticing that the linear phases $e^{\pi i\cdot\delta L^{2\gamma}t_0^q(r\cdot\zeta_0^q)}$ do not affect any part of the proof. Putting together, and noticing that the summation in each $x_0^q$ variables leads to a factor of $L^d$, this proves (\ref{FourierboundI3}).
\end{proof}
\subsection{Application to vines} Using Lemmas \ref{sumintest1}--\ref{sumintest2}, Definition \ref{twistgen} and Remark \ref{fulltwistrep}, we can prove the estimates regarding the part of the expression $\Kc_{(\Qc_{\mathrm{sk}},\As)}$ in (\ref{bigformula2}) where we only sum and integrate over the variables corresponding to a bad or normal vine $\Vb$ in $\Qc_{\mathrm{sk}}$.
\subsubsection{Estimates for vines}\label{vinesubset} Given a couple $\Qc_0$ (not necessarily prime) and a collection $\As$ of regular trees and regular couples as in Proposition \ref{skeleton}, let $\Qc\sim(\Qc_0,\As)$. Fix also a (CL) vine $\Vb\subset\Mb(\Qc_0)$, then we can write  $(\Qc_0,\As)\leftrightarrow(\Qc^{\mathrm{sp}}, \texttt{cod},\mathfrak{n},\texttt{ind},\Bs,\As^{\mathrm{sp}})$ by Remark \ref{fulltwistrep}. we shall fix $(\Qc^{\mathrm{sp}}, \texttt{cod},\mathfrak{n},\As^{\mathrm{sp}})$ and let $\mathtt{ind}$ and $\Bs$ vary as in Remark \ref{fulltwistrep}; in particular $\mathtt{sgn}$ is fixed as in the notion of Definition \ref{twistrep}.

Define $n_1$ to be the number of branching nodes in $\Qc_0[\Vb]\backslash\{\uf_1\}$, and $n_2=n_1+n(\Bs)$. Consider the formula (\ref{bigformula2}) but with $\Qc_{\mathrm{sk}}$ replaced by $\Qc_0$ (and associated notations changed accordingly, like $\Es_{\mathrm{sk}}$ replaced by $\Es_0$ etc.). For later purposes, we also fix a set $W\subset\Qc_0[\Vb]$ containing $\uf_1$, and a Banach space valued function $Z=Z(x_0,k_{\uf_1},k_{\uf_{11}},k_{\uf_{21}},k_{\uf_{22}},t_{\uf_1},t_{\uf_{21}},t_{\uf_{22}},t_{\uf_2})$. Define the expression $\Kc_{(\mathtt{sgn},\mathtt{ind},\Bs)}^{(\Vb,Z,W)}$ similar to $\Kc_{(\Qc_0,\As)}$ in (\ref{bigformula2}), but that:
\begin{enumerate}[{(a)}]
\item We replace the power $(\delta/(2L^{d-\gamma}))^{n_0}$ (where $n_0$ is the order of $\Qc_0$) by $(\delta/(2L^{d-\gamma}))^{n_1}$;
\item We replace the factor $\zeta(\Qc_0)$ by the product of $i\zeta_\nf$ where $\nf$ runs over all branching nodes in $\Qc_0[\Vb]\backslash\{\uf_1\}$;
\item In the summation $\sum_{\Es_0}(\cdots)$, we only sum over the variables $k_\nf$ for $\nf\in\Qc_0[\Vb]\backslash\{\uf_1\}$ (including branching nodes and leaves), and treat the other $k_\nf$ variables as fixed. We also replace $\epsilon_{\Es_0}$ by the product of factors on the right hand side of (\ref{defcoef}), but only for $\nf\in W$. In particular, we have that $k_{\uf_{21}}\neq k_{\uf_{22}}$.
\item In the integral $\int_{\Ec_0}(\cdots)$, we only integrate over the variables $t_\nf$ for all branching nodes $\nf\in\Qc_0[\Vb]\backslash\{\uf_1\}$, and treat the other $t_\nf$ variables as parameters.
\item In the first product $\prod_{\nf\in\Nc_0}(\cdots)$, we only include those factors where $\nf\in\Qc_0[\Vb]$; in the products $\prod_{\mf\in\Lc_0}^{(+)}(\cdots)$ and $\prod_{\mf\in\Nc_0}(\cdots)$, we only include those factors where $\mf\in\Qc_0[\Vb]\backslash\{\uf_1\}$.
\item We include the function $Z$ as a factor in the summand-integrand, where in the place of $x_0$ we plug in $k_{\uf_2}$ if $\mathtt{ind}=-$, and plug in $k_{\uf_{23}}$ otherwise.
\end{enumerate}

Note that $\Kc_{(\mathtt{sgn},\mathtt{ind},\Bs)}^{(\Vb,Z,W)}$ depends on $(\Vb,Z,W)$ and $(\mathtt{ind},\Bs)$, and depends on $(\Qc^{\mathrm{sp}},\nf)$ \emph{only via $\mathtt{sgn}$}. Its output variables are $(k_{\uf_1},k_{\uf_{11}},k_{\uf_{21}},k_{\uf_{22}})$ and time variables $(t_{\uf_1},t_{\uf_{21}},t_{\uf_{22}})$. Since $k_{\uf_1}-k_{\uf_{11}}=\pm(k_{\uf_{21}}-k_{\uf_{22}})$, we may write this function as \begin{equation}\label{eqnx0'}\Kc_{(\mathtt{sgn},\mathtt{ind},\Bs)}^{(\Vb,Z,W)}=\Kc_{(\mathtt{sgn},\mathtt{ind},\Bs)}^{(\Vb,Z,W)}(x_0',k_{\uf_{21}},k_{\uf_{22}},t_{\uf_1},t_{\uf_{21}},t_{\uf_{22}}),\end{equation} where $x_0'$ is replaced by $k_{\uf_1}$ if $\mathtt{sgn}=-$, and by $k_{\uf_{11}}$ otherwise.
\begin{rem}\label{remkv} Suppose the function $Z$ \emph{does not depend on $(k_{\uf_1},k_{\uf_{11}})$.} If we fix $(\Vb,Z,W,\mathtt{ind},\Bs)$ and flip $\mathtt{sgn}$ (which corresponds to flipping $\Qc_0[\Vb]$), then the function $\Kc_{(\mathtt{sgn},\mathtt{ind},\Bs)}^{(\Vb,Z,W)}$, \emph{written exactly as in (\ref{eqnx0'}), does not change.} In fact, consider the decorations of $\Qc_0[\Vb]$ and its flipping, see Figure \ref{fig:flip}, where \emph{the values of $k_{\uf_1}$ and $k_{\uf_{11}}$ are switched}. If we write down the whole expression $\Kc_{(\mathtt{sgn},\mathtt{ind},\Bs)}^{(\Vb,Z,W)}$, which is (\ref{bigformula2}) modified by (a)--(f) above, in these two cases, we can verify that they match exactly term by term, hence the results are the same.
\end{rem}
\begin{prop}\label{estbadvine} Suppose $\Vb$ is a \emph{bad} (CL) vine. In the above setting, let $(\Qc^{\mathrm{sp}}, \mathtt{cod},\mathfrak{n},\As^{\mathrm{sp}})$ and $(Z,W)$ be fixed, then for $\theta\in\{\eta^5,0\}$ we have
\begin{equation}\label{vinebound1}
\bigg\|e^{-\pi i\cdot\delta L^{2\gamma}t_{\uf_1}\Gamma}\sum_{\mathtt{ind},\Bs}\Kc_{(\mathtt{sgn},\mathtt{ind},\Bs)}^{(\Vb,Z,W)}(x_0',k_{\uf_{21}},k_{\uf_{22}},t_{\uf_1},t_{\uf_{21}},t_{\uf_{22}})\bigg\|_{Y_{\mathrm{loc}}^{\theta}}\lesssim (C^+\sqrt{\delta})^{n_2}L^{-\eta^2}\|Z\|_{Y_{\mathrm{loc}}^{\theta}}.
\end{equation} Here $\Gamma:=\zeta_{\uf_{11}}|k_{\uf_{11}}|^2+\zeta_{\uf_{21}}|k_{\uf_{21}}|^2+\zeta_{\uf_{22}}|k_{\uf_{22}}|^2-\zeta_{\uf_1}|k_{\uf_1}|^2$, and the summation in (\ref{vinebound1}) is taken over all choices of $(\mathtt{ind},\Bs)$ as in Remark \ref{fulltwistrep}. More precisely, if $\Vb$ is vine (II-e) then $\mathtt{ind}$ and $\Bs$ are fixed; if $\Vb$ is not vine (II-e) then $\mathtt{ind}\in\{\pm\}$ and $\Bs$ is such that $\Qc^{(\lf,\lf')}$ and $\Tc^{(\mf)}$ are fixed when $\mf\neq\uf_2$ and $(\lf,\lf')\neq (\uf_{23},\uf_0)$, and the value of $n(\Qc^{(\uf_{23},\uf_0)})+n(\Tc^{(\uf_2)})$ is also fixed.
\end{prop}
\begin{proof} The proof is elaborate due to the many cases and arguments involved, so we divide it into several steps. We assume throughout that $\|Z\|_{Y_{\mathrm{loc}}^\theta}=1$.

\medskip
\underline{\it Step 1: Reductions.} By definition, for each fixed $(\mathtt{ind},\Bs)$, the summand on the left hand side of \eqref{vinebound1} takes the form (we omit dependence on $(\Vb,Z,W,\mathtt{sgn})$ for convenience, same below)
\begin{multline}\label{vineboundBfix}
\widetilde \Kc_{(\mathtt{ind},\Bs)}= \bigg(\frac{\delta}{2L^{d-\gamma}}\bigg)^{n_1}\zeta[\Vb] \sum_{\Es [\Vb]}e^{\pi i \widetilde{\Gamma}t_{\uf_1}} \epsilon_{\Es[\Vb]}\cdot\int_{ \Ec[\Vb]} Z(x_0,k_{\uf_1},k_{\uf_{11}},k_{\uf_{21}},k_{\uf_{22}},t_{\uf_1},t_{\uf_{21}},t_{\uf_{22}},t_{\uf_2})\\\times \prod_{\nf\in \Nc[\Vb]} e^{\zeta_\nf\pi i\cdot\delta L^{2\gamma}\Omega_\nf t_\nf}\,{\prod_{\lf\in
\Lc[\Vb]}^{(+)}\Kc_{\Qc^{(\lf,\lf')}}(t_{\lf^p},t_{(\lf')^p},k_\lf)}\prod_{\mf\in\Nc[\Vb]}\Kc_{\Tc^{(\mf)}}^*(t_{\mf^p},t_{\mf},k_\mf)\,\mathrm{d}t_{\mf}.
\end{multline} Here, we define $\zeta [\Vb]$ as the product of $i\zeta_\nf$ where $\nf$ runs over all branching nodes in $\Qc_0[\Vb]\backslash\{\uf_1\}$; similarly define $\Es[\Vb], \epsilon_{ \Es[\Vb]},\Ec[\Vb]$ and $ \Nc[\Vb],\Lc[\Vb]$ as in the definition of $\Kc_{(\mathtt{sgn},\mathtt{ind},\Bs)}^{(\Vb,Z,W)}$ above (In particular $\uf_1\not\in\Nc[\Vb]$). Let $\uf_{12}$ and $\uf_{13}$ be the two children of $\uf_{1}$ other than $\uf_{11}$, note that $$\widetilde \Gamma:=\delta L^{2\gamma}\left(\zeta_{\uf_{12}}|k_{\uf_{12}}|^2+\zeta_{\uf_{13}}|k_{\uf_{13}}|^2-\zeta_{\uf_{21}}|k_{\uf_{21}}|^2-\zeta_{\uf_{22}}|k_{\uf_{22}}|^2\right)=-\delta L^{2\gamma}\sum_{\nf\in\Nc[\Vb]} \zeta_\nf \Omega_\nf,
$$
since each factor $\zeta_\nf |k_{\nf}|^2$ with $\nf\in\Qc_0[\Vb] \setminus\{\uf_1\}$ appears twice in the sum with opposite signs except for $\uf_{12}$ and $\uf_{13}$. Thus, we can rewrite \eqref{vineboundBfix} as 
\begin{multline}\label{vineboundBfix2}
\widetilde \Kc_{(\mathtt{ind},\Bs)}= \bigg(\frac{\delta}{2L^{d-\gamma}}\bigg)^{n_1}\zeta[\Vb] \sum_{\Es [\Vb]}\epsilon_{\Es[\Vb]}\cdot \int_{\Ec[\Vb]} Z(x_0,k_{\uf_1},k_{\uf_{11}},k_{\uf_{21}},k_{\uf_{22}},t_{\uf_1},t_{\uf_{21}},t_{\uf_{22}},t_{\uf_2})\\\times \prod_{\nf\in \Nc[\Vb]} e^{\zeta_\nf\pi i\cdot\delta L^{2\gamma}\Omega_\nf (t_\nf-t_{\uf_1})}\,{\prod_{\lf\in
\Lc[\Vb]}^{(+)}\Kc_{\Qc^{(\lf,\lf')}}(t_{\lf^p},t_{(\lf')^p},k_\lf)}\prod_{\mf\in \Nc[\Vb]}\Kc_{\Tc^{(\mf)}}^*(t_{\mf^p},t_{\mf},k_\mf)\,\mathrm{d}t_{\mf}.
\end{multline}

\underline{\it Step 2: Parametrization for vines (II).} We start by reparametrizing the expression of $\widetilde \Kc$ in \eqref{vineboundBfix2}. For this, we will need the information contained in Figures \ref{fig:couples_cl} and Figure \ref{fig:vineAnnotated} (A) for vine (II). From the latter figure, we denote by $v_1, v_2, v_{2j+1}, v_{2j+2}\,(1\leq j\leq m)$ the $2m+2$ atoms forming the vine with joints at $v_1, v_2$, such that $v_{2j+1}$ is connected by a double bond to $v_{2j+2}$. We denote by $\uf_{1}, \uf_2, \ldots, \uf_{2m+2}$ the corresponding branching nodes in the couple, which are the branching nodes of $\Qc_0[\Vb]$. Given a decoration $(k_{\nf})$ for $\nf\in \Qc_0[\Vb]$, we get a decoration of the bonds of the vine $\Vb$ as explained in Definition \ref{decmole}. Note that there is a total of $n_1=2m+1$ branching nodes in $\Qc_0[\Vb] \setminus \{\uf_1\}$, and as a result a total of $2m+1$ leaf pairs (recall that $\uf_{11}, \uf_{21}, \uf_{22}\not\in\Qc_0[\Vb]$). We will denote the decoration of leaf pairs by $\{k_j: 1\leq j \leq 2m+1\}$, and that of the branching nodes $\uf_2, \ldots, \uf_{2m+2}$ by $\{\ell_j: 1\leq j \leq 2m+1\}$. Those two sets decorate the bonds in vine $\Vb$ as is shown in Figure \ref{fig:vineAnnotated} (A), where we denote by $a_j, b_j$ the decoration of double bonds between $v_{2j+1}$ and $v_{2j+2}$, by $c_j$ that between $v_{2j+1}$ and $v_{2j+3}$, and by $d_{j}$ that between $c_{2j+2}$ and $c_{2j+4}$ for $1\leq j \leq m-1$ (with $c_0$ and $d_0$ decorating the two bonds connecting $(v_2, v_3)$ and $(v_2, v_4)$ respectively, and similarly for $c_m$ and $d_m$ decorating the bonds connected to $v_1$ from $v_{2m+1}$ and $v_{2m+2}$). Finally, we also define $\zeta_j=+1$ if bond decorated by $c_j$ is outgoing from $v_{2j+1}$ for $1\leq j \leq m$ or from $v_2$ for $j=0$, and $\zeta_j=-1$ otherwise. Note that specifying $\zeta_j$ completely specifies the directions of all the bonds in $\Vb$ except for the horizontal double bonds in Figure \ref{fig:vineAnnotated} (A).

\begin{figure}[h!]
  \includegraphics[scale=.4]{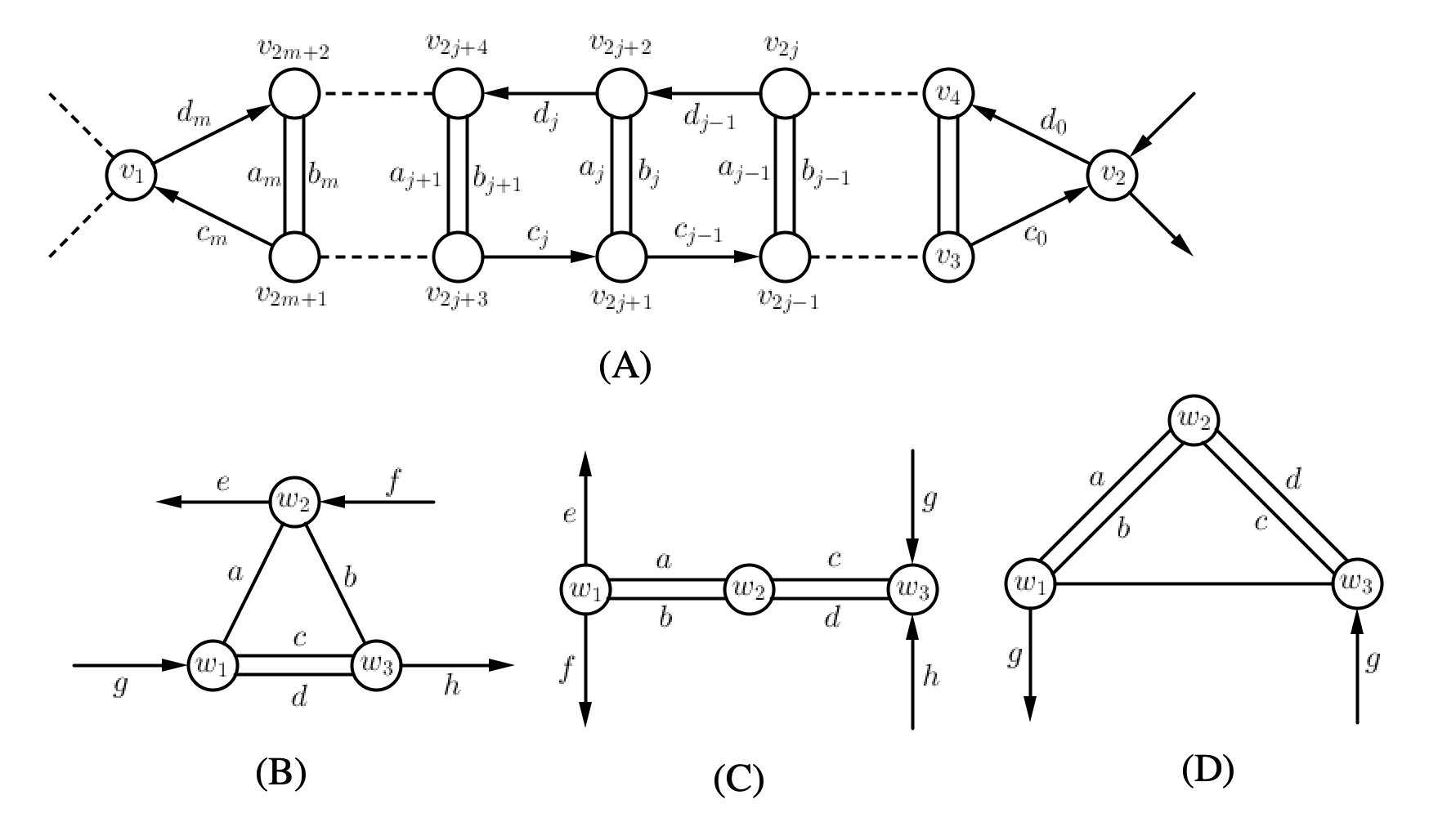}
  \caption{(A) is the annotation of the decoration of a Vine (II) as referred to in Step 3 of the proof of Proposition \ref{estbadvine}. (B)--(D) are the annotations of parts of Vine (III), (IV) and (V)--(VIII) respectively as referred to in Step 1 of the proof of Proposition \ref{estnormalvine}.}
  \label{fig:vineAnnotated}
\end{figure}

From this decoration of the bonds and using Definition \ref{decmole}, we have that 
$$
\zeta_j (c_j -d_j)=\zeta_{j-1}(c_{j-1}-d_{j-1}), \quad 1\leq j \leq m, \qquad 
$$
and we call this common value $r$. Note that $r=k_{\uf_{21}}-k_{\uf_{22}}$ is fixed and is nonzero due to Section \ref{vinesubset} (c). Moreover, we have from \eqref{molegammav} that 
\begin{equation}\label{zetaOmegasum}
\zeta_{\uf_{2j+1}}\Omega_{\uf_{2j+1}}+\zeta_{\uf_{2j+2}}\Omega_{\uf_{2j+2}}=-\Gamma_{v_{2j+1}}-\Gamma_{v_{2j+2}}=\pm 2r\cdot(c_{j-1}-c_j\mathrm{\ or\ }c_{j-1}-d_j).
\end{equation}
As a result of this, we can define new variables $x_0, x_j , y_j$ ($1\leq j \leq m$) such that (a) we have $x_0\in \{c_0, d_0\}$ which is specified as in part (f) of Section \ref{vinesubset}, (b) each of $(x_j, y_j)$ is the difference of two vectors among $(a_j, b_j, c_j, c_{j-1})$ for $1\leq j \leq m$, (c) for $1\leq j\leq m$ we have \[\zeta_{\uf_{2j+2}}\Omega_{\uf_{2j+2}}=-\Gamma_{v_{2j+2}}= 2x_j \cdot y_j,\quad\zeta_{\uf_{2j+1}}\Omega_{\uf_{2j+1}}=-2x_j\cdot y_j +2r\cdot \mu_j\] where $\mu_j=\alpha_j x_j+\beta_j y_j+\theta_j r$ for some $\alpha_j, \beta_j, \theta_j \in \{0, \pm 1\}$ with $\alpha_j^2+\beta_j^2\neq 0$, and (d) for $j=0$ we have that \[\zeta_{\uf_{2}}\Omega_{\uf_{2}}=r\cdot \mu_0\] where $\mu_0=\alpha_0 x_0+\beta_j y_0+\theta_0 r$ for some $\alpha_j, \beta_j, \theta_j \in \{0, \pm 1\}$ and with $y_0 \in \{k_{\uf_{21}}, k_{\uf_{22}}\}$.

 As a result of all this, we get a set of variables $(x_0, x_j, y_j)_{1\leq j \leq m}$ to replace the variables $k_{\uf_j}$ for $\uf_j\in \Qc_0[\Vb]$ (of which there are $2m+1$ corresponding to leaf pairs $(k_\lf)_{\lf\in \Lc[\Vb]}$, and another $2m+1$ corresponding to branching nodes $(k_\mf)_{\mf\in \Nc[\Vb]}$). If we write $\boldsymbol{x}=(x_0, x_j, y_j)_{1\leq j \leq m}$, ${\boldsymbol{k}}=(k_\lf)_{\lf \in  \Lc[\Vb]}$ and $\boldsymbol{\ell}=(k_{\nf})_{\nf \in \Nc[\Vb]}$, then we have that: (i) ${\boldsymbol{k}}=T_1{\boldsymbol{x}}+{\boldsymbol{h}}_1$ and $\boldsymbol{\ell}=T_2{\boldsymbol{x}}+{\boldsymbol{h}}_2$ for some matrices $T_j$ and some constant vectors ${\boldsymbol{h}}_j$ depending only on $(k_{\uf_1},k_{\uf_{11}}, k_{\uf_{21}}, k_{\uf_{22}})$, such that all coefficients of $T_1,\,T_1^{-1},\,T_2$ are integers $\lesssim 1$; (ii) for each component $k_\mf$ ($\mf\in \Nc[\Vb]$) of $\boldsymbol{\ell}$, there exists $k_{\mf'}$, where $\mf'$ is a descendent of $\mf$, such that $k_{\mf}\pm k_{\mf'}$ is an integer linear combination of $(k_\lf)$ and $(k_{\uf_1},k_{\uf_{11}}, k_{\uf_{21}}, k_{\uf_{22}})$ with absolute value sum $\lesssim \Lambda_p$, such that $\Lambda_1\cdots \Lambda_{2m+1}\lesssim C^m$; (iii) any component of ${\boldsymbol{x}}$ is the sum or difference of two variables, each of which is a component of either ${\boldsymbol{k}}$ or $\boldsymbol{\ell}$. Here the proof of (i) and (iii) and obvious, and (ii) follows from Lemma 6.6 of \cite{DH21}.

With the above reparametrization, \eqref{vineboundBfix2} can also be written as 
\begin{equation}\label{vineboundBfix3}
\begin{aligned}
\widetilde \Kc_{(\mathtt{ind},\Bs)}&= \bigg(\frac{\delta}{2L^{d-\gamma}}\bigg)^{2m+1}\zeta[\Vb] \sum_{(x_0,x_j,y_j):1\leq j\leq m}\epsilon_{\Es[\Vb]}\int_{\Ec[\Vb]} e^{-\pi i\cdot\delta L^{2\gamma} (t_{\uf_1}-t_{\uf_2})(r\cdot \mu_0)}\\&\times \prod_{j=1}^m e^{\pi i\cdot\delta L^{2\gamma} [(t_{\uf_{2j+2}}-t_{\uf_{2j+1}})x_j \cdot y_j-(t_{\uf_1}-t_{\uf_{2j+1}})(r\cdot \mu_j)]}\cdot \prod_{\lf\in
\Lc[\Vb]}^{(+)}\Kc_{\Qc^{(\lf,\lf')}}(t_{\lf^p},t_{(\lf')^p},k_\lf)\\
&\times Z(x_0,k_{\uf_1},k_{\uf_{11}},k_{\uf_{21}},k_{\uf_{22}},t_{\uf_1},t_{\uf_{21}},t_{\uf_{22}},t_{\uf_2})\prod_{\mf\in \Nc[\Vb]}\Kc_{\Tc^{(\mf)}}^*(t_{\mf^p},t_{\mf},k_\mf)\,\mathrm{d}t_{\mf}.
\end{aligned}
\end{equation} This expression will not be used in Step 3, but will be used in Step 4 below. Moreover, once we confirm $r\neq 0$, we will replace $\epsilon_{\Es[\Vb]}$ factor by $1$, since $\epsilon_{\Es[\Vb]}\neq 1$ corresponds to the case where some $x_j$ or $y_j$ equals $0$ or some other fixed constant, which means we can sum over them trivially, gain a power using $L^d\ll L^{2(d-\gamma)}$, and treat the rest of the variables using the same arguments below.

\medskip
\underline{\it Step 3: A counting argument.} We now treat the case where $\Vb$ is vine (II), and $\gamma>\frac45-\eta$ or $|r|>L^{-\gamma-\eta}$. In this case we do not need the cancellation structure, and as such we shall estimate each $\widetilde \Kc_{(\mathtt{ind},\Bs)}$ by itself, and use that there or $O(C^{n_2})$ elements in the sum over $(\Bs, \mathtt{ind})$.

Start with the expression of $\widetilde \Kc_{(\mathtt{ind},\Bs)}$ in \eqref{vineboundBfix2}. At this point, suppose we restrict each of $(k_{\uf_1},k_{\uf_{11}},k_{\uf_{21}},k_{\uf_{22}})$ to a unit ball as in the $Y_{\mathrm{loc}}^\theta$ in (\ref{vinebound1}). We may repeat the arguments in the proof of Lemma \ref{sumintest1}, using (i)--(iii) above, to also restrict $x_0$ and each $(x_j,y_j)$, as well as $k_\nf$ for all nodes $\nf\in\Qc_0[\Vb]$, to a unit ball; all these unit balls will be fixed throughout the rest of the proof. As such we may assume $Z\in X^{\theta, 0}$, and henceforth expand it as a Fourier integral, so that the $Z$ function in \eqref{vinebound1} is replaced by the product \begin{equation}\label{defztilde}e^{\pi i \lambda_{\uf_1}t_{\uf_1}}e^{\pi i \lambda_{\uf_{21}}t_{\uf_{21}}}e^{\pi i \lambda_{\uf_{22}} t_{\uf_{22}}}e^{\pi i \lambda_{\uf_2} t_{\uf_2}}\cdot \widetilde{Z}(x_0,k_{\uf_1},k_{\uf_{11}},k_{\uf_{21}},k_{\uf_{22}})\end{equation} where $\widetilde{Z}$ is a bounded function. Note that by doing this we can also exploit the weights of form $\max(\langle \lambda_{\uf_1}\rangle,\langle\lambda_{\uf_{21}}\rangle,\langle\lambda_{\uf_{22}}\rangle,\langle\lambda_{\uf_2}\rangle)^{-\theta}$ whenever needed. In the same way, we can expand all the $\Kc_\Qc$ and $\Kc_{\Tc}^*$ functions as time Fourier integrals to obtain a linear combination of functions\
\[(C^+\delta)^{n(\Bs)/2}\cdot \prod_{\nf\in \Nc[\Vb]}e^{\pi i\vartheta_\nf t_\nf}\cdot\prod_{\lf\in \Lc[\Vb]}^{(+)}\langle k_\lf\rangle^{-40d}\cdot\Xc(k[\Qc_{0}[\Vb]])\] 
for different choices of $\vartheta[\Nc[\Vb]]$, with the coefficient being a weighted $L^1$ integrable function with a suitable weight, and where $\Xc$ is a bounded function of all the $k_\mf$ variables. Below we may fix one choice of $(\vartheta[\Nc[\Vb]])$; by doing so we may also exploit the weight $(\max_{\nf\in \Nc[\Vb]}\langle \vartheta_\nf \rangle)^{-\eta}$ whenever needed. As a result of, we can write $\widetilde \Kc_{(\mathtt{ind},\Bs)}$ in \eqref{vineboundBfix2} as a linear combination of  
\begin{equation}\label{newbexpr}
(C^+\delta)^{n(\Bs)/2} \bigg(\frac{\delta}{2L^{d-\gamma}}\bigg)^{n_1}\sum_{\Es} e^{\pi i (\lambda_{\uf_{21}}t_{\uf_{21}}+\lambda_{\uf_{22}}t_{\uf_{22}}+\widetilde \lambda_{\uf_{1}} t_{\uf_{1}})}\Bc\big(t_{\uf_1},t_{\uf_{21}}, t_{\uf_{22}},\alpha[\Nc[\Vb]]\big)\cdot\widetilde \Xc(k[\Qc_{0}[\Vb]])
\end{equation} where $\alpha_{\mf}=\delta L^{2\gamma}\zeta_m\Omega_{\mf}+\vartheta_\mf$ for $\mf\in  \Nc[\Vb]$, and 
\[
\Bc\big(t_{\uf_1},t_{\uf_{21}}, t_{\uf_{22}},\alpha[\Nc[\Vb]]\big)=\int_{\Ec[\Vb]} \prod_{\nf\in \Nc[\Vb]} e^{\pi i \alpha_\nf (t_\nf-t_{\uf_1})} \, \mathrm{d}t_{\nf}, 
\]
where $\widetilde \lambda_{\uf_1}$ is a shift of $\lambda_{\uf_1}$ by elements of $\vartheta[ \Nc[\Vb]]$, and we also allow a shift of $\alpha_{\uf_2}$ by $\lambda_{\uf_{2}}$. Also, $\widetilde \Xc(k[\Qc_{0}[\Vb]])$ is a bounded function that is localized in a fixed unit size box for each of its variables. Fixing the integer parts of $\alpha_\mf$ as $\sigma_{\mf}$ (each of which having at most $L^{10d}$ possibilities), by summing over all these integer parts, we can bound the $(t_{\uf_1}, t_{\uf_{21}}, t_{\uf_{22}})$ Fourier transform of \eqref{newbexpr} by
\begin{multline}\label{afterloc7.3}
(C^+\delta)^{n(\Bs)/2} \bigg(\frac{\delta}{2L^{d-\gamma}}\bigg)^{n_1} \sum_{\sigma[ \Nc[\Vb]]}\sup_{|\alpha_\mf-\sigma_\mf|\leq 1}|\widehat \Bc(\tau_{\uf_1}+\widetilde{\lambda}_{\uf_1},\tau_{\uf_{21}}+\lambda_{\uf_{21}}, \tau_{\uf_{22}}+\lambda_{\uf_{22}},\alpha[\Nc[\Vb]])| \\\times\sup_{\sigma[\Nc[\Vb]]}\sum_{\Es}^*\widetilde \Xc(k[\Qc_{0}[\Vb]])
\end{multline}
where the sum $\sum^*$ satisfies the additional localization that $|\alpha_{\mf}-\sigma_{\mf}|\leq 1$ for each $\mf\in  \Nc[\Vb]$. The result \eqref{vinebound1} will follow from the following two estimates:
\begin{align}
&\bigg\|\max(\langle\tau_{\uf_1}\rangle,\langle\tau_{\uf_{21}}\rangle,\langle\tau_{\uf_{22}}\rangle)^{\eta^5}\sum_{\sigma[\Nc[\Vb]]}\sup_{|\alpha_\mf-\sigma_\mf|\leq 1}|\widehat \Bc(\tau_{\uf_1},\tau_{\uf_{21}}, \tau_{\uf_{22}},\alpha[\Nc[\Vb]])|\bigg\|_{L_{\tau,\alpha}^1}\lesssim C^m (\log L)^{2m+1} L^{4\eta^5} ,\label{twoestbadv2}\\
&\sup_{(k_{\uf_1},k_{\uf_{11}},k_{\uf_{21}},k_{\uf_{22}})}\sup_{\sigma[ \Nc[\Vb]]}\sum_{\Es}^*\widetilde \Xc(k[\Qc_{0}[\Vb]])\leq(C^+\delta^{-1})^{m+1}L^{(2m+1)(d-\gamma)-5\eta}.\label{twoestbadv1}
\end{align}
Those two estimates are enough to give \eqref{vinebound1}, since shifts of $\widetilde{\lambda}_{\uf_1}$, $\lambda_{\uf_{21}}$ and $\lambda_{\uf_{22}}$ can be absorbed by the weights $\max(\langle \lambda_{\uf_1}\rangle,\langle\lambda_{\uf_{21}}\rangle,\langle\lambda_{\uf_{22}}\rangle,\langle\lambda_{\uf_2}\rangle)^{-\theta}$ and $(\max_{\nf\in \Nc[\Vb]}\langle \vartheta_\nf \rangle)^{-\eta}$.

To prove \eqref{twoestbadv2}, we note that $\Bc(t_{\uf_1},t_{\uf_{21}},t_{\uf_{22}})=\Cc(t_{\uf_1}-\max(\tau_{\uf_{21}},\tau_{\uf_{22}}))$ for some function $\Cc$ by changing time variables, and by Lemma \ref{timeFourierlem} (1) we can bound $\widehat{\Cc}(\tau,\alpha[\Nc[\Vb]])$ by 
\[|\widehat{\Cc}(\tau,\alpha[ \Nc[\Vb]])|\lesssim\langle \tau-\gamma\rangle^{-10}
 \prod_{\mf\in  \Nc[\Vb]} \langle q_\mf\rangle^{-1},
\]
where $(q_\mf)_{\mf \in  \Nc[\Vb]}$ is obtained from $\alpha[\Nc[\Vb]]$ by an invertible (actually lower triangular) linear transformation with integer coefficients, and $\gamma$ is the sum of at most two $q_\mf$ variables. Since each $\sigma_\mf$ has at most $L^{10d}$ choices, we can easily sum over all $(\sigma_\mf)$, equivalently all $(q_\mf)$, and apply Lemma \ref{timeFourierlem} (2) to transform from $\Cc$ to $\Bc$, and get \eqref{twoestbadv2}.

To prove \eqref{twoestbadv1}, we upper bound $\widetilde \Xc$ by 1 which reduces the sum into a counting problem for $k[\Qc_0[\Vb]]$ satisfying all the decoration and localization assumptions as before, and in addition the condition that $\delta L^{2\gamma}\Omega_\nf$ for each $\nf\in\Nc[\Vb]$ belongs to an interval of length $1$. By the reparametrization in Step 2 above, this is a equivalent to counting the number of choices for the variables $(x_0, x_j, y_j)_{1\leq j \leq m}$, each restricted to a fixed unit ball, such that each of $(r\cdot x_0, x_j\cdot y_j, r\cdot \mu_j)$ $(1\leq j \leq m$) is restricted to an interval of length $O(\delta^{-1}L^{-2\gamma})$.

Now if $\gamma\geq\frac45-\eta$, we can bound the number of choices of $x_0$ by $O(\delta^{-1}L^{d-\gamma+(1-\gamma)})$ using Lemma \ref{basiccount} (1), and the number of choices for each $(x_j,y_j)$ by $\delta^{-2} L^{2(d-\gamma)-(1-\gamma)-20\eta}$ using Lemma \ref{basiccount0} (2). This implies \eqref{twoestbadv1}, noticing that $m\geq 1$ for vines (II). If $\gamma<\frac45-\eta$ and $|r|\gtrsim L^{-\gamma-\eta}$, then we can bound the number of choices of $x_0$ by $O(\delta^{-1}L^{d-\gamma+2\eta})$ using Lemma \ref{basiccount} (1), and the number of choices for each $(x_j,y_j)$ by $\delta^{-2} L^{2(d-\gamma)-20\eta}$ using Lemma \ref{basiccount0} (2), which again implies \eqref{twoestbadv1}.

\medskip
\underline{\it Step 4: The cancellation argument.} We now treat the case where $\Vb$ is vine (II), $\gamma\leq \frac45-\eta$ and $|r|\leq L^{-\gamma-\eta}$. In this case we need to rely on the cancellation happening in the sum over $(\Bs, \mathtt{ind})$ of the terms $\widetilde \Kc_{(\mathtt{ind},\Bs)}$ defined in \eqref{vineboundBfix3}. For this cancellation to manifest itself, we need to utilize the couple structure near $\uf_2$, which is depicted in Figure \ref{fig:block_mole}.

We start with the easiest case, which is that of vine (II-e). In this case, we will estimate each element of the sum over $(\Bs, \mathtt{ind})$ separately and use that there are $O(C^{n_2})$ elements in this sum. At this point we can repeat the arguments in Step 3 to localize each of the $(x_0,x_j,y_j)$ and $k_\nf$ variables to a fixed unit ball, and consequently expand $Z$ into (\ref{defztilde}) and each $\Kc_{\Qc^{(\lf,\lf')}}$ and $\Kc_{\Tc^{(\mf)}}^*$ as time Fourier integrals, to reduce (\ref{vineboundBfix3}) to at most $O(C^m)$ expressions of form
\begin{multline*}
\bigg(\frac{\delta}{2L^{d-\gamma}}\bigg)^{2m+1}(C^+\delta)^{n(\Bs)/2}\sum_{|x_0-a_0|\leq 1} e^{\pi i (\lambda_{\uf_{21}}t_{\uf_{21}}+ \lambda_{\uf_{22}} t_{\uf_{22}}+(\lambda_{\uf_1}+\gamma_1)t_{\uf_1})}\\\times \widetilde{Z}(x_0,k_{\uf_1},k_{\uf_{11}},k_{\uf_{21}},k_{\uf_{22}})\cdot\Ic\left(x_0, k_{\uf_1}, k_{\uf_{11}}, k_{\uf_{21}}, k_{\uf_{22}}, t_*\right)
\end{multline*}
Here $a_0$ is a fixed vector, $\gamma_1$ is a linear combination of $\lambda_{\uf_2}$ and the Fourier variables occurring in the expansions of $\Kc_\Qc^{(\lf, \lf')}$ and $\Kc^*_{\Tc^{(\mf)}}$ as above, and $\Ic$ is an expression in the form \eqref{sumintI1}, which can be obtained after defining the new time variables 
$t_j=t_{\uf_1}-t_{\uf_{2j+1}}$, $s_j=t_{\uf_1}-t_{\uf_{2j+2}}$ for $1\leq j \leq m$, and $t_0=t_{\uf_1}-t_{\uf_2}$. The domain of integration after the change of variables is the same as described in Lemma \ref{sumintest1}, with $t$ replaced by $t_*:=t_{\uf_1}-\max(t_{\uf_{21}},t_{\uf_{22}})\}$. Moreover, due to the structure of the couple in vine (II-e) (see Figure \ref{fig:block_mole}), namely the fact that $\uf_2$ is a parent of $\uf_4$ and child of $\uf_3$, we have the additional condition (a) in Lemma \ref{sumintest1}. As a result, by \eqref{FourierboundI}, we have that
$$
\|\Ic\|_{X_{\mathrm{loc}}^{2\eta^5, 0}(t_*)}\lesssim  (C^+\delta^{-1})^{m}L^{2m(d-\gamma)-\gamma-\eta^2}.
$$
After summing over $x_0$, applying Lemma \ref{timeFourierlem} (2), and including the factors $(\delta/(2L^{d-\gamma}))^{2m+1}$ and $(C^+\delta)^{n(\Bs)/2}$ etc., this implies (\ref{vinebound1}). Note that the shifts $(\lambda_{\uf_{21}},\lambda_{\uf_{22}},\lambda_{\uf_1}+\gamma_1)$ is again absorbed by the weights.

Now we turn to vines (II-a)--(II-d). We will consider the case of a pair of couples with vines (II-a) and (II-b) that are twists of each other; the case of vines (II-c) and (II-d) are similar. If $\Bs$ is fixed and we sum over $\mathtt{ind}\in\{0,1\}$, then starting from (\ref{vineboundBfix3}), and recalling Figure \ref{fig:block_mole} (assuming without loss of generality that $\uf_4$ has positive sign), we get the expression
\begin{equation}\label{step4sumind}
\begin{aligned}
\sum_{\mathtt{ind}}\widetilde \Kc_{(\mathtt{ind},\Bs)}&=  \bigg(\frac{\delta}{2L^{d-\gamma}}\bigg)^{2m+1}\zeta[\Vb^{(a)}]\sum_{(x_0,x_j,y_j):1\leq j\leq m} \int_{ \Ec^*[\Vb]}\prod_{j=2}^{2m+2} \mathrm{d}t_{\uf_j}\\
&\times  \prod_{j=1}^m e^{\pi i\cdot\delta L^{2\gamma} (t_{\uf_{2j+2}} -t_{\uf_{2j+1}})x_j \cdot y_j}e^{\pi i\cdot\delta L^{2\gamma} (t_{\uf_{2j+1}}-t_{\uf_1})(r\cdot \mu_j)} \cdot e^{\pi i\cdot\delta L^{2\gamma} (t_{\uf_2}-t_{\uf_1})(r\cdot (k_{\uf_{22}}-x_0))}\\&\times Z(x_0,k_{\uf_1},k_{\uf_{11}},k_{\uf_{21}},k_{\uf_{22}},t_{\uf_1},t_{\uf_{21}},t_{\uf_{22}},t_{\uf_2})\cdot \prod_{\lf\in\Lc[\Vb]\setminus \{\lf^*\}}^{(+)}\Kc_{\Qc^{(\lf,\lf')}}(t_{\lf^p},t_{(\lf')^p},k_\lf)\\
&\times \prod_{\mf\in \Nc[\Vb]\setminus \{\uf_2\}}\Kc_{\Tc^{(\mf)}}^*(t_{\mf^p},t_{\mf},k_\mf)\left[\Mc^{(a)}(t_{\uf_2}, t_{\uf_3}, t_{\uf_4},x_0, y_0)-\Mc^{(b)}(t_{\uf_2},t_{\uf_3}, t_{\uf_4}, x_0,y_0)\right].
\end{aligned}
\end{equation}
Here, we denote by $\zeta[\Vb^{(a)}]$ the $\zeta[\Vb]$ for the couple with vine (II-a), which is the negative of the (II-b) couple, and by $\Ec^*[\Vb]$ the domain of integration derived from $\Ec[\Vb]$ by removing the condition $t_{\uf_4}>t_{\uf_2}$ in case (II-a) and $t_{\uf_3}>t_{\uf_2}$ in case (II-b). Moreover $x_0$ equals $k_{\uf_{23}}$ in case (II-a) and equals $k_{\uf_2}$ in case (II-b), and $y_0$ is defined to be the other element. We also denoted by $\lf^*$ the positive leaf in the leaf pair $(\uf_{23}, \uf_0)$ (using the notation in Figure \ref{fig:block_mole}), and introduced 
\begin{align*}
&\Mc^{(a)}(t_{\uf_2}, t_{\uf_3},t_{\uf_4}, x_0, y_0)=\mathbf{1}_{t_{\uf_2}<t_{\uf_4}}(t_{\uf_4})\overline{\Kc_{\overline{\Qc_a}}}(t_{\uf_3},t_{\uf_2},x_{0})\cdot\Kc_{\Tc_a}^*(t_{\uf_4},t_{\uf_2},y_0)\\
&\Mc^{(b)}(t_{\uf_2}, t_{\uf_3},t_{\uf_4}, x_0, y_0)=\mathbf{1}_{t_{\uf_2}<t_{\uf_3}}(t_{\uf_3})\cdot\Kc_{\Qc_b}(t_{\uf_4},t_{\uf_2},y_0)\cdot\overline{\Kc_{\overline{\Tc_b}}^*}(t_{\uf_3},t_{\uf_2},x_{0}),
\end{align*}
where we denoted by $\Qc^a$ and $\Tc^a$ (resp. $\Qc^b, \Tc^b$) the couples $\Qc^{(\uf_{23}, \uf_0)}$ and $\Tc^{(\uf_2)}$ (resp. $\Qc^{(\uf_0, \uf_{23})}$ and ${\Tc^{(\uf_2)}}$). Recall also that $y_0=x_0\pm r$ which holds for both cases (II-a) and (II-b). 
This puts us in the position to apply Lemma \ref{sumintest1} to conclude. In fact, the difference $\Mc^{(a)}-\Mc^{(b)}$ leads to a sum involving one of the following assumptions or terms:
\begin{enumerate}
\item The assumption $t_{\uf_4}< t_{\uf_2}<t_{\uf_3}$ or $t_{\uf_3}< t_{\uf_2}<t_{\uf_4}$.
\item Factors $\Rs_{\Qc}=\Kc_{\Qc}-(\Kc_{\Qc})_{\textrm{app}}$ or $\Rs^*_{\Tc}=\Kc^*_{\Tc}-(\Kc^*_{\Tc})_{\textrm{app}}$ replacing at least one of the $\Kc_{\Qc}$ or the $\Kc^*_{\Tc}$ in \eqref{step4sumind}, for $\Qc\in\{\overline{\Qc_a},\Qc_b\}$ and $\Tc\in\{\Tc_a,\overline{\Tc_b}\}$; here we use Proposition \ref{regcpltreeasymp}.
\item Factors \[\overline{\Jc(t_{\uf_3}, t_{\uf_2})} \big(\Jc^*(t_{\uf_4}, t_{\uf_2})- \Jc^*(t_{\uf_3}, t_{\uf_2})\big)\quad \textrm{or}\quad\big(\Jc(t_{\uf_4}, t_{\uf_2})-\Jc(t_{\uf_3}, t_{\uf_2})\big)\overline{\Jc^*(t_{\uf_3}, t_{\uf_2})},\] which equals $|t_{\uf_3}-t_{\uf_4}|^{1-\eta}$ multiplies a weighted Fourier $L^1$ function; this comes from expanding $(\Kc_\Qc)_{\mathrm{app}}$ and $(\Kc_\Tc^*)_{\mathrm{app}}$ as in \eqref{kqterms}.
\item Factors $(\Kc_{\Qc})_{\mathrm{app}}(\cdot,\cdot,y_0)-(\Kc_{\Qc})_{\mathrm{app}}(\cdot,\cdot,x_0)$ and $(\Kc_{\Tc}^*)_{\mathrm{app}}(\cdot,\cdot,x_{0})-(\Kc_{\Tc}^*)_{\mathrm{app}}(\cdot,\cdot,y_{0})$, for some $\Qc$ and $\Tc$, with the time variables being the same in both functions, which is bounded by $|r|$ using that $|x_0-y_0|=r$.
\item The leading factor of the form
$$
\left[ \overline{(\Kc_{\Qc_0}})_{\mathrm{app}}(t_{\uf_3},t_{\uf_2},x_0)(\Kc_{\Tc_0}^*)_{\mathrm{app}}(t_{\uf_3},t_{\uf_2},x_{0})-(\Kc_{\Qc_0'})_{\mathrm{app}}(t_{\uf_3},t_{\uf_2},x_0)\overline{(\Kc_{\Tc_0'}^*)_{\mathrm{app}}}(t_{\uf_3},t_{\uf_2},x_{0})\right].
$$
\end{enumerate}

Note that the contribution of the last term (5) vanishes after summing over $\Bs$, due to Lemma \ref{regcpltreecancel}; therefore we are left with terms (1)--(4). For each term, we may argue as in the case of vine (II-e) above, where we insert them back into \eqref{step4sumind}, then localize each vector to a fixed unit ball, expand $Z$ into (\ref{defztilde}), and expand all the $\Kc_{\Qc^{(\lf, \lf')}}$ and $\Kc_{\Tc^{(\mf)}}$ (and $\Rs_{\Qc}$ and $\Rs_{\Tc}^*$ if present) as time Fourier integrals using Proposition \ref{regcpltreeasymp}, to reduce to at most $O(C^m)$ expressions of form
\begin{multline}
\bigg(\frac{\delta}{2L^{d-\gamma}}\bigg)^{2m+1}(C\delta)^{n(\Bs)/2}\sum_{|k_0-a_0|\leq 2} e^{\pi i( \lambda_{\uf_{21}}t_{\uf_{21}}+\pi i \lambda_{\uf_{22}} t_{\uf_{22}}+(\lambda_{\uf_1}+\gamma_1)t_{\uf_1})} \\\times \widetilde{Z}(x_0,k_{\uf_1},k_{\uf_{11}},k_{\uf_{21}},k_{\uf_{22}})\cdot\Ic\left(k_0, k_{\uf_1}, k_{\uf_{11}}, k_{\uf_{21}}, k_{\uf_{22}}, t_*\right).
\end{multline}
Here $a_0$ is a fixed vector, $\gamma_1$ is a linear combination of $\lambda_{\uf_2}$ and the Fourier variables occurring in the expansions of $\Kc_\Qc^{(\lf, \lf')}$ and $\Kc^*_{\Tc^{(\mf)}}$ as above, and $\Ic$ is an expression in the form \eqref{sumintI1}, which can be obtained after defining the new time variables 
$t_j=t_{\uf_1}-t_{\uf_{2j+1}}$, $s_j=t_{\uf_1}-t_{\uf_{2j+2}}$ for $1\leq j \leq m$, and $t_0=t_{\uf_1}-t_{\uf_2}$. The domain of integration after the change of variables is the same as described in Lemma \ref{sumintest1}, with $t$ replaced by $t_*:=t_{\uf_1}-\max(t_{\uf_{21}},t_{\uf_{22}})\}$. Clearly, for each of the terms (1)--(4), one of the conditions (a)--(d) in Lemma \ref{sumintest1} will be satisfied, so we can apply \eqref{FourierboundI} to get that

$$
\|\Ic\|_{X_{\mathrm{loc}}^{2\eta^5, 0}(t_*)}\lesssim  (C^+\delta^{-1})^{m}L^{2m(d-\gamma)-\gamma-\eta^2}.
$$
After summing over $x_0$, applying Lemma \ref{timeFourierlem} (2), and including the factors $(\delta/(2L^{d-\gamma}))^{2m+1}$ and $(C^+\delta)^{n(\Bs)/2}$ etc., this implies (\ref{vinebound1}); note that the shifts $(\lambda_{\uf_{21}},\lambda_{\uf_{22}},\lambda_{\uf_1}+\gamma_1)$ is again absorbed by the weights.

\medskip
\underline{\it Step 5: The vine (I) case.} We now treat the case where $\Vb$ is a (CL) vine (I), i.e. one double bond. Here we will apply the cancellation structure in Step 4, combined with the counting arguments in Step 3, but both will be in this extremely simple setting of one double bond. In fact, in this case we have $m=0$, so there is only one variable $x_0$ which we sum in. Moreover, an easy examination of Figure \ref{fig:block_mole} similar to Step 4 implies that, instead of the $\Mc^{(a)}-\Mc^{(b)}$ factor in (\ref{step4sumind}), we have the factor
\[\overline{\Kc_{\overline{\Qc_a}}}(t_{\uf_1},t_{\uf_2},x_0)\Kc_{\Tc_a}^*(t_{\uf_1},t_{\uf_2},y_0)-\Kc_{\Qc_b}(t_{\uf_1},t_{\uf_2},y_0)\overline{\Kc_{\overline{\Tc_b}}}(t_{\uf_1},t_{\uf_2},x_0),\]assuming $\uf_1$ has sign $+$ in Figure \ref{fig:block_mole}. Upon summing in $\Bs$ and applying Lemma \ref{regcpltreecancel}, this factor is bounded in $X_{\mathrm{loc}}^{\eta,0}$ by $|r|$ since $|x_0-y_0|=r$.

Then, we reduce to the counting problem as in Step 3. Note that the result corresponding to (\ref{twoestbadv2}) is proved in the same way; as for the one corresponding to (\ref{twoestbadv1}), we are reduced to a counting problem with only one variable $x_0$ involved, which is restricted to a unit ball and satisfies that $r\cdot x_0$ belongs to a fixed interval of length $\delta^{-1}L^{-2\gamma}$. By Lemma \ref{basiccount} (1) and using the extra factor $|r|$ exhibited above, we get the total contribution factor
\[(L^{d-1}+\delta^{-1}|r|^{-1}L^{d-2\gamma})\cdot |r|\lesssim L^{d-\gamma-\gamma_0+O(\eta)},\] which suffices for the proof of (\ref{vinebound1}) just as in Step 3.

\medskip
The proof of Proposition \ref{estbadvine} is now finished.
\end{proof}

\medskip
\begin{prop}\label{estnormalvine} Suppose $\Vb$ is a \emph{normal} (CL) vine. Let $(\Qc^{\mathrm{sp}}, \mathtt{cod},\mathfrak{n},\As^{\mathrm{sp}})$ and $(Z,W)$ be fixed as in Section \ref{vinesubset} and Proposition \ref{estbadvine}, and let $(\mathtt{ind},\Bs)$ also be fixed. Then, for any choice of $\theta\in\{\eta^5,0\}$ and in the same notations as in Proposition \ref{estbadvine}, we have
\begin{equation}\label{vinebound2}
\big\|e^{-\pi i\cdot\delta L^{2\gamma}t_{\uf_1}\Gamma}\cdot\Kc_{(\mathtt{sgn},\mathtt{ind},\Bs)}^{(\Vb,Z,W)}(x_0',k_{\uf_{21}},k_{\uf_{22}},t_{\uf_1},t_{\uf_{21}},t_{\uf_{22}})\big\|_{Y_{\mathrm{loc}}^{\theta}}\lesssim(C^+\delta^{1/4})^{n_2}L^{\eta^4}\|Z\|_{Y_{\mathrm{loc}}^{\theta}}.
\end{equation}
Moreover, if $\Vb$ is replaced by a normal (CL) vine chain $\Vb\Cb$, define $\Qc_0[\Vb\Cb]$ as in Proposition \ref{block_clcn} for the block $\Vb\Cb$ and let $\Qc^{\mathrm{sp}}$ be the couple obtained by splicing the whole vine chain $\Vb\Cb$. Fix $(\Qc^{\mathrm{sp}}, \mathfrak{n},\As^{\mathrm{sp}})$ and $(\mathtt{cod},\mathtt{ind}, \Bs)$ for each ingredient vine in the chain, and define the expression $\Kc_{(\mathtt{ind},\Bs)}^{(\Vb\Cb, Z, W)}$ in the same way as in Section \ref{vinesubset}. Then, for $\theta\in\{\eta^5,0\}$ we have 
\begin{equation}\label{vinebound3}
\big\|e^{-\pi i\cdot\delta L^{2\gamma}t_{\uf_1}\Gamma}\cdot\Kc_{(\mathtt{sgn},\mathtt{ind},\Bs)}^{(\Vb\Cb,Z,W)}(x_0',k_{\uf_{21}},k_{\uf_{22}},t_{\uf_1},t_{\uf_{21}},t_{\uf_{22}})\big\|_{Y_{\mathrm{loc}}^{\theta}}\lesssim(C^+\delta^{1/4})^{n_2}L^{\eta^4}\|Z\|_{Y_{\mathrm{loc}}^{\theta}},
\end{equation}
where $n_2$ is the total number of branching nodes in $\Qc_0[\Vb\Cb]\setminus \{\uf_1\}$ plus the sum of all the $n(\Bs)$, as defined in Section \ref{vinesubset}.
\end{prop}
\begin{proof} The proof of this proposition goes along the same lines as that of Proposition \ref{estbadvine}, so we will just sketch the similar arguments and only elaborate on the differences, which are mainly in the reparametrization in Step 1 below. We start with the case of a single normal (CL) vine. We replace $Y_{\mathrm{loc}}^\theta$ by $Y^\theta$ and assume throughout $\|Z\|_{Y^\theta}=1$.

\medskip
\underline{\it Step 1: Reparametrization.} Define the notations $\Es[\Vb]$, $\Nc[\Vb]$ and $\Lc[\Vb]$ etc. as in the proof of Proposition \ref{estbadvine}. We first perform the reparametrization. By examining the form of the normal vine $\Vb$ in Figure \ref{fig:vines}, we see that it contains $n_1=2m+4$ atoms (excluding the joint atom $v_1$) for some $m$, which is equal to the cardinality of $\Qc_0[\Vb]\setminus \{\uf_0\}$. These atoms are split into three groups: (a) the joint $v_2$, (b) the $m$ pairs of atoms connected by double bounds that belong to the ladders represented colored dashed lines in Figure \ref{fig:vines} (for Vine (VII) we also include here the pair of atoms where all three ladders intersect), and (c) the three remaining atoms $v_1', v_2', v_3'$. At the level of the couple, $\Qc_0[\Vb]\backslash\{\uf_1\}$ has $2m+4$ leaf pairs $(\lf, \lf')$, as well as $2m+4$ branching nodes $\mf$. We will replace the variables $k_\lf$ and $k_\mf$ (where $\lf\in\Lc[\Vb]$ and $\mf\in\Nc[\Vb]$) occurring in the decoration $\Es$ by a new set of variables $(x_0, x_j, y_j, u_1, u_2, u_3)_{1\leq j \leq m}$ as follows.

By the same argument and notation as in Step 2 of the proof of Proposition \ref{estbadvine}, we have that $\Omega_{\uf_2}=2r\cdot\mu_0$ where $r=k_{\uf_{21}}-k_{\uf_{22}}$ is fixed and nonzero, and $\mu_0=\alpha_0x_0+\beta_0 y_0+\theta_0 r$ for some $\alpha_0, \beta_0, \theta_0 \in \{0,\pm 1\}$ with $y_0 \in \{k_{\uf_{21}}, k_{\uf_{22}}\}$, $x_0$ equals $k_{\uf_{23}}$ if $\uf_2$ has sign $+$ and $k_{\uf_2}$ otherwise, and $\alpha_0^2+\beta_0^2\neq 0$. Next, for each pair of atoms connected by a double bond in a ladder, the same argument as in Step 3 of the proof of Proposition \ref{estbadvine} shows that if $\uf$ and $\widetilde \uf$ are the branching nodes of $\Qc_0[\Vb]$ corresponding to those two atoms, then 
$$
\zeta_{\uf} \Omega_{\uf}+\zeta_{\widetilde \uf}\Omega_{\widetilde \uf}=2\widetilde r \cdot \mu.
$$
Here $\mu$ is the difference of two $k_\nf$ vectors corresponding to two of the four single bonds at these pair of atoms, and $\widetilde r\in \{0, r\}$ is the the same for all pairs of atoms in the same ladder and is equal to $r$ for the ladders attached to the joints and zero otherwise. In Figure \ref{fig:vines} the ladders whose $\widetilde r$ value is 0 are colored in light blue, and we call such ladders \emph{zero-gap ladders} (note that for the pair of atoms where all three ladders intersect in vine (VII), the argument needs to be slightly adjusted but the result remains the same, with $\widetilde{r}=r$ in this case). As such, we can define $(x_j, y_j)\in \Zb^{2d}$ such that $\zeta_{\uf} \Omega_{\widetilde \uf}=2x_j \cdot y_j$, so that $\zeta_{\widetilde \uf}\Omega_{\widetilde \uf}=-2x_j\cdot y_j+2\widetilde r\cdot \mu$, and $\mu$ can be written as $a_jx_j+b_jy_j+c_j\widetilde r$ for $a_j,b_j,c_j\in\{0,\pm1\}$ and $a_j^2+b_j^2\neq 0$. It now remains to define the variables $(u_1, u_2, u_3)$. Here we only discuss vine (III) in detail below, as arguments in other cases are similar.

(1) For vine (III), using the notation in Figure \ref{fig:vineAnnotated} (B), we may denote $\omega_j=\zeta_{\wf_j}\Omega_{\wf_j}$ where $\wf_j$ is the node in $\Qc_0[\Vb]$ corresponding to the atom $w_j$. In Figure \ref{fig:vineAnnotated} (B) and Figure \ref{fig:vines}, note that $e-g=f-h=\pm r$, and $(e,g)$ is determined by $x_0$ and some of the $(x_j,y_j)$ variables. Now, if the bonds decorated by $c$ and $d$ have opposite directions (say $c$ goes from $w_1$ to $w_3$, and $d$ goes from $w_3$ to $w_1$), we may define $(u_1,u_2,u_3)=(c-d,d-a,d-b)$. If the bonds decorated by $c$ and $d$ have the same direction (which has to go from $w_1$ to $w_3$), then we may define $(u_1,u_2,u_3)=(c-a,a-d,c-b)$. Then we have $\omega_1=2u_1\cdot u_2$, and $\omega_2$ equals $2u_1\cdot u_3$ or $2u_3\cdot (u_1+u_2-u_3)$ in the first and second case. Moreover we have $\omega_1+\omega_2+\omega_3=\pm(|e|^2-|f|^2-|g|^2+|h|^2)=2r\cdot\xi$ where $\xi=au_1+bu_2+cu_3+dr$ with $a,b,c,d\in\{0,\pm1\}$. In any case, the variables $(u_1,u_2,u_3)$ determines $(a,b,c,d,f,h)$ and allows one to proceed with parametrizing the next ladder starting from $(f,h)$ by the rest of $(x_j,y_j)$ variables. The argument for Vine (IV) is similar, see Figure \ref{fig:vineAnnotated} (C).

(2) For Vines (V)--(VIII), the argument is again similar, and in fact much easier. In Figure \ref{fig:vineAnnotated} (D) and Figure \ref{fig:vines}, note that the two bonds going in and out the triangle are both decorated by $g$ (which is determined by the $(x_0,x_j,y_j)$ variables), which means that the vector $r$ for vine (III) above is replaced by $\widetilde{r}=0$. In particular we have $\omega_1+\omega_2+\omega_3=0$ where $\omega_j=\zeta_{\wf_j}\Omega_{\wf_j}$. Then we argue as above, with $(u_1,u_2,u_3)=(a-e,b-e,a-c)$ if bonds decorated by $a$ and $b$ have the same direction, and $(u_1,u_2,u_3)=(e-g,b-g,d-g)$ if they have opposite directions, then the same results will hold.

As a result, in all cases we can define $(u_1, u_2, u_3)$ so that the full decoration of the vine $\Vb$ and hence that of $\Qc_0[\Vb]$ is completely determined by $(x_0, x_j, y_j, u_1, u_2, u_3)_{1\leq j \leq m}$. The factors $\zeta_{\uf_{2j+1}}\Omega_{\uf_{2j+1}}$ and $\zeta_{\uf_{2j+2}}\Omega_{\uf_{2j+2}}$ for $1\leq j \leq m$ are given by $2x_j \cdot y_j$ and $-2x_j \cdot y_j +2\widetilde {r}_j \cdot \mu_j$ where $\mu_j=a_jx_j+b_jy_j+c_jr$ with $a_j,b_j,c_j\in\{0,\pm1\}$ and $a_j^2+b_j^2\neq 0$, and $\widetilde {r}_j=r$ if $j\leq m_1$ and $\widetilde{r}_j=0$ otherwise. The remaining three branching nodes in $\Qc_0[\Vb]$ have their resonance factors given by $2u_1 \cdot u_2$ and $\Lambda\in\{2u_1\cdot u_3,2u_3\cdot (u_1+u_2-u_3)\}$, and $-2u_1\cdot u_2-\Lambda+\widetilde r \cdot\xi$ where $\xi=au_1+bu_2+cu_3+dr$ with $a,b,c,d\in\{0,\pm1\}$ and $\widetilde{r}\in\{r,0\}$. In particular, the change of variables from $(k_\nf)_{\nf\in \Qc_0[\Vb]}$ into $(x_0, x_j, y_j , u_1, u_2, u_3)$ satisfies the conditions stated in Lemma \ref{sumintest2}.

\smallskip
Now, with the reparametrization, we can argue in the same way as in Step 2 of the proof of Proposition \ref{estbadvine} to restrict each of the $(x_0,x_j,y_j,u_j)$ and $k_\nf$ variables to a fixed unit ball, and consequently expand $Z$ into (\ref{defztilde}). Moreover, once we confirm $r\neq 0$, we can get rid of the $\epsilon_{\Es[\Vb]}$ factor in the same way as in Step 2 of the proof of Proposition \ref{estbadvine}. Next, by using the $(x_0,x_j,y_j,u_j)$ variables, we can reduce (\ref{vinebound2}) to estimating the expression
\begin{align}
\widetilde \Kc_{(\mathtt{ind},\Bs)}&= e^{\pi i (\lambda_{\uf_{21}}t_{\uf_{21}}+\lambda_{\uf_{22}} t_{\uf_{22}}+\lambda_{\uf_1}t_{\uf_1})} \bigg(\frac{\delta}{2L^{d-\gamma}}\bigg)^{n_1}\zeta[\Vb] \sum_{(x_0,x_j,y_j):1\leq j\leq m}\sum_{(u_1, u_2, u_3)}\int_{ \Ec[\Vb]}\nonumber\\
&\times e^{2\pi i \lambda_{\uf_2} t_{\uf_2}}e^{\pi i\cdot\delta L^{2\gamma} (t_{\uf_2}-t_{\uf_1})(r\cdot \mu_0)} \prod_{j=1}^{m} e^{\pi i\cdot\delta L^{2\gamma} (t_{\uf_{2j+2}} -t_{\uf_{2j+1}})x_j \cdot y_j}\prod_{j=1}^{m_1}e^{\pi i\cdot\delta L^{2\gamma} (t_{\uf_{2j+1}}-t_{\uf_1})(r\cdot \mu_j)} \nonumber\\
&\times e^{\pi i\cdot \delta L^{2\gamma}(t_{\wf_1}-t_{\wf_3})(u_1\cdot u_2)}e^{\pi i\cdot \delta L^{2\gamma}(t_{\wf_2}-t_{\wf_3})\Lambda} e^{\pi i\cdot \delta L^{2\gamma}t_{\wf_3}(\widetilde{r}\cdot\xi)} {\prod_{\lf\in
\Lc[\Vb]}^{(+)}\Kc_{\Qc^{(\lf,\lf')}}(t_{\lf^p},t_{(\lf')^p},k_\lf)}\nonumber\\
&\times\prod_{\mf\in \Nc[\Vb]}\Kc_{\Tc^{(\mf)}}^*(t_{\mf^p},t_{\mf},k_\mf)\cdot \widetilde{Z}(x_0,k_{\uf_1},k_{\uf_{11}},k_{\uf_{21}},k_{\uf_{22}})\prod_{j=1}^m \mathrm{d}t_{j}\mathrm{d}s_{j}\cdot\mathrm{d}t_{0}\mathrm{d}\tau_{1}\mathrm{d}\tau_{2}\mathrm{d}\tau_{3}.\label{gvineboundBfix2}
\end{align}

\underline{\it Step 2: Case $\gamma>\frac12$: counting argument.} The argument here is very similar to Step 3 of the proof of Proposition \ref{estbadvine}, but with one additional ingredient. After expanding all the $\Kc_{\Qc}$ and $\Kc^*_\Tc$ as time Fourier integrals, and localizing each $k_\nf\,(\nf \in \Qc_0[\Vb])$ to a unit ball as in Step 3 of the proof of Proposition \ref{estbadvine}, we can reduce the estimate of (\ref{gvineboundBfix2}) to that of \eqref{afterloc7.3}. Recall here that $m=m_1+m_2$ where $m_2$ is the length of the zero-gap ladder in $\Vb$, and $m_1$ is the total length of other ladders. Now it suffices to prove the following two estimates (in fact they will allow us to gain a small power of $L$ in (\ref{vinebound2})):
\begin{multline}\label{twoestgoodv2}
\bigg\|\max(\langle \tau_{\uf_1}\rangle,\langle \tau_{\uf_{21}}\rangle,\langle \tau_{\uf_{22}}\rangle)^{\eta^5}\sum_{\sigma[\Nc[\Vb]]}\sup_{|\alpha_\mf-\sigma_\mf|\leq 1}|\widehat \Bc(\tau_{\uf_1},\tau_{\uf_{21}}, \tau_{\uf_{22}},\alpha[ \Nc[\Vb]])|\bigg\|_{L_{\tau,\alpha}^1}\\\lesssim C^m \delta^{-(m+2)/2}L^{Cm_1\sqrt \delta} (\log L)^{C m_1}L^{4\eta^5},
\end{multline}
\begin{equation}\label{twoestgoodv1}
\sup_{(k_{\uf_1},k_{\uf_{11}},k_{\uf_{21}},k_{\uf_{22}})}\sup_{\sigma[\Nc[\Vb]]}\sum_{\Es}^*\Xc(k[\Qc_{0}[\Vb]])\leq (C^+\delta^{-1})^{m+2}L^{(2m+4)(d-\gamma)-5\eta (1+m_1)}.
\end{equation}
These two estimates are clearly enough to give \eqref{vinebound2} since the shifts can be absorbed by the weights as in the proof of Proposition \ref{estbadvine}. Moreover, \eqref{twoestgoodv2} follows by combining the arguments in Step 3 of the proof of Proposition \ref{estbadvine} with the ladder $L^1$ estimate proved in Proposition 10.1 of \cite{DH21} (but only to the zero-gap ladder, so the number of atoms not in this ladder is $O(m_1)$); we remark that while Proposition 10.1 of \cite{DH21} is proved for $L^\infty_t$ rather than $X^\theta$, but it extends directly to the space $X^\theta$ by simply relying on Lemma \ref{timeFourierlem} instead of Lemma 10.2 in \cite{DH21}.

To prove \eqref{twoestgoodv1}, we reduce it to a counting problem as in the proof of Proposition \ref{estbadvine}. Here we are counting the number of choices for the variables $(x_0, x_j, y_j, u_1, u_2,u_3)_{1\leq j \leq m}$, each of which is in a fixed unit ball, such that each of $(r\cdot \mu_0, x_j\cdot y_j,u_1\cdot u_2,\Lambda)$ where $1\leq j \leq m$, and each of $r\cdot\mu_j$ where $1\leq j\leq m_1$, belongs to a fixed interval of length $O(\delta^{-1}L^{-2\gamma})$. Since $\gamma>1/2$, we know that the number of choices for $x_0$ is $O(\delta^{-1}L^{d-\gamma-(1-\gamma)}$ by Lemma \ref{basiccount} (1), that the number of choices for each $(x_j,y_j)\,(1\leq j\leq m_1)$ is $O(\delta^{-2}L^{2(d-\gamma)-10\eta})$ by Lemma \ref{basiccount0} (1), the number of choices for each $(x_j,y_j)\,(j>m_1)$ is $O(\delta^{-1}L^{2(d-\gamma)})$ by Lemma \ref{basiccount0} (1), and the number of choices for $(u_1,u_2,u_3)$ is $O(\delta^{-2}L^{3(d-\gamma)-(1-\gamma)-10\eta})$ by Lemma \ref{basiccount0} (3). Putting together, this proves \eqref{twoestgoodv1}.

\medskip
\underline{\it Step 3: Case $\gamma\leq \frac12$: Lemma \ref{sumintest2}.} Here, the argument is basically the same as in Step 4 of the proof of Proposition \ref{estbadvine}, except that we rely on Lemma \ref{sumintest2} instead of \ref{sumintest1}.  
We expand $\Kc_{\Qc}$ and $\Kc^*_{\Tc}$ as time Fourier integrals to obtain an expression of the form
\begin{multline}
\bigg(\frac{\delta}{2L^{d-\gamma}}\bigg)^{2m+4}(C\delta)^{n(\Bs)/2}\sum_{|x_0-a_0|\leq 1} e^{\pi i (\lambda_{\uf_{21}}t_{\uf_{21}}+\pi i \lambda_{\uf_{22}} t_{\uf_{22}}+(\lambda_{\uf_1}+\gamma_1)t_{\uf_1})}\\\times\widetilde{Z}(x_0,k_{\uf_1},k_{\uf_{11}},k_{\uf_{21}},k_{\uf_{22}})\Ic\left(x_0, k_{\uf_1}, k_{\uf_{11}}, k_{\uf_{21}}, k_{\uf_{22}}, t_*\right).
\end{multline}

Here $a_0$ is a fixed vector, $\gamma_1$ is a linear combination of $\lambda_{\uf_2}$ and the Fourier variables occurring in the expansions of $\Kc_\Qc^{(\lf, \lf')}$ and $\Kc^*_{\Tc^{(\mf)}}$ as above, and $\Ic$ is an expression in the form \eqref{sumintI1} but modified as in Lemma \ref{sumintest2}, which can be obtained after defining the new time variables 
$t_j=t_{\uf_1}-t_{\uf_{2j+1}}$, $s_j=t_{\uf_1}-t_{\uf_{2j+2}}$ for $1\leq j \leq m$,  $t_0=t_{\uf_1}-t_{\uf_2}$, and $\tau_j=t_{\uf_1}-t_{\wf_j}$. The domain of integration after the change of variables is the same as described in Lemma \ref{sumintest2}, with $t$ replaced by $t_*:=t_{\uf_1}-\max(t_{\uf_{21}},t_{\uf_{22}})\}$. As a result, by \eqref{FourierboundI2}, we have that
$$
\|\Ic\|_{X_{\mathrm{loc}}^{2\eta^5, 0}(t_*)}\lesssim  (C^+\delta^{-1})^{m+2}L^{(2m+4)(d-\gamma)-d+\eta^4}.
$$
After summing over $x_0$, applying Lemma \ref{timeFourierlem} (2), and including the factors $(\delta/(2L^{d-\gamma}))^{2m+4}$ and $(C^+\delta)^{n(\Bs)/2}$ etc., this implies (\ref{vinebound2}); note that the shifts $(\lambda_{\uf_{21}},\lambda_{\uf_{22}},\lambda_{\uf_1}+\gamma_1)$ is again absorbed by the weights.

\medskip
\underline{\it Step 4: The Vine Chain case.} The proof for the vine chain $\Vb\Cb$ runs exactly as above, except that we apply Corollary \ref{sumintest3} instead of Lemma \ref{sumintest2} in Step 3. Indeed, by reparametrizing the whole vine chain $\Vb\Cb$ by going from bottom to top and using the same reparametrization in Step 1 for each ingredient vine, we can define the variables $(x_0^q, x_j^q, y_j^q, u_1^q, u_2^q, u_3^q)$ where $0\leq q<Q$ and $Q$ is the number of ingredient vines in $\Vb\Cb$. We then get an expression for $\Kc_{(\mathtt{sgn},\mathtt{ind},\Bs)}^{(\Vb\Cb, W, Z)}$ generalizing that in \eqref{gvineboundBfix2}. We treat the case when $\gamma>\frac12$ by reducing to a counting estimate as in Step 2 above, and treat the case $\gamma\leq \frac12$ using Corollary \ref{sumintest3} as in Step 3 above. This then completes the proof.
\end{proof}
\begin{rem}\label{extravar} In the setting of Propositions \ref{estbadvine} and \ref{estnormalvine} (described in Section \ref{vinesubset}) we have assumed that $Z$ is a function of $(x_0,k_{\uf_1},k_{\uf_{11}},k_{\uf_{21}},k_{\uf_{22}},t_{\uf_1},t_{\uf_{21}},t_{\uf_{22}},t_{\uf_2})$. In fact, we may allow $Z$ to depend on other variables (say denoted by $k[\Uc]$ and $t[\Vc]$) provided that they do not appear in the rest of the expression for $\Kc_{(\mathtt{sgn},\mathtt{ind},\Bs)}^{(\Vb,Z,W)}$; in this case, if we consider the norm $Y_{\mathrm{loc}}^\theta$ in all variables including $k[\Uc]$ and $t[\Vc]$, then (\ref{vinebound1}), (\ref{vinebound2}) and (\ref{vinebound3}) still hold with the same implicit constants. 

This is because, in the process of the proof of Propositions \ref{estbadvine} and \ref{estnormalvine}, we have restricted each of the variables $(x_0,k_{\uf_1},k_{\uf_{11}},k_{\uf_{21}},k_{\uf_{22}})$ to a fixed unit ball. If $Z$ depends on $k[\Uc]$, then we may also restrict each variable in $k[\Uc]$ to a unit ball, which reduces the $Y_{\mathrm{loc}}^\theta$ bound to the $X_{\mathrm{loc}}^{\theta,0}$ bound. Then we simply view $Z$ as a function with value in the Banach space $L_{k[\Uc]}^\infty$, and apply Propositions \ref{estbadvine} and \ref{estnormalvine} to handle these extra $k[\Uc]$ variables. As for the extra time variables $t[\Vc]$, note that $\Kc_{(\mathtt{sgn},\mathtt{ind},\Bs)}^{(\Vb,Z,W)}$ is linear in $Z$ and thus commutes with taking time Fourier transforms in $t[\Vc]$. Let the Fourier dual of $t[\Vc]$ be $\xi[\Vc]$, then
\begin{equation}\label{sepvar}\|\Zc\|_{X^{\eta^5,0}}\sim\int\big(\max_{\nf\in\Vc}\langle \xi_\nf\rangle\big)^{\eta^5}\|\Fc_{t[\Vc]}\Zc(\cdot,\xi[\Vc])\|_{X^{0,0}}\,\mathrm{d}\xi[\Vc]+\int\|\Fc_{t[\Vc]}\Zc(\cdot,\xi[\Vc])\|_{X^{\eta^5,0}}\,\mathrm{d}\xi[\Vc]\end{equation} for any $\Zc=\Zc(\cdot,t[\Vc])$, where the $\cdot$ represents variables other than $k[\Uc]$ and $t[\Vc]$, and the norm on the left hand side of (\ref{sepvar}) is the norm in all variables, while the norms on the right hand side are in the $\cdot$ variables only. Therefore, we can apply Propositions \ref{estbadvine} and \ref{estnormalvine} for each fixed $\xi[\Vc]$, and then integrate in these variables, to get the same results as in (\ref{vinebound1}), (\ref{vinebound2}) and (\ref{vinebound3}).
\end{rem}
\section{Reduction to counting estimates}\label{reduct1}
\subsection{Preliminary setup} Recall the notions of vines and vine-chains (VC), hyper-vines (HV) and hyper-vine-chains (HVC), and ladders in Definition \ref{defvine}.
\begin{lem}\label{vinechainlem} Define a \emph{double-vine (V)}, or DV for short, to be the union of two vines (V), see Figure \ref{fig:vines}, that share two common joints and no other common atoms\footnote{This is not a vine-like object; see Definition \ref{defvine}.}. Then, for any molecule $\Mb$, there is a unique collection $\Cs$ of disjoint atomic groups, such that each atomic group in $\Cs$ is an HV, VC, HVC or DV, and any vine-like object in $\Vb$ is a subset of some atomic group in $\Cs$.
\end{lem}
\begin{proof} Consider all the maximal vine-like objects in $\Mb$, where maximal is in the sense that it is not a subset of any other vine-like object; let this collection be $\Cs_1$. We know that Lemma \ref{disjointlem} applies to any two of these objects. If $\Ab,\Bb\in\Cs_1$ and $\Ab\cap\Bb\neq\varnothing$, then we are in one of scenarios (a)--(c) of Lemma \ref{disjointlem}. However scenarios (b) and (c) are impossible, because $\Db=\Ab\cup\Bb$ (in case (b)) or $\Db=\Cb_0\cup\Cb_1\cup\Cb_2$ (in case (c)) would be a larger vine-like object that contains $\Ab$ and $\Bb$. We are then left with case (a), where $\sigma(\Ab)=\sigma(\Bb)=1$. This means that each of $\Ab$ and $\Bb$ must be one vine (V), and they share two common joints and no other common atoms, so their union is a DV. Moreover, in this case, neither $\Ab$ nor $\Bb$ can intersect with any other maximal vine-like object $\Cb$ (otherwise $\Ab$ (for example) would form another DV with $\Cb$, which leads to a $4$-regular component). Therefore, let $\Cs$ be obtained from $\Cs_1$ by replacing the two intersecting vines (V) with one DV, then it satisfies the requirement and clearly is unique.
\end{proof}
\begin{df}\label{defcong}Let $\Qc$ be a couple with skeleton $\Qc_{\mathrm{sk}}$,  and let $\Cs$ be defined for the molecule $\Mb(\Qc_{\mathrm{sk}})$ by Lemma \ref{vinechainlem}. Define a collection $\Vs$ of (CL) vines as follows: for each VC in $\Cs$ whose joints do not both have degree $3$, we include into $\Vs$ all its vine ingredients that are (CL) vines. For each HVC in $\Cs$ and each VC in $\Cs$ whose joints both have degree $3$, we include into $\Vs$ all but one of its vine ingredients, such that (a) if there is a (CN) vine then we only exclude this one, and (b) if all vines are (CL) vines then we only exclude the ``top" vine whose $\uf_1$ node is the ancestor of all other $\uf_1$ nodes (as branching nodes of $\Qc_{\mathrm{sk}}$). We do not include anything in $\Cs$ from any HV or DV.

Since $\Vs$ satisfies the assumptions in Definition \ref{twistgen}, we shall define any couple $\Qc'$ to be \emph{congruent} to $\Qc$, if $\Qc'$ is a full twist of $\Qc$ with respect to $\Vs$. Clearly, performing a full twist does not affect the molecule $\Mb(\Qc_{\mathrm{sk}})$ nor the choice of vines in $\Vs$, and congruence is an equivalence relation and preserves the order of \emph{each tree} in a couple.
\end{df}
\begin{df}\label{defdiff} Given any molecule $\Mb$ and block $\Bb\subset\Mb$, let the four bonds in $\Bb$ at the two joints be $\ell_1,\ell_2\in\Bb$ at one joint, and $\ell_3,\ell_4\in\Bb$ at the other. Then for any decoration $(k_\ell)$ we have $k_{\ell_1}-k_{\ell_2}=\pm(k_{\ell_3}-k_{\ell_4}):=r$. We call this vector the \emph{gap} of $\Bb$ relative to this decoration. Note that once the parameters $(c_v)$ of a decoration are fixed as in Definition \ref{decmole}, then this $r$ can be expressed as a function of the vectors $k_{\ell_j^*}$, where $\ell_j^*$ runs over all bonds connecting a given joint of $\Bb$ to atoms \emph{not} in $\Bb$. If $\Bb$ is concatenated by blocks $\Bb_j$, then all $\Bb_j$ must have the same gap as $\Bb$. For the hyper-block which is adjoint of $\Bb$, we define its gap to be the gap of $\Bb$. Note that the gap of a block can never be $0$ due to Remark \ref{nonresrem}.

More generally, if $v$ is an atom and $\ell_1,\ell_2\sim v$ are two bonds with opposite directions, then we define the \emph{gap} of the triple $(v,\ell_1,\ell_2)$ relative to a given decoration as $r:=k_{\ell_1}-k_{\ell_2}$. In particular the gap of any block or hyper-block equals a suitable gap at either of its joints. Next, for any ladder of length $\geq 1$ (see Definition \ref{defvine} and Figure \ref{fig:vines}), the difference $k_\ell-k_{\ell'}$ for any pair of parallel single bonds $(\ell,\ell')$ must be equal (up to a sign change), which we also define to be the \emph{gap} of the ladder. In particular, if $\Vb$ is a vine (or VC) with gap $r$, then for any ladder contained in $\Vb$ that is inserted between parallel dashed bonds of the same color in Figure \ref{fig:vines}, the gap of this ladder is either $\pm r$ or $0$. Finally, for all the gaps defined above, we say it is \emph{small gap} (or SG for short) if $|r|\leq L^{-\gamma+\eta}$ (including $r=0$), and \emph{large gap} (or LG) if $|r|> L^{-\gamma+\eta}$.
\end{df}
With the above preparations, we can reduce Propositions \ref{mainprop1}--\ref{mainprop4} to the following
\begin{prop}
\label{kqmainest1} We can define a value $\rho=\rho(\Qc)$ associated to a non-regular couple $\Qc$, which is an integer and $1\leq\rho\leq n$ (where $n$ is the order of $\Qc$), such that (i) it takes the same value for $\Qc$ in the same congruence class, (ii) the number of couples $\Qc$ of order $n$ such that $\rho(\Qc)=\rho$ is at most $(C\rho)!C^n$, and (iii) for any couple $\Qc$ of order $n$, we have
\begin{equation}\label{kqmainest1-1}\bigg|\sum_{\Qc'}\Kc_{\Qc'}(t,t,k)\bigg|\lesssim\langle k\rangle^{-20d}(C^+\delta^{1/4})^n\cdot L^{-\eta^7\cdot\rho(\Qc)},
\end{equation} where $\Qc'$ runs over all couples congruent to $\Qc$.
\end{prop}
\begin{proof}[Proof of Propositions \ref{mainprop1}--\ref{mainprop4} assuming Proposition \ref{kqmainest1}] Consider the sum on the left hand side of (\ref{mainest1}) and (\ref{mainest1.5}). The sum over all regular couples $\Qc$ is taken care of by Propositions \ref{regcpltreeasymp} and \ref{regcpltreesum} (in particular by the bound (\ref{kqbd}) and the equality (\ref{matchn})), so we only need to consider the sum over non-regular couples $\Qc$.

Note that by Definition \ref{twistgen}, if $\Qc=(\Tc^+,\Tc^-)$ and $\Qc'=((\Tc')^+,(\Tc')^-)$ are two congruent couples, then $n(\Tc^\pm)=n((\Tc')^\pm)$. Therefore both sums on the left hand side of (\ref{mainest1}) and (\ref{mainest1.5}), over non-regular couples $\Qc$, can be written as a sum of subset sums, such that each subset sum has the form of the left hand side of (\ref{kqmainest1-1}). We then classify these subset sums according to the value $\rho(\Qc)$, and apply Proposition \ref{kqmainest1} to get that
\begin{equation}\label{suminrho1}\bigg|\sum_{\Qc}\Kc_\Qc(t,t,k)\bigg|\lesssim (C^+\delta^{1/4})^m\sum_{\rho=1}^m(C\rho)!C^m L^{-\eta^7\rho}.
\end{equation} Here in (\ref{suminrho1}) the value $m\in\{2n,2n+1\}$ is fixed, and the sum is either taken over all non-regular couples $\Qc$ of order $m$, or taken over all non-regular couples $\Qc=(\Tc^+,\Tc-)$ with $n(\Tc^+)=n(\Tc^-)=m/2=n$. In either case, since $\rho\leq m$, we have \[(C\rho)^C\leq (Cm)^C\lesssim (\log L)^C\ll L^{\eta^9},\] so the $\rho$-sum on the right hand side of (\ref{suminrho1}) is easly bounded by $L^{-\eta^8}$, which then proves both (\ref{mainest1}) and (\ref{mainest1.5}).
\end{proof}
\subsection{Stage 1 reduction: Small gap (CL) vines}\label{stage1red} We start the proof of Proposition \ref{kqmainest1}. First, using (\ref{bigformula2}), we can rewrite the left hand side of (\ref{kqmainest1-1}) as
\begin{multline}\label{bigformula2new}\mathrm{LHS\ of\ }(\ref{kqmainest1-1})=\sum_{(\Qc_{\mathrm{sk}},\As)}\bigg(\frac{\delta}{2L^{d-\gamma}}\bigg)^{n_0}\zeta(\Qc_{\mathrm{sk}})\sum_{\Es_{\mathrm{sk}}}\int_{\Ec_{\mathrm{sk}}}\epsilon_{\Es_{\mathrm{sk}}}\prod_{\nf\in \Nc_{\mathrm{sk}}} e^{\zeta_\nf\pi i\cdot\delta L^{2\gamma}\Omega_\nf t_\nf}\,\mathrm{d}t_\nf\\\times{\prod_{\lf\in\Lc_{\mathrm{sk}}}^{(+)}\Kc_{\Qc^{(\lf,\lf')}}(t_{\lf^p},t_{(\lf')^p},k_\lf)}\prod_{\mf\in\Nc_{\mathrm{sk}}}\Kc_{\Tc^{(\mf)}}^*(t_{\mf^p},t_{\mf},k_\mf).
\end{multline}
Here the $\Qc_{\mathrm{sk}}$ runs over all twists of a given prime couple, $\As$ runs over collections of regular trees and regular couples that satisfy a certain set of assumptions (see Definition \ref{twistgen}), and $\Es_{\mathrm{sk}}$ runs over all $k$-decorations of $\Qc_{\mathrm{sk}}$ (note that, different choices of $\Qc_{sk}$ are twists of each other, so their decorations are in one-to-one correspondence as in Remark \ref{dectwist}). In particular, if $\Cs$ is defined by Lemma \ref{vinechainlem} for $\Mb(\Qc_{\mathrm{sk}})$, then for each vine-like object in $\Cs$, we can specify whether it has SG or LG under the decoration (each such specification imposes a set of restrictions on the decoration). This reduces the left hand side of (\ref{kqmainest1-1}) to a superposition of at most $C^n$ terms, such that each vine-like object is specified to be either SG or LG in each term. Let $\Vs_0$ be the collection of (CL) vines in $\Vs$ (as Definition \ref{defcong}) that are SG. Note that $\Qc$ runs over a congruence class, which is defined by full twists at all vines in $\Vs$; if we strengthen the equivalence relation by allowing only full twists at all vines in $\Vs_0$, then the sum in $\Qc$ can again be written as a superposition of at most $C^n$ terms, each of which is a sum over the new equivalence class. As such, we only need to consider one of these new terms, which we shall refer to as $\Ks$ for below. 

We define a new ordered collection $\Us_0$ from $\Vs_0$, whose elements are \emph{bad (CL) vines and normal (CL) vine chains,} as follows. First organize vines in $\Vs_0$ into disjoint VC, then order these VC arbitrarily; for each VC, we divide it into \emph{units} and order them from bottom to top, where each unit is either a bad (CL) vine or a (longest) sub-VC formed by consecutive normal (CL) vines.

Now let $\Qc_0=\Qc_{\mathrm{sk}}$ and $\As_0=\As$ as in (\ref{bigformula2new}), then $\Qc\sim(\Qc_0,\As_0)$. Denote $\Us_0$, with elements ordered as above, by $\Us_0=\{\Ub_0,\cdots,\Ub_{q-1}\}$, and denote the $\uf_1,\uf_{23}$ nodes for $\Ub_j$ by $\uf_1^j,\uf_{23}^j$ etc. Define $\Qc_{j+1}$ to be the result of splicing $\Qc_j$ at $\Ub_j$ (so $\Mb(\Qc_{j+1})$ is the result of merging the $\Ub_j$ in $\Mb(\Qc_j)$ into a single atom), and let $\As_{j}$ be the $\As$ collection corresponding to $\Qc_j$, and $\Us_j:=\{\Ub_j,\cdots,\Ub_{q-1}\}$. As $\Qc$ runs over all full twists of a given couple at vines in $\Vs_0$ (or $\Us_0$), we know that $\Qc_j$ runs over all twists of a given couple at vines in $\Us_j$. Moreover, $\As_j$ runs over all collections of regular couples and regular trees, such that the value of $n(\Qc^{(\uf_{23}^i,\uf_0^i)})+n(\Tc^{(\uf_2^i)})$ is fixed for all $i\geq j$ such that $\Ub_i$ is a core vine, and all other regular couples and regular trees are uniquely fixed.

By Remark \ref{fulltwistrep}, for each $j$ we have $(\Qc_j,\As_j)\leftrightarrow(\Qc_{j+1}, \texttt{cod}_j,\mathfrak{n}_j,\texttt{ind}_j,\Bs_j,\As_{j+1})$, with $(\texttt{cod}_j,\mathfrak{n}_j)$ etc. corresponding to $\Ub_j$ in $\Qc_j$ (these matter only when $\Ub_j$ is a bad vine rather than a normal VC). By losing at most $C^n$ we may fix the codes $\texttt{cod}_j$ for each $j$. Moreover, the branching node $\mathfrak{n}_j$ for each $j$ is also uniquely fixed because it corresponds to the atom formed by merging $\Ub_i\,(0\leq i\leq j)$ in $\Mb(\Qc_{\mathrm{sk}})$, and the molecule $\Mb(\Qc_{\mathrm{sk}})$ (as a directed graph) does not vary when $\Qc_{\mathrm{sk}}$ is twisted (note however that $\zeta_{\nf_j}$ may change if we twist at $\Ub_{j+1}$). As such, we know that $(\Qc_j,\As_j)$ uniquely corresponds to the quadruple $(\Qc_{j+1}, \As_{j+1},\texttt{ind}_j,\Bs_j)$. Let $\Oc_j$ be the set of all nodes in $\Qc_{j}$ that do \emph{not} belong to any $\Qc_{j}[\Ub_i]\backslash\{\uf_1^i\}$ for $i\geq j$. Now we prove, by induction in $j$, the following result:
\begin{prop}\label{red1step} For each $j$, we have
\begin{multline}\label{vinered}
\Ks=(C^+\delta^{1/4})^{M_j}\sum_{(\Qc_{j},\As_{j})}\bigg(\frac{\delta}{2L^{d-\gamma}}\bigg)^{n_{j}}\zeta(\Qc_{j})\sum_{\Es_{j}}\int_{\Ec_{j}}\epsilon_{\Es_{j}}^*\prod_{\nf\in \Nc_{j}} e^{\zeta_\nf\pi i\cdot\delta L^{2\gamma}\Omega_\nf t_\nf}\,\mathrm{d}t_\nf\\\times\Zc_{j}\cdot{\prod_{\mf\in\Lc_{j}}^{(+)}\Kc_{\Qc^{(\mf,\mf_-)}}(t_{\mf^p},t_{\mf_-^p},k_\mf)}\prod_{\mf\in\Nc_{j}}\Kc_{\Tc^{(\mf)}}^*(t_{\mf^p},t_{\mf},k_\mf).
\end{multline} Here $n_{j}$ is the order of $\Qc_{j}$, $M_j=(n_0-n_j)+n(\As)-n(\As_j)$, and all symbols with subscript or superscript $j$ are associated with the couple $\Qc_{j}$. The summation in $(\Qc_j,\As_j)$ is taken as above: $\Qc_j$ runs over all twists of a given couple at vines in $\Us_j$. Moreover, $\As_j$ runs over all collections of regular couples and regular trees, such that the value of $n(\Qc^{(\uf_{23}^i,\uf_0^i)})+n(\Tc^{(\uf_2^i)})$ is fixed for all $i\geq j$ such that $\Ub_i$ is a core vine, and all other regular couples and regular trees are uniquely fixed.

In (\ref{vinered}), the factor $\epsilon_{\Es_j}^*$ is defined as in (\ref{defcoef}) but with the product containing only $\nf\in \Wc_j$, where $\Wc_j$ is a subset of $\Nc_j$; the function $\Zc_j=\Zc_j(x_0^j,k[\Xc_j],t[\Yc_j],t_{\uf_2^j})$ where $\Xc_j$ and $\Yc_j$ are two subsets of $\Oc_j$ (when $j=q$ there is no $x_0^j$ and $t_{\uf_2^j}$), and in the summation one replaces $x_0^j$ by $k_{\uf_2^j}$ if $\zeta_{\uf_2^j}=-$, and by $k_{\uf_{23}^j}$ otherwise. The objects $\Wc_j,\Xc_j,\Yc_j,\Zc_j$ etc. do not depend on the choice of $\Qc_j$ or $\As_j$, and the function $\Zc_j$ satisfies $\|\Zc_j\|_{Y_{\mathrm{loc}}^{\eta^5}}\lesssim (C^+)^jL^{\Delta_j}$, where $\Delta_0=0$ and $\Delta_{j+1}-\Delta_{j}$ equals $-\eta^2$ or $\eta^4$ depending on whether $\Us_j$ is a bad vine or a normal vine chain. Finally, in the support of $\Zc_j\cdot\epsilon_{\Es_j}^*$, the values of $k_\nf$ variables must be inherited from a $k$-decoration of $\Qc_{\mathrm{sk}}$ that satisfies all the SG and LG assumptions specified above, as well as the non-degeneracy condition $\epsilon_{\Es_{\mathrm{sk}}}\neq 0$.
\end{prop}
\begin{proof} For $j=0$, recall that we have fixed the term $\Ks$ by specifying whether each vine-like object in $\Cs$ has LG or SG; this is done by attaching factors to the expression in (\ref{bigformula2new}) which are indicator functions of differences of various $k_\mf$ in the decoration. Let the product of these functions be $\Zc_{0}$, then it only depends on the $k_\mf$ variables for $\mf\in\Oc_0$ (see Definition \ref{defdiff}), and does not depend on any time variables $t_\mf$. Let $\Wc_0=\Nc_{\mathrm{sk}}$, then (\ref{vinered}) is true for $j=0$.

Suppose (\ref{vinered}) is true for $j$, we will prove it for $j+1$. Start with the expression $\Ks$ as in (\ref{vinered}), note that the summation
 \[\sum_{(\Qc_j,\As_j)}=\sum_{(\Qc_{j+1},\As_{j+1})}\sum_{(\mathtt{ind}_j,\Bs_j)}\] as described above (recall that $\mathtt{cod}$ has been fixed). We now fix $(\Qc_{j+1},\As_{j+1})$, and consider the part of the sum and integral in (\ref{vinered}) that involves only $(\mathtt{ind}_j,\Bs_j)$ and the $k_\nf$ and $t_\nf$ variables for $\nf\in\Qc_j[\Ub_j]\backslash\{\uf_1^j\}$. Note also that $\Zc_j=\Zc_j(x_0^j,k[\Xc_j],t[\Yc_j],t_{\uf_2^j})$; by  inserting finitely many smooth time cutoff functions, we may formally write \[\Zc_j=\Zc_j(x_0^j,k_{\uf_1^j},k_{\uf_{11}^j},k_{\uf_{21}^j},k_{\uf_{22}^j},t_{\uf_1^j},t_{\uf_{21}^j},t_{\uf_{22}^j},t_{\uf_2^j},k[\Uc],t[\Vc]),\] for some $\Uc\subset\Xc_j\subset\Oc_j,\,\Vc\subset\Yc_j\subset\Oc_j$ where $\{\uf_1^j,\uf_{11}^j,\uf_{21}^j,\uf_{22}^j\}\cap\Uc=\varnothing$ and $\{\uf_1^j,\uf_{21}^j,\uf_{22}^j\}\cap\Vc=\varnothing$.

Then, this part of summation and integration is exactly of the form
\begin{equation}\label{subsummation}\sum_{(\mathtt{ind}_j,\Bs_j)}\Kc_{(\mathtt{ind}_j,\Bs_j)}^{(\Ub_j,Z,W)}\end{equation} defined as in (\ref{vinebound1}) in Section \ref{vinesubset}, where the set $W=\Wc_j\cap\Qc_j[\Ub_j]$, the function \[\qquad Z(x_0^j,k_{\uf_1^j},k_{\uf_{11}^j},k_{\uf_{21}^j},k_{\uf_{22}^j},t_{\uf_1^j},t_{\uf_{21}^j},t_{\uf_{22}^j},t_{\uf_2^j})=\Zc_j(x_0^j,k_{\uf_1^j},k_{\uf_{11}^j},k_{\uf_{21}^j},k_{\uf_{22}^j},t_{\uf_1^j},t_{\uf_{21}^j},t_{\uf_{22}^j},t_{\uf_2^j},k[\Uc],t[\Vc])\] with $k[\Uc]$ and $t[\Vc]$ viewed as parameters, and the sum over $(\mathtt{ind}_j,\Bs_j)$ is exactly as in Propositions \ref{estbadvine} and \ref{estnormalvine}. Here again we plug in $x_0^j=k_{\uf_{2}^j}$ if $\mathtt{ind}_j=\zeta_{\uf_2^j}=-$, and $x_0^j=k_{\uf_{23}^j}$ otherwise.

Now define $\Wc_{j+1}=\Wc_j\backslash\Qc_j[\Ub_j]$ (where we identify branching nodes in $\Qc_{j+1}$ with the corresponding ones in $\Qc_j$), and
\begin{equation}\label{inductzj}
\begin{aligned}\Zc_{j+1}&=\Zc_{j+1}(k_{\uf_1^j},k_{\uf_{11}^j},k_{\uf_{21}^j},k_{\uf_{22}^j},t_{\uf_1^j},t_{\uf_{21}^j},t_{\uf_{22}^j},k[\Uc],t[\Vc])\\&:=(C^+\delta^{1/4})^{-n_*}\exp(-\pi i\cdot\delta L^{2\gamma}t_{\uf_1^j}\Gamma)\sum_{(\mathtt{ind}_j,\Bs_j)}\Kc_{(\mathtt{sgn}_j,\mathtt{ind}_j,\Bs_j)}^{(\Ub_j,Z,W)}(k_{\uf_1^j},k_{\uf_{11}^j},k_{\uf_{21}^j},k_{\uf_{22}^j},t_{\uf_1^j},t_{\uf_{21}^j},t_{\uf_{22}^j})
\end{aligned}\end{equation} as in (\ref{vinebound1}), where $\mathtt{sgn}_j=\zeta_{\uf_1^j}$, and $n_*$ equals the number of branching nodes in $\Qc_j[\Ub_j]\backslash\{\uf_1^j\}$ plus $n(\Bs_j)$. Then, the expression $\Ks$ in (\ref{vinered}) can be reduced to the same expression (\ref{vinered}) with $j$ replaced by $j+1$ (note that $M_{j+1}=M_j+n_*$, and similarly $n(\Qc_{j})-n(\Qc_{j+1})$ equals the number of branching nodes in $\Qc_j[\Ub_j]\backslash\{\uf_1^j\}$), with the new quantities $\Wc_{j+1}$ and $\Zc_{j+1}$.

It remains to prove that the new expression $\Zc_{j+1}$ verifies our assumptions. In addition to $k[\Uc]$ and $t[\Vc]$, the function $\Zc_{j+1}$ also depends on the new variables $(k_{\uf_1^j},k_{\uf_{11}^j},k_{\uf_{21}^j},k_{\uf_{22}^j},t_{\uf_1^j},t_{\uf_{21}^j},t_{\uf_{22}^j})$; therefore we may define $\Xc_{j+1}=(\Xc_j\cup\{\uf_1^j,\uf_{11}^j,\uf_{21}^j,\uf_{22}^j\})\cap\Oc_{j+1}$ and $\Yc_{j+1}=(\Yc_j\cup\{\uf_1^j,\uf_{21}^j,\uf_{22}^j\})\cap\Oc_{j+1}$. Note that with such recursive definition of $\Xc_j$ and $\Yc_j$, it is easy to verify that, if $\Ub_j$ is concatenated with $\Ub_{j+1}$, then $\{\uf_1^j,\uf_{11}^j\}\cap \Xc_j=\varnothing$; this is because $\uf_1^j=\uf_2^{j+1}$ and $\uf_{11}^j=\uf_{23}^{j+1}$, so these nodes belong to $\Ub_{j+1}$, and are thus not involved in any vine merged \emph{before} $\Ub_j$.

Now, if any of these new variables $(k_{\uf_1^j},k_{\uf_{11}^j},k_{\uf_{21}^j},k_{\uf_{22}^j},t_{\uf_1^j},t_{\uf_{21}^j},t_{\uf_{22}^j})$ coincides with $k_\nf$ or $t_\nf$ for any $\Ub\in\Us_{j+1}$ and $\nf\in\Qc_{j+1}[\Ub]\backslash\{\uf_1\}$ (where $\uf_1$ is associated with $\Ub$ as in Proposition \ref{block_clcn})--or otherwise the assumptions for $\Zc_{j+1}$ is already verified--then $\Ub_j$ and $\Ub$ must share a common joint in $\Mb(\Qc_j)$, so $\Ub=\Ub_{j+1}$ and it is concatenated with $\Ub_j$ as above. Consequently the function $\Zc_j$, as well as $Z$ defined above, does not depend on the variables $k_{\uf_1^j}=k_{\uf_2^{j+1}}$ and $k_{\uf_{11}^j}=k_{\uf_{23}^{j+1}}$. By applying Remark \ref{remkv} to (\ref{inductzj}), we see that $\Zc_{j+1}$ depends only on the vector variables $k[\Xc_{j+1}]$ and $x_0^{j+1}$ (which is $k_{\uf_1^j}=k_{\uf_2^{j+1}}$ if $\mathtt{ind}_{j+1}=\zeta_{\uf_2^{j+1}}=-$ and is $k_{\uf_{11}^j}=k_{\uf_{23}^{j+1}}$ otherwise), and that it does not depend on the choice of $\mathtt{ind}_{j+1}$ when regarded as a function. The $t_\nf$ variables are similar, as in this case $t_{\uf_1^j}=t_{\uf_2^{j+1}}$, so $\Zc_{j+1}$ is also allowed to depend on $t_{\uf_2^{j+1}}$ and $t[\Yc_{j+1}]$, as desired.
 
Finally, it is also clear by definition that the support of $\Zc_{j+1}\cdot\epsilon_{\Es_{j+1}}^*$ contains only those decorations that are inherited from decorations in the support of $\Zc_j\cdot\epsilon_{\Es_j}^*$, so it remains to prove that
\begin{equation}\label{inductineqZ}\|\Zc_{j+1}\|_{Y_{\mathrm{loc}}^{\eta^5}}\lesssim L^{\Delta_{j+1}-\Delta_j} \|\Zc_{j}\|_{Y_{\mathrm{loc}}^{\eta^5}},\end{equation} but this follows from Propositions \ref{estbadvine}, \ref{estnormalvine} and Remark \ref{extravar}.
\end{proof}
Define the couple $\Qc_{\mathrm{sub}}$ to be the result of doing splicing at all vines in $\Vs_0$ (equivalently, all ingredient vines of all vine-like objects in $\Us_0$) from $\Qc_{\mathrm{sk}}$, so the molecule $\Mb(\Qc_{\mathrm{sub}})$ is obtained by merging all the vines in $\Vs_0$ from $\Mb(\Qc_{\mathrm{sk}})$. This whole splicing process will be called \emph{stage 1 reduction}. Using Proposition \ref{red1step}, we can reduce Proposition \ref{kqmainest1} to the following counting estimate:
\begin{prop}\label{kqmainest2} Suppose we fix $k\in\Zb_L^d$ and $k_\ell^0\in\Zb_L^d$ for each bond $\ell$ of $\Mb(\Qc_{\mathrm{sub}})$, and $\beta_v\in\Rb$ for each atom $v$ of $\Mb(\Qc_{\mathrm{sub}})$. Consider all the maximal ladders in $\Mb(\Qc_{\mathrm{sub}})$; assume they are $\Lc_j\,(1\leq j\leq q_{\mathrm{sub}})$ with length $z_j$, and fix a dyadic number $P_j\in[L^{-1},1]\cup\{0\}$ for each $j$. Define $m_{\mathrm{sub}}'$ as the number of atoms not in any of these maximal ladders, and $\rho_{\mathrm{sub}}=q_{\mathrm{sub}}+m_{\mathrm{sub}}'$.

Now consider all the $k$-decorations $(k_\ell)$ of $\Mb(\Qc_{\mathrm{sub}})$, with the following additional requirements:
\begin{enumerate}[{(i)}]
\item This $k$-decoration $(k_\ell)$ is inherited from a $k$-decoration $\Es_{\mathrm{sk}}$ of $\Qc_{\mathrm{sk}}$ (and $\Mb(\Qc_{\mathrm{sk}})$) that satisfies all the SG and LG assumptions specified in the proof above, as well as the non-degeneracy condition $\epsilon_{\Es_{\mathrm{sk}}}\neq 0$.
\item The decoration is restricted by $(\beta_v)$ and $(k_\ell^0)$, i.e. $|\Gamma_v-\beta_v|\leq\delta^{-1}L^{-2\gamma}$ and $|k_\ell-k_\ell^0|\leq 1$ (see Definition \ref{decmole}), and the gap $r_j$ of each ladder $\Lc_j$ satisfies that $|r_j|\sim P_j$ (or $|r_j|\gtrsim P_j$ if $P_j=1$).
\end{enumerate}
Let $n_{\mathrm{sub}}$ be the order of $\Qc_{\mathrm{sub}}$, $\Delta_{\mathrm{sub}}$ is the end value $\Delta_j\,(j=q)$ in Proposition \ref{red1step}, and define $\Xf_j=\min((\log L)^2,1+\delta L^{2\gamma}P_j)$ (so $1\leq \Xf_j\lesssim (\log L)^2$). Then, the number $\Cf$ of such restricted $k$-decorations is bounded by
\begin{equation}\label{kqmainest2-1}\Cf\leq (C^+\delta^{-1/2})^{n_{\mathrm{sub}}}L^{(d-\gamma)n_{\mathrm{sub}}}\cdot L^{-\Delta_{\mathrm{sub}}}L^{-\eta^6\rho_{\mathrm{sub}}}\prod_{j=1}^{q_{\mathrm{sub}}}\Xf_j^{-z_j}.
\end{equation}
\end{prop}
\begin{proof}[Proof of Proposition \ref{kqmainest1} assuming Proposition \ref{kqmainest2}] We start by defining $\rho(\Qc)=\rho_{\mathrm{sub}}=q_{\mathrm{sub}}+m_{\mathrm{sub}}'$ where $q_{\mathrm{sub}}$ and $m_{\mathrm{sub}}'$ are defined as above. Clearly its value does not depend on the congruence class of $\Qc$ (as $\Qc_{\mathrm{sub}}$ does not), and for non-regular couples $\Qc$, the couple $\Qc_{\mathrm{sk}}$ is nontrivial, and so is $\Qc_{\mathrm{sub}}$. This means that $1\leq\rho(\Qc)\leq n$ (the latter inequality is trivial as $\rho_{\mathrm{sub}}=q_{\mathrm{sub}}+m_{\mathrm{sub}}'$ does not exceed $n_{\mathrm{sub}}\leq n$).

Next, once $\rho_{\mathrm{sub}}=\rho$ is fixed, the number of choices for $\Mb(\Qc_{\mathrm{sub}})$ is clearly $\lesssim (C\rho)!C^n$ because this molecule must equal at most $\rho$ ladders of total length $\leq n$, plus at most $\rho$ extra atoms. Then, going back from $\Mb(\Qc_{\mathrm{sub}})$ to $\Mb(\Qc_{\mathrm{sk}})$ again involves inserting finitely many VC with total number of atoms $\leq n$, which again has $\lesssim C^n$ choices, as a VC of given length $m$ has $\lesssim C^m$ possibilities, thus the number of choices for $\Mb(\Qc_{\mathrm{sk}})$ is also $\lesssim (C\rho)!C^n$. By Propositions \ref{recover} and \ref{skeleton}, we know that the same bound holds for $\Qc_{\mathrm{sk}}$ and $\Qc$, and certainly also for congruence classed of $\Qc$.

It suffices to prove (\ref{kqmainest1-1}). We only need to control a single term $\Ks$ as defined above, and using Proposition \ref{red1step}, we can start with the expression $\Ks$ in (\ref{vinered}) with $\Qc_j$ replaced by $\Qc_{\mathrm{sub}}$ (which is the final result of all the splicing described above), and all subscripts $(\cdots)_j$ replaced by $(\cdots)_{\mathrm{sub}}$. There is now no summation in $(\Qc_{\mathrm{sub}}, \As_{\mathrm{sub}})$, as the $\Us_0$ is reduced to empty set in the end, and the choice of $\Qc_{\mathrm{sub}}$ and $\As_{\mathrm{sub}}$ are fixed. For simplicity we will denote $(\rho_{\mathrm{sub}},q_{\mathrm{sub}},m_{\mathrm{sub}}')$ by $(\rho,q,m')$.

Due to the $\langle k_\lf\rangle^{-40d}$ decay for each leaf $\lf$, we may also restrict $k_\lf$ for each leaf $\lf$ to a unit ball in $\Rb^d$, say $|k_\lf-k_\lf^0|\leq 1$. By Lemma \ref{explem}, we can decompose the full expression into $\lesssim C^n$ terms, such that in each term, the variable $k_\mf$ belongs to a fixed unit ball not only for \emph{leaves} $\mf=\lf$, but also for \emph{all nodes} $\mf$. With this localization, we may use the $X_{\mathrm{loc}}^{\eta,40d}$, $X_{\mathrm{loc}}^{\eta,0}$ and $Y_{\mathrm{loc}}^{\eta^5}$ bounds for the factors $\Kc_{\Qc^{(\lf,\lf')}}$, $\Kc_{\Tc^{(\mf)}}$ and $\Zc$, see Propositions \ref{regcpltreeasymp} and \ref{red1step}, to reduce the product
\[\Zc_{\mathrm{sub}}\cdot{\prod_{\lf\in\Lc_{\mathrm{sub}}}^{(+)}\Kc_{\Qc^{(\lf,\lf')}}(t_{\lf^p},t_{(\lf')^p},k_\lf)}\prod_{\mf\in\Nc_{\mathrm{sub}}}\Kc_{\Tc^{(\mf)}}^*(t_{\mf^p},t_{\mf},k_\mf)\] in (\ref{vinered}) to a linear combination of functions
\[L^{\Delta_{\mathrm{sub}}}(C^+\delta)^{n(\As_{\mathrm{sub}})/2}\cdot e^{\pi i\mu t}\cdot\prod_{\nf\in\Nc_{\mathrm{sub}}}e^{\pi i\lambda_\nf t_\nf}\cdot\prod_{\lf\in\Lc_{\mathrm{sub}}}^{(+)}\langle k_\lf\rangle^{-40d}\cdot\Xc(k[\Qc_{\mathrm{sub}}])\] for different choices of $(\mu,\lambda[\Nc_{\mathrm{sub}}])$, with the coefficient being an $L^1$ integrable function of $(\mu,\lambda[\Nc_{\mathrm{sub}}])$, where $\Xc$ is a bounded function of all the $k_\mf$ variables. Below we may fix one choice of $(\mu,\lambda[\Nc_{\mathrm{sub}}])$ (in which our estimates will be uniform). To match Proposition \ref{kqmainest2}, we also identify the $q$ maximal ladders in $\Mb(\Qc_{\mathrm{sub}})$ and fix a dyadic number $P_j\in[L^{-1},1]\cup\{0\}$ such that the gap of the $j$-th ladder $\Lc_j$ is $|r_j|\sim P_j$ (or $|r_j|\gtrsim P_j$ if $P_j=1$). This causes a loss of $(\log L)^q$ by summing over all choices of $P_j$, which will be ignored in view of the $L^{-\eta^6\rho}$ gain expected in (\ref{kqmainest2-1}).

\smallskip
Now consider such a term (we call it $\Ms$), rewrite it as
\begin{multline}\label{expressionM}\Ms=e^{\pi i\mu t}\prod_{\lf\in\Lc_{\mathrm{sub}}}^{(+)}\langle k_\lf^0\rangle^{-40d}\cdot(C^+\delta^{1/4})^{M_{\mathrm{sub}}}(C^+\delta)^{n(\As_{\mathrm{sub}})/2}\bigg(\frac{\delta}{2L^{d-\gamma}}\bigg)^{n_{\mathrm{sub}}}\\\times L^{\Delta_{\mathrm{sub}}}\sum_{\Es_{\mathrm{sub}}}\Bc(t,t,\alpha[\Nc_{\mathrm{sub}}])\Yc(k[\Qc_{\mathrm{sub}}]),
\end{multline} where in (\ref{expressionM}), the sum is taken over all $k$-decorations $\Es_{\mathrm{sub}}$ of $\Qc_{\mathrm{sub}}$ restricted by some fixed $(k_\ell^0)$, $n_{\mathrm{sub}}$ is the order of $\Qc_{\mathrm{sub}}$ and $M_{\mathrm{sub}}=(n_0-n_{\mathrm{sub}})+n(\As)-n(\As_{\mathrm{sub}})$ where $n_0$ is the order of $\Qc_{\mathrm{sk}}$ and $\As$ and $\As_{\mathrm{sub}}$ are as in (\ref{bigformula2new}) and (\ref{vinered}). The function $\Yc$ is uniformly bounded, and is supported on the $k$-decorations of $\Qc_{\mathrm{sub}}$ that satisfy the above gap assumptions for ladders, and are inherited from $k$-decorations of $\Qc_{\mathrm{sk}}$ that satisfy all the SG, LG and non-degeneracy assumptions specified above, due to Proposition \ref{red1step}. Finally, the $\Bc$ expression is defined as
\begin{equation}\label{defintB}\Bc(t,t,\alpha[\Nc_{\mathrm{sub}}]):=\int_{\Ec_{\mathrm{sub}}}\prod_{\nf\in\Nc_{\mathrm{sub}}}e^{\pi i \alpha_\nf t_\nf}\,\mathrm{d}t_\nf
\end{equation} and $\alpha_\nf:=\zeta_\nf\delta L^{2\gamma}\Omega_\nf+\lambda_\nf$. From (\ref{expressionM}), by first fixing the values of $\lfloor\alpha_\nf\rfloor:=\sigma_\nf\in\Zb$ (where each $\sigma_\nf$ belongs to a fixed set of $\leq L^{10d}$ elements) and summing over all choices of $\sigma_\nf$, we can bound
\begin{multline}\label{omegasum}|\Ms|\lesssim\langle k\rangle^{-20d}\prod_{\lf\in\Lc_{\mathrm{sub}}}^{(+)}\langle k_\lf^0\rangle^{-20d}\cdot(C^+\delta^{1/4})^{M_{\mathrm{sub}}}(C^+\delta)^{n(\As_{\mathrm{sub}})/2}\bigg(\frac{\delta}{2L^{d-\gamma}}\bigg)^{n_{\mathrm{sub}}}\\\times L^{\Delta_{\mathrm{sub}}}\sum_{\sigma[\Nc_{\mathrm{sub}}]}\sup_{\alpha[\Nc_{\mathrm{sub}}]:|\alpha_\nf-\sigma_\nf|\leq 1}|\Bc(t,t,\alpha[\Nc_{\mathrm{sub}}])|\cdot\sup_{\sigma[\Nc_{\mathrm{sub}}]}\sum_{\Es_{\mathrm{sub}}}1,
\end{multline} where the first summation in (\ref{omegasum}) is taken over all $\sigma[\Nc_{\mathrm{sub}}]$, and the second summation is taken over all $k$-decorations $\Es_{\mathrm{sub}}$ that satisfies all the above assumptions, as well as $|\delta L^{2\gamma}\Omega_\nf\pm(\sigma_\nf-\lambda_\nf)|\leq 1$ for each node $\nf$.

Now, with fixed $\sigma[\Nc_{\mathrm{sub}}]$, consider the $k$-decoration $\Es_{\mathrm{sub}}$ in the second summation in (\ref{omegasum}), and the corresponding $k$-decoration of $\Mb(\Qc_{\mathrm{sub}})$ defined by Definition \ref{decmole}, then the latter decoration will have to satisfy all the requirements made in the statement of Proposition \ref{kqmainest2} (for some choice of $k_\ell^0$ and $\beta_v$). Thus, by (\ref{kqmainest2-1}), the second summation in (\ref{omegasum}) is bounded by
\begin{equation}\label{2ndsumbound}\sup_{\sigma[\Nc_{\mathrm{sub}}]}\sum_{\Es_{\mathrm{sub}}}1\lesssim(C^+\delta^{-1/2})^{n_{\mathrm{sub}}}L^{(d-\gamma)n_{\mathrm{sub}}}\cdot L^{-\Delta_{\mathrm{sub}}}L^{-\eta^6\rho}\prod_{j=1}^q\Xf_j^{-z_j}.\end{equation} Then, if we can prove that the first summation is bounded by
\begin{equation}\label{1stsumbound}
\sum_{\sigma[\Nc_{\mathrm{sub}}]}\sup_{\alpha[\Nc_{\mathrm{sub}}]:|\alpha_\nf-\sigma_\nf|\leq 1}|\Bc(t,t,\alpha[\Nc_{\mathrm{sub}}])|\lesssim (C^+\delta^{-1/4})^{n_{\mathrm{sub}}}\cdot L^{C\rho\sqrt{\delta}}\prod_{j=1}^q\Xf_j^{z_j}\prod_{\lf\in\Lc_{\mathrm{sub}}}^{(+)}\langle k_\lf^0\rangle^{20d},
\end{equation} then putting together (\ref{omegasum}), (\ref{2ndsumbound}) and (\ref{1stsumbound}) (noticing also equalities like $n=n_0+n(\As)$ etc.). we can get \[|\Ms|\lesssim (C^+\delta^{1/4})^{n}\langle k\rangle^{-20d}L^{-\eta^7\rho},\] which then proves Proposition \ref{kqmainest1}.

Finally let us prove (\ref{1stsumbound}); of course we only need to prove it under the restriction on $(\sigma_\nf)$ that, there exists some $k$-decoration $\Es_{\mathrm{sub}}$ satisfying all the previously specified assumptions, such that the corresponding quantities $\alpha_\nf=\zeta_\nf\delta L^{2\gamma}\Omega_\nf+\lambda_\nf$ satisfies $|\alpha_\nf-\sigma_\nf|\leq 1$. Then, this is basically a consequence of Proposition 10.1 of \cite{DH21}, but with some additional twists. Consider all the ladders $\Lc_j$ such that $1+\delta L^{2\gamma}P_j\leq (\log L)^2$, so $\Xf_j=1+\delta L^{2\gamma}P_j$; assume these correspond to $1\leq j\leq p$ for some $1\leq p\leq q$. For each ladder $\Lc_j\,(j\leq p)$ and each pair of atoms $(v_1,v_2)$ connected by a double bond, we consider them together, as well as the corresponding branching nodes $\nf_j=\nf(v_j)$ in $\Qc_{\mathrm{sub}}$. If the gap of this ladder is $|r_j|\sim P_j$, then it is easy to see that for some choice of $\pm$, we have that $\Omega_{\nf_1}\pm\Omega_{\nf_2}=r_j\cdot(k_{\mf}\pm k_{\mf'})$ for some fixed nodes $\mf$ and $\mf'$ depending on $v_1$ and $v_2$.

By Lemma \ref{explem}, we can decompose the set of all $k$-decorations into at most $C^{n_{\mathrm{sub}}}\prod_{\lf\in\Lc_{\mathrm{sub}}}^{(+)}\langle k_\lf^0\rangle^{10d}$ subsets such that whenever a $k$-decoration belongs to a fixed subset, and whenever $|r_j|\sim P_j$, the value of $r_j\cdot(k_{\mf}\pm k_{\mf'})$ must belong to an fixed interval of length $P_j$, which does not depend on the choice of $r_j$ or the decoration. This implies that $\sigma_{\nf_1}\pm\sigma_{\nf_2}$ belongs to a fixed interval of length $1+\delta L^{2\gamma}P_j$, and thus it has at most $O(1+\delta L^{2\gamma}P_j)$ choices. This means that, after losing a factor
\[C^{n_{\mathrm{sub}}}\prod_{j=1}^p(1+\delta L^{2\gamma}P_j)^{z_j+1}=C^{n_{\mathrm{sub}}}\prod_{j=1}^p\Xf_j^{z_j+1},\] we can assume the value of $\sigma_{\nf_1}\pm\sigma_{\nf_2}$ is fixed for each pair $(\nf_1,\nf_2)$ as above. Note that we may replace the above power $z_j+1$ of $\Xf_j$ by $z_j$, as $\Xf_j\leq (\log L)^2$ and $p\leq \rho_{\mathrm{sub}}$, so the loss $(\log L)^{Cp}$ here can be covered by the gain $L^{-\eta^6\rho_{\mathrm{sub}}}$ in (\ref{kqmainest2-1}).

Notice that the ladders are the same as type II chains defined in \cite{DH21}, we can now run the same arguments in the proof of Proposition 10.1 in \cite{DH21}, to get that
\begin{multline}\label{1stsumbound2}
\sum_{\sigma[\Nc_{\mathrm{sub}}]}\sup_{\alpha[\Nc_{\mathrm{sub}}]:|\alpha_\nf-\sigma_\nf|\leq 1}|\Bc(t,t,\alpha[\Nc_{\mathrm{sub}}])|\\\lesssim C^{n_{\mathrm{sub}}}\prod_{j=1}^p\Xf_j^{z_j}\prod_{\lf\in\Lc_{\mathrm{sub}}}^{(+)}\langle k_\lf^0\rangle^{10d}\cdot(C^+)^{n_{\mathrm{sub}}}\delta^{-n_{\mathrm{sub}}/4}L^{Cp\sqrt{\delta}}(\log L)^{Cp+m''},
\end{multline} where $m''$ is the number of atoms not in the ladders $\Lc_j\,(1\leq j\leq p)$, so in particular $m''=m'+2(z_{p+1}+\cdots +z_q)$; namely each ladder $\Lc_j\,(1\leq j\leq p)$ causes a loss of $L^{C\sqrt{\delta}}(\log L)^C$ and each remaining atom causes a loss of $\log L$, see for example (10.10) in Section 10.2 of \cite{DH21}. It is easily seen that (\ref{1stsumbound2}) is stronger than (\ref{1stsumbound}), using the fact that $\rho=q+m'\geq p+m'$. This completes the proof.
\end{proof}
\section{Reduction to large gap molecules}\label{reduct2}
\subsection{Preliminary setup} We now start the proof of Proposition \ref{kqmainest2}. The idea is to perform operations to further reduce the molecule $\Mb(\Qc_{\mathrm{sub}})$ to simpler molecules (this reduction will be called stage 2, compared to stage 1 reduction in Section \ref{stage1red}). In this process we will keep track of the corresponding counting estimate, as well as certain parameters of the molecule, especially the characteristics $\chi=E-V+F$, where $E$, $V$ and $F$ are the number of bonds, atoms, and components (Definition \ref{defmole0}).

Before introducing the stage 2 reduction operations, we first need to setup the properties of the molecule $\Mb(\Qc_{\mathrm{sub}})$, compared to the original molecule $\Mb(\Qc_{\mathrm{sk}})$ before stage 1 reduction, as well as the corresponding decorations. These are summarized in Proposition \ref{subpro} below.
\begin{prop}\label{subpro} Suppose the molecule $\Mb(\Qc_{\mathrm{sk}})$ is reduced to $\Mb(\Qc_{\mathrm{sub}})$ via stage 1 reduction, as described in Section \ref{stage1red}. Let $\Cs$ be defined for $\Qc_{\mathrm{sk}}$ as in Lemma \ref{vinechainlem}. Consider $k$-decoration $(k_\ell)$ of $\Mb(\Qc_{\mathrm{sub}})$ that is inherited from a $k$-decoration $\Es_{\mathrm{sk}}$ of $\Qc_{\mathrm{sk}}$ (and $\Mb(\Qc_{\mathrm{sk}})$), as in Proposition \ref{kqmainest2}. Then we have the following properties:
\begin{enumerate}[{(a)}]
\item Each SGHVC in $\Cs$ is reduced to an SGHV, and each SGHV remains unchanged. Each SGVC in $\Cs$ is reduced to an SG vine that is also a (CN) vine, or an SG vine that is also a root (CL) vine (see Proposition \ref{block_clcn}), or a single atom. Each DV in $\Cs$ remains unchanged, and still contains two vines (V).
\item The molecule $\Mb(\Qc_{\mathrm{sub}})$ is connected, there are only two atoms of degree $3$, and all other atoms have degree $4$. We define an atom $v$ in $\Mb(\Qc_{\mathrm{sub}})$ to be a \emph{hinge atom}, if it is the atom that results from merging a SGVC in $\Mb(\Qc_{\mathrm{sk}})$. These include the single atoms defined in (a), as well as some of the joints of the SGHV and SG vines defined in (a).
\item Any degenerate atom (where there are two bonds $(\ell_1,\ell_2)$ at $v$ of opposite directions and $k_{\ell_1}=k_{\ell_2}$ in the decoration) must be a hinge atom, and any triple bond must have an endpoint that is a hinge atom.
\item For any hinge atom $v$, we can find two bonds $(\ell_1,\ell_2)$ at $v$, so that in the original molecule $\Mb(\Qc_{\mathrm{sk}})$ they are the two bonds at one joint of an SGVC in $\Cs$ that do \emph{not} belong to this VC (this VC is merged in Stage 1). Also the triple $(v,\ell_1,\ell_2)$ has SG and $k_{\ell_1}\neq k_{\ell_2}$. Moreover, if $v$ has degree $4$, then the other two bonds $(\ell_3,\ell_4)$ at $v$ satisfy the same properties as $(\ell_1,\ell_2)$ above.
\item Any SG vine-like object that is not a subset of an object in (a), which also does \emph{not} contain a hinge atom, must be a subset of a LG vine-like object in $\Cs$. Note that it must be in one of the cases defined in Proposition \ref{subsetvc}.
\item Any SG vine-like object (say $\Ub$) that is not a subset of an object in (a), which also \emph{contains} a hinge atom, \emph{must} contain a hinge atom $v$, such that either $v$ is an interior atom of $\Ub$, or $v$ is a joint of $\Ub$ and exactly one bond in $(\ell_1,\ell_2)$ belongs to $\Ub$, where $(\ell_1,\ell_2)$ is defined as in (d).
\end{enumerate}
\end{prop}
\begin{proof} First, (a) directly follows from the definition of $\Vs_0$ and $\Us_0$ in Section \ref{stage1red} and the definition of splicing operation in Proposition \ref{block_clcn}, using also Corollary \ref{blockchainprop}. Also (b) follows from Proposition \ref{moleproperty} and Remark \ref{moleremark} (note that degree $2$ atoms cannot be generated by splicing due to our treatment for root (CL) vines), and (c) follows from the fact that there is no degenerate atoms nor triple bonds in $\Mb(\Qc_{\mathrm{sk}})$. As for (d), if $v$ is a hinge atom, since $v$ has degree $3$ or $4$, we must have in $\Mb(\Qc_{\mathrm{sk}})$ an SGVC (say $\Vb\Cb$) with one of its joints, say $v_1$, having degree $4$. Let the two bonds $(\ell_1',\ell_2')$ at $v_1$ belong to $\Vb\Cb$ and the other two bonds $(\ell_1,\ell_2)$ at $v_1$ not belong to $\Vb\Cb$, then the triple $(v,\ell_1,\ell_2)$ in $\Mb(\Qc_{\mathrm{sub}})$ is just the triple $(v_1,\ell_1,\ell_2)$ in $\Mb(\Qc_{\mathrm{sk}})$ and must have SG as $\Vb\Cb$ has SG. If $d(v)=4$, then the other joint $v_2$ of $\Vb\Cb$ also has degree $4$ and we can repeat the above argument to show that $(\ell_3,\ell_4)$ satisfy the same properties as $(\ell_1,\ell_2)$.

Now let an SG vine-like object, say $\Ub$, be in $\Mb(\Qc_{\mathrm{sub}})$ which is not a subset of an object in (a). If $\Ub$ does not contain a hinge atom, then the same $\Ub$ must exist in the original molecule $\Mb(\Qc_{\mathrm{sk}})$ as a vine-like object, and is not changed in the process. By Lemma \ref{vinechainlem}, $\Ub$ must be the subset of an object in $\Cs$, and this object must be an LG vine-like object, since otherwise it will be involved in the splicing process and then $\Ub$ would be a subset of an object in (a)--(b). This proves (e).

Next, suppose $\Ub$ contains a hinge atom. If (f) does not hold, then $\Ub$ must contain one hinge atom as a joint, such that the bonds $(\ell_1,\ell_2)$ defined in (d) both belong to $\Ub$ (if not, then replace $(\ell_1,\ell_2)$ by $(\ell_3,\ell_4)$ defined in (d)). The other joint $v_2$ of $\Vb$ may also be a hinge atom; if so then again the bonds $(\ell_1',\ell_2')$ defined in (d) (corresponding to $v_2$) both belong to $\Ub$. Then the same $\Ub$ must exist in the original molecule $\Mb(\Qc_{\mathrm{sk}})$ as an SGVC, and is not changed in the process. Suppose $v_1$ is formed by merging an SGVC called $\Vb\Cb_1$, and $v_2$ (if applicable) is formed by merging an SGVC called $\Vb\Cb_2$. Then in $\Mb(\Qc_{\mathrm{sk}})$ one can concatenate $\Ub$ with $\Vb\Cb_1$ (and $\Vb\Cb_2$ if applicable) to form a larger SG vine-like object, which must be a subset of an SG vine-like object in $\Cs$. This implies that $\Ub$ must be a subset of an object in (a), which is a contradiction and proves (f).
\end{proof}
\begin{rem}\label{subcondition} Consider any $k$-decoration $(k_\ell)$ of $\Mb(\Qc_{\mathrm{sub}})$, restricted by $(\beta_v)$ and $(k_\ell^0)$, that is inherited from a $k$-decoration $\Es_{\mathrm{sk}}$ of $\Qc_{\mathrm{sk}}$ (and $\Mb(\Qc_{\mathrm{sk}})$), as in Proposition \ref{kqmainest2}. Then, the following conditions hold for this decoration: 
\begin{enumerate}
\item[(i)] Those imposed by Proposition \ref{subpro},
\item[(ii)] The gaps $|r_j|\sim P_j$ for each maximal ladder $\Lc_j$, where $P_j$ are fixed as in Proposition \ref{kqmainest2}. 
\item[(iii)] By losing at most $C^{n_{\mathrm{sub}}}$, we may assume that each atom that is not an interior atom of an SG vine-like object in (a) and (e) of Proposition \ref{subpro} is fixed to be either SG or LG. For any SG atom $v$ we fix the corresponding bonds $(\ell_1,\ell_2)$ according to the following rules: for any hinge atom we must fix the bonds as in (d) of Proposition \ref{subpro}, for any joint of any SG vine-like object defined in (a) and (e) of Proposition \ref{subpro} we must fix the bonds $(\ell_1,\ell_2)$ to belong to the corresponding vine-like object (if both vines (V) in one DV are SG, then $(\ell_1, \ell_2)$ is chosen from only one them), and for any endpoint of any triple bond in (c) we must fix the bonds using the endpoint that is a hinge atom. For all other SG atoms we may fix the bonds $(\ell_1, \ell_2)$ arbitrarily. 
\item[(iv)] Finally, for each SG atom $v$, we fix its gap as $|r|\sim R_v$, where $R_v\in[L^{-1},L^{-\gamma+\eta}]$ is a dyadic number. Note that this condition leads to a loss of $(\log L)^p$, where $p$ is the number of SG atoms, which we will also treat in the proof below.
\end{enumerate}
\end{rem}
\subsection{The cutting operation}\label{reductcut0}
Now we define the basic operation in the stage 2 reduction, namely the \emph{cutting} operation, which involves dividing an atom of degree $3$ or $4$ into an atom of degree $2$ and another atom of degree $1$ or $2$.
\begin{df}\label{defcut} Given a molecule $\Mb$ and an atom $v$ of degree $3$ or $4$. Suppose $v$ has two bonds $\ell_1$ and $\ell_2$ of opposite directions, then we may \emph{cut} the atom $v$ along  the bonds $\ell_1$ and $\ell_2$, by replacing $v$ with two atoms $v_1$ and $v_2$, such that the endpoint $v$ for the bonds $\ell_1$ and $\ell_2$ is moved to $v_1$, and the endpoint $v$ for the other bond(s) is moved to $v_2$, see Figure \ref{fig:cut}. We also call this cut an $\alpha$- (resp. $\beta$-) cut, if it does not (resp. does) generate a new connected component, and accordingly we call the resulting atoms $\alpha$- or $\beta$- atoms. If we are also given a decoration, then we may define the gap of this cut to be $r:=k_{\ell_1}-k_{\ell_2}$.
\begin{figure}[h!]
\includegraphics[scale=0.45]{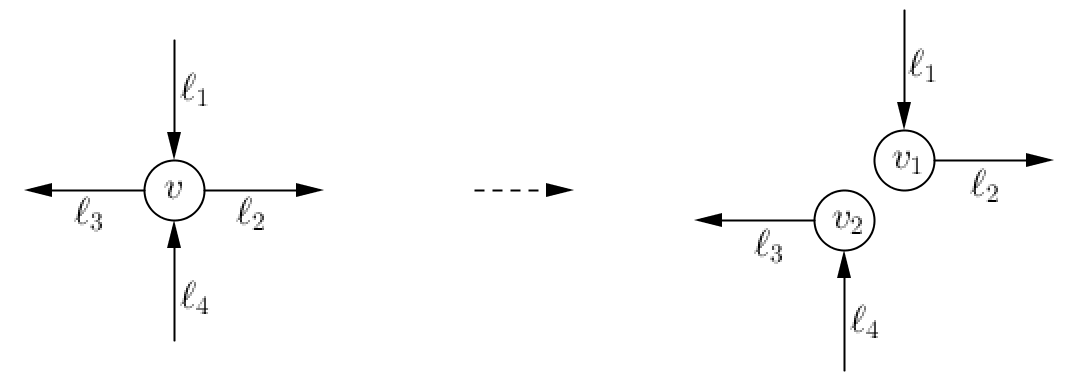}
\caption{A cutting operation executed at a degree $4$ atom $v$, see Definition \ref{defcut}.}
\label{fig:cut}
\end{figure}
\end{df}
The stage 2 reduction, which will be applied to $\Mb(\Qc_{\mathrm{sub}})$ involves cutting various atoms as in Definition \ref{defcut}, and possibly removing some connected components created in this process, until reaching a final molecule $\Mb_{\mathrm{fin}}$. At each step, let the molecule before and after the operation be $\Mb_{\mathrm{pre}}$ and $\Mb_{\mathrm{pos}}$, then a decoration $(k_\ell)$ of $\Mb_{\mathrm{pre}}$ naturally leads to a decoration of $\Mb_{\mathrm{pos}}$. 

For any molecule $\Mb$ (which could be $\Mb_{\mathrm{pre}}$ or $\Mb_{\mathrm{pos}}$), consider the possible $(c_v)$-decorations of $\Mb$, also restricted by $(\beta_v)$ and $(k_\ell^0)$; we also assume that this decoration is inherited from a $k$-decoration of $\Mb(\Qc_{\mathrm{sub}})$ that satisfies all assumptions in Remark \ref{subcondition}. Then we consider the number of such restricted decorations, take supremum over the parameters $(c_v,\beta_v,k_\ell^0)$, and define it to be $\Cf$. In view of the right hand side of (\ref{kqmainest2-1}) and the logarithmic loss in Remark \ref{subcondition}, we also define
\begin{equation}\label{defequan}\Af:=\Cf\cdot L^{-(d-\gamma)\chi(\Mb)}(C^+\delta^{-1/2})^{-\chi(\Mb)}\prod_{j=1}^q\Xf_j^{z_j}\cdot (\log L)^{Cp},
\end{equation} where $\chi(\Mb)$ is the characteristic of $\Mb$, $p$ is the number of remaining SG atoms in $\Mb$, and the other notations are under the setting of Proposition \ref{kqmainest2} but adapted to $\Mb$ instead of $\Mb(\Qc_{\mathrm{sub}})$, for example $\Lc_j$ are the maximal ladders in $\Mb$. Define also $m'$ to be the number of atoms not in the maximal ladders, and $\rho=q+m'$ as in Proposition \ref{kqmainest2} before, and denote the product $\prod_{j=1}^q\Xf_j^{z_j}$ in (\ref{defequan}) as $\Pf$. The notations $(\Af,\Cf,\Pf)$ and $(\rho,q,m')$ will apply to all the molecules appearing in the rest of this paper, with possible subscripts matching those of $\Mb$ (so $\Af_{\mathrm{pre}}$ is defined for the molecule $\Mb_{\mathrm{pre}}$ etc.; sometimes for emphasis we may also write $\Af(\Mb)$ or $\Af(\Mb_{\mathrm{pre}})$ instead of $\Af$ and $\Af_{\mathrm{pre}}$).

We will prove, for each operation, an inequality of form \begin{equation}\label{defdev}\Af_{\mathrm{pre}}\lesssim \Df\cdot \Af_{\mathrm{pos}}\end{equation} for some quantity $\Df$, which we define to be the \emph{deviation} of this operation. Clearly, if we know the deviation in each operation step, and an upper bound for the counting problem associated with the final molecule $\Mb_{\mathrm{fin}}$, then we can deduce from this information an upper bound for the counting problem associated with $\Mb(\Qc_{\mathrm{sub}})$ using (\ref{defequan}) and (\ref{defdev}).

In Section \ref{reductcut} we define the stage 2 reduction process, and study the deviation in each step; in Section \ref{largegapmole} and Section \ref{lgmole} we solve the counting problem associated with the final molecule $\Mb_{\mathrm{fin}}$, and completes the proof of Proposition \ref{kqmainest2}.
\subsection{Stage 2 reduction: Cutting degree $3$ and $4$ atoms}\label{reductcut} Start with the molecule $\Mb(\Qc_{\mathrm{sub}})$ and a restricted decoration $(k_\ell)$ satisfying the assumptions in Remark \ref{subcondition}. Recall the notion of $V,E,F$ and $\chi$ as in Definition \ref{defmole0}; define also $V_\alpha$ and $V_\beta$ to be the number of $\alpha$- and $\beta$-atoms, and use $\Delta$ to denote increments.

\emph{\underline{Step 1: removing SG vine-like objects.}} In Step 1, we collect all the SG vine-like objects defined in (a) and (e) of Proposition \ref{subpro}, including also the triple bonds in (c). Note that if both vines (V) in one DV in (a) are SG then we only collect one of them. For each of these objects $\Ub$, we cut the molecule at each joint, along the two bonds $(\ell_1,\ell_2)$ fixed as in Remark \ref{subcondition} (i.e. along the two bonds that belong to $\Ub$). This disconnects a VC, say $\Vb\Cb$ (which is $\Ub$ or $\Ub$ minus a bond in case of HV or HVC), from the rest of the molecule, and then we remove $\Vb\Cb$. The two joints then become degree $1$ or $2$ atoms, and we \emph{define them as $\alpha$-atoms}; we also label each of them by the dyadic number $R\in[L^{-1},L^{-\gamma+\eta}]$ such that $|r|\sim R$ for the gap $r$ of $\Vb\Cb$ (note that we must have $r\neq 0$).

\emph{\underline{Step 2: removing triangles.}} In Step 2, we consider the possible triangles $v_1v_2v_3$ in the molecule, such that there are bonds $\ell_j$ connecting $v_{j+1}$ and $v_{j+2}$ (where $v_4=v_1$ and $\ell_4=\ell_1$ etc.), and $(v_j,\ell_{j+1},\ell_{j+2})$ is an SG triple as fixed in Remark \ref{subcondition} for $j\in\{1,2\}$. Let $|k_{\ell_{j+1}}-k_{\ell_{j+2}}|\sim R_{j}$ for $j\in\{1,2\}$ with $R_j\in[L^{-1},L^{-\gamma+\eta}]$. If $\mathrm{deg }\,v_j \geq 3$, then we cut the molecule at each $v_j$ along the bonds $(\ell_{j+1},\ell_{j+2})$, which disconnects the triangle formed by $v_j$ and $\ell_j$ from the rest of the molecule, and remove the triangle. This leaves $3$ atoms $v_j$ of degree $1$ or $2$. We call $v_3$ a $\beta$-atom, and call $v_j\,(1\leq j\leq 2)$ an $\alpha$- (resp. $\beta$-) atom if it belongs to the same (resp. different) component with $v_3$, except when $v_1$ and $v_2$ are in the same component different from $v_3$, in which case we call $v_1$ an $\alpha$-atom and $v_2$ a $\beta$-atom. For any $\alpha$-atom $v_j$ we label it by the corresponding $R_j$.

\emph{\underline{Step 3: remaining SG cuts.}} In Step 3, we select each of the remaining SG atoms $v$ of degree $\geq 3$, and cut them along the designated bonds $(\ell_1,\ell_2)$ in Remark \ref{subcondition}, in arbitrary order. Note that each cut may be $\alpha$- or $\beta$-cut; we call the resulting atoms $\alpha$- or $\beta$-atoms, and label any $\alpha$-atom by the dyadic number $R\in[L^{-1},L^{-\gamma+\eta}]$, as in Definition \ref{defcut} (again $R\neq 0$ due to our choice of $(\ell_1,\ell_2)$).

\emph{\underline{Step 4: remaining $\beta$-cuts.}} After Step 3, there is now no more SG atoms left in the molecule. In Step 4, we look for all the possible degree $3$ or $4$ atoms where a $\beta$-cut is possible, and perform the corresponding $\beta$-cut, until this can no longer be done. 

\smallskip
After all the cutting operations, the resulting graph will contain some $\alpha$-atoms. For each $\alpha$-atom $v$ and a given decoration we define an auxiliary number $\iota_v\in\{0,1\}$, such that $\iota_v=1$ if a cutting operation happened before the cutting at the atom $v$, such that the gaps $r,r'$ of these cuttings satisfy $0<|r\pm r'|\leq L^{-50\eta}R_v$; if no such cutting operation exists then define $\iota_v=0$. Note that the number of choices for all $(\iota_v)$ is at most $C^p$, which can be absorbed into the last factor on the right hand side of (\ref{defequan}). Therefore we may assume a choice of $(\iota_v)$ is fixed in the proof below.
\begin{prop}\label{excessprop1} After each operation in Steps 1--4, we have that \begin{equation}\label{excesseqn2}\Df\lesssim L^{-\eta^3}\cdot\prod_{v}^{(\alpha)}L^{(\gamma+3\eta)/2-\kappa_v\eta}R_v\cdot L^{(2\gamma_0+5\eta^2)\Delta F-(\gamma_0+2\eta^2)\Delta V_\beta}
\end{equation} for any $R_v\in[L^{-1},L^{-\gamma+\eta}]$, where the product is taken over all the newly created $\alpha$-atoms $v$, and $R_v$ denotes the label of $v$. The parameter $\kappa_v$ equals $0$ if $\iota_v=0$, and equals $50$ if $\iota_v\neq 0$.
\end{prop}
\begin{proof} First assume $\iota_v=0$ for all $\alpha$-atoms $v$. We start with the simplest case (in Steps 3--4) where we cut at a single atom $v$. Note that, if this is an $\alpha$- (resp. $\beta$-) cut then we have $(\Delta\chi,\Delta F,\Delta V_\beta)=(-1,0,0)$ (resp. $(0,1,2)$). In the case of $\beta$-cut at $v$ along $(\ell_1,\ell_2)$, the value of $k_{\ell_1}-k_{\ell_2}$ is uniquely fixed (by summing the equation (\ref{decmole1}) over all atoms $v'\neq v$ that belongs to one component after making the cut). Similarly, using (\ref{defomegadec}), we know that the number of choices of $|k_{\ell_1}|^2-|k_{\ell_2}|^2$, up to distance $\delta^{-1}L^{-2\gamma}$, is $\lesssim n_{\mathrm{sub}}\leq (\log L)^{C}$. This means that $\Cf_{\mathrm{pre}}\lesssim (\log L)^C\Cf_{\mathrm{pos}}$, hence $\Df\lesssim(\log L)^{C}$ using (\ref{defequan}) and (\ref{defdev}), and noticing that the cut only changes the length of at most one maximal ladder by $O(1)$, using also $1\leq \Xf_j\leq (\log L)^2$. Therefore (\ref{excesseqn2}) is true in this case. As for $\alpha$-cuts, let the corresponding gap $R$ be fixed, then the number of choices for $k_{\ell_1}-k_{\ell_2}$ is $\lesssim (RL)^d$; once $k_{\ell_1}-k_{\ell_2}$ is fixed (and $|k_{\ell_1}-k_{\ell_2}|\sim R$), the number of choices of $|k_{\ell_1}|^2-|k_{\ell_2}|^2$, up to distance $\delta^{-1}L^{-2\gamma}$, is $\lesssim 1+\delta RL^{2\gamma}$ because $k_{\ell_1}+k_{\ell_2}$ belongs to a fixed ball of radius $\sim 1$. This implies that
\[\Df\lesssim (RL)^d(1+RL^{2\gamma})L^{-(d-\gamma)}(\log L)^C\lesssim L^{-\eta^3}\cdot L^{\gamma+3\eta}R^2,\] where the $(\log L)^C$ factor is similar to above, and the last inequality can be verified using $R\in[L^{-1},L^{-\gamma+\eta}]$ and $d\geq 3$. Since both new $\alpha$-atoms are labeled by $R$, this proves (\ref{excesseqn2}).

Next consider Step 2, for which we know $\Delta F\in\{0,1,2\}$, $\Delta V_\beta=\Delta F+1$, and $\Delta\chi=\Delta F-3$; moreover $\Df\lesssim L^{(d-\gamma)\Delta\chi}(\log L)^C\Ff$, where $\Ff$ is the number of choices for $(k_{\ell_1},k_{\ell_2},k_{\ell_3})$. If $\Delta F=2$, then similar to above, for \emph{each} $j$ we know that 
\begin{multline}\label{triangle1}k_{\ell_j}-k_{\ell_{j+1}}\textrm{\ is uniquely fixed,\ and\ the\ number\ of\ choices}\\\textrm{for\ }|k_{\ell_j}|^2-|k_{\ell_{j+1}}|^2,\textrm{\ up\ to\ distance\ }\delta^{-1}L^{-2\gamma},\textrm{\ is\ }\lesssim(\log L)^{C}.\end{multline} By Lemma \ref{basiccount} (1) (where we also use that $R\geq L^{-1}$), we have
\[\Df\lesssim \delta^{-1}\min(L^{d-2\gamma+1},L^d)\cdot L^{-(d-\gamma)}(\log L)^C\lesssim L^{-\eta^3}\cdot L^{\gamma_0+4\eta^2}\] using that $\gamma_0=\min(\gamma,1-\gamma)$, which proves (\ref{excesseqn2}). If $\Delta F=1$, then (\ref{triangle1}) still holds for \emph{some} $j$. Moreover, once $(k_{\ell_j},k_{\ell_{j+1}})$ is fixed, the number of choices for $k_{\ell_{j+2}}$ is at most $(RL)^d$ where $R$ is the label of the unique $\alpha$-atom. This implies that
\[\Df\lesssim\delta^{-1} \min(L^{d-2\gamma+1},L^d)\cdot L^{-2(d-\gamma)}(RL)^d(\log L)^C\lesssim L^{-\eta^3}\cdot L^{\eta^2}L^{(\gamma+3\eta)/2}R\] using that $R\in[L^{-1},L^{-\gamma+\eta}]$ and $d\geq 3$, so (\ref{excesseqn2}) is again true. Now if $\Delta F=0$, then simply by using the gap assumptions we have $\Ff\lesssim L^d(R_1L)^d(R_2L)^d$, hence
\[\Df\lesssim L^{3d}(R_1R_2)^dL^{-3(d-\gamma)}(\log L)^C\lesssim L^{-\eta^3}\cdot L^{-\gamma_0-2\eta^2}L^{\gamma+3\eta}R_1R_2,\] again using that $R_1,R_2\in[L^{-1},L^{-\gamma+\eta}]$ and $d\geq 3$, which proves (\ref{excesseqn2}).

Now consider Step 1. By Proposition \ref{subsetvc}, we know that $\Vb\Cb$ is either a single vine, or is formed by two double bonds. First assume it is a single vine. For the operation we have $\Delta F=\Delta V_\beta=0$ by combining Proposition \ref{block_clcn}, Proposition \ref{subpro} (a) (e), and Proposition \ref{subsetvc} (and an easy verification for DV). If $\Vb\Cb$ contains $m$ atoms (including the joints) then it is easy to verify that $\Delta\chi=-m$; let the the number of choices for $(k_\ell)$ for all bonds $\ell\in\Vb\Cb$ be $\Ff$, then
\begin{equation}\label{fbyg}\Df\lesssim \Ff\cdot L^{-m(d-\gamma)}\cdot (\log L)^C\delta^{m/2}\prod_{1\leq j\leq q:\Lc_j\subset\Vb\Cb}\Xf_j^{z_j},\end{equation} where the product is taken over all $j$ such that the maximal ladder $\Lc_j$ (see Proposition \ref{kqmainest2}) is a subset of $\Vb\Cb$ (so there are at most three such $j$, see Figure \ref{fig:vines}), and $z_j$ is the length of $\Lc_j$.

To calculate $\Ff$, first fix $x_0=k_{\ell_1}$ and $y_0=k_{\ell_2}$, where $\ell_1$ and $\ell_2$ are the two bonds at one joint of $\Vb\Cb$ that belong to $\Vb\Cb$. These have $L^d(RL)^d$ choices since the gap of $\Vb\Cb$ is $|r|\sim R$. Then, we invoke the reparametrization introduced in the proof of Propositions \ref{estbadvine}--\ref{estnormalvine}, and define the new variables $(x_j,y_j)\,(1\leq j\leq m_1)$ if $\Vb\Cb$ is bad vine, or $(x_j,y_j)\,(1\leq j\leq \widetilde{m})$ and $(u_1,u_2,u_3)$ if $\Vb\Cb$ is normal vine, where $m=2\widetilde{m}+2$ for bad vine and $m=2\widetilde{m}+5$ for normal vine. In either case, since each $\Gamma_v$ belongs to a fixed interval of length $\delta^{-1}L^{-2\gamma}$, we know that each $(x_j,y_j)$ satisfies a system of form (\ref{basiccount01}), and $(u_1,u_2,u_3)$ satisfies a system of form (\ref{basiccount02}) with $r=x_0-y_0$ or $r=0$ (so $P=R$ or $P=0$) and $v=0$. Therefore, by applying Lemma \ref{basiccount0}, we get in either case that
\[\Ff\lesssim L^d(RL)^dL^{(m-2)(d-\gamma)}\delta^{-m/2}(\log L)^C\prod_{1\leq j\leq q:\Lc_j\subset\Vb\Cb}\Xf_j^{-z_j},\] and hence $\Df\lesssim R^dL^{2\gamma}(\log L)^C\lesssim L^{-\eta^3}\cdot R^2L^{\gamma+3\eta}$, so (\ref{excesseqn2}) is true. Finally, if $\Vb\Cb$ is formed by two double bonds, then $\Delta\chi=-3$. We again first fix $(x_0,y_0)$, and then apply Lemma \ref{basiccount} (1) to get \[\Df\lesssim L^{-3(d-\gamma)}L^d(RL)^d(L^{d-1}+\delta^{-1}\min(R^{-1}L^{d-2\gamma},L^d))(\log L)^C\lesssim L^{-\eta^3}\cdot R^2L^{\gamma+3\eta},\] using that $R\in[L^{-1},L^{-\gamma+\eta}]$ and $d\geq 3$.

In the case where $\iota_v\neq 0$ for some $\alpha$-atom $v$, it is clear that the corresponding gap $r$ must belong to a fixed ball of radius $L^{-50\eta}R$ where $|r|\sim R$, and the number of possible choices of such balls is bounded by $(\log L)^C$. As such, the factor $(RL)^d$ in the above proof, which indicates the number of choices for this gap $r$, is replaced by $(\log L )^C(RL^{1-50\eta})^d$, and the other factors remain the same, therefore (\ref{excesseqn2}) is true with the improved $\kappa_v$.
\end{proof}
\subsection{The molecule $\Mb_{\mathrm{fin}}$}\label{largegapmole} Let the result of stage 2 reduction be $\Mb_{\mathrm{fin}}$. If it is not connected, let $\Mb_1$ be any of its components. Consider also a decoration of $\Mb_{\mathrm{fin}}$ inherited from a decoration $(k_\ell)$ of $\Mb(\Qc_{\mathrm{sub}})$ as in Remark \ref{subcondition}. Then they have the following properties:
\begin{prop}\label{finalmole} There is only one component, which we call the \emph{odd component}, that contain two atoms of degree $1$ and $3$. All the other atoms have degree $2$ and $4$. The atoms of degree $1$ and $2$ are classified as $\alpha$-, or $\beta$-atoms. Each $\alpha$-atom $v$ is labeled by a dyadic number $R_v\in [L^{-1},L^{-\gamma+\eta}]$, such that if $v$ has two bonds $(\ell_1,\ell_2)$ then $|k_{\ell_1}-k_{\ell_2}|\sim R_v$ in the decoration; recall also $\iota_v$ and $\kappa_v$ introduced in Section \ref{reductcut} and Proposition \ref{excessprop1}. Atoms that are neither $\alpha$- nor $\beta$-atoms are called \emph{$\varepsilon$-atoms}; any $\varepsilon$-atom in $\Mb_{\mathrm{fin}}$ \emph{must} be $LG$ in the decoration. Moreover $\Mb_{\mathrm{fin}}$ contains no triple bond.

Any even (i.e. non-odd) component has at least one $\beta$-atom. If an even component $\Mb_0$ is a cycle, then it is either a double bond (which is also vine (I)), or a cycle of length at least $4$, or a triangle with at most one $\alpha$-atom. If $\Mb_0$ is not a cycle, then all its $\alpha$- and $\beta$-atoms form several disjoint chains, such that each chain has two distinct $\varepsilon$-atom at both ends. Finally, if any component $\Mb_0$ is a \emph{vine} (with two joints having degree $2$), then it \emph{must} be LG in the decoration.
\label{proplg}
\end{prop}
\begin{proof} The total number of odd degree atoms is not changed under the cutting operation and removing of isolated components (which only contain even degree atoms). This value is $2$ initially (Proposition \ref{subpro}), so it will remain $2$. The two odd degree atoms have to be in the same component as any component must have an even number of odd degree atoms. The degree $1$ and $2$ atoms are classified as $\alpha$-, or $\beta$-atoms, and $(R_v,\iota_v,\kappa_v)$ for $\alpha$-atoms $v$ are defined as before. $\Mb_{\mathrm{fin}}$ contains no triple bond because any triple bond is destroyed in Step 1.

Now consider an even component $\Mb_0$ with only degree $2$ and $4$ atoms. Such a component can only be formed after a $\beta$-cut, so it will contain at least one $\beta$-atom. If it is a cycle, then it is either a double bond, or a triangle, or has at least length $4$. If it is a triangle, then it cannot have at least two $\alpha$-atoms, since any $\alpha$-atom must have SG in the original molecule $\Mb(\Qc_{\mathrm{sub}})$, and any triangle with at least two SG atoms will be removed in Step 2. If $\Mb_0$ is not a cycle, then it has at least one degree $4$ (i.e. $\varepsilon$-) atom. For any degree $2$ atom, consider the longest chain of degree $2$ atoms containing it, which ends at two degree $4$ atoms $v_1$ and $v_2$ in both directions; they cannot coincide, otherwise we can perform a $\beta$-cut at this common atom according to Step 4 above.

Finally, each $\varepsilon$-atom must have LG by definition, and $|k_{\ell_1}-k_{\ell_2}|\sim R_v$ for any $\alpha$-atom $v$ of degree $2$ labeled by $R_v$. If any component $\Mb_1$ is an SG vine $\Vb$, then this vine $\Vb$ must exist in the original molecule $\Mb(\Qc_{\mathrm{sub}})$, and is not changed in the process. But the SG vine $\Vb$ cannot be any vine occurring in (a) or (e) of Proposition \ref{subpro} (or any triple bond in (c)), because then it would be removed in Step 1. Therefore, $\Vb$ has to be a vine occurring in (f) of Proposition \ref{subpro}. However, since our reduction involves cutting the molecule at the hinge atom along the two bonds $(\ell_1,\ell_2)$, by (f) of Proposition \ref{subpro}, the vine $\Vb$ cannot remain intact after this cutting operation. This completes the proof.
\end{proof}
With all the above preparations, we can now reduce Proposition \ref{kqmainest2} to the following
\begin{prop}
\label{lgmolect} For each component $\Mb_0$ of the final molecule $\Mb_{\mathrm{fin}}$, define $(\Cf_{0},\Af_0,\Pf_0)$ and $\rho_0$ as in Section \ref{reductcut0} and (\ref{defequan}), but associated with $\Mb_0$ (here $p$ should be replaced by $0$ in (\ref{defequan}) as there is no more SG atoms in $\Mb_0$). Let also each $\alpha$-atom $v$ be labeled by the dyadic number $R_v$, then we have
\begin{equation}\label{finalcount}\Af_0\lesssim \prod_v^{(\alpha)}L^{-(\gamma+3\eta)/2+\kappa_v\eta}R_v^{-1}\cdot L^{(\gamma_0+2\eta^2)V_\beta-(2\gamma_0+5\eta^2)G}\cdot L^{-\eta^5\rho_0},
\end{equation} where the product is taken over all $\alpha$-atoms $v$, and $V_\beta$ is the number of $\beta$-atoms in $\Mb_0$; moreover $G$ is $0$ or $1$ depending on whether $\Mb_0$ is odd or even component.
\end{prop}
\begin{proof}[Proof of Proposition \ref{kqmainest2} assuming Proposition \ref{lgmolect}] Since (\ref{finalcount}) holds for each component $\Mb_0$, clearly it also holds for the union $\Mb_{\mathrm{fin}}$, if all the expressions and quantities on the right hand side are replaced by the ones corresponding to $\Mb_{\mathrm{fin}}$ (and $G$ replaced by $F-1$ where $F$ is the number of components of $\Mb_{\mathrm{fin}}$). Using (\ref{finalcount}) and the deviation bound (\ref{excesseqn2}) for each operation, we get that
\begin{equation}\label{finalcount2}\Af_{\mathrm{sub}}\leq (C^+)^{n_{\mathrm{sub}}}\cdot L^{-\eta^5(\sigma+\rho_{\mathrm{fin}})-\eta^3\sigma/2},\end{equation} where $\Af_{\mathrm{sub}}$ and $\rho_{\mathrm{fin}}$ are defined as before, and $\sigma$ is the total number of operations that are done in Steps 1--4. It is easy to see that $\Delta_{\mathrm{sub}}\leq C\eta^4\sigma$ for the value $\Delta_{\mathrm{sub}}$ in Proposition \ref{kqmainest2}. This is because each SGVC described in Proposition \ref{subpro} (a) leads to a hinge atom as in Proposition \ref{subpro} (b) at which we perform a cut, so the number of these SGVC is at most $\sigma$, while by the definition of $\Delta_j$ in Proposition \ref{red1step} the contribution of each such SGVC to $\Delta_{\mathrm{sub}}$ is at most $\eta^4$ (or negative). Therefore, the term $L^{\eta^3\sigma/2}$ in (\ref{finalcount2}) takes care of the $L^{-\Delta_{\mathrm{sub}}}$ term in (\ref{kqmainest2-1}).

Comparing (\ref{finalcount2}) with (\ref{kqmainest2-1}), using also the definition (\ref{defequan}) and noticing the logarithmic loss in Remark \ref{subcondition}, it now suffices to prove that $\rho_{\mathrm{sub}}\leq C(\sigma+\rho_{\mathrm{fin}})$ for the quantity $\rho$ defined in Proposition \ref{kqmainest2}. However, the value of $\rho_{\mathrm{sub}}$ becomes $\rho_{\mathrm{fin}}$ after all the operations, and each operation changes the value of $\rho$ by at most $O(1)$ because each vine contains at most $O(1)$ maximal ladders and at most $O(1)$ atoms apart from these maximal ladders (which is clear from Figure \ref{fig:vines}), so it is clear that $|\rho_{\mathrm{sub}}-\rho_{\mathrm{fin}}|\leq C\sigma$, as desired.
\end{proof}
\section{Counting problem for large gap molecules}\label{lgmole}
\subsection{Preliminary setup}\label{lgmole1} We now start the proof of Proposition \ref{lgmolect}. We will first get rid of the expression on the right hand side of (\ref{finalcount}), and reduce Proposition \ref{lgmolect} to the following Propositions \ref{moleprop1}--\ref{moleprop4}. In these propositions, we always consider a connected molecule or pseudomolecule $\Mb$ (cf. Definition \ref{defmole0}), and a $(c_v)$-decoration $(k_\ell)$ of $\Mb$ which is also restricted by $(\beta_v)$ and $(k_\ell^0)$; however, \emph{in Propositions \ref{moleprop1}--\ref{moleprop4} only}, we will relax the definition of decorations by \emph{not} requiring $c_v=0$ for degree $4$ atoms $v$ as in Definition \ref{decmole}. For each atom $v\in\Mb$ we also assume the decoration is LG at $v$ (i.e. for any bonds $(\ell_1,\ell_2)$ of opposite directions at $v$ we have $|k_{\ell_1}-k_{\ell_2}|\geq L^{-\gamma+\eta}$).

Let $(V,E,\chi)$ etc. be associated with $\Mb$, and let $\Cf$ and $(\rho,q,m')$ be defined for $\Mb$ as before. Note that, for a given ladder, the differences $k_{\ell}-k_{\ell'}$ for different pairs $(\ell,\ell')$ of parallel single bonds may not be the same as in Definition \ref{defdiff}, due to the relaxation of the assumption $c_v=0$ for degree $4$ atoms $v$. Therefore a ladder of length $z_j\geq 1$ will have $z_j$ different gaps $|r_{ji}|\sim P_{ji}\,(1\leq i\leq z_j)$; consequently we define the quantity $\Af$ as in (\ref{defequan}) but with $p$ replaced by $0$ and the $\Xf_j^{z_j}$ factor in $\Pf$ replaced by $\prod_{i=1}^{z_j}\min((\log L)^2,1+\delta L^{2\gamma}P_{ji})$ where $|r_{ji}|\sim P_{ji}$ for the $i$-th gap of the ladder $\Lc_j$ with the notations of Proposition \ref{kqmainest2}. Note that we may drop any $C^n$ factor below, since they can be absorbed into the definition of $\Af$.
\begin{prop}\label{moleprop1} If $\Mb$ is a molecule as above, then we always have $\Af\lesssim 1$.
\end{prop}
\begin{prop}\label{moleprop2}  If $\Mb$ is a molecule that contains no triple bond and $E=2V-1$, then $\Af\lesssim L^{-3\gamma_0/5}$; if $\Mb$ contains no triple bond and $E=2V-2>0$, then $\Af\lesssim L^{-\eta/3}$. Furthermore, if the number of atoms in $\Mb$ not of degree $4$ is $w$, \emph{and we also allow at most $w$ atoms to be SG in the decoration}, then we have $\Af\lesssim L^{-\eta^2\cdot\rho+C(w+1)}$.
\end{prop}
\begin{prop}\label{moleprop3} If $\Mb$ is a molecule that contains only single bonds, and has two degree $3$ atoms with all other atoms having degree $4$, then $\Af\lesssim L^{-\gamma_0-\eta/2}$.
\end{prop}
\begin{prop}\label{moleprop4} (1) If $\Mb$ is molecule that has at most one triple bond and $E=2V-1$, and one cannot make a $\beta$-cut in $\Mb$ such that one of the new components has all atoms of degree $4$ except the newly formed $\beta$-atom which has degree $2$, and $\Mb$ is not formed by removing the two joints (with their bonds) of a vine (II) and adding one new bond between the two atoms connected to one of the joints, then $\Af\lesssim L^{-4\gamma_0/7}$.

(2) Suppose $\Mb$ is $4$-regular pseudomolecule that contains at most two triple bonds, and is \emph{not} formed by removing the two joints of a vine and adding one new bond between the pair of atoms connected to each joint. Fix any bond $\ell$ in $\Mb$, then either $\Af(\Mb)\lesssim L^{-\eta/4}$ for $\Mb$, or $\Af(\Mb\backslash\{\ell\})\lesssim L^{-\gamma_0-\eta/4}$ for the molecule after removing $\ell$ (note that the exact case may depend on assumptions we impose on the decoration, which will be made clear in the proof).
\end{prop}
We start with Proposition \ref{moleprop1} as it is the simplest, and some ingredients in its proof will be reused later.
\begin{proof}[Proof of Proposition \ref{moleprop1}] We will reduce the molecule $\Mb$ to the empty set by the following operations: \underline{Operation (a)} which consists of removing one atom (and all the bonds), and \underline{Operation (b)} which consists of removing two atoms $v_1$ of degree 3, and $v_2$ of degree 3 or 4, that are connected by a double bond, and all the bonds attached to them. Note that $\Mb$ is a molecule which does not have components of only degree $4$ atoms, so we can always assume that the removed atom has degree $\leq 3$ in any operation (a); assume also that in the whole process, we perform operation (a) only when (b) is not possible.

For each operation we will consider the deviation, i.e. the quantity $\Df$ such as $\Af_{\mathrm{pre}}\lesssim\Df\cdot\Af_{\mathrm{pos}}$ as in (\ref{defdev}), where the meaning of the quantities should be obvious. Note that each operation (a) does not affect any ladder of length at least one (since the removed atom cannot belong to such a ladder, or otherwise we should perform operation (b)), and hence does not affect the factor $\Pf$. Moreover, if the removed atom $v$ has degree $r\in\{1,2,3\}$, then $0\leq \Delta F\leq r-1$ and $\Delta\chi=\Delta F-r+1$. By Lemma \ref{basiccount} (1)--(2), we see that
\[\Df\lesssim\delta^{-1}L^{2(d-\gamma)}\cdot L^{-2(d-\gamma)}(C^+\delta^{-1/2})^{-2}\lesssim1\] if $\Delta\chi=-2$ (so $(r,\Delta F)=(3,0)$), and that
\[\Df\lesssim\delta^{-1}L^{d-\gamma-\eta}\cdot L^{-(d-\gamma)}\lesssim L^{-\eta/2}\] if $\Delta\chi=-1$ (so $(r,\Delta F)\in\{(2,0),(3,1)\}$), and $\Df\lesssim1$ if $\Delta\chi=0$ (so $(r,\Delta F)\in\{(1,0),(2,1),(3,2)\}$). Here we have also used the fact that $R\gtrsim L^{-\gamma+\eta}$ in Lemma \ref{basiccount} (1) due to LG assumption; moreover if $\Delta F\geq 1$ then one of the values $k_\ell$ for bonds $\ell$ at $v$ will be uniquely fixed (if $\Delta F\geq 2$ then all $k_\ell$ will be uniquely fixed).

Now consider operation (b). Here, we have two cases. Either the operation does not affect any ladder of length at least one, or it can reduce the length of one such ladder by one. Now, in the former case we clearly has $\Df\lesssim 1$ as the operation (b) can be split into two operations (a). In the latter case, the operation removes a factor $\min((\log L)^2,1+\delta L^{2\gamma}P)$ from the product $\Pf$, where $P\sim|k_{\ell}-k_{\ell'}|$ for the two single bonds $(\ell,\ell')$ at the two removed atoms. If both atoms have degree 3, then we have $\Delta\chi=-2$, and by Lemma \ref{basiccount} (3) we have
\begin{multline}\label{laddersharp}\Df\lesssim \delta^{-1}L^{2(d-\gamma)}\cdot\max(\delta^{-1}L^{-\gamma_0},(\log L)^{-2},(1+\delta L^{2\gamma}P)^{-1})\\\times L^{-2(d-\gamma)}(C^+\delta^{-1/2})^{-2}\cdot \min((\log L)^2,1+\delta L^{2\gamma}P)\lesssim 1.\end{multline}
If $v_2$ has degree 4, then either $\Delta\chi=-2$ and we have the same bound as in (\ref{laddersharp}), or $\Delta \chi=-3$ and a better bound holds using Lemma \ref{basiccount} (4). Therefore, in all cases we will have $\Af_{\mathrm{pre}}\lesssim\Af_{\mathrm{pos}}$. By choosing the constant $C^+$ in (\ref{defequan}) large enough we can assume $\Af_{\mathrm{pre}}\leq\Af_{\mathrm{pos}}$; this implies that we must have $\Af\leq 1$ in the beginning, since it trivially holds in the end when $\Mb$ has no bonds. This completes the proof.
\end{proof}
In the rest of this subsection we prove Proposition \ref{lgmolect} assuming Propositions \ref{moleprop2}--\ref{moleprop4}.
\begin{proof}[Proof of Proposition \ref{lgmolect} assuming Propositions \ref{moleprop2}--\ref{moleprop4}] Let $\Mb_0$ be a component of $\Mb_{\mathrm{fin}}$ as in Proposition \ref{lgmolect}. If $\Mb_0$ is a double bond, then it must be LG and both atoms must be $\beta$-atoms, due to Proposition \ref{finalmole}. In this case we have $\Af_0\lesssim \delta^{-1}L^{d-\gamma-\eta}\cdot L^{-(d-\gamma)}\lesssim L^{-\eta/2}\lesssim L^{-\eta^5}\cdot L^{2(\gamma_0+2\eta^2)-(2\gamma_0+5\eta^2)}$ using Lemma \ref{basiccount} (1) and the LG condition, so (\ref{finalcount}) is true. If $\Mb_0$ is a triangle, then there are at least two $\beta$-atoms due to Proposition \ref{finalmole}. Moreover by Lemma \ref{basiccount} (1) we see that \[\Af_0\lesssim \delta^{-1}L^{d-\gamma+\gamma_0}\cdot L^{-(d-\gamma)}\lesssim L^{-3\eta^5}\cdot L^{3(\gamma_0+2\eta^2)-(2\gamma_0+5\eta^2)}\] if there is no $\alpha$-atom, and that \[\Af_1\lesssim (L^{d-1}+\delta^{-1}R^{-1}L^{d-2\gamma})L^{-(d-\gamma)}\lesssim L^{-3\eta^5}\cdot L^{-(\gamma+3\eta)/2}R^{-1}\cdot L^{2(\gamma_0+2\eta^2)-(2\gamma_0+5\eta^2)}\] if there is one $\alpha$-atom labeled by $R\in [L^{-1},L^{-\gamma+\eta}]$. In either case (\ref{finalcount}) is true. If $\Mb_1$ is a cycle of length $s\geq4$, then there is at least one $\beta$-atom. If there is no $\alpha$-atom the proof is same as the triangle case; if there is at least one $\alpha$-atom labeled by $R$, then
\[\Af_1\lesssim (L^{d-1}+\delta^{-1}R^{-1}L^{d-2\gamma})L^{-(d-\gamma)}\lesssim L^{-s\cdot \eta^5}L^{-(\gamma+3\eta)/2}R^{-1}\cdot (L^{(\gamma-5\eta)/2})^{s-2}L^{\gamma_0+2\eta^2}\cdot L^{-(2\gamma_0+5\eta^2)},\] which proves (\ref{finalcount}) since $L^{-(\gamma+3\eta)/2}R^{-1}\gtrsim L^{(\gamma-5\eta)/2}$ and $L^{\gamma_0+2\eta^2}\gtrsim L^{(\gamma-5\eta)/2}$  when $R\leq L^{-\gamma+\eta}$. This completes case when $\Mb_0$ is a cycle.

Now assume $\Mb_0$ has at least one $\varepsilon$-atom. Thanks to Proposition \ref{moleprop2} we always have
\begin{equation}\label{roughbd}\Af_0\lesssim L^{-\eta^2\rho_0+C(V_\alpha+V_\beta+1)}.
\end{equation} Then we need to prove some other estimates to interpolate with (\ref{roughbd}). To achieve this we need to perform some operations on $\Mb_0$; these operations include Operation (a) and Operation (b) defined in the proof of Proposition \ref{moleprop1} above, as well as the new ones defined below.

\underline{Operation (c):} Remove a chain of $\alpha$- and $\beta$-atoms (as in Proposition \ref{finalmole}), which has two distinct $\varepsilon$-atoms at both its ends, and all bonds. By Lemma \ref{basiccount} (1) we have
\begin{equation}\label{operc}\Df\lesssim1+\delta^{-1}R^{-1}L^{-\gamma}
\end{equation} for both $\alpha$- and $\beta$-atoms, where $R$ is such that $\max|r|\sim R$ for all the gaps $r$ at the $\alpha$- and $\beta$-atoms in this chain.

\underline{Operation (d):} Remove a chain of $\alpha$- and $\beta$-atoms, which has two distinct $\varepsilon$-atoms $v_1$ and $v_2$ at both its ends, and all bonds, and then add a new bond between $v_1$ and $v_2$ in the same direction as the chain. This operation does not change $\chi$, and we will show that $\Df\lesssim (\log L)^C$.

To see this, fix a decoration of the molecule $\Mb_{\mathrm{pre}}$ before operation. Let the two bonds in the chain at the two $\varepsilon$-atoms be $\ell_1$ and $\ell_2$ respectively, and the new bond added be $\ell_3$. Then $k_{\ell_1}-k_{\ell_2}$ and $|k_{\ell_1}|^2-|k_{\ell_2}|^2$ are fixed (the latter up to distance $\delta^{-1}L^{-2\gamma}$) in the decoration. For any integer $|g|\leq 3$, the value is also fixed of $|k_{\ell_1}|^2-|k_{\ell_2}+g(k_{\ell_1}-k_{\ell_2})|^2$ up to distance $O(1)\delta^{-1}L^{-2\gamma}$, thus we obtain a decoration for the molecule $\Mb_{\mathrm{pos}}$ after operation, by setting $k_{\ell_3}=k_{\ell_2}+g(k_{\ell_1}-k_{\ell_2})$. The LG condition for the new decoration is still satisfied, if we choose a suitable $g$, and weaken LG condition to $|r|\geq L^{-\gamma+\eta}/10$ (if not, then by pigeonhole principle we must have $|k_{\ell_1}-k_{\ell_2}|\leq L^{-\gamma+\eta}/10$, and thus the LG condition for $\Mb_{\mathrm{pre}}$ implies the weakened LG condition for $\Mb_{\mathrm{pos}}$). This implies that $\Cf_{\mathrm{pre}}\lesssim\Cf_{\mathrm{pos}}$ and hence $\Df\lesssim (\log L)^C$, because this operation affects at most $O(1)$ nodes in $O(1)$ ladders, and modifies $\Pf$ by at most a $(\log L)^C$ factor.

\underline{Operation (e):} Suppose after operation (d), a triple bond forms between $v_1$ and $v_2$. If these two atoms have at most one external bond then we remove  them and all bonds; otherwise, if they are connected to two $\varepsilon$-atoms $v_3\neq v_4$ by two single bonds, then we remove $(v_1,v_2)$ and all bonds, then add a new bond between $v_3$ and $v_4$, matching the directions of the removed single bonds.

By using the same arguments above (assigning a suitable $k_{\ell_2}+g(k_{\ell_1}-k_{\ell_2})$ to the new bond $\ell_3$, where $\ell_1$ and $\ell_2$ are the two single bonds connecting to $v_3$ and $v_4$ to $v_1$ and $v_2$) and also using Lemma \ref{basiccount} (3) we can show that as a result of applying Operations (d) and (e) consecutively,  \begin{equation}\label{opere}\Df\lesssim L^{-\gamma_0+200\eta},\end{equation} provided that $\max|r|\gtrsim L^{-\gamma-100\eta}$ for all the gaps $r$ at the $\alpha$- and $\beta$-atoms in this chain. Note that, after operations (d) and (e), we no longer require $c_v=0$ for degree $4$ atoms $v$ as in Definition \ref{decmole}.

\smallskip
We treat the remaining cases of Proposition \ref{lgmolect}. For even component, note that $V_\beta\geq 1$ because the last cut that separate $\Mb_0$ from the other components must be $\beta$-cut; moreover if $V_\beta=1$ then we must have $V_\alpha\geq 1$, since otherwise the only $\beta$-atom will have gap $r=0$ which is not possible.

(1) If $\Mb_0$ is the odd component, then we perform the same operations (a) and (b) as in the proof of Proposition \ref{moleprop1} above, but we only remove $\varepsilon$-atoms (even if they may become degree $1$ or $2$ in the process). Note that the value of $\Af$ becomes $1$ after removing all $\varepsilon$-atoms (because the remaining $\alpha$- and $\beta$-atoms can only form finitely many chains for which $\chi=0$), and $\Df\lesssim 1$ for each step in the same way as in the proof of Proposition \ref{moleprop1}, we conclude that $\Af_0\lesssim 1$. Moreover, let $V_\alpha$ and $V_\beta$ be the number of $\alpha$- and $\beta$-atoms respectively; if $w:=V_\alpha+V_\beta=0$ then $E=2V-1$ and $\Af_0\lesssim L^{-3\gamma_0/5}$ by Proposition \ref{moleprop2}, and if $w>0$ then $\Af_0\lesssim 1$ and the right hand side of (\ref{finalcount}) is at least $L^{(\gamma_0-5\eta)w/2-\eta^5\rho_0}$ (note also that $G=0$ for the odd component). In either case (\ref{finalcount}) follows from an interpolation with (\ref{roughbd}).

(2) From now on we assume $\Mb_0$ is an even component. If $V_\beta\geq 3$, then we first choose any chain of $\alpha$- and $\beta$-atoms and perform operation (c). After this, the molecule will no longer have any component such that all $\varepsilon$-atoms have degree $4$, so we can perform operations (a) and (b) as in (1) above, to prove that $\Af_0\lesssim\delta^{-1}L^{\gamma_0}$ using also (\ref{operc}). As the right hand side of (\ref{finalcount}) is at least \[L^{\eta^2V_\alpha}\cdot L^{V_\beta(\gamma_0+2\eta^2)-2\gamma_0-5\eta^2-\eta^5\rho_0}\gtrsim L^{\gamma_0+\eta^2/2-\eta^5\rho_0}\cdot L^{(V_\alpha+V_\beta)\eta^2/10},\] we can interpolate (\ref{roughbd}) with the bound $\Af_0\lesssim\delta^{-1}L^{\gamma_0}$ to prove (\ref{finalcount}).

(3) If $V_\beta=2$ and $V_\alpha\geq 1$, or if $V_\beta=1$ and $V_\alpha\geq 3$, then we choose any chain containing at least one $\alpha$-atom and perform operation (c), and proceed as in (2) above. Let the label of the $\alpha$-atom in this chain be $R$, then by (\ref{operc}) we have $\Af_0\lesssim1+\delta^{-1}R^{-1}L^{-\gamma}$. Interpolating with (\ref{roughbd}) we get
\begin{multline}\label{case3est}\Af_0\lesssim(1+\delta^{-1}R^{-1}L^{-\gamma})\cdot L^{-\eta^5\rho_0+C\eta^3(V_\alpha+1)}\\\lesssim L^{-\eta^5\rho_0}\cdot L^{-(\gamma+3\eta)/2}R^{-1}\cdot L^{V_\beta(\gamma_0+2\eta^2)-2\gamma_0-5\eta^2}\cdot L^{(V_\alpha-1)(\gamma-5\eta)/2}\end{multline} which is better than (\ref{finalcount}) and is easily verified when $V_\beta=2$ and $V_\alpha\geq 1$ or $V_\beta=1$ and $V_\alpha\geq 3$.

(4) If $V_\alpha=0$ and $V_\beta=2$, then choose all the chains of $\beta$-atoms and perform operation (d) to them. The resulting molecule $\Mb_1$ is $4$-regular and has at most two triple bonds. Therefore, by Proposition \ref{moleprop4} (2), we know that either $\Af_1\lesssim L^{-\eta/5}$ or $\Mb_1$ is formed by removing the two joints of a vine and adding one new bond between the pair of atoms connected to each joint. (note that, the number of choices for vector $k_\ell$ for a newly added bond $\ell$ is bounded by $\delta^{-1}L^{d-\gamma+\gamma_0}$ due to definition of operation (d) and Lemma \ref{basiccount} (1), so if $\Af\lesssim L^{-\gamma_0-\eta/4}$ for the molecule $\Mb_1\backslash\{\ell\}$ we also have $\Af_1\lesssim L^{-\eta/5}$). In the former case interpolating with (\ref{roughbd}) yields $\Af_0\lesssim L^{-\eta/5-\eta^5\rho_0}\lesssim L^{\eta^2-\eta^5\rho_0}$ which implies (\ref{finalcount}).

In the latter case, an enumeration shows that $\Mb_0$ must be a vine and is thus LG due to Proposition \ref{finalmole}, so we may now perform operation (c) to the chains of $\beta$-atoms in $\Mb_0$. For the resulting molecule $\Mb_1$ we have $\Af_1\lesssim 1$ by Proposition \ref{moleprop1}, and using also (\ref{operc}) and the LG assumption we get that $\Af_0\lesssim L^{-\eta/2}$, so (\ref{finalcount}) again follows by interpolation.

(5) If $V_\alpha=V_\beta=1$, then we choose the chain containing the $\alpha$-atom and perform operation (c),  and choose any other possible chain and perform operation (d), to get a molecule $\Mb_1$ with only $\varepsilon$-atoms. Note that $\Mb_1$ has at most one triple bond, satisfies the $\beta$-cut assumption in Proposition \ref{moleprop4} (1) (which follows because one cannot make any $\beta$-cuts in $\Mb_0$) as well as $E=2V-1$, so by Proposition \ref{moleprop4} (1) we know either $\Af_1\lesssim L^{-4\gamma_0/7}$, or $\Mb_1$ is formed by removing the two joints of a vine (II) and adding one new bond between the two atoms connected to one of the joints. But the latter case is impossible, because then $\Mb_0$ has to be a vine (II) and thus has LG due to Proposition \ref{finalmole}, which is impossible as one of its joints is an $\alpha$-atom.

 Let the $\alpha$-atom be labeled by $R$, then using (\ref{operc}) we get $\Af_0\lesssim (1+\delta^{-1}R^{-1}L^{-\gamma})L^{-4\gamma_0/7}$, and interpolating with (\ref{roughbd}) we get
 \[\Af_0\lesssim(1+\delta^{-1}R^{-1}L^{-\gamma})L^{-4\gamma_0/7}\cdot L^{-\eta^5\rho_0+C\eta^3}\lesssim L^{-\eta^5\rho_0}\cdot L^{(\gamma_0+2\eta^2)-2\gamma_0-5\eta^2}\cdot L^{-(\gamma+3\eta)/2}R^{-1},\] which is easily proved using $R\leq L^{-\gamma+\eta}$. This proves (\ref{finalcount}).

(6) Finally suppose $V_\alpha=2$ and $V_\beta=1$. Here we will use the $\iota_v$ and $\kappa_v$ parameters defined in Section \ref{reductcut}. First, let the labels of the two $\alpha$-atoms $(v_1,v_2)$ be $R_1\geq R_2$, we choose the chain containing $v_1$ and perform operation (c), then choose the other chains and perform operation (d), to reduce to $\Mb_1$. By (\ref{operc}) and interpolation, we know that (\ref{finalcount}) is true as long as the inequality
\begin{multline}\label{case6est} L^{-\eta^5\rho_0+C\eta^3}\cdot(1+\delta^{-1}R_1^{-1}L^{-\gamma})(\log L)^C\cdot \Af_1\\\lesssim L^{-\eta^5\rho_0}\cdot L^{-(\gamma+3\eta)+(\kappa_{v_1}+\kappa_{v_2})\eta}(R_1R_2)^{-1}\cdot L^{-(\gamma_0+3\eta^2)}\end{multline} holds. Note that $\Af_1\lesssim 1$ and $R_j\lesssim L^{-\gamma+\eta}$, an easy calculation shows that (\ref{case6est}) is true if $\kappa_{v_1}+\kappa_{v_2}>0$, or if $\Mb_1$ has no triple bond (so Proposition \ref{moleprop2} implies that $\Af_1\lesssim L^{-3\gamma_0/5}$), or if $R_2\leq L^{-\gamma-50\eta}$.

Now if $\kappa_{v_1}+\kappa_{v_2}=0$, $\Mb_1$ has a triple bond, and $R_1\geq R_2\geq L^{-\gamma-50\eta}$. Note that an algebraic sum of the three gaps at the three $\alpha$- and $\beta$-atoms equal to $0$, and none of the three gaps is $0$ itself, so by the definition of $(\iota_v,\kappa_v)$, we must have $|r|\gtrsim L^{-50\eta}R_2\geq L^{-\gamma-100\eta}$ for the gap $r$ at the $\beta$-atom. Then we may perform operation (e) instead of (d) in the last step at one of the chains not containing $v_1$, and reduce to a molecule $\Mb_2$. Clearly $\Af_2\lesssim 1$, and using (\ref{opere}) and interpolation gives that
\[\Af_0\lesssim L^{-\eta^5\rho_0+C\eta^3}\cdot(1+\delta^{-1}R_1^{-1}L^{-\gamma})L^{-\gamma_0+200\eta},\] and this extra $L^{-\gamma_0}$ gain (with $L^{O(\eta)}$ loss) is more than enough to imply (\ref{finalcount}) as above. This finishes the proof of Proposition \ref{lgmolect}.
\end{proof}
\subsection{Proof of Proposition \ref{moleprop2}}\label{mole2proof} In this subsection we prove Proposition \ref{moleprop2}. Under the \emph{large gap} assumption, this proof relies on the steps and the algorithm that are almost identical to those defined in the proof of Proposition 9.10 in \cite{DH21} (see \cite{DH21}, Sections 9.3--9.4). For completeness, we have included the definitions and properties of these steps and algorithm, with suitable modifications adapted to the current scaling law, in Appendix \ref{appalg}.

With these preparations we can prove Proposition \ref{moleprop2}, by adopting the same arguments as in Section 9.5 of \cite{DH21}, which we present below. We start with the case when $E\in\{2V-1,2V-2\}$ (so $\Mb$ is a molecule), and $\Mb$ has no triple bond, and apply the algorithm described in Section \ref{alg}. The algorithm contains $O(n)$ operations, where $n$ is the size of the molecule, and in some cases we are making binary choices depending on properties of the decoration, leading to at most $C^n$ possibilities. Such $C^n$ factors are always negligible by choosing the constant $C^+$ in (\ref{defequan}); Below we will fix one such possibility (and hence an operation sequence). Let $r_1$ be the total number of \emph{fine} operations, and $r_2$ be the total number of \emph{good} operations. Note that the change of any of the quantities we will study below, caused by any single operation we defined above, is at most $O(1)$.
\subsubsection{Increments of $\eta$ and $V_3$}\label{incre1} First, note that operations (TB-1N) and (TB-2N) only occur once after (3S3-3G) or (3D3-3G) which are good operations, the number of those is at most $Cr_2$. Let the number of (BR-N) where $d(v_1)=d(v_2)=3$ (see Proposition \ref{brprop}) be $z_1$, the number of other (BR-N) be $z_1'$. Let the number of (3S3-1N) be $z_2$, the number of (3R-1N) be $z_3$, the numbers of (2R-2F)--(2R-4F) be $z_4$, $z_5$ and $z_6$, and the number of (2R-1F) be $z_7$. By Propositions \ref{brprop}--\ref{2rprop}, we can examine the increment of $\nu$ in the whole process and get
\begin{equation}\label{increeta}-2z_1-2z_1'-2z_2+2z_3-2z_5\geq -\nu_0-Cr_2,
\end{equation} where $\nu_0\in\{0,-2\}$ is the initial value of $\nu$, and in the end $\nu=0$. In the same way, by examining the increment of $V_3$ we get
\begin{equation}\label{increv3}-2z_1-z_1'+2z_2+2z_3+z_4\leq  -V_{30}-Cr_2,
\end{equation} where $V_{30}\geq 0$ is the initial value of $V_3$ and in the end $V_3=0$. Subtracting these two inequalities yields $z_1'+z_2+z_4+z_5\leq \nu_0-V_{30}+Cr_2$. In particular we have $z_1'+z_2+z_4+z_5\leq Cr_2$, and if $E=2V-1$ then $r_2\geq 1$. Note also that $z_6+z_7\leq r_1$ because (2R-1F) and (2R-4F) are fine.
\subsubsection{The other operations} Next we will prove that $z_1+z_3\leq Cr_2$. By (\ref{increv3}) we have $z_3\leq z_1+Cr_2$, so we only need to prove $z_1\leq Cr_2$. Let $V_2^*$ be the number of degree 2 atoms with two single bonds. It is clear that $\Delta V_2^*=0$ for (3D3-1N), (3R-1N) and (2R-4F), and $\Delta V_2^*\geq 0$ for (2R-1F), and $\Delta V_2^*\geq 0$ for (BR-N) assuming $d(v_1)=d(v_2)=3$. Moreover, equality holds for (BR-N) if and only if the bridge removed is special. Therefore, with at most $Cr_2$ exceptions, all the bridges removed in (BR-N) are special.

Consider the increment of the number of special bonds, denoted by $\xi$. Clearly $\Delta\xi=0$ for (2R-1F) and (2R-4F); for (BR-N) which removes a special bridge, we can check that this operation cannot make any existing non-special bond special, so $\Delta\xi=-1$. Moreover, by  our algorithm, whenever we perform (3R-1N), it is always assumed that the component contains no special bond after this step, so $\Delta\xi\leq 0$. Similarly, whenever we perform (3D3-1N) we are always in (3-b) or (3-c-iii) (or in (3-c-i) but then the next operation will be good). For (3-c-iii), $v_3$ and $v_4$ are the only two degree 3 atom in the component after performing (3D3-1N), and they are not connected by a special bond (otherwise we are in (3-c-i)), so this step also does not create any special bond, hence $\Delta\xi\leq 0$.

Now let us consider operations (3D3-1N) occurring in (3-b). By our algorithm, if we also include the possible (3D3-2G), then such steps occur in the form of sequences which follow the type II chains in the molecule. For any operation in this sequence \emph{except} the last one, we must have $\Delta\xi=0$ (because in this case, after (3D3-1N), neither $v_3$ nor $v_4$ is connected to a degree 3 atom by a single bond). Moreover, if for the last one in the sequence we do have $\Delta\xi>0$, then immediately after this sequence we must have a good operation (because in this case, after we finish the sequence and move to (3-c), either $v_3$ or $v_4$ will have degree 3 instead of 4, so we must be in (3-c-ii)). Since the number of good operations is at most $r_2$, we know that the number of operations for which $\Delta\xi>0$ is at most $Cr_2$. Thus, considering the increment of $\xi$, we see that $z_1\leq Cr_2$.
\subsubsection{Ladders}\label{lad1} Now we see that the number of steps \emph{different from} (3D3-1N) is at most $C(r_1+r_2)$. In particular steps (3D3-1N) occurring in (3-c-i) and (3-c-iii) is also at most $C(r_1+r_2)$ because each of them must be followed by an operation different from (3D3-1N). As for the sequences of (3D3-1N) or (3D3-2G) occurring in (3-b), each sequence corresponds to a ladder, and each chain can be as long as $Cn$, but the number of chains must be at most $C(r_1+r_2)$ for the same reason. Note that some of the bonds in the ladders may not exist in the original base molecule, but the number of those bonds is again at most $C(r_1+r_2)$ because (3S3-3G) and (3D3-3G) are both good steps. Upon further dividing, we can find these (at most $C(r_1+r_2)$) ladders in the original molecule $\Mb$, such that the number of atoms not belonging to one of these ladders is at most $C(r_1+r_2)$. By the definition of $\rho$ (see Proposition \ref{kqmainest2}), we know that $\rho\leq C(r_1+r_2)$.
\subsubsection{Conslusion} Now we can prove Proposition \ref{moleprop2}. First, if $E=2V-1$, then as shown in Section \ref{incre1} we must have $r_2\geq 1$, hence $\Af\lesssim L^{-3\gamma_0/5}$ by definition of good operations; here note that, for operation (3D3-1N), which is the only operation whose number is not controlled by $C(r_1+r_2)$, we do not have any logarithmic loss due to Proposition \ref{3d3prop}, so th possible logarithmic losses can be easily accommodated. In the same way, if $E=2V-2>0$, then the total number of operations is $\leq C(r_1+r_2)$ as shown above, and this total number must be positive if $E>0$, so we know $r_1+r_2\geq 1$ and hence $\Af\lesssim L^{-\eta/3}$ (again considering possible log losses).

Finally, suppose $\Mb$ has at most $w$ atoms not of degree $4$ and is allowed to have at most $w$ SG atoms. Then we may first remove each of the $w$ SG atoms, where for each operation we trivially have $\Df\lesssim L^C$. The resulting molecule still has at most $C(w+1)$ atoms not of degree $4$, which allows us to apply the algorithm described in Section \ref{alg}. All the arguments in Sections \ref{incre1}--\ref{lad1} still apply, if one allows remainders of size $C(w+1)$ (for example the values of $\nu_0$ and $V_{30}$ in (\ref{increeta}) and (\ref{increv3}) will both be $\leq C(w+1)$, etc.). In the end, using the definition of good and fine operations, we get that
\[\Af\lesssim L^{C(w+1)}\cdot L^{-(r_1+r_2)\eta/3}\qquad\mathrm{and}\qquad \rho\leq C(r_1+r_2)+C(w+1),\] which clearly implies $\Af\lesssim L^{-\eta^2\rho+C(w+1)}$, as desired.
\subsection{Proof of Proposition \ref{moleprop3}} In this subsection we prove Proposition \ref{moleprop3}. The proof involves a different procedure with the following operations, which also occur in other sections, but for simplicity we shall give them specific names that are used \emph{only in this subsection}. These include: (R), where we remove a degree $2$ or $3$ atom, (B), where we remove a bridge (in the sense of Section \ref{secbr}), and ($\beta$), where we perform a $\beta$-cut at a degree $4$ atom such that none of its bonds is a bridge. Note that removing a bridge does not affect whether or not any other bond is a bridge, and also does not create any new possibility of $\beta$-cut as in ($\beta$).

Start with the molecule $\Mb$ described in Proposition \ref{moleprop3}; note in particular that $\Mb$ has no ladders. We use $\Mb\rightarrow (3,3)[4]$ to indicate that there $\Mb$ has two degree $3$ atoms and the other atoms have degree $4$ (similarly $\Mb\rightarrow (2)[3,4]$ means that $\Mb$ has one degree $2$ atom and the other atoms have degree $3$ and $4$, etc.). If $\Mb$ has a bridge, then we remove it by (B) to get two components $\Mb_1$ and $\Mb_2$. It is easy to check that either $\Mb_j\rightarrow (3,3)[4]$ or $\Mb_j\rightarrow (2)[4]$ for each $j$ (since the sum of degrees is always even), but if $\Mb_j\rightarrow (2)[4]$ then removing the degree $2$ atom by (R) also yields a molecule $\widetilde{\Mb_j}\rightarrow(3,3)[4]$. If $\Mb$ admits a $\beta$-cut as specified above, then we perform operation ($\beta$), then one of the resulting components will be $\rightarrow (2)[4]$, and the other satisfies $\Af\lesssim1$ by Proposition \ref{moleprop1}. As all the above operations satisfy $\Df\lesssim 1$ due to the LG assumption, we may always reduce to the case where $\Mb\rightarrow(3,3)[4]$ has no bridge and admits no $\beta$-cuts as above (so in particular removing any atom in $\Mb$ does not create any new component).

Now we remove one degree $3$ atom $v_1$ in $\Mb$ by (R), and denote the bonds by $(\ell_1,\ell_2,\ell_3)$. Consider the following possibilities after this operation:
\begin{enumerate}[{(1)}]
\item If there is a degree $3$ atom $v_2$ with bonds $(\ell_4,\ell_5,\ell_6)$ such that only $\ell_4$ is a bridge, then by Lemma 9.14 of \cite{DH21}, for some $j\in\{1,2,3\}$ we must have $k_{\ell_4}\pm k_{\ell_j}$ equals constant and $|k_{\ell_4}|^2\pm|k_{\ell_j}|^2$ equals constant up to distance $n\delta^{-1}L^{-2\gamma}$ for some choice of $\pm$, where $n$ is the size of $\Mb$. Then we remove $(v_1,v_2)$ and all bonds; since this results in $\Delta F=1$ (recall $F$ is the number of components) and $\Delta\chi=-3$, using Lemma \ref{basiccount} (4), we see that this composition operation has $\Df\lesssim L^{-\gamma_0}(\log L)^C$. The same result holds if there is a degree $2$ atom $v_2$ with two bonds that are not bridge.
\item If there is a degree $4$ atom $v_2$ with no bridge at which a $\beta$-cut is possible, say along the bonds $(\ell_4,\ell_5)$ and $(\ell_6,\ell_7)$. Then by Lemma 9.14 of \cite{DH21}, for some choice of $\pm$ and $j\in\{1,2,3\}$ we have $k_{\ell_j}\pm k_{\ell_4}\pm k_{\ell_5}$ equals constant and $|k_{\ell_j}|^2\pm |k_{\ell_4}|^2\pm |k_{\ell_5}|^2$ equals constant up to distance $n\delta^{-1}L^{-2\gamma}$. Then we remove $(v_1,v_2)$ and all bonds; since $\Delta\chi=-4$, by using Lemma \ref{basiccount} (4) for $(k_{\ell_1},\cdots,k_{\ell_5})$ and Lemma \ref{basiccount} (1) plus LG assumption for $(k_{\ell_6},k_{\ell_7})$, we see that this composition operation has $\Df\lesssim L^{-\gamma_0-\eta}(\log L)^C$.
\item If there is a degree $4$ atom $v_2$ with bonds $(\ell_4,\cdots,\ell_7)$ such that exactly two bonds (say $\ell_4$ and $\ell_5$) are bridges (we may assume $\ell_4$ and $\ell_5$ are in opposite directions or the proof will be much easier using the better bounds in Lemma \ref{basiccount} (1)), then there are three components after removing $v_1$ (and all bonds) and $(\ell_4,\ell_5)$, denote them by $X_j\,(1\leq j\leq 3)$ with $v_2\in X_1$. Then, since $v_4$ and $v_5$ are not bridges before removing $v_1$, and one cannot make a $\beta$-cut before removing $v_1$, we see that the three other endpoint for $\ell_j\,(1\leq j\leq 3)$ must be in $X_j$ respectively. By Lemma 9.14 of \cite{DH21}, this implies that $k_{\ell_1}\pm k_{\ell_6}\pm k_{\ell_7}$ equals constant and $|k_{\ell_1}|^2\pm |k_{\ell_6}|^2\pm |k_{\ell_7}|^2$ equals constant up to distance $n\delta^{-1}L^{-2\gamma}$, for some choice of $\pm$. Then we remove $(v_1,v_2)$ and all bonds; since $\Delta\chi=-3$, by using Lemma \ref{basiccount} (4) we see that this composition operation has $\Df\lesssim L^{-\gamma_0}(\log L)^C$.
\item If none of (1)--(3) holds, then we remove $v_1$ and all bonds, and then remove all the subsequent bridges. In each resulting component, there will be no degree $2$ atom (which would correspond to one of (1)--(3)) nor degree $1$ atom (which would correspond to bridges), so each (nontrivial) component will be $\rightarrow[3,4]$, and there is no bridge nor $\beta$ cut possibilities. We then remove another degree $3$ atom and repeat the above procedure, until one of (1)--(3) happens or the molecule becomes trivial (i.e. with no bonds). But the latter case is impossible, since removing one degree $3$ atom has $\Delta\chi=-2$, removing one bridge has $\Delta\chi=0$, and for any molecule $\Mb\rightarrow [3,4]$ with only single bonds we have $\chi\geq V/2+1\geq 3$.
\end{enumerate}

By the above argument, we have shown that we will be able to perform at least one good operation in (1)--(3) such that either $\Df\lesssim L^{-\gamma_0}(\log L)^C$ and $\Delta\chi=-3$, or $\Df\lesssim L^{-\gamma_0-\eta}(\log L)^C$ and $\Delta\chi=-4$. In the latter case we are already done since all other operations in the sequence trivially have $\Df\lesssim 1$; in the former case, we further remove all bridges after the good operation. If the resulting molecule is nontrivial, then either repeating the above argument or exploiting a degree $2$ atom using Lemma \ref{basiccount} (1) and LG assumption will gain another power $L^{-\eta}$ which allows us to close. Finally, if the resulting molecule is trivial, then we have $\chi=0$ after the good operation which has $\Delta \chi=-3$. Since the molecule $\Mb'$ before the good operation has only single bonds, and $\Mb'\rightarrow[3,4]$ and $\chi=3$, the only possibility for $\Mb'$ is $K_4$ (i.e. a complete graph of $4$ atoms with only single bonds). But in this last case, we can work directly with $\Mb'$ and apply Lemma \ref{basiccount} (5) to get $\Af'\lesssim L^{-\gamma_0-\eta}$ for the molecule $\Mb'$. This completes the proof of Proposition \ref{moleprop3}.
\subsection{Proof of Proposition \ref{moleprop4}}\label{lgmole4} In this subsection we prove Proposition \ref{moleprop4}. For this purpose we need to introduce one more operation (and operation sequence), which we denote by (Y), as follows. Suppose $\Mb$ contains a triple bond between two atoms $(v_1,v_2)$. If these two atoms have at most one extra bond, or if they have two extra single bonds connecting to the same third atom, then we remove them and all the bonds, and call this (Y1); otherwise there are two extra single bonds connecting $v_1$ to $v_3$ and $v_2$ to $v_4$ (with $v_3\neq v_4$), then we remove $(v_1,v_2)$ and all the bonds, and add one new bond between $v_3$ and $v_4$ matching the directions of the removed single bonds, and call this (Y2). If after (Y2) a new triple bond forms between $v_3$ and $v_4$, we then apply another operation (Y) to them, and so on, until there is no more triple bonds, and call this (Y) sequence. Note that (Y) sequence may involve a ladder, as illustrated in Figure \ref{fig:yoper}.
  \begin{figure}[h!]
  \includegraphics[scale=.42]{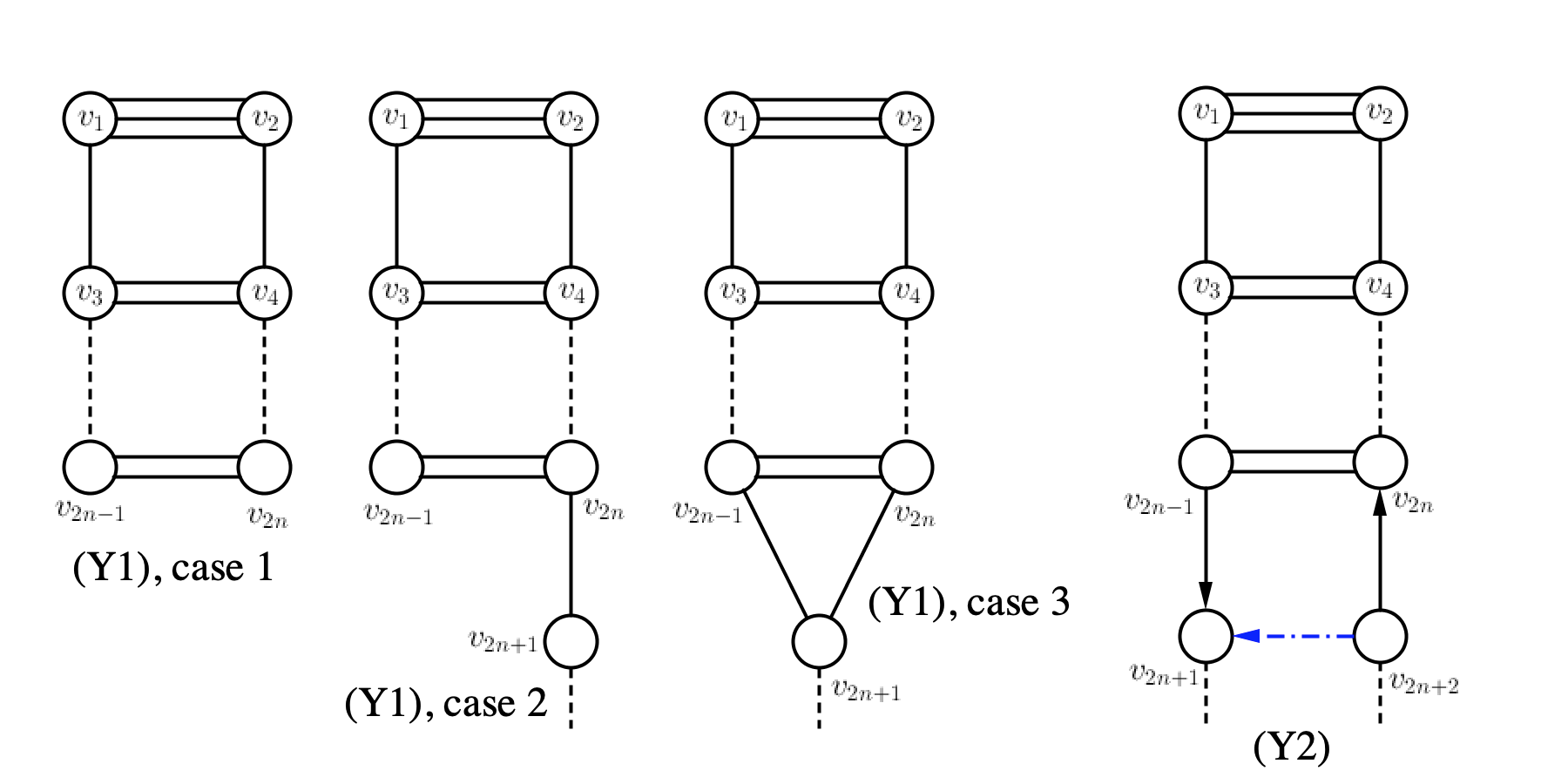}
  \caption{The structure involved in (Y) sequence, which may contain a ladder. In the sequence we remove all atoms up to $v_{2n}$ and all bonds; if the last operation is (Y2), we also add a new bond (the blue one) between $v_{2n+1}$ and $v_{2n+2}$.}
  \label{fig:yoper}
\end{figure} 

We need two lemmas concerning (Y) sequences before proving Proposition \ref{moleprop4}.
\begin{lem}\label{ylem1} Suppose $\Mb$ is $4$-regular with \emph{at most two triple bonds} and $\Mb'$ is formed from $\Mb$ by (Y) sequences. If $\Mb'$ is \emph{either a quadruple bond or a triangle formed by $3$ double bonds}, then $\Mb$ is formed by removing the two joints of a vine and adding one new bond between the pair of atoms connected to each joint.
\end{lem}
\begin{proof} The proof is an enumeration of all possibilities. Clearly going from $\Mb$ to $\Mb'$ involves at most two (Y) sequences with last operation (Y2). To invert one such sequence, one simply selects a non-triple bond from $\Mb'$ (see e.g. the blue one in Figure \ref{fig:yoper}), remove it, then insert the new structure shown in case (Y2) of Figure \ref{fig:yoper}; we call this (Z). If $\Mb'$ is a quadruple bond, then by applying (Z) once, we get a pseudomolecule formed from vine (II) (by removing the two joints of a vine and adding one new bond between the pair of atoms connected to each joint; same below). This already has two triple bonds, so one cannot further apply (Z) at any other bond (or one would produce a third triple bond), and hence $\Mb$ is formed by vine (II).

Now suppose $\Mb'$ is a triangle formed by $3$ double bonds. We may choose any bond $\ell\in \Mb$ and apply (Z) to get an intermediate pseudomolecule $\widetilde{\Mb}$, which can be formed from vines (III), (IV), (V), (VII) or (VIII), and contains only one triple bond. Then we may choose one non-triple bond $\ell'$ in $\widetilde{\Mb}$ and apply (Z) again to get $\Mb$. The structure of $\Mb$ depends on which bond we choose:
\begin{itemize}
\item If $\ell'$ is the other bond in the double bond containing $\ell$, then $\Mb$ is formed from vine (IV);
\item If $\ell'$ is from another double bond in the triangle $\Mb'$, then $\Mb$ is formed from vine (III);
\item If $\ell'$ is a double bond inserted in the first (Z) operation, then $\Mb$ is formed from vine (VII);
\item If $\ell'$ is a single bond inserted in the first (Z) operation, then $\Mb$ is formed from (VI).
\end{itemize}
In any case, this proves  Lemma \ref{ylem1}.
\end{proof}
\begin{lem}\label{ylem2} Suppose $\Mb'$ is formed from $\Mb$ by one (Y) sequence. Fix a bond $\ell\in\Mb$, and if the last operation is (Y2), let the newly added bond be $\ell'$. Then we have $\Af(\Mb)\lesssim (\log L)^C\cdot\Af(\Mb')$. Moreover, if the last operation is (Y2), and $\ell$ is removed in the (Y) sequence, then $\Af(\Mb\backslash\{\ell\})\lesssim (\log L)^C\cdot\Af(\Mb'\backslash\{\ell'\})$.
\end{lem}
\begin{proof} We use the notation of the different cases in Figure \ref{fig:yoper}. First note that, in any case, the (Y) sequence removes a ladder of length $n-2$ as in Figure \ref{fig:yoper}. If the last operation is (Y1) in case 1 or case 2, then we have $\Delta\chi=-2n$; once a decoration of $\Mb'$ is fixed, we may examine the remaining part of decoration, going from bottom to top, using Lemma \ref{basiccount} (3) for each step and Lemma \ref{basiccount} (2) for the last step, to bound $\Df\lesssim (\log L)^C$ for this (Y) sequence. Note that for all but $O(1)$ operations in this (Y) sequence we have the sharp bound $\Df\lesssim 1$, in the same way as (\ref{laddersharp}), in view of the $\Xf^{-1}$ factor in Lemma \ref{basiccount} (3) and the definition of the $\Pf$ product. Next, if the last operation is (Y1) in case 3, then we have $\Delta\chi=-(2n+1)$; once a decoration of $\Mb'$ is fixed, we again go from bottom to top in exactly the same way as above, the only difference being that we consider the two bonds at $v_{2n+1}$ in the first step, using Lemma \ref{basiccount} (1) and the LG assumption at $v_{2n+1}$, to get $\Df\lesssim L^{-\eta/2}$.

Now suppose the last operation is (Y2) with $\Delta\chi=-2n$. Given a decoration of $\Mb$, let the two bonds connecting $v_{2n-1}$ and $v_{2n}$ to $v_{2n-3}$ and $v_{2n-2}$ (see Figure \ref{fig:yoper}) be $\ell_1$ and $\ell_2$, and let the new bond be $\ell'$, then $k_{\ell_1}-k_{\ell_2}$ is fixed and $|k_{\ell_1}|^2-|k_{\ell_2}|^2$ is fixed up to distance $\delta^{-1}L^{-2\gamma}$ due to Lemma 9.14 of \cite{DH21}. We then define a decoration of $\Mb'$, as in operation (d) and (e) in the proof of Proposition \ref{kqmainest2} in Section \ref{lgmole1}, by assigning $k_{\ell'}=k_{\ell_2}+g(k_{\ell_1}-k_{\ell_2})$ for some $|g|\les 3$. This will keep the LG assumption for $\Mb'$ (which is weakened by a constant multiple, but this does not matter since we will only ever perform (Y) sequence $O(1)$ times). Once a decoration of $\Mb'$ is fixed, then $k_{\ell_1}$ and $k_{\ell_2}$ are also fixed, and we can go from bottom to top just as above to show that $\Df\lesssim (\log L)^C$ for this (Y) sequence.

Finally, assume the last operation is (Y2) and $\ell$ is a bond removed in this sequence, then $\Mb'\backslash\{\ell'\}$ is formed from $\Mb\backslash\{\ell\}$ by removing all atoms up to $v_{2n}$ and all bonds other than $\ell$, an operation with $\Delta\chi=-2n$. Once a decoration of $\Mb'\backslash\{\ell'\}$ is fixed, we simply start from the atom $v_j$ containing $\ell$ and apply Lemma \ref{basiccount} (2), then apply Lemma 9.14 of \cite{DH21} to fix the values of $k_{\ell_i}$ for bonds $\ell_i$ at the atom $v_{j\pm 1}$ connected to $v_j$ by a double or triple bond. Next, we simply go from $(v_j,v_{j\pm1})$ both upwards and downwards, using Lemma \ref{basiccount} (2) and (3) exactly as above, to show that $\Df\lesssim (\log L)^C$ for this sequence of operation. This completes the proof.
\end{proof}
With Lemmas \ref{ylem1} and \ref{ylem2}, we can now prove Proposition \ref{moleprop4}.
\begin{proof}[Proof of Proposition \ref{moleprop4}] Start with (1). Since $\Mb$ has at most one triple bond, we can apply (Y) sequence once to get $\Mb'$ which has no triple bond. The last operation cannot be (Y1) case 3 due to the $\beta$-cut assumption for $\Mb$, and cannot be (Y1) case 1 because then $\Mb$ would be formed from vine (II) by removing the two joints of a vine (II) and adding one new bond between the two atoms connected to one of the joints. Now, if the last operation is either (Y1) case 2 or (Y2), then we have $E=2V-1$ for $\Mb'$, so by Proposition \ref{moleprop2} and Lemma \ref{ylem2} we get \[\Af\lesssim (\log L)^C\Af'\lesssim L^{-3\gamma_0/5}(\log L)^C\lesssim L^{-4\gamma_0/7}.\]

Now consider (2). Since $\Mb$ has at most two triple bonds and is not formed from a vine, using Lemma \ref{ylem1}, we can apply at most two (Y) sequences to reduce it to $\Mb'$, which does not contain triple or quadruple bond, and is not a triangle formed by $3$ double bonds. By Lemma \ref{ylem2} it suffices to prove the same result (with slightly better powers) for $\Mb'$ (say with a fixed bond $\ell'$).

If the last operation is (Y1), which must be case 3 (since $\Mb$ is 4-regular), then we have $E=2V-1$ for $\Mb'$, so the result for $\Af(\Mb')$ follows from Proposition \ref{moleprop2}; now we assume the last operation is (Y2), so $\Mb'$ is $4$-regular (hence cannot contain any bridge, because otherwise we get a component with odd total degree after removing the bridge). If either $\Mb'$ admits a $\beta$-cut or $\Mb'\backslash\{\ell'\}$ has a bridge, then by preforming this cut or removing this bridge we can divide $\Mb'$ or $\Mb'\backslash\{\ell'\}$ into two molecules with $E=2V-1$ and no triple bonds, so the result for $\Af(\Mb')$ or $\Af(\Mb'\backslash\{\ell'\})$ follows from Proposition \ref{moleprop2} (note also that $(L^{-3\gamma_0/5})^2\lesssim L^{-\gamma_0-\eta_0/2}$). Finally, if $\Mb'$ has only single bonds, then the result for $\Af(\Mb'\backslash\{\ell'\})$ follows from Proposition \ref{moleprop3}.

From now on, assume that $\Mb'$ does not admit any $\beta$-cut, has at least one double bond but no triple bond, and $\Mb'\backslash\{\ell'\}$ has no bridge. Consider the following cases, where in each case we also assume that no earlier cases happen:

(I) Suppose a double bond, say between two atoms $v_1$ and $v_2$, shares a common atom with the fixed bond $\ell'$, then we choose to prove $\Af(\Mb'\backslash\{\ell'\})\lesssim L^{-\gamma_0-\eta/4}$. In fact since $\Mb'$ does not admit any bridge or $\beta$-cut, we see that removing the two atoms $(v_1,v_2)$ and all the bonds from $\Mb'\backslash\{\ell'\}$ has either $\Delta\chi=-3$ or $\Delta\chi=-2$. In the former case we have $\Df\lesssim L^{-\gamma_0}(\log L)^C$ by Lemma \ref{basiccount} (4), and the resulting molecule has $E=2V-2>0$ without triple bond, so by Proposition \ref{moleprop2} we have $\Af(\Mb'')\lesssim L^{-\eta/3}$, and hence $\Af(\Mb'\backslash\{\ell'\})\lesssim L^{-\gamma_0-\eta/4}$. In the latter case we have $\Df\lesssim (\log L)^C$ by Lemma \ref{basiccount} (3), and the resulting molecule $\Mb''$ has two components each satisfying $E=2V-1$ and having no triple bond. By Proposition \ref{moleprop2} we have $\Af(\Mb'')\lesssim (L^{-3\gamma_0/5})^2$, which again implies $\Af(\Mb'\backslash\{\ell'\})\lesssim L^{-\gamma_0-\eta/4}$.

(II) Suppose two bonds $(\ell_1,\ell_2)$ form a \emph{bi-bridge}, such that $\Mb'$ becomes disconnected after removing both bonds; by assumption we know $\ell'\not\in\{\ell_1,\ell_2\}$. By Lemma 9.14 of \cite{DH21} we know that $k_{\ell_1}- k_{\ell_2}$ is fixed and $|k_{\ell_1}|^2- |k_{\ell_2}|^2$ is fixed up to distance $O(n\delta^{-1}L^{-2\gamma})$, where $n\lesssim(\log L)^C$ is the size of $\Mb'$. If $k_{\ell_1}\neq k_{\ell_2}$, then we choose to prove $\Af(\Mb')\lesssim L^{-\eta/4}$, and remove the bonds $(\ell_1,\ell_2)$ to get a new molecule $\Mb''$. By Lemma \ref{basiccount} (1) we have $\Df\lesssim L^{\gamma_0}(\log L)^C$ for this operation, and $\Mb''$ has two components with $E=2V-1$ and no triple bonds, so Proposition \ref{moleprop2} implies that \[\Af(\Mb')\lesssim L^{\gamma_0}(\log L)^C\Af(\Mb'')\lesssim L^{\gamma_0}(\log L)^C(L^{-3\gamma_0/5})^2\lesssim L^{-\eta/4}.\]

Now if $k_{\ell_1}=k_{\ell_2}$, then we choose to prove $\Af(\Mb'\backslash\{\ell'\})\lesssim L^{-\gamma_0-\eta/4}$. We remove the bonds $(\ell_1,\ell_2)$ from $\Mb'\backslash\{\ell'\}$, but add one new bond $\ell_3$ between the two endpoints of $\ell_1$ and $\ell_2$ that belong to the component containing $\ell'$, matching the directions of $(\ell_1,\ell_2)$. This generates a new molecule $\Mb''$, and any decoration of $\Mb'\backslash\{\ell'\}$ leads to a decoration of $\Mb''$ by defining $k_{\ell_3}=k_{\ell_1}=k_{\ell_2}$, so the operation going from $\Mb'\backslash\{\ell'\}$ to $\Mb''$ has $\Df\lesssim(\log L)^C$. By our choice, $\Mb''$ has two components with $E=2V-1$ and at most one triple bond. Moreover the component with (possibly) one triple bond cannot be the exceptional case described in part (1) above, because then it would have a double bond between the only two degree $3$ atoms, which is impossible because $\Mb'$ has no triple bond and the two endpoints of $\ell'$ cannot be connected by a double bond in $\Mb'\backslash\{\ell'\}$. Therefore, by using Propositions \ref{moleprop2} and part (1) just proved, first considering the component of $\Mb''$ containing $\ell_3$ and then the one not containing $\ell_3$, we get
\[\Af(\Mb'\backslash\{\ell'\})\lesssim(\log L)^C\Af(\Mb'')\lesssim (\log L)^C(L^{-3\gamma_0/5})^2\lesssim L^{-\gamma_0-\eta/4}.\]

(III) Suppose $\Mb'$ has a double bond $(\ell_1,\ell_2)$ between two atoms $v_1$ and $v_2$, and each $v_j$ has two extra bonds $(\ell_{2j+1},\ell_{2j+2})$, such that $|k_{\ell_i}\pm k_{\ell_j}|\leq L^{-\gamma+\eta}$ for some $i\in\{3,4\}$ and $j\in\{5,6\}$. We then choose to prove $\Af(\Mb')\lesssim L^{-\eta/4}$, and remove $(v_1,v_2)$ and all the bonds. This does not disconnect $\Mb'$ (otherwise $\Mb'$ would admit a $\beta$-cut or bi-bridge), so $\Delta\chi=-4$. The number of choices for $(k_{\ell_1},\cdots,k_{\ell_6})$ is bounded by first fixing $(k_{\ell_i},k_{\ell_j})$ and applying Lemma \ref{basiccount} (3), which results in
\[L^d(L^{1-\gamma+\eta})^dL^{2(d-\gamma)}\lesssim L^{4(d-\gamma)-\gamma_0+d\eta},\] so $\Df\lesssim L^{-\gamma_0/2}$ for this operation, and the new molecule $\Mb''$ satisfies $\Af(\Mb'')\lesssim 1$ by Proposition \ref{moleprop1}, which then implies $\Af(\Mb')\lesssim L^{-\eta/4}$.

(IV) Finally, suppose $\Mb'$ has a double bond as in (III), but no inequality $|k_{\ell_i}\pm k_{\ell_j}|\leq L^{-\gamma+\eta}$ holds (and $\ell'\not\in\{\ell_1,\cdots,\ell_6\}$). Then we merge the two atoms $(v_1,v_2)$ into one atom which has bonds $(\ell_3,\cdots,\ell_6)$, to get a new pseudomolecule $\Mb''$. A decoration of $\Mb'$ naturally leads to a decoration of $\Mb''$, which is also LG by our assumptions; the operation going from $\Mb'$ to $\Mb''$ has $\Delta\chi=-1$ and $\Df\lesssim L^{-\eta/2}$ by Lemma \ref{basiccount} (1) and the LG assumption at $v_1$. Therefore the bound for $\Af(\Mb')$ (or $\Af(\Mb'\backslash\{\ell'\}$) follows from the same bound for $\Af(\Mb'')$ (or $\Af(\Mb''\backslash\{\ell'\}$). Note that $\Mb''$ has no quadruple bond (otherwise $\Mb'$ would be a triangle of double bonds) and no triple bond (otherwise $\Mb'$ would contain a triangle with one single bond and two double bonds, and the two outgoing bonds of this triangle would form a bi-bridge), so if $\Mb''$ is not a triangle of double bonds, we can repeat the same arguments above for $\Mb''$ until it either becomes a triangle of double bonds or runs out of double bonds (in this latter case we prove the bound for $\Af(\Mb''\backslash\{\ell'\})$ using Proposition \ref{moleprop3}). If $\Mb''$ is a triangle of double bonds, then $\Mb'$ must be a quadrilateral of double bonds, in which case we can prove the bounds for $\Af(\Mb'\backslash\{\ell'\})$ by first using Lemma \ref{basiccount} (4) and then using Lemma \ref{basiccount} (1) plus the LG assumption. This completes the proof of Proposition \ref{moleprop4}.
\end{proof}
\subsection{Addressing degenerate cases}\label{extradegen} In this subsection we discuss the possibility of degenerated cases, defined by $k_2\in\{k_1,k_3\}$ (and hence $k_1=k_2=k_3$) in (\ref{defset}), see Remark \ref{nonresrem}. Such degeneracies may occur at various stages in the main proof above, but due to the strong restriction $k_1=k_2=k_3$, they enjoy much better summation and counting estimates etc. than non-degenerate cases $k_2\not\in\{k_1,k_3\}$, and are easily addressed. We briefly demonstrate this below.

(1) Regular couples, trees and vines: Consider the regular couples $\Qc^{(\lf,\lf')}$ and regular trees $\Tc^{(\mf)}$ during the reduction process from $\Qc$ to $\Qc_{\mathrm{sk}}$ in Section \ref{primered}, or in the vine-like objects $\Ub_j$ during the stage 1 reduction from $\Qc_{\mathrm{sk}}$ to $\Qc_{\mathrm{sub}}$ in Section \ref{stage1red}. As is clear from the proofs of Proposition \ref{regcpltreeasymp} and Lemmas \ref{sumintest1}--\ref{sumintest2}, any degeneracies occurring in these expressions only produce lower order remainder terms (for example, they correspond to $x_j=y_j=0$ for some $j$ in Lemma \ref{sumintest1}), so they do not affect the proof.

(2) Molecule structure: As shown in Remark \ref{moleremark}, the molecule $\Qc_{\mathrm{sub}}$ may have a degree $2$ atom instead of two degree $3$ atoms if degeneracy is allowed. However in this case the degree $2$ atom $v$ must be degenerate, and the values $k_{\ell_1}=k_{\ell_2}$ for its two bonds must be fixed, so we simply remove this atom (and more atoms connected to it if needed) to reduce to the case of two degree $3$ atoms. This operation will have a huge gain $\Df\lesssim L^{-d+\gamma+\eta}$, which is enough to cover all possible losses that may occur later, so we just proceed normally thereafter. As for self connecting bonds, they are left for now and will be treated in later steps.

(3) The cutting process: Consider the cutting operations during stage 2 reduction in Section \ref{reductcut0}. Note that in selecting the collection $\Vs_0$ of SG vines in Section \ref{stage1red} we shall exclude those vines of \emph{zero gap} (i.e. with degenerate joints), so all the hinge atoms (see Proposition \ref{subpro}) produced in stage 1 reduction will not be degenerate once they are cut according to the rules in Section \ref{reductcut0}. Moreover, we shall not make any cut at non-hinge degenerate atoms, so there will be no degenerate atoms involved in the cutting process, and any $\alpha$- or $\beta$- atom will not be degenerate; however, we also allow degenerate $\varepsilon$-atoms in the resulting molecule $\Mb_{\mathrm{fin}}$.

(4) The final molecule: Suppose there is a degenerate $\varepsilon$-atom $v$ in a component $\Mb_0$ of the final molecule $\Mb_{\mathrm{fin}}$. Assume $v$ has $4$ bonds $\ell_j\,(1\leq j\leq 4)$, as other cases are similar and easier, then $k_{\ell_1}=\cdots =k_{\ell_4}:=k_0$. We shall remove this atom $v$; if $\Delta\chi\leq -2$ then this operation has a big gain $\Df\lesssim L^{-d+2\gamma+\eta}$ which is enough to cover all possible losses and we just proceed normally (using Proposition \ref{moleprop1}) thereafter. If $\Delta\chi\geq -1$, then one of $k_j$ must be a bridge, so in the specific counting problem for $\Mb_0$, the value of $k$ will be uniquely fixed due to Lemma 9.14 of \cite{DH21}; in this case the exact value of $k_{\ell_j}$ is not important, so we may replace them by some arbitrary non-degenerate configurations and proceed normally as in Sections \ref{lgmole1}--\ref{lgmole4} above.

Combining cases (1)--(4) above, this finishes the discussion of degenerate cases and concludes the proof of Propositions \ref{mainprop1}--\ref{mainprop4}.
\section{Linearization and the end of the proof}\label{linoper}
\subsection{Proof of Proposition \ref{mainprop2}}\label{linoper1} In this subsection we prove Proposition \ref{mainprop2}. The proof is a slight modification of the proofs of Propositions \ref{mainprop1} and \ref{mainprop4}, following the same arguments in Section 11 of \cite{DH21}, but with a few differences specific to this paper. 
\subsubsection{Construction of $\Xs$} Recall the notion of \emph{flower trees} and \emph{flower couples} defined in Definition 11.1 of \cite{DH21}: a flower tree is a ternary tree with one leaf specified (called \emph{flower}), and a flower couple is formed two flower trees with their leaves paired, such that the two flowers are paired to each other. The \emph{stem} of a flower tree is the unique path from its root to flower.

For any flower tree $\Tc$ and flower couple $\Qc$, define the quantities
\begin{equation}\label{jtflower}\widetilde{\Jc}_\Tc(t,s,k,k')=\bigg(\frac{\delta}{2L^{d-1}}\bigg)^m\zeta(\Tc)\sum_\Ds\epsilon_\Ds\int_{\Dc}\prod_{\nf\in\Nc}e^{\zeta_\nf\pi i \delta L^2\Omega_\nf t_\nf}\mathrm{d}t_\nf\cdot\dirac(t_{\ff^p}-s)\prod_{\ff\neq\lf\in\Lc}\sqrt{n_{\mathrm{in}}(k_\lf)}\eta_{k_{\lf}}^{\zeta_{\lf}}(\omega)\mathbf{1}_{k_\ff=k'},\end{equation}
\begin{equation}\label{kqflower}\widetilde{\Kc}_\Qc(t,s,k,k')=\bigg(\frac{\delta}{2L^{d-1}}\bigg)^{2m}\zeta(\Qc)\sum_\Es\epsilon_\Es\int_{\Ec}\prod_{\nf\in\Nc}e^{\zeta_\nf\pi i \delta L^2\Omega_\nf t_\nf}\mathrm{d}t_\nf\prod_{\ff}\dirac(t_{\ff^p}-s){\prod_{\ff\neq\lf\in\Lc^*}^{(+)}n_{\mathrm{in}}(k_\lf)}\mathbf{1}_{k_\ff=k'},\end{equation} which are slight modifications of (\ref{defjt}) and (\ref{defkq}) in the same way as (11.2) and (11.3) of \cite{DH21}. Here in (\ref{jtflower}), $\Ds$ is a $k$-decoration of $\Tc$, $\Dc$ is defined as in (\ref{defdomaind}), and the other objects are associated with the tree $\Tc$. In (\ref{kqflower}), $\Es$ is a $k$-decoration of $\Qc$, the other objects are associated with the couple $\Qc$, and the set $\Ec$ is defined as in (\ref{defdomaine}) but with $s$ replaced by $t$; the second product is taken over the two flower leafs $\ff$ and in the last product we assume $\lf$ has sign $+$ and is not one of the two flowers $\ff$ of the flower couple $\Qc$.

We now define the $\Rb$-linear operators $\Xs$ and $\Xs_m$ in Proposition \ref{mainprop2} such that its kernel
\begin{equation}\label{kernelx}(\Xs_m)_{kk'}^\zeta(t,s)=\sum_{\Tc}\widetilde{\Jc}_\Tc(t,s,k,k'),
\end{equation} where the sum is taken over all flower trees $\Tc$ such that $\Tc$ has order $m$, and the root $\rf$ and flower $\ff$ of $\Tc$ has signs $\zeta_\rf=+$ and $\zeta_\ff=\zeta$. Then, by multiplying out the definition (\ref{jtflower}), similar to the proof of Proposition 11.2 of \cite{DH21}, it is easy to see that the operators $\Ys_m$ and $\Ws_m$ defined in Proposition \ref{mainprop2} have the same expression as in (\ref{kernelx}), but with the sum taken over different sets of $\Tc$. Namely, in both cases we still require $\Tc$ has order $m$ and $\zeta_\rf=+$ and $\zeta_\ff=\zeta$, but in $\Ys_m$ we additionally require that (Y-1) the value $m>N$, (Y-2) the subtree rooted at each child node of $\rf$ has order $\leq N$. In $\Ws_m$ we additionally require that (W-1) the value $m>N$, (W-2) the subtree rooted at each of the two sibling nodes of $\ff$ has order $\leq N$, and (W-3) the flower tree obtained by replacing the parent $\ff^p$ of $\ff$ with a new flower has also order $\leq N$, see Figure \ref{fig:flowertree} for illustration. Note that the above requirement imposes that $N+1\leq m\leq 3N+1$ for both $\Ys_m$ and $\Ws_m$.
  \begin{figure}[h!]
  \includegraphics[scale=.45]{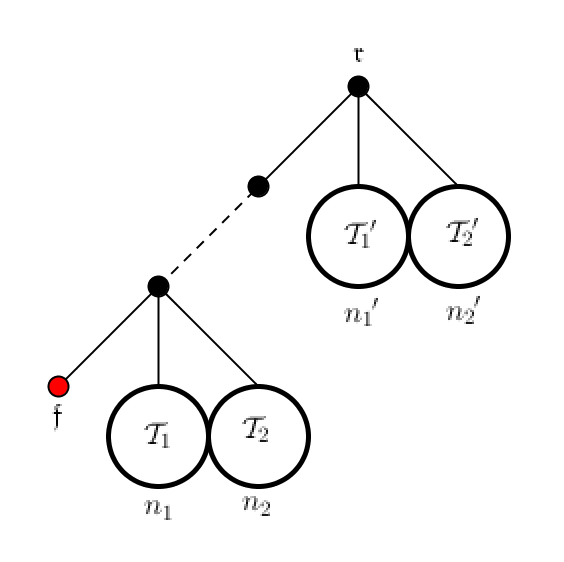}
  \caption{A flower tree $\Tc$ with root $\rf$ and flower $\ff$. Let the order of subtrees $\Tc_j$ be $n_j$ etc., and the order of $\Tc$ be $m$. Then the condition for $\Ys_m$ is that $N<m\leq N+n_1'+n_2'+1$ and $n_j'\leq N$, the condition for $\Ws_m$ is that $N<m\leq N+n_1+n_2+1$ and $n_j\leq N$.}
  \label{fig:flowertree}
\end{figure} 

By definitions (\ref{jtflower}) and (\ref{kqflower}), similar to the proof of Proposition 11.2 of \cite{DH21}, it is easy to show
\begin{equation}\label{flowerexp2}\Eb|(\Xs_m)_{kk'}^\zeta(t,s)|^2=\sum_{\Qc}\widetilde{\Kc}_\Qc(t,s,k,k'),\end{equation} where the sum is taken over all flower couples $\Qc=(\Tc^+,\Tc^-)$ such that both trees have order $m$, and the flower of tree $\Tc^{\pm}$ has sign $\pm\zeta$. The expressions for $\Ys_m$ and $\Ws_m$ are the same, except that both trees in $\Qc$ also have to satisfy the assumptions (Y-1)--(Y-2), or (W-1)--(W-3) above.

We now need to prove that the right hand side of (\ref{flowerexp2}) satisfies (\ref{mainest2}). Since the definition of $\widetilde{\Kc}_\Qc$ is almost the same as $\Kc_\Qc$ in (\ref{defkq}), the proof of (\ref{mainest2}) can also be done in almost the same way as in the proof of (\ref{mainest1}) in Proposition \ref{mainprop1}, with slight modifications due to the few differences between $\widetilde{\Kc}_\Qc$ and $\Kc_\Qc$ that lead to a loss of at most $L^{40d}$. In fact, this modification follows the same arguments in Section 11 of \cite{DH21}, with only two differences which we will discuss below.

The first difference concerns the vine cancellation (Proposition \ref{estbadvine}), which requires to group together couples that are full twists of each other (Definition \ref{twistgen}). Let [X], [Y] and [W] be the set of conditions posed on the flower tree $\Tc$ by $\Xs_m$, $\Ys_m$ and $\Ws_m$ respectively, as described above (see Figure \ref{fig:flowertree}). We only need to prove that, if $\widetilde{\Qc}$ is a full twist of $\Qc$, then both trees of $\widetilde{\Qc}$ satisfy [X] (or [Y] or [W]) if and only if both trees of $\Qc$ satisfy the same set of conditions. Note that each set of conditions only depend on the values of $(n_1,n_2,n_1',n_2')$ and $m$ as in Figure \ref{fig:flowertree}, we just need to show that each full twist (in fact each full unit twist at vine $\Vb$) does not change any of these values $n_j$ or $n_j'$. But this is obviously true, provided that the vine $\Vb$ in $\Qc_{\mathrm{sk}}$ does not contain any of the following ``special" nodes: the root $\rf$, and child of $\rf$, the flower $\ff$ (or a leaf in $\Qc_{\mathrm{sk}}$ such that the regular couple in $\Qc$ attached at it contains $\ff$), or the parent of $\ff$. In fact, in this case, suppose (for example) the vine $\Vb$ is like in Figure \ref{fig:vinescancel}, then the flower $\ff$ must belong to one of parts (A)--(D) in Figure \ref{fig:vinescancel}, so changing $\Qc$ to its full unit twist obviously does not change the values of $n_1$ and $n_2$; similarly it does not change the values of $n_1'$ and $n_2'$. Finally, as for the vines containing any of the special nodes, clearly the number of such vines does not exceed the number of special nodes which is less than $10$. Since each vine without exploiting cancellation only leads to loss of at most $L^{1/2}$, all these exceptional vines will lead to at most $L^5$ loss which is acceptable in view of (\ref{mainest2}).
\subsubsection{An extra argument with ladders} The second difference between the current proof and Section 11 of \cite{DH21} is as follows. Note that the factor $n_{\mathrm{in}}(k_\ff)$ is absent from (\ref{kqflower}) and replaced by $\mathbf{1}_{k_\ff=k'}$, so one does not have a decay factor in $k_\ff$, instead this $k_\ff$ is equal to a fixed vector $k'$. In fact, for any node $\nf$ on the stem, we do not have a decay factor in $k_\nf$, but instead a decay factor in $k_\nf-k'$ (or equivalently $k_\nf-k$). Note that the shift $k'$ is \emph{fixed but may be arbitrarily large}; fortunately most of the proof in the previous sections is translation invariant. In fact, the only part that is not translation invariant is the proof of the $L^1$ bound (\ref{1stsumbound}) in Section \ref{stage1red}. Here, in that proof we are using the fact that $\Omega_{\nf_1}\pm\Omega_{\nf_2}=r_j\cdot (k_\mf\pm k_{\mf'})$ to control the number of possibilities for $\sigma_{\nf_1}\pm\sigma_{\nf_2}$, where $|r_j|\sim P_j$ is the gap of a given ladder, $(\nf_1,\nf_2)$ is a pair of branching nodes corresponding to two atoms in the ladder connected by a double bond, and $\sigma_{\nf_j}=\lfloor\delta L^{2\gamma}\Omega_{\nf_j}\rfloor$. Now if $|r_j|\sim P_j$ and each $k_\mf$ belongs to a unit ball \emph{centered at $0$}, then we always have $|\sigma_{\nf_1}\pm\sigma_{\nf_2}|\leq \delta L^{2\gamma}P_j$; however, if $k_\mf$ belongs to a unit ball \emph{centered at $k'$}, then with $r_j$ also allowed to vary, we can no longer restrict $\sigma_{\nf_1}\pm\sigma_{\nf_2}$ to a fixed interval of length $\delta L^{2\gamma}P_j$, but only have $\sigma_{\nf_1}\pm\sigma_{\nf_2}=\kappa\delta L^{2\gamma}(r_j\cdot k')+O(\delta L^{2\gamma}P_j)$ with $\kappa\in\{-2,-1,0,1,2\}$, which may cause problems in counting the number of possibilities for $\sigma_{\nf_1}\pm\sigma_{\nf_2}$.

This issue is resolved by examining the proof of (\ref{1stsumbound}) in Section \ref{stage1red}, which follows the same arguments as in Sections 10.1--10.2 of \cite{DH21}. By going through the combinatorial arguments in Section 10.1 of \cite{DH21}, we can show that, apart from at most $O(\rho)$ counterexamples which are negligible (where $\rho=\rho_{\mathrm{sub}}$ is defined in Proposition \ref{kqmainest2}), for each double bond connecting two atoms in a given ladder, either (1) both bonds are LP bonds (see Definition \ref{defmole}), or (2) one bond is LP, the other bond is PC, and each of the two pairs of parallel single bonds at these two atoms is also one LP and one PC. Moreover, we may assume the gap $|r_j|\sim P_j\lesssim 1$, otherwise the power gain from (\ref{countc1-1}) easily covers the log loss. Then, in case 2, for each of the four single bonds $\ell$, the decoration $k_\ell$ is \emph{not} shifted by $k'$; indeed, this is true for the LP bond $\ell$ since $k_\ell=k_\lf$ for some leaf $\lf$ which may not belong to the stem apart from at most one counterexample, and is also true for the parallel PC bond due to the assumption $|r_j|\lesssim 1$. Then, we would have $|\sigma_{\nf_1}\pm\sigma_{\nf_2}|\lesssim \delta L^{2\gamma}P_j$ again, so the proof in Section \ref{stage1red} still goes through.

It remains to consider case 1 of an LP-LP double bond. In this case we will assume $\sigma_{\nf_1}\pm\sigma_{\nf_2}=A+O(\delta L^{2\gamma}P_j)$, where $A$ is a quantity that depends only on the gap $r_j$ of the ladder. By repeating the arguments in the proof of (\ref{1stsumbound}) before, we may reduce to the case $\sigma_{\nf_1}\pm\sigma_{\nf_2}=A+\mu$ with at most $C^n$ loss, where $\mu$ is now a fixed integer. We next classify the possibilities of $A$; recall that in the setting of Section 10 of \cite{DH21}, each $\Omega_\nf$ (and hence $\sigma_\nf$) variable belongs to a fixed set which is the union of at most $L^{10d}$ unit intervals, so $A$ has $\lesssim L^{20d}$ choices. For each fixed $A$, the arguments in Section 10.2, case 1 of \cite{DH21} implies that the left hand side of (\ref{1stsumbound}) is bounded by the same expression but without the $(\sigma_{\nf_1},\sigma_{\nf_2})$ variables, multiplied by a factor of
\begin{equation}\label{smallloss}\int_\Rb\frac{1}{\langle \alpha\rangle}\cdot\frac{1}{\langle \alpha+A-\mu\rangle}\,\mathrm{d}\alpha\lesssim\frac{\log (2+|A-\mu|)}{\langle A-\mu\rangle}.\end{equation} Once $A$ is fixed, the whole ladder can be treated in the same way as in \cite{DH21} without further loss; finally we sum in different choices of $A$, and summing up the factor in (\ref{smallloss}) yields a factor $\lesssim (\log L)^2$, which is acceptable. This completes the proof of Proposition \ref{mainprop2}.
\subsection{Proof of Theorem \ref{main}}\label{linoper2} In this subsection we prove Theorem \ref{main}. Note that with Proposition \ref{mainprop1} is proved, by (\ref{mainest1}) and (\ref{correlation}) we have that \[\Eb|(\Jc_n)_k(t)|^2\lesssim\langle k\rangle^{-20d}(C^+\sqrt{\delta})^n\]for any $0\leq n\leq N^3$. This plays the role of Proposition 2.5 in \cite{DH21}, while Propositions \ref{mainprop2} and \ref{mainprop4} play the roles of Propositions 2.6 and 2.7 in \cite{DH21}.

Therefore, we may repeat the arguments in Section 12 of \cite{DH21} to control the remainder term $\textit{\textbf{b}}$. Note that, strictly speaking, we are actually applying the version of this argument in Section 4 of \cite{DH21-2}, because here we have $N=\lfloor(\log L)^4\rfloor$ as in \cite{DH21-2}, but the proof can be easily adapted. Another difference here concerns the invertibility of $1-\Ls$, which follows from Proposition \ref{mainprop2}. In fact, the same proof as in Section 12 of \cite{DH21} (but with $N=\lfloor(\log L)^4\rfloor$ and with a corollary in the form of Corollary 11.3 of \cite{DH21}, which follows from the same proof) yields that
\[\|\Xs_m\|_{Z\to Z}+\|\Ys_m\|_{Z\to Z}+\|\Ws_m\|_{Z\to Z}\lesssim (C^+\sqrt{\delta})^{n/2}L^{60d}\] with probability $\geq 1-e^{-(\log L)^2}$, which plays the role of Proposition 12.2 of \cite{DH21}. This then implies that
\[\|\Xs\|_{Z\to Z}\leq L^{61d},\quad \|\Ys-1\|_{Z\to Z}+\|\Ws-1\|_{Z\to Z}\leq 1/2.\] But $\Ys=(1-\Ls)\Xs$ and $\Ws=\Xs(1-\Ls)$, so the invertibility of \emph{both} $\Ys$ and $\Ws$ by Von Neumann series, implies that $1-\Ls$ has both left and right inverse, hence it is invertible. In particular
\[\|(1-\Ls)^{-1}\|_{Z\to Z}\leq \|\Xs\|_{Z\to Z}\cdot\|\Ys^{-1}\|_{Z\to Z}\leq L^{62d}.\] The rest of the proof in Section 12 of \cite{DH21} then carries along, which allows one to control the remainder term $\textit{\textbf{b}}$, and complete the proof of Theorem \ref{main}.
\appendix
\section{Auxiliary results}\label{aux}
\subsection{Counting estimates}\label{basiccounting} We collect the vector counting estimates used in Sections \ref{reduct2} and \ref{lgmole}.
\begin{lem}\label{basiccount0pre} (1) Let $A\subset\Rb^2$ be the intersection of an annulus with radius $Q$ and width $Q^{-1}\ll\rho\ll Q$, and a disc of radius $L$. Then we have
\begin{equation}\label{disclattice}\#(A\cap\Zb^2)\lesssim\min\big(Q\rho+Q^{7/11},(Q\rho)^{1/2}+L(Q\rho)^{3/2}Q^{-1}+L(Q\rho)^{1/2}Q^{-1/3}\big).
\end{equation}

(2) Fix two dyadic numbers $1\ll\rho\ll L$ and $1\lesssim\theta\ll \rho$, let $A,B\subset \Rb^d$ be two balls of radius $L$ and $\alpha\in\Rb$, then we have
\begin{equation}\label{estcross}
\#\bigg\{x\in A\cap \Zb^d:\sup_{B,\alpha}\# \{y\in B\cap\Zb^d:|x\cdot y-\alpha|\lesssim\rho\}\gtrsim L^{d-1}\theta^{-1}\bigg\}\lesssim \rho^d+L^{d-2}(\log L)^C\rho^2\theta^2.
\end{equation}
\end{lem}
\begin{proof} (1) The first upper bound $Q\rho+Q^{7/11}$ follows from ignoring the disc or radius $L$ and using the error term bound of Huxley \cite{Hux93} for counting the number of lattice points inside a disc (or ellipse). Now let us consider the second upper bound.

Clearly $A$ is contained in an annulus section of radius $\sim Q$, width $\rho$ and angle $\alpha\lesssim LQ^{-1}$. We may divide it into at most $1+\alpha/\theta$ annulus sections of angle $\theta$ where $\theta\ll\min(Q^{-2/3},(Q\rho)^{-1})$. Note that it is easily calculated by elementary geometry, that the area of the convex hull of each smaller annulus section is $\lesssim Q\rho\theta+Q^2\theta^3\ll 1$, so it may not contain any three lattices points that are not collinear. On the other hand, the length of any line segment contained in the whole annulus is clearly $\lesssim(Q\rho)^{1/2}$, thus
\[\#(A\cap \Zb^d)\lesssim (1+LQ^{-1}\max(Q^{2/3},Q\rho))\cdot(Q\rho)^{1/2},\] which implies (\ref{disclattice}).

(2) Assume $|x^1|=\max |x^j|$ where $x_j$ are the coordinates of $x$. If $|x^1|\lesssim \rho$ then we get the trivial upper bound $\rho^d$; suppose now $|x^1|\gg \rho$. In the set of $y$ defined in (\ref{estcross}), we may fix the last $d-2$ coordinates, by pigeonhole principle and translation, up to a constant factor, we will have
\[\#\{y\in B(0,L)\cap \Zb^2:|x^1y^1+x^2y^2|\lesssim\rho\}\gtrsim L\theta^{-1}.\] Note that for any $y^2$ there exists at most one $y^1$ satisfying the above inequality, by pigeonhole principle again we see that there exist $(a^1,a^2)$ such that $0<|a^2|\lesssim \theta$ and $|a^1x^1+a^2x^2|\lesssim\rho$. In the same way, if we first fix the coordinates $(y^2,y^4,\cdots y^d)$ of $y$, we also get that $|b^1x^1+b^3x^3|\lesssim\rho$ for some $(b^1,b^3)$ with $0<|b^3|\lesssim\theta$.

Now let $|x^j|\sim X_j$ etc., then the number of choices for $(x^1,a^2,b^3)$ is at most $LA_2B_3$. When they are fixed, we must have $|a^1x^1|\lesssim |a^2x^2|+\rho$ etc., so the number of choices for $(a^1,b^1)$ is at most $(1+A_2X_2/X_1)(1+B_3X_3/X_1)$. Finally when $x^1$ and $(a^j,b^j)$ are fixed, the number of choices for $(x^2,\cdots,x^d)$ is at most $\rho^2A_2^{-1}B_3^{-1}L^{d-3}$, so the number of choices for $x$ is at most
\[\sum_{A_2,B_3,X_1,X_2,X_3}LA_2B_3\cdot\bigg(1+\frac{A_2X_2}{X_1}\bigg)\bigg(1+\frac{B_3X_3}{X_1}\bigg)\cdot \rho^2A_2^{-1}B_3^{-1}L^{d-3}\lesssim L^{d-2}(\log L)^C\rho^2\theta^2,\] noticing that each dyadic variable has at most $\log L$ choices in the summation. This completes the proof.
\end{proof}
\begin{lem}\label{basiccount0} Fix $\alpha,\beta\in\Rb$ and $r,v\in\Zb_L^d$, such that $|r|\sim P$ with $P\in[L^{-1},1]\cup\{0\}$ (with $|r|\gtrsim 1$ if $P=1$). Let each of $(x,y,z)$ belong to $\Zb_L^d$ intersecting a fixed unit ball, assume $\xi\in\{x,y,x+y\}$ and $\zeta\in\{x-v,x+y-v,x+y-v-z\}$, and define $\Xf:=\min((\log L)^2,1+\delta L^{2\gamma}P)$ as in Proposition \ref{kqmainest2}. Consider the counting problems
\begin{equation}\label{basiccount01}\big\{(x,y):|x\cdot y-\alpha|\leq \delta^{-1}L^{-2\gamma},\,\, |r\cdot\xi-\beta|\leq \delta^{-1}L^{-2\gamma}\big\}:=\Cf_1,\end{equation}
\begin{equation}\label{basiccount02}
\big\{(x,y,z):|x\cdot y-\alpha|\leq \delta^{-1}L^{-2\gamma},\,\, |\zeta\cdot z-\beta|\leq \delta^{-1}L^{-2\gamma}\big\}:=\Cf_2.
\end{equation}

(1) We have 
\begin{equation}\label{countc1-1}\Cf_1\lesssim\left\{
\begin{aligned}&\delta^{-1}L^{2(d-\gamma)}(1+\delta L^{2\gamma}P)^{-1},&&\mathrm{if}\,\,\gamma\leq1/2\,\,\mathrm{or}\,\,P=0\\
& \delta^{-2}L^{2(d-\gamma)}\min((LP)^{-1}+L^{-(2-2\gamma)},L^{-10\eta}),&&\mathrm{if}\,\,\gamma>1/2\,\,\mathrm{and}\,\,P\neq 0
\end{aligned}
\right\}\lesssim\delta^{-1}L^{2(d-\gamma)}\Xf^{-1}.
\end{equation} 

(2) If $\gamma>(4/5)-10\eta$ and $P\neq 0$, then we have $\Cf_1\lesssim\delta^{-2} L^{2(d-\gamma)-(1-\gamma)-20\eta}$.

(3) We have \begin{equation}\label{countc2-1}\Cf_2\lesssim\left\{
\begin{aligned}&\delta^{-2}L^{3d-4\gamma},&&\mathrm{if}\,\, \gamma\leq 1/2\\
&\delta^{-2}L^{3(d-\gamma)-\gamma_0-10\eta},&&\mathrm{if}\,\,\gamma>1/2
\end{aligned}
\right\}\lesssim\delta^{-2}L^{3(d-\gamma)-\gamma_0}.
\end{equation}
\end{lem}
\begin{proof} (1) First we have $\Cf_1\lesssim \delta^{-1}L^{2(d-\gamma)}$ using only the first inequality $|x\cdot y-\alpha|\leq \delta^{-1}L^{-2\gamma}$, due to the classical counting estimates for $\gamma=1$ proved in \cite{DH21} (see Proposition 6.1 and Lemma A.9 (2) in \cite{DH21}, where the upper bound also holds for square torus. This already implies (\ref{basiccount01}) if $P=0$, so below we will assume $P\neq 0$; moreover by subdividing the conditions in (\ref{basiccount01}) and (\ref{basiccount02}) we can also $\delta=1$ (same for (2) and (3) below).

Next we prove that $\Cf_1\lesssim L^{2d-1-2\gamma}P^{-1}+L^{2d-2}$ when $\gamma>1/2$, and $\Cf_1\lesssim L^{2d-4\gamma}P^{-1}$ when $\gamma\leq 1/2$. In fact, assume either $\xi=x$ and $|x|\sim R$ (with $|x|\gtrsim 1$ if $R=1$, same below) or $\xi=x+y$ and $|x-y|\sim R$, then the number of choices for $\xi$ is $\lesssim (LR)^{d-1}(1+L^{1-2\gamma}P^{-1})$ but with $(LR)^{d-1}$ replaced by $L^{d-1}$ if $\xi=x+y$, and the number of choices for $y$ with $\xi$ fixed is $\lesssim L^{d-1}(1+L^{1-2\gamma}R^{-1})$ but with $L^{d-1}$ replaced by $(LR)^{d-1}$ if $\xi=x+y$. In either case
\begin{equation}\label{4vecnew}\Cf_1\lesssim (L^2R)^{d-1}(1+L^{1-2\gamma}P^{-1})(1+L^{1-2\gamma}R^{-1});\end{equation} by summing over $R$, we get the desired bound both when $\gamma>1/2$ and when $\gamma\leq 1/2$.

We now need to prove that $\Cf_1\lesssim L^{2(d-\gamma)}L^{-10\eta}$ when $\gamma>1/2$. Recall from the choice of $\eta$ that $0<\eta\ll \gamma-1/2$. If $R\leq L^{-30\eta}$ or $LP\geq L^{30\eta}$ then the desired bound already follows from (\ref{4vecnew}), so we will assume $R>L^{-30\eta}$ and $LP<L^{30\eta}$.  If $\xi=x+y$, then the number of choices for $\xi$ is $\lesssim L^{d+1-2\gamma+O(\eta)}$; when $\xi$ is fixed, the number of choices for $x$, using a rescaled version of (\ref{disclattice}), is bounded by
\[L^{d-2}\cdot\min\big(L^{2-2\gamma}+Q^{7/11},L^{1-\gamma}+L^{4-3\gamma}Q^{-1}+L^{2-\gamma}Q^{-1/3}\big)\] for some value $Q$ (with $\rho\sim L^{2-2\gamma}$); an easy calculation implies that the above quantity is always $\lesssim L^{d-1-\min(10^{-3},2\gamma-1)}$ for any $Q$ and any $\gamma>1/2$, which gives the desired result.

Assume now that $\xi=x$. In the arguments leading to (\ref{4vecnew}), we can choose a dyadic variable $\theta\geq 1$ and assume that $x$ is such that the number of choices for $y$ is $\sim L^{d-1}\theta^{-1}$. We may assume $\theta\leq L^{300\eta}$, or the desired bound again follows from the improved version of (\ref{4vecnew}); then, by a rescaled version of (\ref{estcross}) with $\rho\sim L^{2-2\gamma}$, we can bound the number of choices for $x$ by $L^{d-2+2(2-2\gamma)+O(\eta)}$. This, combined with the number of choices for $y$ above, again yields the desired bound.

(2) The proof is similar to the last parts of (1). If $\xi=x+y$, then the number of choices for $\xi$ is $\lesssim L^{d-2\gamma+1+O(\eta)}$. For fixed $\xi$, the number of choices for $x$ is bounded, using (\ref{disclattice}), by
\[L^{d-2+O(\eta)}\cdot\min\big(L^{2/5}+Q^{7/11},L^{1/5}+L^{8/5}Q^{-1},L^{6/5}Q^{-1/3}\big),\] where the $\min(\cdots)$ factor is bounded by $L^{(4/5)-10^{-4}}$ by a simple calculation, hence the desired bound (recall $\gamma>(4/5)-10\eta$).

Now if $\xi=x$, then we may choose a dyadic variable $\theta$ and assume that $x$ is such that the number of choices for $y$ is $\sim L^{d-1}\theta^{-1}$. If $1\lesssim\theta\ll L^{2-2\gamma}$, by the second equation in (\ref{basiccount01}) and (\ref{estcross}), we can bound the number of choices for $x$ by
$\min\big(L^{d-2\gamma+1+O(\eta)},L^{d-4\gamma+2+O(\eta)}\theta^2\big)$, and hence
\[\Cf_1\lesssim\sum_{\theta\geq 1}L^{d-1+O(\eta)}\theta^{-1}\cdot \min\big(L^{d-(2\gamma-1)},L^{d-(4\gamma-2)}\theta^2\big)\lesssim L^{2d-3\gamma+(1/2)+O(\eta)},\] which suffices when $\gamma>(4/5)-10\eta$. If $\theta\gtrsim  L^{2-2\gamma}$ we use only the first bound above which suffices, and if $\theta\ll 1$ then necessarily $|x|\lesssim L^{2-2\gamma}\leq L^{(2/5)+O(\eta)}$ and the estimate is easily proved.

(3) If $\zeta\in\{x+y-v,x+y-v-z\}$, then we first fix $(x+y,z)$ which has $\lesssim L^{2(d-\gamma)}$ choices, and then count $(x,y)$. If $\gamma\leq 1/2$ the number of choices is easily seen to be $L^{d-2\gamma}$ which suffices. If $\gamma>1/2$, then the same argument as in the proof of (1), using (\ref{disclattice}), shows that the number of choices for $(x,y)$ is at most $L^{d-1-\min(10^{-3},2\gamma-1)}$, which suffices due to the choice of $\eta$ so that $0<\eta\ll \gamma-1/2$.

Now suppose $\xi=x-v$, and let $|\xi|\sim R$, then with $(x,y)$ fixed, the number of choices for $z$ is $\lesssim L^{d-1}(1+L^{1-2\gamma}R^{-1})$; moreover a simple variation of Proposition 6.1 and Lemma A.9 (2) of \cite{DH21} allows to control the number of choices for $(x,y)$ by $L^{2-2\gamma}L^{d-1}(LR)^{d-1}$, hence we get
\begin{equation}\label{bdc2new}\Cf_2\lesssim (L^2R)^{d-1}L^{2-2\gamma}L^{d-1}(1+L^{1-2\gamma}R^{-1})\lesssim L^{3(d-\gamma)-\gamma_0}.\end{equation} This is already enough if $\gamma\leq 1/2$. If $\gamma>1/2$ we need to gain extra $L^{-10\eta}$, which is provided by (\ref{bdc2new}), unless $R\geq L^{-300\eta}$ (and similarly also $|x|\geq L^{-300\eta}$). In the latter case, we choose two dyadic variables $1\leq\theta_1,\theta_2\leq L^{300\eta}$ and assume for $x$ that with this fixed $x$, the number of choices for $y$  is $\sim L^{d-1}\theta_1^{-1}$ and the number of choices for $z$ is $\sim L^{d-1}\theta_2^{-1}$. By (\ref{estcross}), the number of choices for $x$ is at most $L^{d-2+O(\eta)}L^{2(2-2\gamma)}\min(\theta_1,\theta_2)^2$, and putting together we get $\Cf_2\lesssim L^{3d-4\gamma+O(\eta)}$.
\end{proof}
Now let $k_j\,(j=1,2,\cdots)$ be vector variables such that $k_j\in\Zb_L^d$ and $|k_j-k_j^0|\leq 1$ for some fixed values $k_j^0$. For any tuple $(j_1^{\epsilon_1},\cdots,j_r^{\epsilon_r})$ with $r\leq 4$ and $\epsilon_i\in\{\pm\}$, we associate with it a system
\begin{equation}\label{counteqn}
\sum_{i=1}^r\epsilon_ik_{j_i}=k,\quad \bigg|\sum_{i=1}^r\epsilon_i|k_{j_i}|^2-\beta\bigg|\leq\delta^{-1}L^{-2\gamma},
\end{equation} where $k\in\Zb_L^d$ and $\beta\in\Rb$ are fixed (we label different tuples differently, and do not require $k=0$ when $r=4$). If some $j_i$ in the tuple does not come with an $\epsilon_i$, we understand that the sign of $k_{j_i}$ in (\ref{counteqn}) can be arbitrary, but subject to the restrictions stated below. Each counting problem is represented by a set of tuples (such as $\{(1^+,2^-)\}$ or $\{(1,2,3),(1,4,5)\}$); we require that no tuple contains $3$ elements of same sign, each $j$ appears in at most two tuples, and the signs of $k_j$ in the two tuples must be opposite. Denote the corresponding number of solutions by $\Cf$.
\begin{lem}\label{basiccount} We have the following bounds for the number of solutions to counting problems.
\begin{enumerate}
\item For $\{(1^+,2^+)\}$ we have $\Cf\lesssim \delta^{-1}L^{d-\gamma-\gamma_0}$.

\noindent For $\{(1^+,2^-)\}$, let $|k_1-k_2|\sim R\in[L^{-1},1]$ (with $|k_1-k_2|\gtrsim R$ when $R=1$), then we have $\Cf\lesssim L^{d-1}+\delta^{-1}\min(R^{-1}L^{d-2\gamma},L^d)$.
\item For $\{(1,2,3)\}$ we have $\Cf\lesssim \delta^{-1}L^{2(d-\gamma)}$.
\item For $\{(1,2,3^+),(1,2,4^+)\}$ or $\{(1,2,3),(3^+,4^+)\}$ we have $\Cf\lesssim \delta^{-2}L^{2(d-\gamma)-\gamma_0}$.

\noindent For $\{(1,2,3^+),(1,2,4^-)\}$ or $\{(1,2,3),(3^+,4^-)\}$, let $|k_3-k_4|\sim P\in[L^{-1},1]\cup\{0\}$ (with $|k_3-k_4|\gtrsim P$ when $P=1$), then $\Cf$ satisfies the same bounds as in (\ref{countc1-1}) and Lemma \ref{basiccount0} (2), where $\Xf:=\min((\log L)^2,1+\delta L^{2\gamma}P)$.
\item For $\{(1,2,3),(1,4,5)\}$ or $\{(1,2,3),(1,2,4,5)\}$, $\Cf$ satisfies the same bounds as in (\ref{countc2-1}).
\item For $\{(1,2,4),(2,3,5),(3,4,6)\}$ we have $\Cf\lesssim \delta^{-3}L^{3(d-\gamma)-1.01\gamma_0}$.
\end{enumerate}
\end{lem}
\begin{proof} In all proofs we may assume $\delta=1$. The first half of (1) follows by assuming $k_1+k_2$ is fixed and $|k_1-k_2|\sim R$, then the number  of choices for $(k_1,k_2)$ is at most $\delta^{-1}(LR)^{d-1}(1+L^{1-2\gamma}R^{-1})\lesssim \delta^{-1}L^{d-\gamma-\gamma_0}$. The second half of (1) follow similarly, by first choosing the first $d-1$ coordinates of $k_1$ (assuming the $d$-th coordinate of $k_1-k_2$ is $\sim R$) and then considering the last coordinate. Next, (2) follows from Proposition 6.1 and Lemma A.9 (2) in \cite{DH21}. For the first part of (3), if the signs associated with $k_1$ and $k_2$ are the same then it follows from applying (1) twice (first for $(k_3,k_4)$ and then for $(k_1,k_2)$); otherwise, we can assume $|k_1-k_2|\sim R$, then once $(k_3,k_4)$ is fixed, the number of choices for $(k_1,k_2)$ is $\lesssim L^{d-1}(1+L^{1-2\gamma}R^{-1})$, while the number of choices for $(k_3,k_4)$ is now $\lesssim (LR)^{d-1}(1+L^{1-2\gamma})$, which gives $\Cf\lesssim L^{2d-2}(1+L^{1-2\gamma})^2\lesssim L^{2(d-\gamma-\gamma_0)}$ which proves the first part of (3). Moreover, the second part of (3) follows from Lemma \ref{basiccount0} (1) and (2), and (4) follows from Lemma \ref{basiccount0} (3), after a suitable reparametrization using the factorization $|k_1|^2-|k_2|^2+|k_3|^2-|k|^2=2(k_1-k)\cdot (k_3-k)$ when $k_1-k_2+k_3=k$ (and the case $\{(1,2,3),(1,2,4,5)\}$ follows as a consequence).

Now let us prove (5). We may assume (say) the signs corresponding to $k_3$ and $k_5$ in the triple $(2,3,5)$ are the same, and the signs corresponding to $k_3$ and $k_6$ in the triple $(3,4,6)$ are the opposite. We will first fix $(k_1,k_2,k_4)$, which has $\lesssim L^{2-2\gamma}(L^2R)^{d-1}$ choices assuming $|k_3-k_6|\sim R$, due to a variation of Proposition 6.1 and Lemma A.9 (2) of \cite{DH21}; then $k_3+k_5$ and $k_3-k_6:=r$ are fixed, and $|k_3|^2+|k_5|^2$ and $k_3\cdot r$ are fixed up to distance $L^{-2\gamma}$. We may assume $r\neq 0$ (otherwise the bound is trivial), and let $q:=k_3-k_5$, then $q\in\Zb_L^d$ belongs to a fixed unit ball, and both $|q|^2$ and $r\cdot q$ are fixed up to distance $L^{-2\gamma}$. We then fix the first $d-2$ coordinates of $q$ (which has $\lesssim L^{d-2}$ choices) and reduce to counting the number of $(x,y)\in \Zb_L$ such that
\[|x-x_0|\leq 1,\,\,|y-y_0|\leq 1;\qquad x^2+y^2=\alpha+O(L^{-2\gamma}),\,\,ax+by=\beta+O(L^{-2\gamma})\] where $x_0,y_0,a,b,\alpha,\beta\in\Rb$ are fixed and $|a|+|b|\sim R$. Assuming also $|x|+|y|\sim M$, we can reduce this lattice point counting estimate to a area counting estimate by considering the $O(L^{-1})$ neighborhood of these lattice points, so that $x^2+y^2=\alpha+O(L^{-2\gamma}+ML^{-1})$ and $ax+by=\beta+O(L^{-2\gamma}+RL^{-1})$. By doing a rotation to calculate the area, we can easily prove that\[\#\mathrm{\ of\ choices\ for\ }(x,y)\lesssim R^{-1}L^2(L^{-1}+L^{-2\gamma})(L^{-\gamma}+M^{1/2}L^{-1/2}),\] which is enough provided $M\leq L^{1-\gamma_0/10}$. Finally, if $M\geq L^{1-\gamma_0/10}$ then the single condition $x^2+y^2=\alpha+O(L^{-2\gamma})$ suffices, using (\ref{disclattice}), to show that the number of choices for $(x,y)$ is $\lesssim L^{\max(1,2-2\gamma)-\gamma_0/10}$, which implies the desired result.
\end{proof}
\subsection{Miscellaneous}\label{misc} We collect some miscellaneous results.
\begin{lem}\label{explem} Given a tree $\Tc$ of order $n$, consider the collection of all decorations $(k_\nf)$, such that $|k_\lf-k_\lf^0|\leq 1$ for each leaf node $\lf$, where $k_\lf^0\in\Zb_L^d$ are fixed vectors for each leaf $\lf$. Then:
\begin{enumerate}
\item This collections of decorations can be divided into at most $C^n$ sub-collections, such that for any decoration $(k_\nf)$ in a given sub-collection, each $k_\nf$ belongs to a fixed unit ball that depends only on $(k_\nf^0)$, $\nf$ and the sub-collection but not $(k_\nf)$ itself.
\item This collections of decorations can be divided into at most \[C^n\prod_{\lf\in\Lc}\langle k_{\lf}^0\rangle^{4d}\] sub-collections, such that for any decoration $(k_\nf)$ in a given sub-collection, and any vector $r\in\Rb^d$ with $|r|\leq 1$, the value $r\cdot k_\nf\in\Rb$ belongs to a fixed unit interval that depends only on $(k_\nf^0)$, $\nf$ and the sub-collection but not $(k_\nf)$ itself and not on $r$.
\end{enumerate}
\end{lem}
\begin{proof} (1) The key tool is Lemma 6.6 of \cite{DH21}. It implies that: for any branching node $\nf\in\Tc$, there exists $\nf'$ which is a child of $\nf$, and a ball $B_\nf$ which only depends on $(k_\nf^0)$ and $\nf$, such that
\[\prod_{\nf}\rho(B_\nf)\leq \frac{3^n}{2n+1}\leq C^n\] with $\rho(B_\nf)\geq 1$ being the radius of $B_\nf$, and for any decoration satisfying $|k_\lf-k_\lf^0|\leq 1$ for each leaf $\lf$, we must have that $k_{\nf}\pm k_{\nf'}\in B$. This implies that $\lfloor k_\nf\rfloor\pm\lfloor k_{\nf'}\rfloor$ is an integer vector in a fixed ball of radius $\rho(B_\nf)+1\leq 2\rho(B_\nf)$, where $\lfloor k_\nf\rfloor$ denotes the point whose coordinates are the integer parts of coordinates of $k_\nf$. Since the values of $\lfloor k_\nf\rfloor\pm\lfloor k_{\nf'}\rfloor$ for all branching nodes $\nf$ and the values $\lfloor k_\lf\rfloor$ for all leaves $\lf$ uniquely determine $\lfloor k_\nf\rfloor$ for all nodes $\nf$, we know that the collection $(\lfloor k_\nf \rfloor)_{\nf\in\Tc}$ has at most $C^n$ possible choices, hence the result.

(2) By running a similar (inductive) proof as in Lemma 6.6 of \cite{DH21}, we can choose a child $\nf'$ of each branching node $\nf$ as in (1), and the corresponding ball $B_\nf$, but we require that $B_\nf$ be centered at the origin, and the radii satisfy that $\rho(B_\nf)\geq 1$ and
\[\prod_{\nf}\rho(B_\nf)\leq \frac{3^n\prod_{\lf\in\Lc}(|k_\lf^0|+1)}{\sum_{\lf\in\Lc}(|k_\lf^0|+1)}\leq C^n\prod_{\lf\in\Lc}\langle k_\lf^0\rangle.\] This implies that for any $|r|\leq 1$, we have $|r\cdot k_\nf\pm r\cdot k_{\nf'}|\leq \rho(B_\nf)$, hence $\lfloor r\cdot k_\nf\rfloor\pm \lfloor r\cdot k_{\nf'}\rfloor$ is an integer of absolute value $\leq 2\rho(B_\nf)$. Also for each leaf $\lf$ we have that $\lfloor r\cdot k_\lf\rfloor$ is an integer of absolute value $\leq 2\langle k_\lf^0\rangle$, so in the same way as in (1), the collection $(\lfloor r\cdot k_\nf \rfloor)_{\nf\in\Tc}$ has at most $C^n\prod_{\lf\in\Lc}\langle k_{\lf}^0\rangle^{2}$ choices as $r$ and the decoration $(k_\nf)$ vary. This completes the proof.
\end{proof}
\begin{prop}\label{subsetvc} Suppose a vine-like object $\Ub_1$, is contained in another vine-like object $\Ub_2$, and they do not have any vine as a common ingredient. Then $\Ub_1$ is either one double bond, or the adjoint of vine (V) (which is an HV), or the adjoint of two double bonds (which is an HVC). In any case, $\Ub_2$ still remains connected after removing $\Ub_1$ or the VC which $\Ub_1$ is the adjoint of.
\end{prop}
\begin{proof} If $\Ub_1$ and $\Ub_2$ do not have any vine as a common ingredient, then $\Ub_1$ has to be contained in a single ingredient of $\Ub_2$, which is a single vine (I)--(VIII) as shown in Figure \ref{fig:vines}. The result is then self-evident by examining the structures of these vines, and we omit the proof.
\end{proof}
\begin{lem}\label{timeFourierlem} (1) Let $\Tc$ be a ternary tree, and denote by $\Nc$ the set of branching nodes. Consider 
	\begin{equation}\label{modifiedtimeint2}
	\Uc_{\Tc}(t,\alpha[\Nc])=\chi_0(t)\int_{\widetilde{\Dc}}\prod_{\nf\in\Nc}e^{\pi i\alpha_\nf t_\nf}\,\mathrm{d}t_\nf,
	\end{equation} where the domain $\widetilde{\Dc}=\big\{t[\Nc]:0<t_{\nf'}<t_\nf<t\mathrm{\ whenever\ }\nf'\mathrm{\ is\ a\ child\ node\ of\ }\nf\big\}$.
	
	For every choice of $d_{\nf}\in \{0, 1\}\,(\nf\in \Nc)$, we define  $q_{\nf}$ for $\nf \in \Nc$ inductively as follows: Set $q_{\nf}=0$ if $\nf$ is a leaf, and otherwise define $q_{\nf}=\alpha_\nf+d_{\nf_1}q_{\nf_1}+d_{\nf_2}q_{\nf_2}+d_{\nf_3}q_{\nf_3}$ where $\nf_1, \nf_2, \nf_3$ are the three children of $\nf$. The following estimate holds:
	\begin{equation}\label{timeintlemma1}
	|\widehat \Uc_\Tc (\tau,\alpha[\Nc])| \leq C^n\sum_{d_\nf\in \{0, 1\}} \langle \tau-d_\rf q_\rf\rangle^{-10}\prod_{\nf \in \Nc}\frac{1}{\langle q_\nf\rangle }.
	\end{equation}

(2) Suppose $F:\Rb\to \Rb$ is in $X^\alpha(\Rb)$ for some $\alpha<1$, then for any $\epsilon>0$, there holds that $\chi_0(t_1)\chi_0(t_2)F(\max(t_1, t_2)) \in X^{\alpha-\epsilon}(\Rb^2)$.
\end{lem}

\begin{proof}
The proof of (1) is contained in Proposition 2.3 of \cite{DH19}. Note that the result can be extended to the case of multiple trees, which will be used in the proof of Propositions \ref{estbadvine} and \ref{estnormalvine}. For (2), we expand the Fourier transform of $G(t_1, t_2)=\chi_0(t_1)\chi_0(t_2)F(\max(t_1, t_2))$ as 
\[
\int_{\Rb}\widehat F(\mu) \int_{\Rb\times \Rb} e^{-2\pi i \lambda_1 t_1}e^{-2\pi i \lambda_2 t_2}\chi(t_1) \chi(t_2) e^{2\pi i \max(t_1, t_2)} \mathrm{d}t_1\mathrm{d}t_2\mathrm{d}\mu.
\]
Splitting the integral into $t_1\leq t_2$ and $t_1\geq t_2$ and integrating by parts once in $t_1$, we get 
$$
\widehat G(\lambda_1, \lambda_2)=-\int_{|\mu-\lambda_1|\geq 1}\widehat F(\mu) \widehat{(\chi^2)}(\lambda_2+\lambda_1-\mu)\left(\frac{1}{\lambda_1-\mu}-\frac{1}{\lambda_1}\right) \mathrm{d}\mu + \textrm{Rem}
$$
where the remainder term Rem can be easily seen (possibly by one more additional integration by parts) to be in $X^\alpha$. Using the Schwartz decay of $\widehat{(\chi^2)}(\lambda_2+\lambda_1-\mu)$, and the fact that 
$$
\int_{|\lambda_1|, |\lambda_1 -\mu|\geq 1} |\lambda|^\beta \frac{|\mu|}{|(\lambda_1-\mu) \lambda_1|} d\lambda_1 \lesssim |\mu|^\beta \log |\mu|, \qquad 0\leq \beta<1,
$$
gives the result.
\end{proof}

\section{The algorithm for large gap molecules}\label{appalg}
\subsection{The setup} In this appendix we define and analyze the operations and algorithm that occur in the proof of Proposition \ref{moleprop2}, which are almost the same as those in Sections 9.3--9.4 of \cite{DH21}, except for some minor differences (for example we are omitting the degenerate atoms, etc.).

For each operation we will consider the corresponding $\Df$ value. There will be three cases, which we denote by \emph{normal} (N), \emph{fine} (F) and \emph{good} (G) operations, where we have $\Df\lesssim 1$ (or $\Df\lesssim(\log L)^C$), $\Df\lesssim L^{-\eta/2}$ and $\Df\lesssim L^{-5\gamma_0/8}$ respectively. Recall that for any decoration under consideration, each atom is assumed to be LG, apart from at most $w$ atoms; all the arguments below are based on this LG assumption, and the exceptional atoms will be discussed separately in the proof of Proposition \ref{moleprop2} (see Section \ref{mole2proof}). Note also that we do not require $c_v=0$ for degree $4$ atoms $v$ as in Definition \ref{decmole}, but this plays no role in the proof. In consistence with Sections 9.3--9.4 of \cite{DH21}, in the end of the operations we will reduce $\Mb$ to a molecule of isolated atoms only (and no bonds).

For each operation and each decoration of the molecule $\Mb_{\mathrm{pre}}$ before the operation, we will define a corresponding decoration of the molecule $\Mb_{\mathrm{pos}}$ after the operation (either obviously, or in a precise way we will describe below). In some cases the exact operation we perform will depend on some specific assumptions for the decoration (of form $|k_{\ell}-k_{\ell'}|\geq L^{-\gamma+2\gamma_0/3}$ or $|k_{\ell}-k_{\ell'}|<L^{-\gamma+2\gamma_0/3}$, corresponding to the extra conditions $\mathtt{Ext}$ described in Section 9.3 of \cite{DH21}). Finally, define as in \cite{DH21} that $\nu:= V_3+2V_2+3V_1+4V_0-4F=4V-2E-4F$, where $V_j$ is the number of degree $j$ atoms.
\subsection{The operations}\label{oper} We now define all the different operations.
\subsubsection{Triple bonds} Recall that the molecule $\Mb_{\mathrm{pre}}$ does not contain any degenerate atom (i.e. $k_{\ell}\neq k_{\ell'}$ in the decoration for any two bonds $(\ell,\ell')$ of opposite direction at one atom $v$, see Remark \ref{nonresrem}), and does not contain any self-connecting bonds.

In this operation, assume there is a triple bond between two atoms $v_1$ and $v_2$ in $\Mb_{\mathrm{pre}}$, such that $d(v_1)$ and $d(v_2)$ are not both $4$. In (TB-1N) we assume $d(v_1)=d(v_2)=3$, so the triple bond is separated from the rest of the molecule; in (TB-2N) we assume $d(v_1)=3$ and $d(v_2)=4$, so $v_2$ has an extra single bond.
\begin{itemize}
\item Operations (TB-1N)--(TB-2N): we remove atoms $v_1$, $v_2$ and all the bonds.
\end{itemize}
\begin{prop}\label{tbprop} We have $\Df\lesssim (\log L)^C$ for operations (TB-1N)--(TB-2N).
\end{prop}
\begin{proof} This follows from Lemma \ref{basiccount} (2), noticing that $\Delta\chi=-2$; the factor $(\log L)^C$ is due to the possible change it may cause to the $\Pf$ factor in $\Af$ (same below).
\end{proof}
\subsubsection{Bridge removal}\label{secbr} In all subsequent operations, we assume $\Mb_{\mathrm{pre}}$ has no triple bonds. In this operation, we assume $\Mb_{\mathrm{pre}}$ contains a \emph{bridge} $\ell$, which is a single bond connecting atoms $v_1$ and $v_2$, such that removing this bond will create a new component.
\begin{itemize}
\item Operation (BR-N): we remove the bond $\ell$.
\end{itemize}
\begin{prop}\label{brprop} We have $\Df\lesssim (\log L)^C$ for operation (BR-N). We also have $\Delta\nu=2$ and $\Delta V_3\geq -2$, with equality holding only when $d(v_1)=d(v_2)=3$.
\end{prop}
\begin{proof} The bound $\Df\lesssim (\log L)^C$ follows from Lemma 9.14 of \cite{DH21}. The effect of (BR-N) reduces the degrees of two atoms each by $1$, and adds one new component. by definition of $\nu$ we have $\Delta\nu=2-4=-2$, because the contribution to $V_3 +2V_2 +3V_1 +4V_0$ of each of the two atoms connected by $\ell$ increased by $1$ after the removal of $\ell$. The other statements are obvious.
\end{proof}
\subsubsection{Degree 3 atoms connected by a single bond}\label{3s3} In all subsequent operations, we assume there is no bridge in $\Mb_{\mathrm{pre}}$. In this operation, we assume that there are two degree 3 atoms $v_1$ and $v_2$, connected by a single bond $\ell_1$. Then $\Mb_{\mathrm{pre}}$ must contain one of the atomic groups shown in Figures \ref{fig:3s3}--\ref{fig:3s3new}.

In operations (3S3-1N)--(3S3-4G) we assume that $v_1$ and $v_2$ each has two more single bonds $\ell_2,\ell_3$ and $\ell_4,\ell_5$, connecting to four different atoms $v_3,v_4$ and $v_5,v_6$, see Figure \ref{fig:3s3}; in operation (3S3-5G) we assume this does not hold, see Figure \ref{fig:3s3new}. In (3S3-1N)--(3S3-3G) we assume that (i) after removing $(v_1,v_2)$ and all the bonds, $(v_3,v_5)$ is in one new component, and $(v_4,v_6)$ is in the other new component, and that (ii) the bonds $\ell_2$ and $\ell_4$ have opposite directions (viewing from $(v_1,v_2)$), and the bonds $\ell_3$ and $\ell_5$ also have opposite directions. In (3S3-4G) we assume either (i) or (ii) is false. Moreover, in (3S3-1N) we assume that $d(v_3)=\cdots=d(v_6)=4$, and in (3S3-3G) we assume that $d(v_3)$ and $d(v_5)$ are not both 4.
  \begin{figure}[h!]
  \includegraphics[scale=.5]{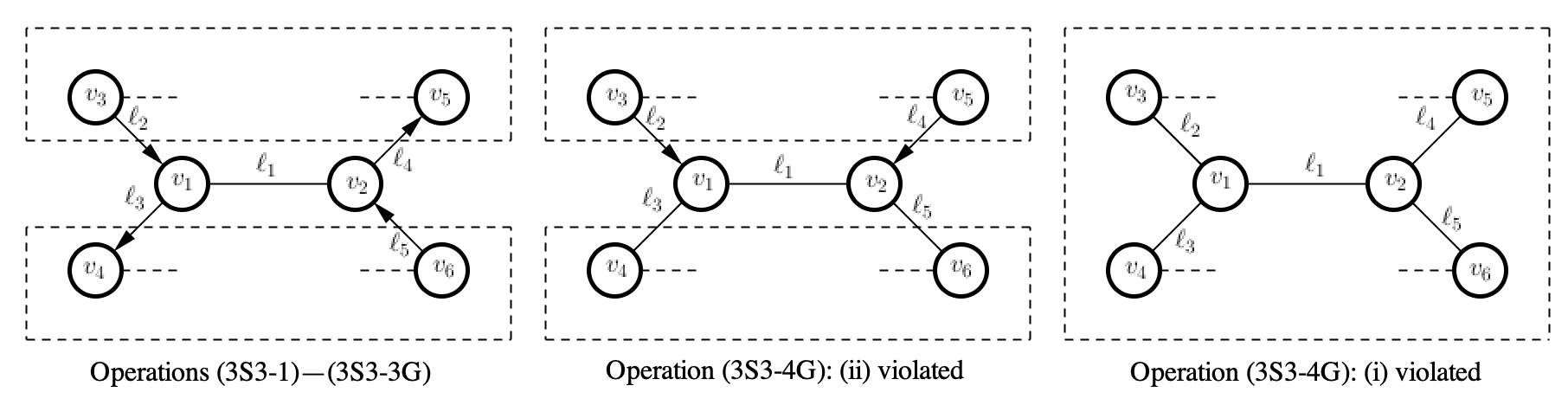}
  \caption{The atomic group involved in operations (3S3-1N)--(3S3-4G). In the first two pictures $(v_3,v_5)$ are $(v_4,v_6)$ are not in the same component after removing $v_1$ and $v_2$, while in the third picture they are.}
  \label{fig:3s3}
\end{figure} 
  \begin{figure}[h!]
  \includegraphics[scale=.5]{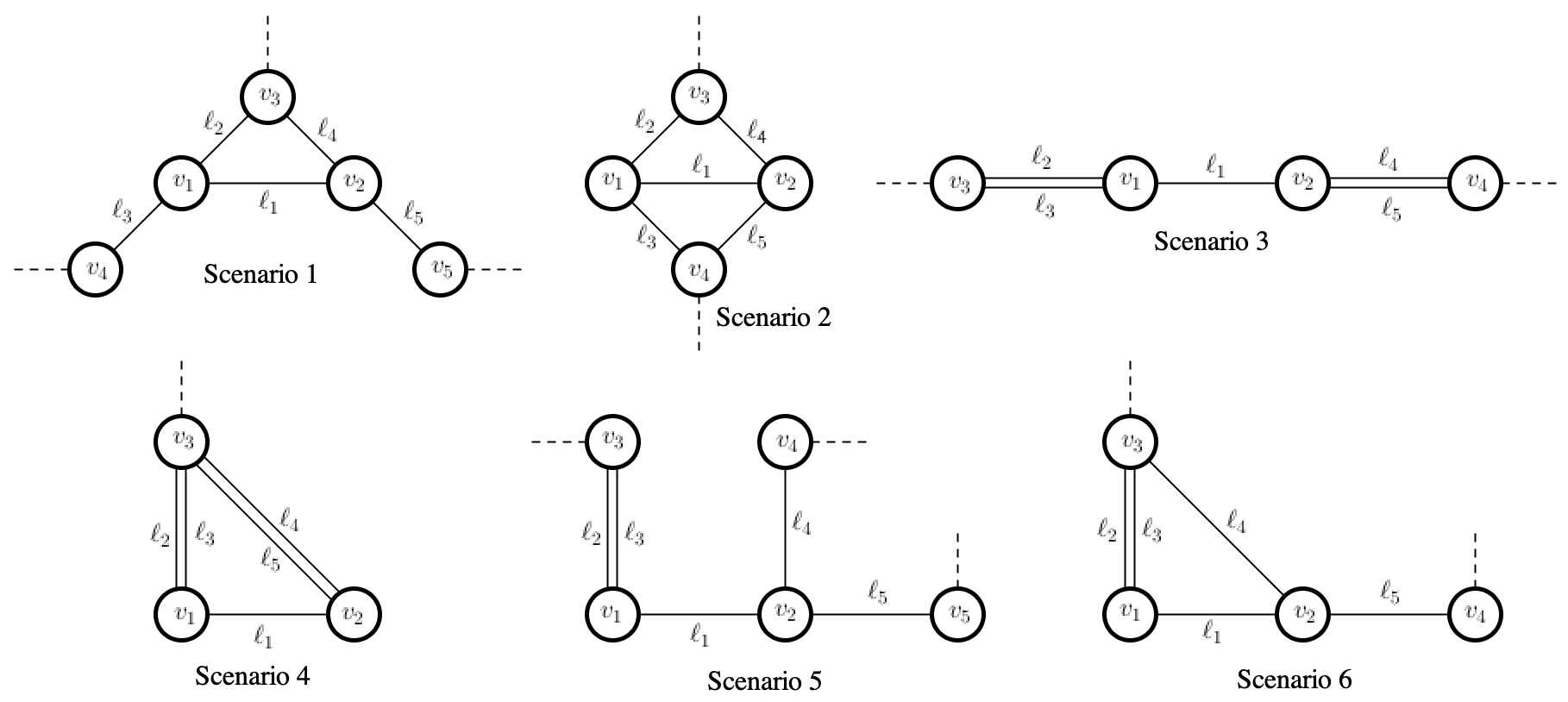}
  \caption{The atomic groups involved in operation (3S3-5G). In total there are 6 scenarios.}
  \label{fig:3s3new}
\end{figure} 
\begin{itemize}
\item Operation (3S3-1N): assuming that $|k_{\ell_2}-k_{\ell_4}|+|k_{\ell_3}-k_{\ell_5}|\leq L^{-\gamma}$, and $|k_{\ell_1}-k_{\ell_3}|\geq L^{-\gamma+2\gamma_0/3}$ if $(\ell_1,\ell_3)$ have opposite directions viewing from $v_1$, we remove the atoms $(v_1,v_2)$ and all the bonds.
\item Operation (3S3-2G): assuming the negation of the conditions in (3S3-1N), we remove $(v_1,v_2)$ and all the bonds.
\item Operation (3S3-3G): assuming the conditions in (3S3-1N), we remove $(v_1,v_2)$ and all the bonds, but add a new bond $\ell_6$ between $v_3$ and $v_5$ (not drawn in Figure \ref{fig:3s3}), which goes from $v_3$ to $v_5$ if $\ell_2$ goes from $v_3$ to $v_1$ and vice versa.
\item Operations (3S3-4G)--(3S3-5G): we remove $(v_1,v_2)$ and all the bonds.
\end{itemize}
\begin{prop}\label{3s3prop} We have $\Df\lesssim (\log L)^C$ for operation (3S3-1N), and $\Df\lesssim L^{-5\gamma_0/8}$ for operations (3S3-2G)--(3S3-5G). For (3S3-1N) we also have $\Delta\nu=-2$ and $\Delta V_3=2$.
\end{prop}
\begin{proof} First consider the four operations other than (3S3-3G). In each scenario in Figures \ref{fig:3s3}--\ref{fig:3s3new}. If $\Delta\chi=-3$ then we have $\Df\lesssim L^{-5\gamma_0/8}$ using Lemma \ref{basiccount} (4) (the possible $(\log L)^C$ due to the $\Pf$ factor is also absorbed; same below). If $\Delta\chi=-2$, then Lemma 9.14 of \cite{DH21} implies that the values of $k_{\ell_2}-k_{\ell_4}$ and $k_{\ell_3}-k_{\ell_5}$ must be fixed. In this case Lemma \ref{basiccount} (3) implies $\Df\lesssim (\log L)^C$. Moreover, this bound can be improved to $\Df\lesssim L^{-5\gamma_0/8}$ if the bonds $(\ell_2,\ell_4)$ and $(\ell_3,\ell_5)$ do not both have opposite directions viewing from $(v_1,v_2)$, or if $|k_{\ell_2}-k_{\ell_4}|+|k_{\ell_3}-k_{\ell_5}|\geq L^{-\gamma}$ (if the latter happens the we have $|P|\gtrsim L^{-\gamma}$ in Lemma \ref{basiccount} (3)). The same improvement works if $|k_{\ell_1}-k_{\ell_3}|\leq L^{-\gamma+2\gamma_0/3}$ and $(\ell_1,\ell_3)$ have opposite directions, because in this case $k_{\ell_2}$ belongs to a fixed ball of radius $L^{-\gamma+2\gamma_0/3}$, so a simple variation of Lemma \ref{basiccount} (2) gives that the number of choices for $(k_{\ell_1},k_{\ell_2},k_{\ell_3})$ is at most
\[\delta^{-1}L^{2-2\gamma}L^{d-1}(L^{1-\gamma+2\gamma_0/3})^{d-1}\lesssim \delta^{-1}L^{2(d-\gamma)}L^{-2\gamma_0/3}.\] These observations are sufficient to prove the bounds for the four operations other than (3S3-3G) (the fact about $\Delta\nu$ and $\Delta V_3$ for (3S3-1N) is also clear), noticing also the LG assumption for $v_3$ in Scenarios 1--2 in Figure \ref{fig:3s3new}, and that we must have $\Delta\chi=-3$ in Scenarios 3--6 in Figure \ref{fig:3s3new}.

Now consider (3S3-3G), where $\Delta\chi=-1$. To get a decoration of $\Mb_{\mathrm{pos}}$ from that of $\Mb_{\mathrm{pre}}$, we simply define $k_{\ell_6}=k_{\ell_2}$ for the newly added bond $\ell_6$. Note that if the decoration for $\Mb_{\mathrm{pre}}$ is LG at each atom, the the gap at each atom for the resulting decoration for $\Mb_{\mathrm{pos}}$ is still at least $L^{-\gamma+\eta}-L^{-\gamma}$ since $|k_{\ell_2}-k_{\ell_4}|\leq L^{-\gamma}$. Therefore, even after at most $n\leq (\log L)^C$ iterations (where $n$ is the size of the molecule $\Mb$) this gap is still $\geq L^{-\gamma+\eta}/2$ which does not affect any estimate later. Moreover, we have $\Cf_{\mathrm{pre}}\lesssim \delta^{-1}L^{d-\gamma-2\gamma_0/3}\Cf_{\mathrm{pos}}$ by first looking at the decoration $(k_\ell)$ for $\ell$ in the component of $\Mb_{\mathrm{pos}}$ containing $\ell_6$, then looking at the two vectors $(k_{\ell_1},k_{\ell_3})$, and then looking at the decoration $(k_\ell)$ for $\ell$ in the component of $\Mb_{\mathrm{pos}}$ not containing $\ell_6$. Here note that, once $k_{\ell_6}=k_{\ell_2}$ is fixed, the number of choices for $(k_{\ell_1},k_{\ell_3})$ is at most
\[\delta^{-1}L^{d-\gamma-\gamma_0}+\delta^{-1}L^{d-2\gamma}(L^{-\gamma+2\gamma_0/3})^{-1}\lesssim \delta^{-1}L^{d-\gamma}L^{-2\gamma_0/3}\] due to Lemma \ref{basiccount} (1), where we have $R\gtrsim L^{-\gamma+2\gamma_0/3}$ in Lemma \ref{basiccount} (1) when applicable. Taking into account of the possible changes to the $\Pf$ factor, we get that $\Df\lesssim L^{-5\gamma_0/8}$, as desired.
\end{proof}
\subsubsection{Degree 3 atoms connected by a double bond} In this operation, we assume there are two degree $3$ atoms $v_1$ and $v_2$, connected by a double bond $(\ell_1,\ell_2)$, which are also connected to two other atoms $v_3$ and $v_4$ by two single bonds $\ell_3$ and $\ell_4$, see Figures \ref{fig:3d3} and \ref{fig:3d3new}. In (3D3-1N)--(3D3-3G) and (3D3-6G) we assume $v_3\neq v_4$ and $\ell_3$ and $\ell_4$ are in opposite directions (viewing from $(v_1,v_2)$); in (3D3-1N) we assume $d(v_3)=d(v_4)=4$, and in (3D3-3G) we assume that \emph{not} all atoms in the current component other than $(v_1,v_2)$ have degree 4. In (3D3-4G) we assume $v_3\neq v_4$ and $\ell_3$ and $\ell_4$ are in the same direction, and in (3D3-5G) we assume $v_3=v_4$. Finally, in (3D3-6G) we assume that $v_3$ is connected to $v_4$ via a single bond $\ell_5$, and $v_3$ and $v_4$ are each connected to different atoms $v_5$ and $v_6$ via double bonds $(\ell_6,\ell_7)$ and $(\ell_8,\ell_9)$, see Figure \ref{fig:3d3new}.
  \begin{figure}[h!]
  \includegraphics[scale=.5]{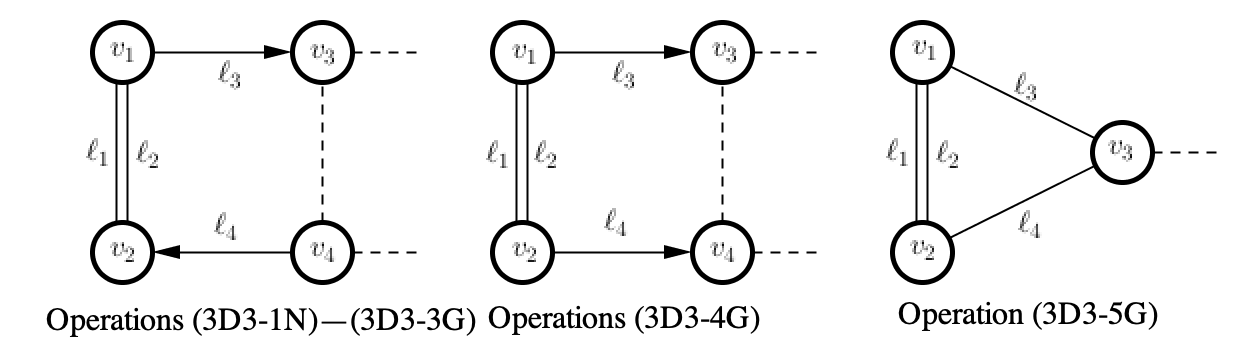}
  \caption{The atomic groups involved in operations (3D3-1N)--(3D3-5G).}
  \label{fig:3d3}
\end{figure} 
  \begin{figure}[h!]
  \includegraphics[scale=.5]{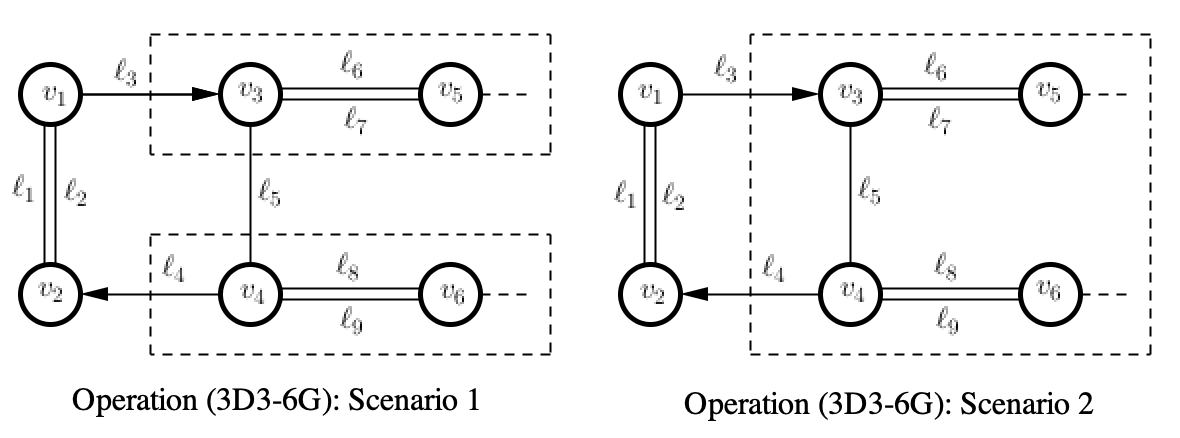}
  \caption{The atomic group involved in operation (3D3-6G). In the left picture $\ell_5$ becomes a bridge after removing $\{v_1,v_2\}$, while in the right picture it does not.}
  \label{fig:3d3new}
\end{figure} 
\begin{itemize}
\item Operation (3D3-1N): assuming that $|k_{\ell_3}-k_{\ell_4}|\leq L^{-\gamma}$, and $|k_{\ell_1}-k_{\ell_2}|\geq L^{-\gamma+2\gamma_0/3}$ if $(\ell_1,\ell_2)$ have opposite directions, we remove the atoms $(v_1,v_2)$ and all bonds.
\item Step (3D3-2G): assuming the negation of the conditions in (3D3-1N), we remove $(v_1,v_2)$ and all the bonds.
\item Step (3D3-3G): assuming the conditions in (3D3-1N), we remove $(v_1,v_2)$ and all the bonds, but add a new bond $\ell_5$ between $v_3$ and $v_4$ (not drawn in Figure \ref{fig:3d3}), which goes from $v_4$ to $v_3$ if $\ell_3$ goes from $v_1$ to $v_3$ and vice versa.
\item Steps (3D3-4G)--(3D3-5G): we remove the atoms $(v_1,v_2)$ and all the bonds.
\item Step (3D3-6G): we remove the atoms $(v_1,\cdots,v_4)$ and all the bonds.
\end{itemize}
\begin{prop}\label{3d3prop} For operation (3D3-1N) we have $\Df\lesssim 1$ and $\Delta\nu=\Delta V_3=0$. For operations (3D3-2G)--(3D3-6G) we have $\Df\lesssim L^{-5\gamma_0/8}$. Note the absence of $(\log L)^C$ loss for (3D3-1N).
\end{prop}
\begin{proof} First consider the four operations other than (3D3-1N) and (3D3-3G). In each scenario in Figures \ref{fig:3d3}--\ref{fig:3d3new}, except for (3D3-6G), we always have $\Delta\chi=-2$, so Lemma \ref{basiccount} (3) implies $\Df\lesssim L^{-5\gamma_0/8}$ if the bonds $(\ell_3,\ell_4)$ have the same direction viewing from $(v_1,v_2)$, or if $|k_{\ell_3}-k_{\ell_4}|\geq L^{-\gamma}$, or if $|k_{\ell_1}-k_{\ell_2}|\leq L^{-\gamma+2\gamma_0/3}$ and $(\ell_1,\ell_2)$ have opposite directions, in the same way as in the proof of Proposition \ref{3s3prop} above. These observations are sufficient to prove $\Df\lesssim L^{-5\gamma_0/8}$ for (3D3-2G), (3D3-4G) and (3D3-5G), using also the LG assumption for $v_3$ in (3D3-5G). As for (3D3-6G) we have $\Delta\chi\in\{-4,-5\}$; if $\Delta\chi=-5$ the result follows from applying Lemma \ref{basiccount} (3) for $(k_{\ell_1},\cdots k_{\ell_4})$ and then applying Lemma \ref{basiccount} (4) for $(k_{\ell_5},\cdots k_{\ell_9})$. If $\Delta\chi=-4$, then by Lemma 9.14 of \cite{DH21} we know that  $k_{\ell_3}\pm k_{\ell_5}$ is fixed and $|k_{\ell_3}|^2\pm|k_{\ell_5}|^2$ is fixed up to distance $O(n\delta^{-1}L^{-2\gamma})$, with a suitable choice of $\pm$ (note also that $n\lesssim (\log L)^C$). We then apply Lemma \ref{basiccount} (3) for $(k_{\ell_1},k_{\ell_2},k_{\ell_3},k_{\ell_5})$ using also the LG assumption for $v_3$, and then apply Lemma \ref{basiccount} (1) for $(k_{\ell_6},k_{\ell_7})$ and $(k_{\ell_8},k_{\ell_9)}$ to get $\Df\lesssim L^{-5\gamma_0/8}$.

Next consider (3D3-1N). The fact $\Delta\nu=\Delta V_3=0$ is clear. Moreover this operation may not affect any ladder of length $\geq1$, or it may reduce the length of one such ladder by one; in the latter case it removes a factor $\min((\log L)^2,1+\delta L^{2\gamma}P)$ from the product $\Pf$, where $P\sim |k_{\ell_3}-k_{\ell_4}|$. In either case, since $\Delta\chi=-2$, by Lemma \ref{basiccount} (3) we have $\Df\lesssim 1$ in the same way as (\ref{laddersharp}).

Finally consider (3D3-3G), where $\Delta\chi=-1$. To get a decoration of $\Mb_{\mathrm{pos}}$ from that of $\Mb_{\mathrm{pre}}$, we simply define $k_{\ell_5}=k_{\ell_3}$ for the newly added bond $\ell_5$. This does not affect the LG assumptions, in the same way as in the proof of Proposition \ref{3s3prop} above. Moreover, we have $\Cf_{\mathrm{pre}}\lesssim\delta^{-1}L^{d-\gamma-2\gamma_0/3}$ and consequently $\Df\lesssim L^{-5\gamma_0/8}$, also in the same way as in the proof of Proposition \ref{3s3prop}, using the assumption $|k_{\ell_1}-k_{\ell_2}|\geq L^{-\gamma+2\gamma_0/3}$ if $(\ell_1,\ell_2)$ have the same direction.
\end{proof}
\subsubsection{Degree 3 and 4 atoms connected by a double bond} In this operation, we assume there is an atom $v_1$ of degree 3, and another atom $v_2$ of degree 4, that are connected by a double bond $(\ell_1,\ell_2)$. Then $\Mb_{\mathrm{pre}}$ must contain one of the atomic groups shown in Figure \ref{fig:3d4}.
  \begin{figure}[h!]
  \includegraphics[scale=.5]{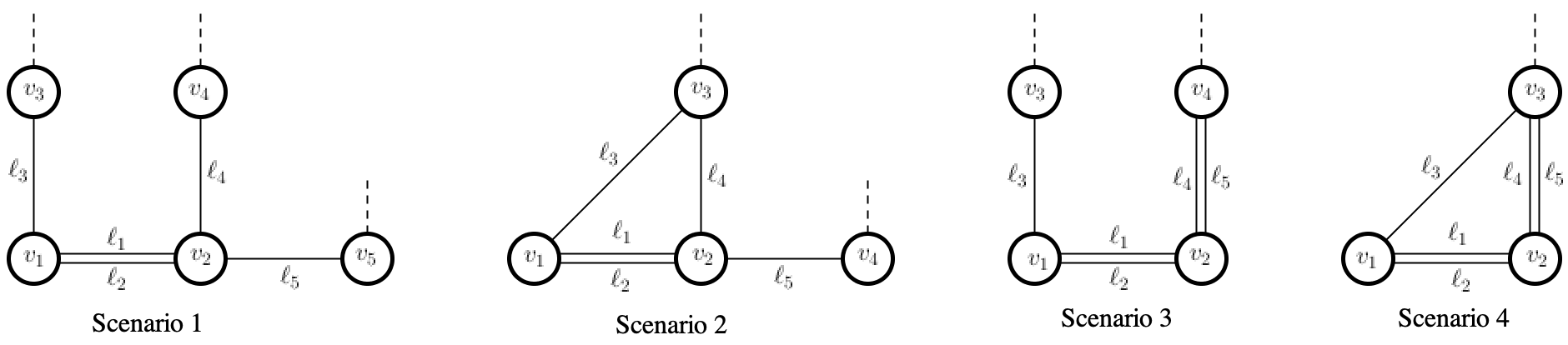}
  \caption{The atomic groups involved in step (3D4G). In total there are 4 scenarios.}
  \label{fig:3d4}
\end{figure} 
\begin{itemize}
\item Operation (3D4-G): we remove the atoms $(v_1,v_2)$ and all the bonds.
\end{itemize}
\begin{prop}\label{3d4prop} We have $\Df\lesssim L^{-5\gamma_0/8}$ for operation (3D4-G).
\end{prop}
\begin{proof} In each case in Figure \ref{fig:3d4} we have $\Delta\chi=-3$, so the result follows from Lemma \ref{basiccount} (4).
\end{proof}
\subsubsection{Degree 3 and 2 atoms connected} In this operation, we assume there is an atom $v_1$ of degree 3, and another atom $v_2$ of degree 2, that are connected. Note that they must be connected by a single bond $\ell_1$, otherwise there would be a bridge. Then, $\Mb_{\mathrm{pre}}$ must contain one of the atomic groups shown in Figure \ref{fig:3s2}.
  \begin{figure}[h!]
  \includegraphics[scale=.5]{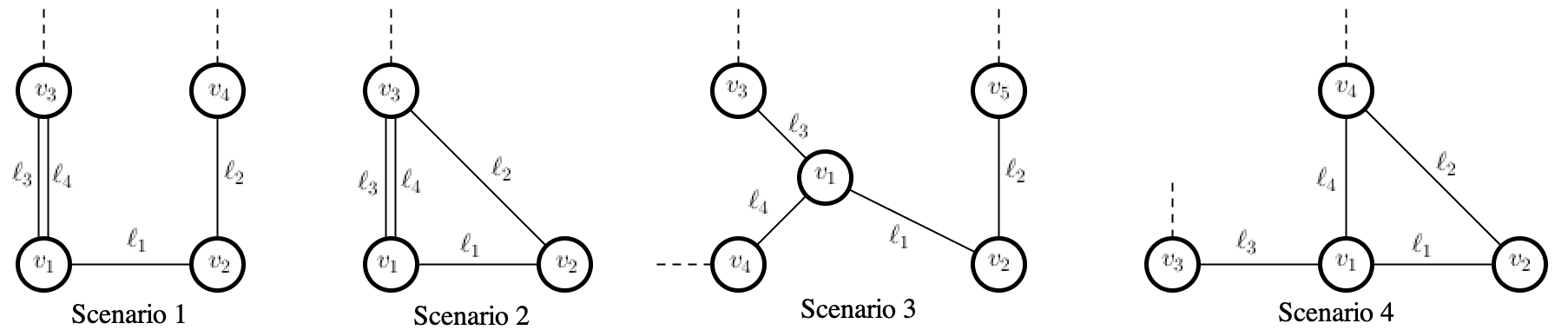}
  \caption{The atomic groups involved in step (3S2-G). In total there are 4 scenarios.}
  \label{fig:3s2}
\end{figure} 
\begin{itemize}
\item Step (3S2-G): we remove the atoms $(v_1,v_2)$ and all the bonds.
\end{itemize}
\begin{prop}\label{3s2prop} We have $\Df\lesssim L^{-5\gamma_0/8}$ for operation (3S2-G).
\end{prop}
\begin{proof} In each case we have $\Delta\chi=-2$, so the result follows from Lemma \ref{basiccount} (3), noting also that $|P|\gtrsim L^{-\gamma+\eta}$ if $\ell_1$ and $\ell_2$ have opposite directions, due to the LG assumption.
\end{proof}
\subsubsection{Degree 3 atom removal}\label{3r} In this operation, we assume there is an atom $v$ of degree 3, which is connected to three atoms $v_j\,(1\leq j\leq 3)$ of degree 4, by three single bonds $\ell_j\,(1\leq j\leq 3)$. In step (3R-2G) we further assume that, there is a \emph{special} bond $\ell_1'$ (i.e. a single bond connecting two degree $3$ atoms, such that they have two double bonds connecting to two different atoms) in the molecule (or component) after removing the atom $v$ and the bonds $\ell_j$. In this case, suppose $\ell_1'$ connects atoms $v_1'$ and $v_2'$, $v_1'$ is connected to $v_3'$ by a double bond $(\ell_2',\ell_3')$, and $v_2'$ is connected to $v_4'$ by a double bond $(\ell_4',\ell_5')$, see Figure \ref{fig:3rproof}.
  \begin{figure}[h!]
  \includegraphics[scale=.5]{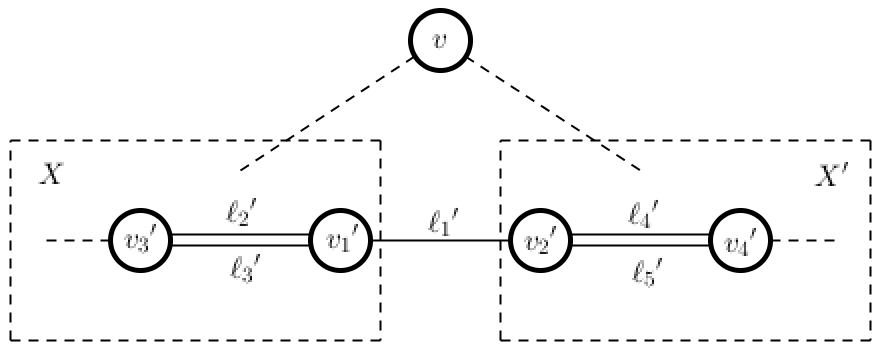}
  \caption{The functional group involved in operation (3R-2G). Here $v_1,v_2,v_3$ are not drawn; some of them may coincide with some $v_j'$. Also we only draw the scenario where $\ell_1'$ becomes a bridge after removing $v$, but the other scenario is also possible.}
  \label{fig:3rproof}
\end{figure} 
\begin{itemize}
\item Operation (3R-1N): we remove the atom $v$ and all the bonds.
\item Operation (3R-2G): we we remove the atoms $(v,v_1',v_2')$ and all the bonds.
\end{itemize}
\begin{prop}\label{3rprop} We have $\Ds\lesssim(\log L)^C$ for operation (3R-1N), and also $\Delta\nu=\Delta V_3=2$. For operation (3R-2G) we have $\Df\lesssim L^{-5\gamma_0/8}$.
\end{prop}
\begin{proof} For (3R-1N) we have $\Delta\chi=-2$, so the result follows from Lemma \ref{basiccount} (2) and simple calculations. As for (3R-2G), we have $\Delta\chi\in\{-4,-5\}$. If $\Delta\chi=-5$, then the result follows from applying Lemma \ref{basiccount} (2) for $(k_{\ell_1},k_{\ell_2},k_{\ell_3})$ and then applying Lemma \ref{basiccount} (4) for $(k_{\ell_1'},\cdots,k_{\ell_5'})$.

If $\Delta\chi=-4$, this means that $\ell_1'$ becomes a bridge after removing $v$, see Figure \ref{fig:3rproof}. Since $\ell_1'$ is not a bridge in $\Mb_{\mathrm{pre}}$, we know $v$ must have at least one bond connecting to each of the two components after removing $v$ and $\ell_1'$. Without loss of generality, assume $v$ has only one bond, say $\ell_1$, connecting to an atom $v_1$ in $X$ (the component containing $(v_1',v_3')$), then by Lemma 9.14 of \cite{DH21} we know that  $k_{\ell_1}\pm k_{\ell_1'}$ is fixed and $|k_{\ell_1}|^2\pm|k_{\ell_1'}|^2$ is fixed up to distance $O(n\delta^{-1}L^{-2\gamma})$, with a suitable choice of $\pm$ (note also that $n\lesssim (\log L)^C$). Therefore, we can apply Lemma \ref{basiccount} (4) for $(k_{\ell_1},k_{\ell_2},k_{\ell_3},k_{\ell_2'},k_{\ell_3'})$ and then apply Lemma \ref{basiccount} (1) for $(k_{\ell_4'},k_{\ell_5'})$ (using also the LG assumption) to get $\Df\lesssim L^{-5\gamma_0/8}$.
\end{proof}
\subsubsection{Degree 2 atom removal} In this operation, we assume there is an atom $v$ of degree 2, connected to one or two atom(s) of degree 2 or 4.
\begin{itemize}
\item Operation (2R-1F): assuming $v$ is connected to a degree 4 atom by a double bond, we remove the atom $v$ and all the bonds.
\item Operation (2R-2F): assuming $v$ is connected to a degree 4 atom by a single bond, and also connected to another atom of degree 2 or 4 by a single bond, we remove the atom $v$ and all the bonds.
\item Operation (2R-3F): assuming $v$ is connected to two degree 2 atoms $v_1$ and $v_2$ by two single bonds, such that neither $v_1$ nor $v_2$ is connected to a degree 3 atom, we remove the atoms $(v, v_1,v_2)$ and all the bonds.
\item Operation (2R-4F): assuming $v$ is connected to a degree 2 atom $v'$ by a double bond, we remove the atoms $(v,v')$ and tall bonds.
\end{itemize}
\begin{prop}\label{2rprop} We have $\Df\lesssim L^{-\eta/2}$ for operations (2R-1F)--(2R-4F). For (2R-1F) and (2R-4F) we have $\Delta V_3=\Delta \nu=0$, for (2R-2F) we have $\Delta\nu=0$ and $\Delta V_3\geq 1$, and for (2R-3F) we have $\Delta V_3\geq 0$ and $\Delta\nu\leq -2$.
\end{prop}
\begin{proof} In each case we have $\Delta\chi=-1$, so the bond for $\Df$ follows from Lemma \ref{basiccount} (1) using also the LG asumption. The other statements follow from simple calculations.
\end{proof}
\subsection{The algorithm}\label{alg}  The algorithm is described as a big loop. Once we enter the loop, we shall follow a set of rules to choose the next operation depending on the current molecule $\Mb$ and possible assumptions made on the decoration. In some cases we may also choose a sequence of successive operations, again following a specific set of rules, until we are done with this execution of the loop and return to the start of the loop. The loop ends when $\Mb$ contains only isolated atoms.
\subsubsection{Description of the loop}\label{loop} The loop is described as follows. Start with a molecule $\Mb$.
\begin{enumerate}
\item If $\Mb$ contains a bridge, then remove it using (BR-N). Go to (1).
\item Otherwise, if $\Mb$ contains two degree 3 atoms $v_1$ and $v_2$ connected by a single bond $\ell_1$, then:
\begin{enumerate}
\item If $\Mb$ contains one atomic group in Figure \ref{fig:3s3new}, then preform (3S3-5G). Go to (1).
\item Otherwise, $\Mb$ contains the atomic group in Figure \ref{fig:3s3}. If it satisfies (i) and (ii) in Section \ref{3s3}, and $d(v_3)=\cdots =d(v_6)=4$, then we perform (3S3-1N) or (3S3-2G), depending on whether the conditions in (3S3-1N) are met. Go to (1).
\item If it satisfies (i) and (ii) in Section \ref{3s3}, but (say) $d(v_3)$ and $d(v_5)$ are not both 4, then we perform (3S3-3G) or (3S3-2G), depending on whether the conditions in (3S3-1N) are met. If after (3S3-3G) a triple bond forms between $v_3$ and $v_5$, immediately remove it by (TB-1N)--(TB-2N). Go to (1).
\item If either (i) or (ii) in Section \ref{3s3} is violated, then we perform (3S3-4G). Go to (1).
\end{enumerate}
\item Otherwise, if $\Mb$ contains two degree 3 atoms $v_1$ and $v_2$ connected by a double bond $(\ell_1,\ell_2)$, then:
\begin{enumerate}
\item If $\Mb$ contains the atomic group in Figure \ref{fig:3d3} corresponding to (3D3-4G) or (3D3-5G), then we perform the corresponding step. Go to (1).
\item Otherwise, $\Mb$ contains the atomic group in Figure \ref{fig:3d3} corresponding to (3D3-1N)--(3D3-3G). This can be seen as the start of a ladder. Now, if and while this ladder \emph{continues} (i.e. $v_3$ and $v_4$ are connected by a double bond, and they are connected to two different atoms $v_5$ and $v_6$ by two single bonds of opposite directions viewing form $(v_3,v_4)$), we perform (3D3-1N) or (3D3-2G), depending on whether the conditions in (3D3-1N) are met. Proceed with (c) below.
\item Now assume the ladder does not continue. Then:
\begin{enumerate}
\item If $v_3$ and $v_4$ are like in Figure \ref{fig:3d3new}, then perform (3D3-6G); if $v_3$ and $v_4$ are like (3D3-5G) in Figure \ref{fig:3d3} after removing $(v_1,v_2)$, then perform (3D3-1N) followed by (3D3-5G). Go to (1).
\item Otherwise, if not all atoms in the current component other than $(v_1,v_2)$ have degree 4, we perform (3D3-3G) or (3D3-2G), depending on whether the conditions in (3D3-1N) are met. If after (3D3-3G) a triple bond forms between $v_3$ and $v_4$, immediately remove it by (TB-1N)--(TB-2N). Go to (1).
\item Otherwise, we preform (3D3-1N). Go to (1) but scan within this component.
\end{enumerate}
\end{enumerate}
\item Otherwise, if $\Mb$ contains a degree 3 atom $v_1$ connected to a degree 4 atom $v_2$ by a double bond $(\ell_1,\ell_2)$, then we have one atomic group in Figure \ref{fig:3d4}. We perform (3D4G). Go to (1).
\item Otherwise, if $\Mb$ contains a degree 3 atom $v_1$ connected to a degree 2 atom $v_2$, then we have one atomic group in Figure \ref{fig:3s2}. We perform (3S2G). Go to (1).
\item Otherwise, if $\Mb$ contains a degree 3 atom $v$, then $v$ must be connected to three degree 4 atoms $v_j\,(1\leq j\leq 3)$ by three single bonds $\ell_j\,(1\leq j\leq 3)$. Then:
\begin{enumerate}
\item If the component after removing $v$ and $\ell_j$ contains a special bond (see Section \ref{3r}), then we perform (3R-2G). Go to (1).
\item Otherwise, we perform (3R-1N). Go to (1).
\end{enumerate}
\item Otherwise, $\Mb$ only contains atoms of degree (0 and) 2 and 4. Then, we find an atom of degree $2$, which must be in one of the cases corresponding to operations (2R-1F)--(2R-4F); perform the corresponding operation. Go to (1).
\end{enumerate}

We make a few remarks about the algorithm, which is useful in the proof of Proposition \ref{moleprop2}.

\begin{enumerate}[{1.}]
\item There is no triple bond when we perform any operation other than (TB-1N)--(TB-2N). This is because only operations (3S3-3G) and (3D3-3G) may create triple bonds, but they are immediately removed using (TB-1N)--(TB-2N), as in (2-c) and (3-c-i). Similarly there is no bridge when we perform any operation other than (TB-1N)--(TB-2N) or (BR-N), because operation (BR-N) has the top priority; moreover, if we are in (3-b), i.e. the ladder continues, then the operations (3D3-1N) and (3D3-2G) cause the same change on $\Mb$, and this change does not create any bridge.
\item In (3-c-ii), after (3D3-3G), it cannot happen that $v_3$ and $v_4$ are connected by a triple bond, and $d(v_3)=d(v_4)=4$. In fact, in this case the two extra single bonds $\ell_1'$ and $\ell_2'$ from $v_3$ and $v_4$ must have opposite directions, so either the ladder continues, or the bonds $\ell_1'$ and $\ell_2'$ shares a common atom. This means we are in either (3-b) or (3-c-i) before performing (3D3-3G), which is impossible.
\item When executing a ``Go to" sentence, we may proceed to scan the whole molecule for the relevant structures, except in (3-c-iii), where we only scan \emph{the current component}. Note that after performing (3D3-1N) in (3-c-iii), $v_3$ and $v_4$ will have degree 3, and all other atoms in the current component will have degree 4. Therefore the next operation(s) we perform in this component, following our algorithm, may be (BR-N), (3S3-1N)--(3S3-5G), (3D3-4G)--(3D3-5G), (3D4G), (3R-1N)--(3R-2G), possibly accompanied by (TB-1N)--(TB-2N), but \emph{cannot} be (3D3-1N)--(3D3-3G) because the ladder does not continue.
\item In the whole process we never have a saturated component (which can be verified for (3S3-3G) and (3D3-3G) by our algorithm, and is obvious for the other operations), thus in (7) there must be at least one degree 2 atom (unless there are only isolated atoms, in which case the loop ends; note that we are also not considering degree 1 atoms, as those imply the existence of bridges).
\end{enumerate}

\end{document}